\documentclass{amsbook}

\newtheorem{thm}{Theorem}[chapter]
\newtheorem{lem}[thm]{Lemma}

\theoremstyle{definition}

\newtheorem{eg}{Example}[chapter]
\newtheorem*{notation}{Notation}

\theoremstyle{rmk}
\newtheorem{rmk}{Remark}[chapter]

\theoremstyle{Proposition}

\newtheorem{conj}[thm]{Conjecture}
\newtheorem{cor}[thm]{Corollary}

\newtheorem{prop}[thm]{Proposition}

\newtheorem*{Question}{Question}

\numberwithin{section}{chapter}
\numberwithin{equation}{chapter}

\numberwithin{section}{chapter}
\numberwithin{equation}{chapter}

\usepackage[english]{babel}

\usepackage{amsmath}
\usepackage{graphicx}
\usepackage{amssymb}
\usepackage{amsthm}
\usepackage{tikz-cd}
\usepackage{mathrsfs}
\usepackage{enumitem}
\usepackage{yfonts}
\usepackage{MnSymbol}
\usepackage{slashed}

\date{\today}

\newcommand{\cf}{\emph{cf.} }

\newcommand{\R}{\mathbb{R}}
\newcommand{\C}{\mathbb{C}}
\newcommand{\Z}{\mathbb{Z}}
\newcommand{\N}{\mathbb{N}}

\newcommand{\norm}[1]{\left\lVert#1\right\rVert}
\newcommand{\Lap}{\Delta}

\DeclareMathOperator{\Tr}{Tr}

\begin{document}

\frontmatter

\title{SYZ geometry for Calabi-Yau 3-folds: Taub-NUT and Ooguri-Vafa type metrics}

\author{Yang Li}
\address{Department of Mathematics,
	Imperial College London,
	180 Queen's Gate, South Kensington
	London, SW7 2AZ,
	England}


\thanks{This work was supported by the Engineering and Physical Sciences Research
Council [EP/L015234/1], the EPSRC Centre for Doctoral Training in
Geometry and Number Theory (The London School of Geometry and Number
Theory), University College London. The author is funded by Imperial College London for his PhD studies.}


\maketitle	

\setcounter{page}{4}
\tableofcontents

\begin{abstract}
We construct a family of Calabi-Yau metrics on $\C^3$ with properties analogous to the Taub-NUT metric on $\C^2$, and construct a family of  Calabi-Yau 3-fold metric models on the positive and negative vertices of SYZ fibrations with properties analogous to the Ooguri-Vafa metric.
\end{abstract}

\mainmatter

\chapter{Introduction and Background Review}

The principal motivation of this paper lies in the metric aspect of the Strominger-Yau-Zaslow (SYZ) conjecture for 3-folds, a strong form of which states
\begin{conj}\cite{SYZ}\cite{KontsevichSoibelman}\label{SYZconjecture}
Let $(X_t, g_t)$ be a degenerating family of (polarized) Calabi-Yau 3-folds near the large complex structure limit, equipped with Calabi-Yau metrics $g_t$. Then for $|t|\ll 1$, the 3-fold $X_t$ admits a \emph{special Lagrangian $T^3$-fibration} over a base $B$ homeomorphic to a 3-sphere, known as the \emph{SYZ fibration}. The base $B$ is equipped with an \emph{affine structure} and a compatible metric $g_B$ solving the \emph{real Monge-Amp\`ere equation} with singularities along a \emph{trivalent graph} $\mathfrak{D}\subset B$, such that the rescaled Calabi-Yau metrics $
(X_t,  \text{diam}(g_t)^{-2} g_t)$ converge to $( B, g_B  )
$ in Gromov-Hausdorff sense as $t\to 0$. 
Morever, suitably away from $\mathfrak{D}$ the metrics $g_t$ are approximated by a \emph{semiflat metric} up to exponentially small errors.
\end{conj}

The SYZ conjecture stands at the crossroad of algebraic, symplectic, Riemannian and calibrated geometry. In the past two decades following the SYZ proposal there has emerged a sophisticated topological, algebro-geometric and symplectic picture \cite{KontsevichSoibelman}\cite{Gross4}. The topological and complex geometric description for SYZ fibrations developed by Gross, Ruan, Joyce and Zharkov will be recalled in Section \ref{GrossRuanJoyce}.

The prototype result in the metric direction is Gross and Wilson's description of the degenerating K3 metrics \cite{GrossWilson}.  Crucial in \cite{GrossWilson} is an explicit metric model for the neighbourhood of the  singular SYZ fibres, known as the \emph{Ooguri-Vafa metric},  constructed via the \emph{Gibbons-Hawking ansatz} (\cf review Section \ref{OoguriVafa}). After hyperk\"ahler rotation, the SYZ fibration turns into a holomorphic fibration by elliptic curves over $\mathbb{P}^1\simeq S^2$, where the fibres have much smaller diameters compared to the base, a phenomenon known as \emph{collapsing}. Gross and Wilson construct the collapsing K3 metrics by gluing the Ooguri-Vafa metric to a semiflat metric with exponentially small gluing error. Some features of their construction persist on higher dimensional hyperk\"ahler manifolds with holomorphic Lagrangian Abelian variety fibrations \cite{Tosatti2}. Variants of the Ooguri-Vafa metric feature prominently in other types of metric degenerations on K3 surfaces \cite{HeinSunViaclovsky}.

Beyond the hyperk\"ahler case, Zharkov et al \cite{Zharkov1}\cite{Zharkov2} established a formal differential geometric framework called the \emph{generalised Gibbons-Hawking ansatz} designed to construct Calabi-Yau metrics with torus symmetry, and pointed out its relevance to the metric aspects of the SYZ conjecture (\cf review Section \ref{GeneralisedGibbonsHawking} and \ref{Degeneratingtorichypersurface}). Despite all the progress, essentially no analytic result was known concerning the Calabi-Yau metric on a quintic 3-fold near the large complex structure limit, beyond its abstract existence due to Yau's solution of the Calabi conjecture.

The main accomplishment of this paper is to use the generalised Gibbons-Hawking ansatz to construct two types of analogues for the Ooguri-Vafa metrics on Calabi-Yau 3-folds, corresponding to the \emph{positive vertex} and the \emph{negative vertex}, referring to the neighbourhoods of the two types of most singular SYZ fibres in a generic 3-fold SYZ fibration according to the Gross-Ruan-Joyce classification. As a byproduct of our project, we construct a family of new exotic Calabi-Yau metrics on $\C^3$ with properties akin to the Taub-NUT metric on $\C^2$.

Here is a crude statement for the main results.  A fuller summary can be found in the introductions to individual Chapters.

\begin{thm}\label{TaubNUTtypeC3theoremintro}
(\textbf{Taub-NUT type metric on $\C^3$}, \cf Chapter \ref{TaubNUTtypemetriconC3}) There is a family of Calabi-Yau metrics  on $(\C^3_{z_0,z_1,z_2}, -dz_0\wedge dz_1\wedge dz_2)$ invariant under the diagonal $T^2$-action, which are  parametrised by positive definite rank 2 real symmetric matrices $(a_{ij})$. The base of the $T^2$-fibration is $\R^2_{\mu_1, \mu_2}\times \C_\eta$ and the discriminant locus is the trivalent graph
\[
\mathfrak{D}= \{ \mu_2\geq 0, \mu_1=0, \eta=0      \}\cup \{  \mu_1\geq 0, \mu_2=0, \eta=0       \}\cup \{  \mu_1=\mu_2\leq 0,  \eta=0          \}.
\]
The tangent cone at infinity is the Euclidean $\R^4$. Near spatial infinity suitably away from $\mathfrak{D}$, the metric is approximately a flat $T^2$ fibration over an open subset of Euclidean $\R^4$, such that the metric on the $T^2$-fibres is asymptotically given by the inverse matrix $(a^{ij})$ in distinguished coordinates. The metric transverse to $\mathfrak{D}$ is modelled on a fibration by Taub-NUT metrics.
\end{thm}

\begin{rmk}
Readers primarily interested in this Taub-NUT type metric on $\C^3$ can treat Chapter \ref{TaubNUTtypemetriconC3} as an indepenent paper, and refer to Section \ref{GeneralisedGibbonsHawking} for backgrounds.
\end{rmk}

\begin{thm}\label{Positivevertextheoremintro}
(\textbf{Ooguri-Vafa type metric on the positive vertex}, \cf Chapter \ref{Positivevertexchapter}) There is a family of incomplete Calabi-Yau metrics with $T^2$-symmetry, which are parametrised by positive definite rank 2 real symmetric matrices $(a_{ij})$, such that
\begin{itemize}
\item The ambient space has the same topology as the positive vertex predicted by Gross-Ruan-Joyce, namely it is a singular $T^2$-bundle over a 4-dimensional base contained in $\R^2_{\mu_1, \mu_2}\times (S^1\times \R)_\eta$ with discriminant locus along
\[
\mathfrak{D}= \{ \mu_2\geq 0, \mu_1=0, \eta=0      \}\cup \{  \mu_1\geq 0, \mu_2=0, \eta=0       \}\cup \{  \mu_1=\mu_2\leq 0,  \eta=0          \}.
\]
\item The holomorphic structure together with the holomorphic volume form agrees with the Zharkov prediction (\cf Section \ref{Degeneratingtorichypersurface}).

\item Suitably away from $\mathfrak{D}$ there is a $T^3$-fibration structure such that the Calabi-Yau metrics decay exponentially to semiflat metrics.

\item
These metrics extend over an exponentially large region, under the unit homological volume normalisation on $T^3$.

\item  Metric behaviour near the origin  is modelled on the Taub-NUT type metrics on $\C^3$ mentioned above. Metric behaviour transverse to $\mathfrak{D}$ but suitably away from the origin is modelled on a fibration by Taub-NUT metrics. Metric behaviour suitably away from $\mathfrak{D}$ is approximately a flat $T^2$-bundle over an open subset of 
$\R^2_{\mu_1, \mu_2}\times (S^1\times \R)_\eta$ with a Euclidean metric.
\item  These Calabi-Yau metrics admit special Lagrangian $T^3$-fibrations.

\end{itemize}
\end{thm}

\begin{thm}\label{Negativevertextheoremintro}
(\textbf{Ooguri-Vafa type metric on the negative vertex}, \cf Chapter \ref{NegativervertexChapter}) There is a family of incomplete Calabi-Yau metrics with $S^1$-symmetry, which are parametrised by rank 2 Hermitian matrices $(a_{p\bar{q}})$, such that
\begin{itemize}
\item The ambient space is topologically a singular $S^1$-bundle over a 5-dimensional base contained in $(\C^*)^2_{z_1,z_2}\times \R_\mu$, with discriminant locus along 
\[
S=\{  z_1+z_2=1, \mu=0  \}\subset (\C^*)^2_{z_1,z_2}\times \R_\mu.
\]
\item The holomorphic structure together with the holomorphic volume form agrees with the Zharkov prediction (\cf Section \ref{Degeneratingtorichypersurface}).

\item Suitably away from $S$ there is a $T^3$-fibration structure such that these Calabi-Yau metrics decay exponentially to semiflat metrics.

\item
These metrics extend over an exponentially large region, under the unit homological volume normalisation on $T^3$.

\item Metric behaviour transverse to $S$ is modelled on a fibration by Taub-NUT metrics. Metric behaviour suitably away from $S$ is approximately a flat $S^1$-bundle over an open subset of $(\C^*)^2_{z_1,z_2}\times \R_\mu$ with a Euclidean metric.
\end{itemize}
\end{thm}

\begin{rmk}
Loftin, Yau and Zaslow \cite{LoftinYauZaslow} attempted to find semiflat metrics on the vertices by a reduction from the 3-dimensional real Monge-Amp\`ere equation to the elliptic affine sphere. However it is unclear whether their construction has the requisite topology, or how it compares with our construction.
\end{rmk}

\begin{rmk}
After the completion of this paper,
S. Sun and R. Zhang inform the author that they have independently expected the same construction.
\end{rmk}

\begin{notation}
Summation convention will be used throughout the paper.
\end{notation}

We now proceed with an extended review of the necessary backgrounds. Some of the main techniques in this paper will be demonstrated below on the Taub-NUT metric and the Ooguri-Vafa metric.

\section{Gross-Ruan-Joyce picture of SYZ fibrations}\label{GrossRuanJoyce}

Here we review the expected picture of special Lagrangian $T^3$-fibrations (=SYZ fibrations) $f: M\to B$ on a Calabi-Yau 3-fold near the large complex structure limit. The primary sources are the work of M. Gross \cite{Gross1}\cite{Gross3} and W. D. Ruan \cite{Ruan}, with important modifications proposed by D. Joyce (\cf \cite{Joyce} Section 8). The survey of Morrison \cite{Morrison} provides good background reading.

Gross \cite{Gross3} observes that if $f: M\to B$ is a smooth SYZ fibration then the discriminant locus $\mathfrak{D}\subset B$ is of Hausdorff codimension 2. Combined with monodromy considerations, this leads to the speculation that for generic such fibrations $\mathfrak{D}$ is a trivalent graph, consisting of smooth edges and two kinds of vertices, which we refer to as positive and negative vertices following \cite{Joyce}.

\subsection{Generic region}\label{Genericregion}

In the generic region $f: M\to B$ is a smooth proper submersion with $T^3$ fibres. A torus fibration is called \textbf{semiflat} if the metric restricts to flat metrics on the tori. The Calabi-Yau structure $(g,\omega, \Omega)$ on a semiflat SYZ $T^3$-fibration can be locally described in action-angle coordinates as
\[
\begin{cases}
\omega= \sum_1^3  d\mu_i\wedge  d\theta_i ,
\\
\Omega=  (\det(g_{ij}))^{-1/2}  \bigwedge_{i=1}^3 (  d\theta_i - \sqrt{-1} g_{ij} d\mu_j   ) ,
\\
g= g^{ij} d\theta_i d\theta_j+  g_{ij} d\mu_i d\mu_j.
\end{cases}
\]
Here $g_{ij}$ is the Hessian of a real valued function $\varphi$ on $B$ solving the \textbf{real Monge-Amp\`ere equation}:
\[
g_{ij}= \frac{\partial^2 \varphi}{\partial \mu_i \partial \mu_j}, \quad 
\det( \frac{\partial^2 \varphi}{\partial \mu_i \partial \mu_j} ) = \text{const}.
\]
and $g_{ij}$ defines a metric on the base $g_B= g_{ij}d\mu_i d\mu_j$ such that $f:M\to B$ is a Riemannian submersion.

The Calabi-Yau structure induce two sets of \textbf{affine structures} on the base $B$: the \emph{symplectic moment coordinates} $\mu_i$ satisfying $d\mu_i= -\omega( \frac{\partial}{\partial \theta_i},\cdot  )$, and the \emph{complex affine coordinates} $y_i$ satisfying $dy_i= \text{Im}\Omega( \frac{\partial}{\partial \theta_j}, \frac{\partial}{\partial \theta_k}, \cdot  )$ for cyclic indices $i,j,k$. These coordinates are related by the Legendre transform
$
y_i=- \frac{\partial \varphi}{\partial \mu_i}.
$

\textbf{Semiflat mirror symmetry} is the observation that if over the same base $B$ we fibrewise replace $T^3$ by the dual tori, then there is a canonical Calabi-Yau structure exhibiting this dual torus fibration as a semiflat SYZ fibration. Furthermore, the base metric $g_B$ is unchanged while the roles of the two affine structures are interchanged.

A major part of the SYZ Conjecture \ref{SYZconjecture} is that Calabi-Yau metrics on (polarised) manifolds near the large complex structure limit are asymptotically described by such semiflat SYZ fibrations in the generic region up to exponentially small errors.

\subsection{Edges}

Along an edge $\mathfrak{\Delta}\subset \mathfrak{D}$, the singular fibres have the topology of $T^3$ with $T^2$ collapsed to $S^1$, alternatively written as $I_1\times S^1$, where $I_1$ refers to the nodal elliptic curve or equivalently $S^2$ with two points identified. Notice the singularity on the fibre is not isolated. These singular fibres have Betti numbers $(b_1,b_2)=(2,2)$ and Euler characteristic 0. The $T^3$-fibration is locally described as the Kodaira type $I_1$ degenerating family of elliptic curves over a disc $D^2\subset \C$, Cartesian product with the trivial $S^1$-bundle $S^1\times \R\to \R$. The monodromy around the edge $\mathfrak{\Delta}\subset \mathfrak{D}$ acting on $H_1(T^3)\simeq \Z^3$ can be written in a suitable basis as
\[
\begin{bmatrix}
	1       & 1 & 0  \\
0       & 1 & 0 \\
0      & 0 & 1 
\end{bmatrix}.
\]

For an alternative viewpoint which ties in better with the Gibbons-Hawking construction (see later Sections \ref{GeneralisedGibbonsHawking}), the base is $B\subset  \R^3$, and the total space $M$ is a singular $S^1$-bundle over $ S^1\times B \times S^1 $, where the $S^1$-fibres collapse to points along the codimension 3 locus $\{0\}\times \mathfrak{\Delta} \times S^1 \subset S^1\times B\times S^1$. In the 3 transverse directions,  the singular $S^1$-bundle structure is topologically modelled on the Hopf map
\begin{equation}\label{GibbonsHawkingC2}
\pi_{\C^2}: \C^2\to \R\times \C, \quad (z_0, z_1)\mapsto \left( \frac{1}{2}(|z_1|^2-|z_0|^2),  z_0 z_1      \right).
\end{equation}


The Chern class 
$c_1\in H^2(S^1\times B\times S^1 \setminus (
 \{0\}\times \mathfrak{\Delta}\times S^1), \Z     )
$ evaluates to 1 on a suitably oriented $S^2$-cycle linking $\{0\}\times \mathfrak{\Delta}\times S^1$ inside $S^1\times B\times S^1$.

\begin{rmk}
In this review Section topology refers to the continuous topology. There are subtleties with extending the smooth structure on the singular $S^1$-bundle across the discriminant locus (\cf Section \ref{Structureneardiscriminantlocus}).
\end{rmk}



\subsection{Positive vertices}\label{Positivevertices}

Let $\mathfrak{D}\subset B$ be a  graph with one vertex emitting 3 edges. Topologically, we can present $\mathfrak{D}$ as
\begin{equation}\label{positivevertexDelta}
\begin{split}
 \mathfrak{D}& = \mathfrak{D}_1\cup \mathfrak{D}_2\cup \mathfrak{D}_3\cup \{0\}\\
 & =
\{ \mu_1=0, \mu_2> 0   \}\cup \{ \mu_2=0, \mu_1> 0    \}\cup \{  \mu_1=\mu_2< 0   \}\cup \{0\} \\
&\subset  \R^2_{\mu_1, \mu_2}\times\{0\}\subset \R^2\times \R=B.
\end{split}
\end{equation}
The total space $M^+$ is built as a \textbf{singular $T^2$-bundle} over $B\times S^1= \R^2_{\mu_1, \mu_2}\times \R\times S^1$ with discriminant locus $\mathfrak{D}\times \{0\}\subset B\times S^1$. Let $e_1, e_2$ denote a basis of $H_1(T^2, \Z)$, and denote $T(ae_1+be_2)$ as the subtorus with homology class $ae_1+be_2$. Over $(B\times S^1)\setminus (\mathfrak{D}\times \{0\})$ the space $M^+$ is a principal $T^2$-bundle, whose Chern class $c_1\in H^2(B\times S^1\setminus (\mathfrak{D}\times \{0\}), \Z e_1\oplus \Z e_2)$ evaluates to $e_1, -e_2, -e_1+e_2$ respectively on the $S^2$-cycles linking $\mathfrak{D}_1\times \{0\}, \mathfrak{D}_2\times \{0\}, \mathfrak{D}_3\times \{0\}$ inside $B\times S^1$. Over the codimension 3 loci $\mathfrak{D}_1\times \{0\}$, $\mathfrak{D}_2\times\{0\}$, $\mathfrak{D}_3\times\{0\}$  inside $B\times S^1$, the $T^2$-fibres collapse to circle fibres  $T^2/T(e_1)$, $T^2/T(-e_2)$, $T^2/T(-e_1+e_2)$ respectively. Finally, over the origin $\{0\}\subset B\times S^1$, the $T^2$-fibre collapses to a point. The singular $T^2$-bundle over a small neighbourhood of the origin $D^4\subset B\times S^1$ is topologically modelled on 
\begin{equation}\label{GibbonsHawkingC3}
\begin{split}
&\pi_{\C^3}: \C^3\to \R^2_{\mu_1, \mu_2}\times \R\times \R, \\
 &(z_0,z_1,z_2)\mapsto \left(\frac{1}{2}(|z_1|^2-|z_0|^2), \frac{1}{2}(|z_2|^2-|z_0|^2), \text{Im}(z_0z_1z_2), \text{Re}(z_0z_1z_2)  \right)
\end{split}
\end{equation}
whose discriminant locus is compatible with $\mathfrak{D}$.

By construction $M^+$ fibres over $B$ with generic fibre $T^3$. The \textbf{singular fibre} over $0\in B$ has the topology of $T^3$ with $T^2$ collapsed to a point, so has Betti numbers $(b_1,b_2)=(1,2)$ and Euler characteristic $+1$ (hence the name `positive vertex'). A basis of $H_1(T^3)$ is given by $e_1, e_2\in H_1(T^2)\subset H_1(T^3)$ and an $S^1$-cycle $e_0$ on the total space lifting the cycle $S^1\subset B\times S^1$.
The monodromies around the 3 edges $\mathfrak{D}_1$, $\mathfrak{D}_2$, $\mathfrak{D}_3$ acting on $H_1(T^3)$ are given in the basis $e_0, e_1, e_2$ as
\[
\begin{bmatrix}
1       & 0 & 0  \\
1       & 1 & 0 \\
0      & 0 & 1 
\end{bmatrix},
\quad
\begin{bmatrix}
1       & 0 & 0  \\
0       & 1 & 0 \\
-1      & 0 & 1 
\end{bmatrix},
\text{ and }
\begin{bmatrix}
1       & 0 & 0  \\
-1       & 1 & 0 \\
1      & 0 & 1 
\end{bmatrix}.
\]

\subsection{Negative vertices}\label{Negativevertices}

We recount here the historical perspective of Gross and Ruan on negative vertices, to be modified in Section \ref{Joycecritique}.
Let $\mathfrak{D}\subset B$ be a graph with one vertex emitting 3 edges. Topologically, we present $\mathfrak{D}$ as
\[
\begin{split}
\mathfrak{D}& = \mathfrak{D}_1\cup \mathfrak{D}_2\cup \mathfrak{D}_3\cup \{0\} =
\{ y_1=0, y_2> 0   \}\cup \{ y_2=0, y_1> 0    \}\cup \{  y_1=y_2< 0   \}\cup \{0\} \\
&\subset  \R^2_{y_1, y_2}\times\{0\}\subset \R^2\times \R=B.
\end{split}
\]
Let $e_1, e_2$ be a basis of $H_1(T^2, \Z)$. Let $S\subset B\times T^2$ be a `pair of pants' (namely a surface homeomorphic to the complement of 3 points in $S^2$) sitting over $\mathfrak{D}\subset B$, such that $S\cap (  \mathfrak{D}_i \times T^2  )$ is a cylinder $\mathfrak{D}_i\times S^1$, where the $S^1$ factor inside $T^2$ has homology class $e_2, e_1, -e_1-e_2$ for $i=1,2,3$ respectively. The fact that these 3 classes add up to zero means the 3 cylinders $\mathfrak{D}_i\times S^1$ can be joined together over $\{0\}\times T^2$. The fibre of $S\to \mathfrak{D}$ over $0\in \mathfrak{D}$ is a `figure 8 diagram'.

Then the total space $M^-$ is built as a \textbf{singular $S^1$-bundle} over $B\times T^2$, which restricts to a principal $S^1$-bundle over the complement of the codimension 3 locus $S\subset B\times T^2$, and along $S$ the $S^1$-fibres collapse to points. The first Chern class of the $S^1$-bundle evaluates trivially on $T^2\subset B\times T^2$ but nontrivially on the $S^2$-cycle wrapping $S$. In the 3 transverse directions, the $S^1$ fibration is modelled topologically on (\ref{GibbonsHawkingC2}).

By construction $M^-$ fibres over $B$ with generic fibre $T^3$, where $T^3$ itself is an $S^1$-bundle over $T^2$. The class of this $S^1\subset T^3$ is denoted $e_3$. The \textbf{singular fibre} of $M^-\to B$ over $0\in B$ is obtained by taking the bundle $T^3\to T^2$, and collapse down its $S^1$-fibres over a `figure 8 diagram' inside $T^2$. This singular fibre has Betti numbers $(b_1, b_2)=(2, 1)$ and Euler characteristic $-1$ (hence the name `negative vertex'). The homology classes $e_1, e_2$ lift to $H_1(T^3)$. The monodromies around the edges $\mathfrak{D}_1, \mathfrak{D}_2, \mathfrak{D}_3$ acting on $H_1(T^3, \Z)$ are given in the basis $e_1, e_2, e_3$ of $H_1(T^3, \Z)$ as
\[
\begin{bmatrix}
1       & 0 & 0  \\
0       & 1 & 0 \\
1      & 0 & 1 
\end{bmatrix},
\quad
\begin{bmatrix}
1       & 0 & 0  \\
0       & 1 & 0 \\
0      & -1 & 1 
\end{bmatrix},
\text{ and }
\begin{bmatrix}
1       & 0 & 0  \\
0       & 1 & 0 \\
-1     & 1 & 1 
\end{bmatrix}.
\]

\subsection{Joyce's critique}\label{Joycecritique}

Joyce \cite{Joyce} gave reasons that the above topological picture of Gross-Ruan cannot literally describe a special Lagrangian fibration in a generic Calabi-Yau 3-fold, based on his study of local $U(1)$-invariant special Lagrangian submanifolds inside $\C^3$. Joyce's critique hinges on two geometric observations:

\begin{itemize}
\item   Special Lagrangian fibrations need not be defined by a smooth map, and the discriminant locus needs not have codimension 2.
\item   The $I_1\times S^1$ singular fibres have non-isolated special Lagrangian singularities, which is an infinite codimensional phenomenon in the parameter space, namely the singularity structure cannot persist under almost any perturbation of the  K\"ahler structure or the boundary data of the special Lagrangian.
\end{itemize}

Furthermore, in the $U(1)$-invariant setting, Joyce constructed examples illustrating the possibility that fibres with $I_1\times S^1$ singularities can break up into fibres with a pair of special Lagrangian $T^2$ cones. Such fibres lie over a thickened version of the original edges in $\mathfrak{D}$, and in particular the discriminant locus of the SYZ fibration now has codimension 1.

As suggested by Morrison \cite{Morrison} this \textbf{thickening} picture is linked to the description of the \textbf{negative vertex} (Section \ref{Negativevertices}) as follows. We can view $B\times T^2$ as $(\C^*)_{z_1,z_2}^2\times \R\simeq \R^2\times \R\times T^2$. The `pair of pants' $S$ is realised topologically by  \[
\{  z_1+z_2=1  \}\subset (\C^*)_{z_1,z_2}^2= (\C^*)^2\times \{0\}\subset (\C^*)^2\times \R
.
\] The algebraic 2-torus $(\C^*)^2$ maps to $\R^2_{y_1,y_2}$ via 
\[
y_1= - \frac{1}{2\pi}\log |z_1|, \quad y_2= -\frac{1}{2\pi} \log |z_2|.
\]
The image of $S$ in $\R^2_{y_1,y_2}$ under this map is an \textbf{amoeba} which can be thought as a thickend version of $\mathfrak{D}\subset \R^2_{y_1, y_2}$. Along the 3 directions defined by $\mathfrak{D}$, the asymptotic geometry of $S$ near infinity approaches 3 cylinders. In the modified construction $M^-$ is a singular $S^1$-bundle over $(\C^*)^2\times \R$ whose fibres collapse to points along the codimension 3 locus $S\subset (\C^*)^2\times \R$. The natural smooth map $M^-\to (\C^*)^2\times \R\to B=\R^2_{y_1, y_2}\times \R$ cannot be exactly a special Lagrangian fibration since its discriminant locus is of codimension 1, but it is still possible to be an \textbf{approximate special Lagrangian fibration}.

 We also wish to resolve a paradox here in advance. Part of our plan is to construct a family of $T^2$-symmetric Ooguri-Vafa type Calabi-Yau metrics on the \textbf{positive vertex}, which admit a special Lagrangian fibration with all the topological features predicted by Gross and Ruan, and in particular the singular fibres will have non-isolated singularities. We emphasize there is no contradiction with Joyce's critique: it is possible for the Ooguri-Vafa type metrics to be a good metric model for a generic Calabi-Yau 3-fold near the large complex structure limit, while the singularity structure of the SYZ fibration changes drastically. Joyce's critique does not rule out the Gross-Ruan picture as a limiting description of SYZ fibrations.

\subsection{Degenerating toric Calabi-Yau hypersurfaces}\label{Degeneratingtorichypersurface}

A familiar picture from Riemann surface theory is that higher genus algebraic curves can be obtained topologically by patching together `pairs of pants' along cylindrical necks. There is a similar picture for Calabi-Yau toric hypersurfaces approaching a large complex structure limit, well studied in tropical geometry. The discussions below are loosely based on Zharkov \cite{Zharkov2}\cite{Zharkov1}, and are included to predict the holomorphic structure of the positive and the negative vertices.

Let  $\mathbb{P}_{\triangle} $ be a toric manifold whose moment polytope is the reflexive integral polytope $\triangle$ in $\R^4$, so the integral points $v\in \triangle$ correspond to a basis $\{s_v\}$ for anticanonical sections. Let $\lambda$ be a (suitably generic) function on $\triangle \cap \Z^4$ whose piecewise linear extension is a convex function on $\R^4$ minimized at $0\in \triangle$ with minimum value 0. We consider a polarised family of hypersurfaces $X_t$ defined by
\begin{equation}\label{torichypersurface}
s_0+ \sum_{ v\in \triangle\setminus \{0\}  } t^{\lambda(v)} a_v s_v=0,
\end{equation}
where $a_v$ are fixed nonzero complex numbers and $t$ is a small positive parameter. The holomorphic volume form is determined from the adjunction formula.

The key point is that when $t$ is very small, the hypersurface $X_t$ decompose into a finite number of regions, on each of which only a small number of monomial functions $\frac{s_v}{s_0}$ dominate the rest. Thus up to scaling coordinates by powers of $t$, there are only a small number of complex geometric local models, typically with some torus symmetry. Furthermore there is some combinatorial structure which controls how these local models patch together to give $X_t$ as a complex manifold.

\begin{eg}(\textbf{Generic region})
Suppose in some region only $s_0$ and $t^{\lambda(v)}a_v s_v$ dominate, so the hypersurface locally looks like $\frac{s_v}{s_0}= \text{const}$. After normalising by powers of $t$ we may write this as
$
\{ z_0 =1\}
$ in the coordinates $z_0, z_1, z_2, z_3$ on the algebraic torus $(\C^*)^4$.
This model has $T^3$-symmetry under the diagonal action on $z_1, z_2, z_3$. The $T^3$-orbits are the natural candidate for approximate SYZ fibres. Thus we naturally look for a K\"ahler metric with potential $\phi=\phi(u_1, u_2, u_3)$ depending only on the logarithms $u_1= \log |z_1|, u_2=\log |z_2|, u_3= \log |z_3|$. The holomorphic volume form $\Omega$ on the hypersurface is up to a scale factor given by
\[
\frac{dz_0}{z_0}\wedge \frac{dz_1}{z_1}\wedge\frac{dz_2}{z_2}\wedge \frac{dz_3}{z_3}= d(z_0-1)\wedge \Omega,
\]
namely $\Omega= \frac{dz_1}{z_1}\wedge\frac{dz_2}{z_2}\wedge \frac{dz_3}{z_3}= d\log z_1 \wedge d\log z_2\wedge d\log z_3$. The complex Monge-Amp\`ere equation $(\sqrt{-1}\partial \bar{\partial} \phi)^3= \text{const} \sqrt{-1}\Omega\wedge \overline{\Omega}$ naturally reduces to the \textbf{real Monge-Amp\`ere equation} $\det( \frac{\partial^2 \phi}{\partial u_i\partial u_j}  )=\text{const}$. One can further calculate that such regions take up most of the volume measure on $X_t$, thus lending some evidence for the SYZ conjecture. We remark that the description only applies to local regions so the metrics are not complete.
\end{eg}

\begin{eg}
The real Monge-Amp\`ere equation governs also the region near the intersection of $X_t$ with \emph{a smooth component} of the toric boundary. Suppose after normalising by powers of $t$, the dominant monomials are
$z_0^{-1}, z_0^{-1}z_1, z_0^{-1}z_2, z_0^{-1}z_3, 1$ in the coordinates $z_0,z_1,z_2,z_3$ on the algebraic torus $(\C^*)^4$, so the hypersurface has the local complex geometric model $\{ z_0^{-1}(1+z_1+z_2+z_3 )+1=0    \}\subset (\C^*)^4$, or equivalently $-z_0=1+z_1+z_2+z_3$. The adjunction formula
\[
\frac{dz_0}{z_0}\wedge \frac{dz_1}{z_1}\wedge\frac{dz_2}{z_2}\wedge \frac{dz_3}{z_3}= d( z_0^{-1}(1+z_1+z_2+z_3)+1  )\wedge \Omega
\] 
leads to $\Omega= \frac{dz_1}{z_1}\wedge\frac{dz_2}{z_2}\wedge \frac{dz_3}{z_3}$ as before. The diagonal $T^3$-action on $z_1, z_2, z_3$ provides the candidate for an approximate SYZ fibration, and a solution to the real Monge-Amp\`ere equation in the $\log |z_1|, \log |z_2|, \log |z_3|$ coordinates induces a local Calabi-Yau metric.
\end{eg}

\begin{eg}
Suppose after normalising by powers of $t$, the dominant monomials are $(z_1z_2)^{-1}, -(z_1z_2)^{-1}z_3, -1$, so the hypersurface admits the local complex geometric model $\{  (z_1z_2)^{-1}(1-z_3)=1  \}\subset (\C^*)^4_{z_0,z_1,z_2,z_3}$, or equivalently $z_1z_2=1-z_3$.  This happens near the intersection of two smooth components of the toric boundary. Up to numerical factors, the holomorphic volume form is given by
\[
\frac{dz_0}{z_0}\wedge \frac{dz_1}{z_1}\wedge\frac{dz_2}{z_2}\wedge \frac{dz_3}{z_3}= d( (z_1z_2)^{-1}(1-z_3)-1   ) \wedge \Omega,
\]
namely $\Omega= d\log z_0\wedge \frac{dz_1\wedge dz_2}{ z_3  }$. This model has a natural $T^2$-symmetry: one $S^1$ acts trivially on $z_1, z_2, z_3$ and rotates $z_0$, while the other $S^1$ acts trivially on $z_0, z_3$ and diagonally on $z_1, z_2$. The model is intimately related to the Ooguri-Vafa metric (\cf Section \ref{OoguriVafaholomorphicviewpoint}), and we expect this region to coincide with the neighbourhood of \textbf{edges} in the Gross-Ruan-Joyce picture.
\end{eg}

In this paper we are primarily interested in the positive and negative vertices. These are relevant for certain regions near the intersection of $X_t$ with some higher depth strata of the toric boundary of $\mathbb{P}_\triangle$.

\begin{eg}
The \textbf{positive vertex}  $M^+$ describes a neighbourhood of the point $(0,0,0,1)$ inside $\{ z_0z_1z_2=1- z_3   \}\subset \C^3\times \C^*_{z_3}$. In the toric hypersurface picture, we are looking at a region where the dominating monomials are up to scale factors $(z_0z_1z_2)^{-1}, -z_3(z_0z_1z_2)^{-1}, 1$, so the defining equation of $X_t$ is  approximately  $ (z_0z_1z_2)^{-1}(1-z_3)=1$ once we absorb the scale factors into $z_i$. The holomorphic volume form $\Omega$ is  up to constant given by
\[
-\frac{\sqrt{-1}}{2\pi}
d\log z_0\wedge d\log z_1\wedge d\log z_2\wedge d\log z_3= d((z_0z_1z_2)^{-1}(1-z_3)-1   )\wedge \Omega,
\]
or equivalently $\Omega=-\frac{\sqrt{-1}}{2\pi}\frac{1}{z_3} dz_0\wedge dz_1\wedge dz_2$. An important feature of this model is the diagonal \textbf{$T^2$-symmetry}: 
\[
e^{i\theta_1}\cdot (z_0, z_1, z_2)=(e^{-i\theta_1}z_0, e^{i\theta_1} z_1, z_2), 
\quad 
e^{i\theta_2}\cdot (z_0, z_1, z_2)=(e^{-i\theta_2}z_0, z_1, e^{i\theta_2} z_2).
\]
We have $\Omega( \frac{\partial }{\partial \theta_1}, \frac{\partial }{\partial \theta_2}, \cdot   )= -\frac{\sqrt{-1}}{2\pi}d \log z_3=d \eta$, where $\eta=  -\frac{\sqrt{-1}}{2\pi} \log (z_3)$ is a holomorphic coordinate with period 1, and takes the value zero at $z_0=z_1=z_2=0$. The relation between this complex geometric perspective and the topological picture in Section \ref{Positivevertices} is perhaps clearest with the generalised Gibbons-Hawking construction in mind (\cf Section \ref{GeneralisedGibbonsHawking} below). Essentially $M^+$ is a singular $T^2$-bundle over a 4-dimensional base contained in $\R^2_{\mu_1, \mu_2}\times (S^1\times \R)_\eta$, where $\mu_1, \mu_2$ are the $T^2$-moment maps normalised to have value 0 at $z_0=z_1=z_2=0$. We shall notice that the discriminant locus $\mathfrak{D}\times \{0\}$ of this singular $T^2$-bundle is not sensitive to the choice of the K\"ahler form (\cf Lemma \ref{discriminantlocusisinsensitive} and its ensuing Remark). The normalising constant on $\Omega$ imply that the SYZ $T^3$-fibres have   $\int_{T^3} \Omega=4\pi^2$.
\end{eg}

\begin{eg}
The \textbf{negative vertex} $M^-$ describes an open subset inside $\{ z_3 z_4= 1-z_1-z_2   \}\subset \C^*_{z_1}\times \C^*_{z_2}\times \C^2_{z_3, z_4}$.
In the toric hypersurface picture, we are looking at a region where the dominating monomials are up to scale factors just $(z_3z_4)^{-1}$, $z_1(z_3z_4)^{-1}$, $z_2(z_3z_4)^{-1}$ and 1, so the defining equation of $X_t$ is  approximately $ (1-z_1-z_2)(z_3z_4)^{-1}-1=0     $ once we absorb the scale factors into $z_i$. The holomorphic volume form $\Omega$ is given up to constant by
\[
 \frac{\sqrt{-1}}{4\pi^2}
d\log z_1\wedge d\log z_2\wedge d\log z_3\wedge d\log z_4
= d \left( (1-z_1-z_2)(z_3z_4)^{-1}-1        \right) \wedge \Omega,
\]
or equivalently
$
\Omega= \frac{-\sqrt{-1}}{4\pi^2}\frac{1}{z_1z_2} dz_2\wedge dz_3\wedge dz_4.
$
This model has \textbf{$S^1$-symmetry}:
\[
e^{i\theta}\cdot (z_1, z_2, z_3, z_4)= ( z_1, z_2,  e^{i\theta}z_3 , e^{-i\theta} z_4       ).
\]
Hence $M^-$ is a singular $S^1$-bundle over $\R_\mu\times \C^*_{z_1}\times \C^*_{z_2}$, where $\mu$ is the $S^1$-moment coordinate which takes the value zero on the singular locus $\{ z_3=z_4=0 \}$ (notice that the degeneracy of the $S^1$ factor implies that the moment map is constant on this singular locus for any choice of K\"ahler form). This agrees with the modified topological description in Section \ref{Joycecritique}.
We calculate \[\iota_{\frac{\partial }{\partial \theta}   }\Omega=- \frac{1}{4\pi^2} d\log z_1\wedge d\log z_2 = d\eta_1\wedge d\eta_2 , \]
 where the logarithmic coordinates $\eta_p= \frac{1}{2\pi \sqrt{-1}}\log z_p$ for $p=1,2$ have period 1. The $\arg z_1, \arg z_2$ coordinates provide a family of 2-tori in $\R\times\C^*_{z_1}\times \C^*_{z_2}$, and the restriction of the $S^1$-bundle over these 2-tori defines a family of 3-tori. The normalising constant on $\Omega$ imply that $\int_{T^3} \Omega=2\pi$.
\end{eg}

\section{Generalised Gibbons-Hawking ansatz}\label{GeneralisedGibbonsHawking}

The materials in this Section draws heavily from the presentation of Zharkov \cite{Zharkov1}. Suppose $M$ is a complex $N$-dimensional K\"ahler manifold with a nonvanishing holomorphic volume form admitting a holomorphic isometric free $T^{\mathfrak{n}}$ action. The \textbf{generalised Gibbons-Hawking ansatz} expresses the K\"ahler and Calabi-Yau conditions in terms of the $\mathfrak{n}$ moment map coordinates and the $N-\mathfrak{n}$ holomorphic coordinates on the K\"ahler quotient.

Let $\mathfrak{t}$ denote the Lie algebra of $T^{\mathfrak{n}}$, and let $\mathfrak{t}_\Z$ be the natural integral lattice in $\mathfrak{t}$. A choice of basis in $\mathfrak{t}_\Z$ defines linear coordinates $\mu_i$ on the dual space $\mathfrak{t}^*\simeq \R^{\mathfrak{n}}$. Let $Y$ be either $\C^{N-\mathfrak{n}}$ or $(\C^*)^{ N-\mathfrak{n} }$, and let $\eta_p$ denote the standard complex coordinates on $\C^{N-\mathfrak{n}}$ or the logarithmic coordinates on $(\C^*)^{N-\mathfrak{n}}$ with period 1 (in which case $e^{2\pi i \eta_p}$ are the standard coordinates on $(\C^*)^{N-\mathfrak{n}}$). Consider a principal $T^{\mathfrak{n}}$-bundle $\pi: M\to \mathcal{B}^0$ over an open set $\mathcal{B}^0$ in $\mathfrak{t}^*\times Y$, whose first Chern class is an element $c_1\in H^2(\mathcal{B}^0, \mathfrak{t}_\Z)$. Later we will partially compactify $M$ into a singular $T^{\mathfrak{n}}$-bundle. Summation convention will be used throughout.

\begin{thm}\label{GibbonsHawkingZharkov}
(\cf Theorem 2.1 in \cite{Zharkov1})
Let $V^{ij}$, respectively $W^{p\bar{q}}$, be real symmetric positive definite/Hermitian matrices of smooth functions on $\mathcal{B}^0$, locally given by some potential function $\Phi$:
\begin{equation}
V^{ij}= \frac{\partial^2 \Phi}{\partial \mu_i \partial \mu_j}, \quad W^{p\bar{q}}= -4 \frac{\partial^2 \Phi}{\partial 
\eta_p \partial \bar{\eta}_q}, \quad 1\leq i,j\leq \mathfrak{n}, \quad 1\leq p,q\leq N-\mathfrak{n}.
\end{equation}
Then the following $\mathfrak{t}$-valued real 2-form is closed:
\begin{equation}\label{GibbonsHawkingcurvature}
F_j= \sqrt{-1} \left(    \frac{1}{2} \frac{\partial W^{p\bar{q}}}{\partial \mu_j} d\eta_p \wedge d\bar{\eta}_q + \frac{\partial V^{ij}}{\partial \eta_p} d\mu_i \wedge d\eta_p -  \frac{\partial V^{ij}}{\partial \bar{\eta}_q} d\mu_i \wedge d\bar{\eta}_q  \right).
\end{equation}
Suppose further that $\frac{1}{2\pi}(F_1, \ldots, F_{\mathfrak{n}})$ is in the cohomology class $ c_1\in H^2(\mathcal{B}^0, \mathfrak{t}_\Z)$. Then there exists a connection $\vartheta$ on the principal bundle $M\to \mathcal{B}^0$ with curvature $d\vartheta_i=F_i$ for $i=1, \ldots, \mathfrak{n}$, such that $M$ is a K\"ahler manifold with metric tensor
\begin{equation}\label{GibbonsHawkingmetrictensor}
h= (V^{-1})^{ij} \zeta_i \otimes \bar{\zeta}_j + W^{p\bar{q}} d\eta_p \otimes d\bar{\eta}_q, \quad \omega= d\mu_j \wedge \vartheta_j + \frac{\sqrt{-1}}{2} W^{p\bar{q}} d\eta_p \wedge d\bar{\eta}_q,
\end{equation}
where $\zeta_j= V^{ij} d\mu_i + \sqrt{-1} \vartheta_j$ and $\eta_p$ form a basis of type (1,0) forms which defines an integrable complex structure. There is a nowhere vanishing holomorphic form on $M$:
\begin{equation}\label{GibbonsHawkingholomorphicvolume}
\Omega= \wedge_{j=1}^{\mathfrak{n}} (-\sqrt{-1}\zeta_j) \bigwedge \wedge_{ p=1  }^{ N-\mathfrak{n} } d\eta_p.
\end{equation}
The \textbf{Calabi-Yau condition} $\omega^N= \frac{ N!  }{2^{N}} \sqrt{-1}^{N^2} \Omega\wedge \overline{\Omega}$ is equivalent to the equation
\begin{equation}\label{GibbonsHawkingCY}
\det (V^{ij})= \det (W^{p\bar{q}}).
\end{equation}

\end{thm}

\begin{rmk}
The inverse matrix $(V^{-1})^{ij}$ describes the  metric restricted to the torus fibres, and the matrix $W^{p\bar{q}}$ describes the metric induced on the K\"ahler quotients.	This viewpoint is taken by Pedersen and Poon \cite{PedersenPoon}, whose argument shows that Calabi-Yau manifolds with Hamiltonian torus symmetries necessarily arise from this construction locally. The local existence of the potential $\Phi$ is equivalent to the linear \textbf{integrability condition}
\begin{equation}\label{GibbonsHawkingintegrability1}
\frac{\partial V^{ij}}{\partial \mu_k}= \frac{\partial V^{ik}}{\partial \mu_j}, \quad \frac{\partial W^{p\bar{q}}}{\partial \eta_r }= \frac{\partial W^{rq}}{\partial \eta_p},
\quad \frac{\partial W^{pr}}{\partial \bar{\eta}_q }= \frac{\partial W^{p\bar{q}}}{\partial \bar{\eta}_r },
\end{equation}
and 
\begin{equation}\label{GibbonsHawkingintegrability2}
\frac{\partial^2 W^{p\bar{q}}}{\partial \mu_i \partial \mu_j }
+
4\frac{\partial^2 V^{ij}}{\partial \eta_p \partial \bar{\eta}_q }=0.
\end{equation}
In particular when $N=2, \mathfrak{n}=1$, the Calabi-Yau condition is $V=W$ and we recover the usual Gibbons-Hawking equation from (\ref{GibbonsHawkingintegrability2}).
\end{rmk}

\begin{proof}
(Theorem \ref{GibbonsHawkingZharkov}, sketch)
By formula (\ref{GibbonsHawkingcurvature})	and the integrability condition (\ref{GibbonsHawkingintegrability1}),
\begin{equation}\label{GibbonsHawkingdF}
dF_j= \frac{\sqrt{-1}}{2} \left(  
\frac{\partial^2 W^{p\bar{q}}}{\partial \mu_i \partial \mu_j }
+
4\frac{\partial^2 V^{ij}}{\partial \eta_p \partial \bar{\eta}_q }
\right) d\mu_i \wedge d\eta_p \wedge d\bar{\eta}_q,
\end{equation}
so the closedness of $F_j$ is equivalent to (\ref{GibbonsHawkingintegrability2}). Since $\frac{1}{2\pi}F$ represents the appropriate first Chern class, $F$ must be the curvature of a $T^{\mathfrak{n}}$-connection $\vartheta$. Modulo gauge $\vartheta$ admits the local formula
\begin{equation}\label{connectionformula}
\vartheta_j= \sqrt{-1} \{  \frac{\partial^2 \Phi}{\partial \eta_p \partial \mu } d\eta_p-   \frac{\partial^2 \Phi}{\partial \bar{\eta}_p \partial \mu } d\bar{\eta}_p   \}.
\end{equation}
Gauge equivalent choices of the connection define the structures on $M$ up to holomorphic isometry.

The integrability of the complex structure follows from the fact that the differential ideal generated by $(1,0)$ forms is closed:
\begin{equation}\label{GibbonsHawkingholomorphicdifferential}
d\zeta_j= \left(  \frac{1}{2}
\frac{\partial W^{p\bar{q}}}{ \partial \mu_j } d\bar{\eta}_q
-
2\frac{\partial V^{ij}}{\partial \eta_p  }d\mu_i
\right)  \wedge d\eta_p ,
\end{equation}
using (\ref{GibbonsHawkingintegrability1})(\ref{GibbonsHawkingcurvature}) and the definition of $\zeta_j$. The K\"ahler condition $d\omega=0$ follows from (\ref{GibbonsHawkingintegrability1})(\ref{GibbonsHawkingcurvature}).
The Calabi-Yau condition follows from the more general formula
\[
\omega^N= \det(W^{p\bar{q}}) \det (V^{ij} )^{-1} \frac{ N!  }{2^N} \sqrt{-1}^{N^2} \Omega\wedge \overline{\Omega}.
\]
\end{proof}

\begin{rmk}\label{T2symmetryimpliesSL}
If $\frac{\partial}{\partial \theta_j  }$ are the Hamiltonian vector fields dual to $\vartheta_i$, namely $\vartheta_i(\frac{\partial}{\partial \theta_j  } )=\delta_{ij}$, then $d\mu_i= -\iota_{  \frac{\partial}{\partial \theta_i  }  }\omega$, namely $\mu_i$ are the symplectic \textbf{moment coordinates} up to sign. When $N-\mathfrak{n}=1$, namely there is only one $\eta$ coordinate, then $d\eta=\Omega( \frac{\partial}{\partial \theta_1  }  , \ldots, \frac{\partial}{\partial \theta_{\mathfrak{n}}  } , \cdot     )$, and accordingly we refer to $\eta$ as the holomorphic moment coordinate. In this situation $M$ admits a fibration
\[
M\xrightarrow{    (\mu_1, \ldots, \mu_{\mathfrak{n}}, \text{Im}(\eta) ) }\mathfrak{t}^*\times \R,
\]
where fibres are \textbf{special Lagrangians} with phase zero: the Lagrangian condition follows from $d\mu_i= -\iota_{\frac{\partial}{\partial \theta_i}}\omega=  0$ on fibres, while the special condition is equivalent to $\text{Im} \Omega(  \frac{\partial}{\partial \theta_1}, \ldots \frac{\partial}{\partial \theta_{\mathfrak{n}} }, \cdot     )= d\text{Im}\eta=0$. 
\end{rmk}

\begin{rmk}
The information contained in the generalised Gibbons-Hawking ansatz can be encoded by a \textbf{Riemannian metric on the base}, written in distinguished coordinates as
\begin{equation}\label{GibbonsHawkingbasemetric}
g_{\mathcal{B}^0}= V^{ij} d\mu_i \otimes d\mu_j + \text{Re}( W^{p\bar{q}} d\eta_p \otimes d\bar{\eta}_q),
\end{equation}
such that the map $M\to \mathcal{B}^0$ is a \textbf{Riemannian submersion}.
\end{rmk}

\begin{rmk}
There is an additional freedom to twist the connection $\vartheta$ by a \textbf{flat connection}. Up to gauge equivalence, the choice of $\vartheta$ is parametrised by $H^1(\mathcal{B}^0, T^n)$.
\end{rmk}

\subsection{Elementary examples}

\begin{eg}\label{Constantsolution}
(\textbf{Constant solution}) The simplest solution is where $V^{ij}$ and $W^{p\bar{q}}$ are independent of the base variables and satisfy (\ref{GibbonsHawkingCY}). We shall see that many interesting solutions can be thought heuristically as \textbf{perturbation of the constant solution} after introducing some \textbf{topology}.
Some important special cases for us are:
\begin{itemize}
\item $N=2, \mathfrak{n}=1$, $V=W=A>0$. The subcase where $\eta$ takes value in $\C$ is relevant for the Taub-NUT metric (\cf Example \ref{TaubNUT}), and the subcase where $\eta$ is a periodic variable is relevant for the Ooguri-Vafa metric (\cf Section \ref{OoguriVafa}). In the periodic case the choice of the connection $\vartheta$ is parametrised by $H^1(S^1\times \R^2, S^1)\simeq S^1$.

\item  $N=3, \mathfrak{n}=2$, $V^{ij}=a_{ij}$ is symmetric positive definite, and $W=A=\det(a_{ij})$. We write $g_a= a_{ij} d\mu_i  d\mu_j+ A |d\eta|^2$. The subcase where $\eta$ takes value in $\C$ will be relevant for constructing new Taub-NUT type Calabi-Yau metrics on $\C^3$, and the subcase where $\eta$ is a periodic variable will be relevant for the positive vertex. In the periodic case the choice of the connection $\vartheta$ is parametrised by $H^1(S^1\times \R^3, T^2)\simeq T^2$.
\\
\item  $N=3, \mathfrak{n}=1$,
$W^{p\bar{q}}=a_{p\bar{q} }$ is Hermitian, and $V=A=\det(a_{p\bar{q}} )$. We write $g_a=\text{Re}( a_{p\bar{q}} d\eta_p  d\bar{ \eta}_q )+ A d\mu\otimes d\mu$. Here $\eta_1, \eta_2$ are periodic coordinates with period 1. The choice of the connection $\vartheta$ is parametrised by $H^1(T^2\times \R^3, S^1)\simeq T^2$. This case will be relevant for the negative vertex.  Notice that if we demand that the fibration on $M$ induced by $\mu, \text{Im}(\eta_1), \text{Im}(\eta_2)$ is a special Lagrangian fibration with phase zero, then we would need $a_{1\bar{2}}= a_{2\bar{1}}$, namely $a_{pq}=a_{p\bar{q}}$ is symmetric. 
\end{itemize}

\end{eg}

\begin{eg}\label{HarveyLawson}
($\C^N$ \textbf{Harvey-Lawson} example) The affine space $\C^N$ with the standard Euclidean metric $\omega= \frac{\sqrt{-1}}{2} \sum_{i=0}^{N-1} dz_i\wedge d\bar{z}_i$ and holomorphic volume form $\Omega= \sqrt{-1}^{N-1} dz_0\wedge\ldots dz_{N-1}$ admits a diagonal $T^{N-1}$-action, where the $k$-th circle factor acts by
\[
e^{i\theta_k}\cdot (z_0, z_1, \ldots , z_{N-1})=( e^{-i\theta_k} z_0 , z_1, \ldots, e^{i\theta_k} z_k, z_{k+1},\ldots, z_{N-1}      ).
\]
The corresponding moment coordinates are
\[
\mu_i= \frac{1}{2} (|z_i|^2-|z_0|^2), \quad i=1, 2, \ldots N-1, \text{ and } \eta= z_0 z_1\ldots z_{N-1}.
\]
This defines a $T^{N-1}$-bundle away from the singular locus $\bigcup \{ z_i=z_j=0\}$. Special cases include (\ref{GibbonsHawkingC2})(\ref{GibbonsHawkingC3}). The inverse matrices are
\[
(V^{-1})^{ij}= |z_0|^2+ \delta_{ij} |z_i|^2, \quad W^{-1}= |z_0z_1\ldots z_{N-1}|^2\left( \frac{1}{|z_0|^2}+\ldots+ \frac{1}{ |z_{N-1}|^2  }     \right) ,
\]
viewed as functions of $\mu_i$ and $\eta$. The special Lagrangian fibration described by Remark \ref{T2symmetryimpliesSL} is the well known Harvey-Lawson example.

We notice in particular when $N=3$ that the \textbf{discriminant locus} of the singular $T^2$-bundle is given by $\mathfrak{D}\subset \R^2_{\mu_1, \mu_2}\times \{0\}\subset \R^2_{\mu_1, \mu_2}\times \C_\eta$ as in (\ref{positivevertexDelta}). This is not an accidental feature of the Euclidean metric:

\begin{lem}\label{discriminantlocusisinsensitive}
Let $\C^3$ be equipped with the holomorphic volume form $\Omega$ above, and let $\omega$ be any $T^2$-invariant K\"ahler form with infinite volume on the singular loci $\{ z_i=z_j=0\}$ for any $i,j$. Then the discriminant locus of the singular $T^2$-bundle is $\mathfrak{D}\subset \R^2_{\mu_1, \mu_2}\times \C_\eta $ in moment coordinates $\mu_1, \mu_2$ and $\eta$.
\end{lem}

\begin{proof}
The discriminant locus is the image of the singular locus $\mathcal{C}_{ij}=\{z_i=z_j=0\}$ under the moment map. We shall focus on $\mathcal{C}_{01}$.
The holomorphic moment coordinate $\eta$ depends only on $\Omega$ and the $T^2$ action, so $\eta=z_0z_1z_2$ as before and vanishes on $\mathcal{C}_{01}$. The symplectic moment coordinates are defined by $d\mu_i=-\iota_{ \frac{\partial}{\partial \theta_i}  }\omega$, and are normalised to be zero at $(z_1, z_2, z_3)=0$. In particular since the Hamiltonian vector field $\frac{\partial}{\partial \theta_1}$ vanishes on $\mathcal{C}_{01}$, the moment $\mu_1$ must be the constant zero on $\mathcal{C}_{01}$. Furthermore $\mu_2>0$ on $\mathcal{C}_{01}$ by considering the weight of the remaining $S^1$ action at the fixed point, so  the image of $\mathcal{C}_{01}$ is contained in $\mathfrak{D}_1\cup\{0\}$.
The infinite volume condition and the formula
\[
0\leq \int_{  \mathcal{C}_{01}\cap \{\mu_2  <m\}  } \omega =- 2\pi \int_{ \mu_2  =m }^0 d\mu_2= 2\pi m, \quad \forall m\geq 0,
\]
 ensure that $\mu_2$ stretches to infinity, so  $\mathfrak{D}_1\cup \{0\}$ is the image of $\mathcal{C}_{01}$. Likewise the image of $\mathcal{C}_{02}$ is $\mathfrak{D}_2\cup\{0\}$
and the image of $\mathcal{C}_{12}$ is $\mathfrak{D}_3\cup \{0\}$.
\end{proof}

\begin{rmk}
The same method shows the complex geometry of the positive vertex in Section \ref{Degeneratingtorichypersurface} is compatible with the discriminant locus described in Section \ref{Positivevertices}.
\end{rmk}

\end{eg}

\begin{eg}\label{TaubNUT}
(\textbf{Taub-NUT})
We take $N=2, \mathfrak{n}=1$, and $V=W= \frac{1}{2\sqrt{\mu^2+|\eta|^2   } }+A$, where $A$ is a positive constant. This defines a Calabi-Yau metric whose asymptotic geometry at infinity approaches the constant solution (\cf Example \ref{Constantsolution}), with asymptotic circles of length $\frac{2\pi}{\sqrt{A}}$ fibred over a flat 3-dimensional base. Different choices of $A$ define the same metric up to scaling. The first Chern class $c_1$ of the $S^1$-bundle over $(\R_\mu\times \C_\eta)\setminus\{0\}$ evaluates to $-1$ on any sphere around the origin in $\R_\mu\times \C_\eta$; equivalently, the 3-current $-\frac{1}{2\pi} dF$ is represented by the origin $0\in \R_\mu\times \C_\eta$ viewed as a codimension 3 cycle. Written in terms of the delta function,
\[
\frac{1}{2\pi} dF=\frac{1}{2\pi} 
( \frac{\partial^2}{\partial \mu^2} + 4\frac{\partial^2}{\partial \eta \partial \bar{\eta}}   ) V d\mu\wedge d \text{Re}\eta\wedge d\text{Im}{\eta}
=-\delta(\mu, \eta, \bar{\eta}) d\mu\wedge d\text{Re} \eta\wedge d \text{Im}{\eta}.
\]

From the \textbf{holomorphic perspective}, LeBrun \cite{LeBrun} observes that the Taub-NUT space is biholomorphic to $\C^2$. To see this, recall
$\zeta=Vd\mu+\sqrt{-1}\vartheta$ and
 notice the (1,0) form $\zeta- \frac{\mu}{2\eta \sqrt{\mu^2+|\eta|^2} }d\eta$ is closed, so locally is the differential of a holomorphic function. The line integrals 
\[
\log z_1= \int \zeta+ (\frac{1}{2\eta} - \frac{\mu}{2\eta \sqrt{\mu^2+|\eta|^2} })d\eta, \quad \log z_0= \int -\zeta+ (\frac{1}{2\eta} + \frac{\mu}{2\eta \sqrt{\mu^2+|\eta|^2} })d\eta
\]
define holomorphic functions up to $2\pi \sqrt{-1} \Z$ over the regions $\R\times \C\setminus \{  \eta=0, \mu\leq 0   \}$ and $\R\times \C\setminus \{  \eta=0, \mu\geq 0   \}$ respectively, so $z_1$ and $z_0$ are well defined over the respective regions. Since $d\log z_1+d\log z_0=d\log \eta$ we can normalise $z_0, z_1$ to satisfy the \textbf{functional equation}  $z_1z_0=\eta$, whence $z_1$ and $z_0$ extend as global holomorphic functions. These coordinates exhibit the biholomorphism to $\C^2$. By considering the Hamiltonian vector field acting on $\log z_1, \log z_0$, we idenitfy the $S^1$ action as
\[
e^{i\theta}\cdot (z_1, z_0)=( e^{i\theta}z_1, e^{-i\theta}z_0).
\]
The holomorphic volume form is
\[
\Omega= -\sqrt{-1} \zeta \wedge d\eta=-\sqrt{-1}d\log z_1 \wedge d(z_1z_0)=\sqrt{-1} dz_0\wedge dz_1.
\]
Thus $\eta=z_0z_1$ defines holomorphic fibration of $\C^2$ by affine quadrics. The generic quadric fibre is topologically a cylinder, and metrically is also approaching the flat cylindrical metric near spatial infinity. When $\eta=0$, the quadric fibre degenerates into a union of two complex lines with simple normal crossing, where each line looks metrically like a cylinder with one capped end.

\end{eg}

\subsection{Compactification and distributional equation}

When we partially compactify the principal $T^n$-bundle over $\mathcal{B}^0$ to a singular $T^{\mathfrak{n}}$-bundle over $\mathcal{B}\supset \mathcal{B}^0$ by allowing torus fibres to degenerate, we need to encode the topology into the generalised Gibbons-Hawking ansatz, by changing the RHS of (\ref{GibbonsHawkingintegrability2}) into a \textbf{distributional term} reflecting the nontriviality of the first Chern class (\cf the Taub-NUT example \ref{TaubNUT}). This has been worked out by Zharkov \cite{Zharkov1} in general dimensions; here we will focus on the vertices in $N=3$.

\begin{eg}\label{compactificationpositivevertex} (\textbf{Positive vertex} and \textbf{Taub-NUT type $\C^3$})
Recall from Section \ref{Positivevertices} that $e_1, e_2$ are the homology classes of the two circle factors in the $T^2$-fibre, or equivalently an integral basis in $\mathfrak{t}$.
The $\mathfrak{t}$-valued curvature 2-form 
$F=F_1 e_1+ F_2 e_2$ satisfies (\cf (\ref{GibbonsHawkingdF}))
\[
\frac{1}{2\pi}dF_j \otimes  e_j=  \frac{\sqrt{-1}}{4\pi} \left(  
\frac{\partial^2 W}{\partial \mu_i \partial \mu_j }
+
4\frac{\partial^2 V^{ij}}{\partial \eta \partial \bar{\eta} }
\right) d\mu_i \wedge d\eta \wedge d\bar{\eta}\otimes e_j,
\]
which is a $\mathfrak{t}$-valued 3-current supported on the codimension 3 discriminant locus $\mathfrak{D}\subset \R^2_{\mu_1, \mu_2}\times (S^1\times \R)_\eta$. Now take  small 3-balls transverse to $\mathfrak{D}_1, \mathfrak{D}_2, \mathfrak{D}_3$ respectively. The integrals of $\frac{1}{2\pi} dF_j \otimes e_j$ over the balls are equal to the integrals of the Chern class representative $\frac{1}{2\pi} F$ over the $S^2$ linking $\mathfrak{D}_1, \mathfrak{D}_2, \mathfrak{D}_3$, which by Section \ref{Positivevertices} are $e_1, -e_2, -e_1+e_2$ up to orientation issues. Thus  
\begin{equation}\label{Positivevertexdistribution}
\frac{\sqrt{-1}}{4\pi} \left(  
\frac{\partial^2 W}{\partial \mu_i \partial \mu_j }
+
4\frac{\partial^2 V^{ij}}{\partial \eta \partial \bar{\eta} }
\right) d\mu_i \wedge d\eta \wedge d\bar{\eta}\otimes e_j= \mathfrak{D}_1\otimes e_1- \mathfrak{D}_2 \otimes e_2+ \mathfrak{D}_3\otimes (e_2-e_1),
\end{equation}
where the RHS is a $\mathfrak{t}$-valued codimension 3 cycle. The orientation here is decided by comparing with the Taub-NUT example. If $d\mu_1\wedge d\mu_2\wedge d\text{Re} \eta \wedge d\text{Im}\eta$ is an orientation form on $\R^2_{\mu_1,\mu_2}\times (S^1\times \R)_\eta$, then the orientation forms on $\mathfrak{D}_1, \mathfrak{D}_2,\mathfrak{D}_3$ are $d\mu_2, d\mu_1, -d\mu_1$, compatible with the directions pointing to infinity.

In the variant situation where $\eta\in \C$ instead of being periodic, to which previous discussions still apply, the generalised Gibbons-Hawking ansatz has a \textbf{scaling symmetry} compatible with the distributional equation (\ref{Positivevertexdistribution}): a new solution $\Phi_\Lambda$ may be constructed from an old solution $\Phi$ by
\begin{equation}\label{scalingsymmetry}
\begin{cases}
\Phi_{\Lambda}(\mu_1,\mu_2, \eta)= \Lambda^{-1} \Phi(\Lambda \mu_1, \Lambda \mu_2, \Lambda^{1.5} \eta), \\
V^{ij}_{\Lambda}(\mu_1,\mu_2, \eta)= \Lambda V^{ij}(\Lambda \mu_1, \Lambda \mu_2, \Lambda^{1.5} \eta), \\
W_{\Lambda}(\mu_1,\mu_2, \eta)= \Lambda^2 W(\Lambda \mu_1, \Lambda \mu_2, \Lambda^{1.5} \eta)
\end{cases}
\end{equation}
These solutions are isometric up to a scaling factor, analogous to Taub-NUT metrics with different asymptotic circle lengths.
The presence of the periodicty condition (or more abstractly an integral lattice structure) breaks down scaling symmetry by singling out a special scale.
\end{eg}

\begin{eg}\label{compactificationnegativevertex}
(\textbf{Negative vertex}) By a similar argument, in the negative vertex setting (\cf Section \ref{Joycecritique})
 the curvature 2-form $F$ satisfies
\begin{equation}\label{distributioanlequationnegativevertex}
-\frac{1}{2\pi}dF=  -\frac{\sqrt{-1}}{4\pi} \left(  
\frac{\partial^2 W^{p\bar{q}}}{\partial \mu \partial \mu }
+
4\frac{\partial^2 V}{\partial \eta_p \partial \bar{\eta}_q }
\right) d\mu \wedge d\eta_p \wedge d\bar{\eta}_q=S,
\end{equation}
where $S=\{ z_1+z_2=1 \}=\{ e^{2\pi i \eta_1}+ e^{2\pi i \eta_2}=1 \}\subset \C^*_{z_1}\times \C^*_{z_2}\times \{0 \} \subset  \C^*_{z_1}\times \C^*_{z_2}\times \R_\mu$ defines a codimension 3 cycle. Here $S$ is endowed with the complex orientation, and the orientation on $\C^*_{z_2}\times \R_\mu$ is defined by the form $d\mu \wedge d\text{Re}\eta_1\wedge d\text{Im}\eta_1\wedge d\text{Re}\eta_2\wedge d\text{Im}\eta_2$.
\end{eg}

\section{Ooguri-Vafa metric}\label{OoguriVafa}

In this Section we will review the Ooguri-Vafa metric, based on Gross and Wilson \cite{GrossWilson}. The recent paper \cite{HeinSunViaclovsky} is an influence to our viewpoint, and the author thanks Song Sun for useful discussions.

\subsection{Gibbons-Hawking viewpoint}

The Ooguri-Vafa metric is an incomplete $S^1$-invariant hyperK\"ahler metric constructed via the \textbf{Gibbons-Hawking ansatz} (\cf Section \ref{GeneralisedGibbonsHawking} with $N=2, \mathfrak{n}=1$). In our normalisation conventions, the metric lives on the singular $S^1$-bundle $M\to \mathcal{B}\subset \R_\mu\times (S^1 \times \R)_\eta $ where the complex variable $\eta$ has period 1, and the $S^1$-fibre collapses to a point over the origin $(\mu, \eta)=0$. The first Chern class $c_1$ of the $S^1$-bundle evaluates to $-1$ on a sphere around the origin in $\R_\mu \times (S^1\times \R)_\eta$. The composition $M\to \mathcal{B}\xrightarrow{(\mu, \text{Im} \eta)} \R^2$ gives a singular $T^2$-fibration, and the periodicity condition on $\eta$ amounts to imposing $\int_{T^2} \Omega=2\pi$.

Let $A\gg 1$ be a large parameter. The Ooguri-Vafa metric can be thought as a perturbation of the constant solution (\cf Example \ref{Constantsolution}) which is encoded by
\[
g_A= A( d\mu^2+ |d\eta|^2  ).
\]
after incorporating some topology.
We denote $|(\mu, \eta)|= \sqrt{  \mu^2+ |\eta|^2     }$, and set 
\begin{equation}\label{OoguriVafapotential}
V(\mu, \eta)=W=A+ \frac{1}{2 |(\mu, \eta) |       } + \sum_{ n\in \Z\setminus \{0\} } \{\frac{1}{2|(\mu, \eta+n)|   }- \frac{1}{2|n|}
\}
\end{equation}
This series is convergent, 1-periodic in the $\eta$ variable, and satisfies the \textbf{Laplace equation} on $\mathcal{B}$ with distributional term which encodes simultaneously the \textbf{Calabi-Yau condition} and the topology:
\[
 \frac{1}{2\pi} \left(  
\frac{\partial^2 }{\partial \mu \partial \mu }
+
4\frac{\partial^2 }{\partial \eta \partial \bar{\eta} }
\right) V d\mu \wedge d \text{Re}\eta \wedge d\text{Im}{\eta}= -\delta_0,
\]
where $\delta_0$ is the delta measure at the origin in $\R_\mu\times S^1\times \R$. The metric on $M$
\[
g= V(   d\mu^2+ |d\eta|^2        )+ V^{-1}\vartheta^2
\]
is called the \textbf{Ooguri-Vafa metric}. Strictly speaking, the connection $\vartheta$ can be twisted by a \textbf{flat connection}, and this choice is parametrised by $H^1( \mathcal{B}\setminus\{0\}, S^1   )=H^1( \mathcal{B}, S^1   )=H^1(S^1\times \R^2, S^1)\simeq S^1$ using that a codimension 3 subset in the base does not affect the fundamental group. We sometimes suppress mentioning this choice as it does not affect the geometry significantly. 
By Remark \ref{T2symmetryimpliesSL} the $\mu, \text{Im} \eta$ coordinates define a \textbf{special Lagrangian fibration} with phase zero on $M$.

The Ooguri-Vafa metric has the important \textbf{exponential decay property} for $\mu^2+ (\text{Im}\eta)^2\geq 1$,
\begin{equation}\label{OoguriVafaexponentialdecay}
| V(\mu, \eta)- A+ \gamma_E-\log 2+ \frac{1}{2} \log (\mu^2+|\text{Im}\eta|^2) |\leq C \exp( -2\pi \sqrt{ \mu^2+ (\text{Im}\eta)^2 }   ) .
\end{equation}
where $\gamma_E=\lim_{n\to \infty} \sum_{k=1}^n \frac{1}{k}-\log n$ is the Euler constant.
For $\sqrt{\mu^2+ (\text{Im}\eta)^2}\gg 1$, the dependence of $V$ on the periodic $\text{Re}(\eta)$-variable decays exponentially, so up to exponentially small error the Ooguri-Vafa metric is asymptotic to 
a \textbf{semiflat metric}. This property is the main reason why the Ooguri-Vafa metric is useful for the gluing construction of Gross and Wilson \cite{GrossWilson}.  On the other hand $V$ becomes negative roughly when 
$\log (\mu^2+ |\text{Im}\eta|^2 )> 2A$, so the metric is only defined on a bounded set and is \textbf{incomplete}.

\subsection{Holomorphic viewpoint}\label{OoguriVafaholomorphicviewpoint}

As a hyperK\"ahler metric, the Ooguri-Vafa metric admits a 2-sphere of compatible integrable complex structures. There is one distinguished complex structure giving rise to a holomorphic elliptic fibration, which is described in detail in \cite{GrossWilson}. Here we wish to focus on another distinguished complex structure where $\eta$ is holomorphic and $\mu$ is the symplectic moment map, which is more natural for special Lagrangian fibrations. 
 The author is not aware of explicit references for the content of this Section.

Our main goal is to \textbf{identify the complex structure} on $M$ explicitly, which requires us to construct holomorphic functions on $M$. 
 We start with the type $(1,0)$ form $\zeta= Vd\mu+ \sqrt{-1} \vartheta$ and recall formula (\ref{GibbonsHawkingholomorphicdifferential}). To turn $\zeta$ into a holomorphic differential, we need to subtract a function times $d\eta$, whose differential cancels out $d\zeta$. Inspired by the Taub-NUT example \ref{TaubNUT}, and taking care of periodicity requirement, we introduce the functions
\[
\begin{cases}
\beta_+=\frac{\pi\sqrt{-1}}{2}- \sqrt{-1}\theta_{\infty}+\lim_{k\to \infty} \sum_{n=-k}^{n=k} \{ 
\frac{1}{2(\eta+n)} - \frac{ \mu}{ 2(\eta+n) \sqrt{ \mu^2+ |\eta+n|^2  }    }
\} \\
\beta_-=\frac{\pi\sqrt{-1}}{2}+\sqrt{-1}\theta_{\infty}+\lim_{k\to \infty} \sum_{n=-k}^{n=k} \{ 
\frac{1}{2(\eta+n)} +\frac{ \mu}{ 2(\eta+n) \sqrt{ \mu^2+ |\eta+n|^2  }    }
\}
\end{cases}
\]
These series are convergent and 1-periodic in $\eta$, such that  the forms $\zeta'= \zeta+ \beta_+ d\eta$, $\zeta''=-\zeta +\beta_- d\eta$ are closed. The real number $\theta_{\infty}$ is chosen to cancel the asymptotic holonomy of the $S^1$-connection $\vartheta$ along the $\text{Re}(\eta)$-circle as $\text{Im}\eta\to \infty$. We have
\[
\begin{split}
\zeta'+ \zeta''&=\pi\sqrt{-1} d\eta+\lim_{k\to \infty} \sum_{n=-k}^k \frac{1}{ \eta+n } d\eta \\
&=\pi \sqrt{-1}d\eta+ d \log ( \eta \prod_{n=1}^\infty (1- \frac{\eta^2}{n^2})      ) \\
&= \pi \sqrt{-1}d\eta+ d\log ( \frac{\sin (\pi\eta)}{\pi }    )
= d\log ( 1- e^{2\pi i \eta}   ),
\end{split}
\]
where we made use of Euler's factorisation identity of $\frac{\sin(\pi \eta)}{\pi \eta}$. The line integrals $\int \zeta'$ and $\int \zeta''$ are locally holomorphic functions on $M$, defined over the complement of $\{\mu\leq 0, \eta=0 \}$ and $\{ \mu\geq 0, \eta=0  \}$ inside $\mathcal{B}$. Their $T^2$-periods lie in $2\pi \sqrt{-1} \Z$: the periods along the $S^1$-fibre over $\mathcal{B}$ is $\int_{S^1} \sqrt{-1}\vartheta= 2\pi \sqrt{-1}$, while the periods along the $S^1$-cycle in $M$  lifting  $S^1\subset \R\times S^1\times \R$ can be evaluated by their asymptotic value as $\text{Im}(\eta)\to +\infty$; in particular if we twist the connection $\vartheta$ by a flat connection, then we can always use the choice of $\theta_\infty$ to cancel that twist. Thus we can define the \textbf{holomorphic functions} without multivalue issues
\[
z_1= \exp( \int \zeta'    ), \quad z_2= \exp(   \int \zeta''     ).
\]
We are free to choose the multiplicative constants on $z_1, z_2$ to satisfy the \textbf{functional equation} $z_1 z_2= 1- e^{2\pi i \eta}$, from which we see that $z_1, z_2$ extend to global holomorphic functions on $M$, with zero locus $\{ \mu\leq 0, \eta=0    \}$ and $\{  \mu\geq 0, \eta=0      \}$ inside $\mathcal{B}$ respectively.
Setting $z_3= \exp( 2\pi \sqrt{-1} \eta   )$, we obtain a holomorphic map
\[
M\to \{  z_1 z_2= 1-z_3     \}\subset \C^2_{z_1, z_2}\times \C^*_{z_3},
\]
which is easily seen to be an open embedding.

By looking at the action of the Hamiltonian vector field $\frac{\partial}{\partial \theta}$, we can identify the \textbf{circle action} as
\[
e^{i\theta} \cdot (  z_1, z_2, z_3)= ( e^{i\theta}z_1, e^{-i \theta} z_2, z_3).
\]
By construction the \textbf{holomorphic volume form} $\Omega$ satisfies  $\iota_{ \frac{\partial}{\partial \theta} }\Omega= d\eta$, which implies
$
\Omega= - \frac{1}{2\pi } \frac{1}{z_3} dz_1\wedge dz_2,
$
or equivalently 
\[
\Omega\wedge d( (z_1z_2)^{-1}(1-z_3)-1      )= \frac{1}{2\pi} d\log z_1\wedge d\log z_2 \wedge d\log z_3.
\]
The reader is advised to compare this discussion to Section \ref{Degeneratingtorichypersurface}.

\begin{rmk}
The viewpoint taken here starts with geometry, and the algebraic structure on the holomorphic functions only emerges a posteriori as a consequence of functional equations on transcendental integrals. This is conceptually rather similar to elliptic curves where algebraic relations arise from theta functions.
\end{rmk}

\subsection{Some conceptual aspects of the Ooguri-Vafa metrics}\label{conceptualaspectsofOoguriVafa}

The Ooguri-Vafa metric comes with an intrinsic parameter $A$, and admits different geometric behaviours at different \textbf{scales}, which can be formalised in terms of \textbf{blow up limits}. Recall by periodicity we may assume $\text{Re} (\eta)$ lies in some interval $[0,1]$. The periodicity condition we chose amounts to the normalisation that $\int_{T^2}\Omega=2\pi$.
\begin{itemize}
\item   When $\sqrt{\mu^2+ |\eta|^2} \lesssim \frac{1}{A}$, the leading order behaviour is $V\sim A+ \frac{1}{ 2\sqrt{\mu^2+ |\eta|^2  }    }$, and the metric is modelled on the \textbf{Taub-NUT} metric with parameter $A$. The length of the  circle fibres $\sim \frac{2\pi}{\sqrt{A}}$. After scaling up the metric by a factor $A$ and taking the limit $A\to \infty$, the pointed Gromov-Hausdorff limit based at the origin is the standard Taub-NUT metric with parameter 1. Most Riemannian curvature is concentrated in this region.
\item 
When $\frac{1}{A}\ll \sqrt{\mu^2+ |\eta|^2} \ll 1$, the leading order behaviour is the \textbf{constant solution} $V\sim A$, and the metric is locally modelled on a flat circle bundle over a flat base $\R^3$. A suitable blow up limit space is flat $\R^3$.
\item 
When $\sqrt{\mu^2+ |\eta|^2}\sim 1$, the leading order behaviour is still $V\sim A$, but the \textbf{periodicity} condition is now visible.  The metric is locally modelled on a flat circle bundle over a flat base $\R_{\mu}\times S^1\times \R$. The length of the circle factor of the base is approximately $\sqrt{A}$. If we scale down the metric by a factor $\frac{1}{A }$ and take the limit $A\to \infty$, the pointed Gromov-Hausdorff limit based at the origin is the flat $\R_{\mu}\times S^1\times \R$.
\item 
When $1\ll \sqrt{\mu^2+ |\eta|^2}\ll  \exp( A   ) $, the metric becomes \textbf{almost semiflat up to exponentially small errors}, and the leading order behaviour is \[
V\sim A+\log 2-\gamma_E- \frac{1}{2} \log (\mu^2+ |\text{Im} (\eta)|^2).
\]
We remark that the log function grows very slowly. Thus within an \textbf{exponentially long neck region}, the constant solution $V\sim A$ is a good approximation. If we view $A$ as related to the average length of circles, then we can think of $V$ as approximated by a family of constant solutions whose parameter slowly drifts down as we move up the logarithmic scale.
The author finds it attractive to call this phenomenon \textbf{running coupling}.
\item 
When $\sqrt{\mu^2+ |\eta|^2}\sim \exp(A)$ the \textbf{incompleteness} of the metric is manifested. This is best understood by viewing the Ooguri-Vafa metric as an \textbf{effective local description} of the hyperK\"ahler metric on a family of collapsing K3 surfaces, and incompleteness is an indication that there is a scale beyond which this description must break down. On the other hand, if we are looking at smaller distance scales, then the Ooguri-Vafa metric becomes better approximations of the K3 metric. In particular, the K3 hyperK\"ahler structure involves 60 parameters while the Ooguri-Vafa metric only involves one scaling parameter $A$ and a gauge parameter in $S^1$, but the metric at smaller distance scales are not sensitive to many extra parameters as long as the K3 surfaces are sufficiently collapsed. This phenomenon may be called \textbf{effective uniqueness} or \textbf{local universality}, which is an essential aspect of Gross and Wilson's gluing construction \cite{GrossWilson}. The  analogy with K. Wilson's philosophy of effective quantum field theory will be further explained in Section \ref{runningcoupling}.
\end{itemize}

The analysis of the blow up limits reveals the \emph{cause d'etre} of the Ooguri-Vafa metric. Recall the Taub-NUT metrics arise in a 1-parameter family, which are related by the \textbf{scaling symmetry}. The Ooguri-Vafa metric is obtained conceptually by \textbf{gluing the Taub-NUT metric to the constant solution}. The \textbf{periodicity condition}, which is a kind of integral lattice structure, \textbf{breaks down the scaling symmetry}, and results in an intrinsic \textbf{gluing parameter} $A$. Another major effect of the periodicity condition is the \textbf{exponential decay of higher Fourier modes}, which works via spectral theory, and results in the semiflat asymptotic picture.

A notable feature in the Ooguri-Vafa metric is the appearance of the Green's function $V$. This is because we are \textbf{perturbing from the constant solution}, and the first order correction to the flat ambient solution natually invovles \textbf{harmonic} functions at least away from the singular locus. The precise nature of the \textbf{singularity} of $V$ is dictated by the \textbf{topology}, or more precisely the Chern class, via the \textbf{distributional equation}.

These principles are sufficient to lead to the discovery of the Ooguri-Vafa metric. While the exact linearity of the equation governing the Gibbons-Hawking ansatz in complex dimension 2 is a fortunate simplifying feature, it does not appear essential in our discussions above. A core idea in this paper is that essentially the same principles dictate how to generalise the Ooguri-Vafa metric to dimension 3.

\section{Synopsis of the new metrics}

A central theme running through this paper is the strong analogy between the Taub-NUT type metric on $\C^3$ and the Ooguri-Vafa type metrics on the positive/negative vertices (\cf Theorem \ref{TaubNUTtypeC3theoremintro}, \ref{Positivevertextheoremintro} and \ref{Negativevertextheoremintro}). The purpose of this Section is to give a unified view on our strategy to produce these three types of new metrics, which involves a geometric part concerning the construction of an ansatz, and an analytic part concerning perturbing the ansatz into a Calabi-Yau metric. Many aspects of these new metrics naturally generalise features of the Taub-NUT metric (\cf Example \ref{TaubNUT}) and the Ooguri-Vafa metric (\cf Section \ref{OoguriVafa}).

\subsection{Geometric aspects}


All three types of metric ansatzs are constructed in the generalised Gibbons-Hawking framework, by \textbf{perturbing from the constant solution} after incorporating \textbf{topology}. The constant solutions in Example \ref{Constantsolution} serve as zeroth order approximations to the metric ansatz, and can be thought as scaling limits (\cf Section \ref{conceptualaspectsofOoguriVafa}). Geometrically they describe a flat torus fibration fibred over a Euclidean base with \emph{distinguished coordinates} related to moment maps. The choice of this Euclidean metric is parametrised by a positive definite rank 2 real symmetric or Hermitian matrix, depending on the 3 cases. The principal difference between the Taub-NUT type metric on $\C^3$ and the Ooguri-Vafa type metrics is that the bases in the latter cases have periodic directions.

In order to build in the Gross-Ruan-Joyce topology (\cf Section \ref{GrossRuanJoyce}) we need to make \textbf{first order corrections} to the constant solutions. Recall the generalised Gibbons-Hawking construction involves three sets of equations:
\begin{itemize}
\item The \textbf{integrability condition} is responsible for the integrability of the complex structure and the K\"ahler condition.
\item The \textbf{distributional equation} captures the topology and the discriminant locus.
\item The \textbf{Calabi-Yau condition} is the only nonlinear equation.
\end{itemize}
It is natural to impose that the first order corrections satisfy the linearised version of these equations; in particular the linearisation of the Calabi-Yau condition gives rise to \textbf{harmonic} functions. These linearised equations combine into a coupled overdetermined system. Our method to solve this system is to first determine by educated guess the \textbf{singularities} of the harmonic functions along the discriminant locus, explicitly construct such harmonic functions using \textbf{Green's representation}, and then verify the other equations in the overdetermined system by means of Liouville theorem type arguments. The first order corrections we obtain are canonical (up to constants) under mild growth constraints. For the Taub-NUT type $\C^3$ case the first order corrections admit elementary formulae. For the Ooguri-Vafa type metrics on the vertices the first order corrections involve infinite series and Green representation integrals, which are a priori divergent but become convergent after subtracting logarithmically divergent terms, much like what happens already for the Ooguri-Vafa metric.

We then extract various \textbf{asymptotes} of the first order ansatz. Transverse to the discriminant locus, the leading asymptotes can be interpreted geometrically as giving rise to \textbf{Taub-NUT metrics}; ultimately this is forced on us by the distributional equation coming from the topology. In the Ooguri-Vafa type situations, we can also perform \textbf{Fourier analysis} in the periodic variables. Suitably away from the discriminant locus, the zeroth Fourier mode is the dominant contribution, giving rise to a \textbf{semiflat metric}. The harmonicity condition implies that the higher Fourier modes satisfy Helmholtz equations, thereby \textbf{decay exponentially}. An additional problem in the Ooguri-Vafa type situations is that the metric ansatzs are only positive definite on a bounded region, whereby metrically \textbf{incomplete}.

The strategy to \textbf{identify the holomorphic structure} is to produce holomorphic differentials with integral periods, in a manner similar to the Taub-NUT metric and the Ooguri-Vafa metric (\cf Section \ref{OoguriVafaholomorphicviewpoint}). The functional equation satisfied by the holomorphic functions allows us to identify the holomorphic volume form. It should be emphasized that while topology is built a priori into the generalised Gibbons-Hawking construction, the holomorphic structure is a nontrivial a posteriori consequence.

The first order corrections are small perturbations suitably away from the discriminant locus, but near the discriminant locus they are large compared to the constant solution. This explains why the first order metric ansatz is \textbf{approximately Calabi-Yau} suitably away from the discriminant locus. In the suitable weighted H\"older norms this approximation continues to hold good near the discriminant locus, except on small balls near the origin
in the Taub-NUT type $\C^3$ case and the 
positive vertex case. Geometrically this problem is caused by the 3 edges of $\mathfrak{D}$ interacting strongly at their intersection point. The same problem does not appear on the negative vertex because the discriminant locus $S$ has no singular point.

Our strategy trifurcates at this point. The small ball is a \emph{fully nonlinear} region in which linear approximation methods fail completely. In the case of the \textbf{Taub-NUT type metric on $\C^3$}, we instead shift to the complex geometric perspective, and solve the \textbf{complex Monge-Amp\`ere} equation with prescribed asymptotes at infinity. This is viable because the exterior of the small ball does admit an approximately Calabi-Yau ansatz. The output is a Calabi-Yau metric on $\C^3$ whose deviation from the first order ansatz satisfies an asymptotically good estimate.

The \textbf{Ooguri-Vafa type metric on the  positive vertex} is best thought as the periodic version of the Taub-NUT type metric on $\C^3$, and is obtained by gluing the Taub-NUT type $\C^3$ to the first order ansatz on the positive vertex. The \textbf{periodicity condition} breaks down the scaling symmetry of the Taub-NUT type metrics, and instead results in the \textbf{gluing picture}, exactly analogous to the relation between the Taub-NUT metric and the usual Ooguri-Vafa metric. The nonlinear effect on the positive vertex is already fully present on the Taub-NUT type $\C^3$. It is worth comparing with the topological prediction of Gross-Ruan (\cf Section \ref{Positivevertices}) where the neighbourhood of the origin is modelled on $\C^3$ with a $T^2$ fibration related to the Harvey-Lawson example \ref{HarveyLawson}. But for metric purposes we need to use an \textbf{exotic Calabi-Yau metric} on $\C^3$, rather than the Euclidean $\C^3$.

The \textbf{Ooguri-Vafa type metric on the negative vertex}, on the other hand, is constructed entirely \textbf{perturbatively} from the first order ansatz.

\subsection{Analytic aspects}

The analytic step is aimed at perturbing the first order ansatz into a genuine Calabi-Yau metric, and the techniques involved overlap substantially in all three cases.

A central issue, roughly put, is to produce a parametrix for the right inverse to the Laplacian with accurate control on weighted H\"older norm estimates. Some of the main difficulties are:
\begin{itemize}
\item
The first order corrected metric is \emph{multiscaled}, namely it has very different characteristic behaviours in different regions and at different length scales. 
\item 
The initial error \emph{decays slowly}. 
\end{itemize}
The core idea in our methodology is \emph{divide and conquer}. We decompose the source function according to its support. The contribution supported sufficiently away from the discriminant locus is inverted approximately using the Euclidean Green operator, reflecting the fact that the constant solution is the zeroth order approximation to the metric ansatz.   Afterwards the source function is effectively supported near the discriminant locus. We then use a Green operator adapted to the Taub-NUT fibration near the discriminant locus to cure the remaining source.

We now turn to specifics. The 3 cases are arranged in pedagogical order, and each case contains most difficulties of previous cases. As a general policy, detailed proofs will be omitted if the main techniques appeared previously.

In the \textbf{Taub-NUT type $\C^3$} case, the parametrix is used to improve the approximation to the Calabi-Yau condition \emph{asymptotically} outside a compact region. Once the decay of the approximation error is \emph{sufficiently fast}, we can appeal to a non-compact version of Yau's solution to the Calabi conjecture, developed in H-J. Hein's thesis \cite{Heinthesis}, to turn the ansatz into a genuine Calabi-Yau metric with effective estimates.

Here a difficulty caused by the \emph{slow decay} of error is that the inverse of the Laplacian is not well behaved in the weighted H\"older spaces. Instead it is preferable to work with the \emph{zeroth order operator} $\nabla^2 \Lap^{-1}$, which controls how to correct a K\"ahler metric for a given amount of volume form error. The advantage is that this operator maps between function spaces with the same H\"older weights, the operator norm is not affected by rescaling the metric, and crucially the Schwartz kernel has two extra order of decay  compared to $\Lap^{-1}$.

In the \textbf{positive vertex} case, the main \emph{new} difficulty is to prove \emph{exponential decay of higher Fourier modes}. This comes down to mapping properties of the periodic Euclidean Green operator, ultimately thanks to the exponential decay of the higher Fourier modes of the periodic Newtonian potential.

The second new difficulty is that that the volume form error \emph{does not decay}, and in fact grows logarithmically at large distance, causing problem for perturbation theory over an exponentially long region. The strategy is to first correct the error inside the generic region in the generalised Gibbons-Hawking framework, using the periodic Green operator. We then switch to the complex geometric viewpoint and solve the complex Monge-Amp\`ere equation perturbatively, which avoids the difficulty of the generalised Gibbons-Hawking equation near the discriminant locus.

The third new difficulty comes from \emph{metric incompleteness}: the Laplacian has no good mapping property in the na\"ive weighted H\"older spaces.  In our approach, this means the parametrix is only defined on compactly supported sources, but the outputs are generally not compactly supported. A formal trick  called \emph{extension norms} \cite{Gabor} effectively allows us to assume the source is compactly supported. This circumvents the need to impose a non-canonical  boundary condition.

In the \textbf{negative vertex} case, the main \emph{new} difficulty comes from the \emph{curved} nature of the discriminant locus $S$, making it harder to produce a parametrix near $S$. A closely related issue is that \emph{there is no obvious a priori choice of  smooth topology} such that the first order metric ansatz is smooth along $S$. These problems force us to work in weighted H\"older spaces with \emph{low regularity}, in which it makes no sense to speak of an arbitrarily high order of differentiability. Crucially there is enough regularity to make the Laplacian well defined.
The \emph{smooth topology emerges a posteriori} only after solving the complex Monge-Amp\`ere equation. The solution itself defines a complex structure, hence induces a smooth topology, and the compatibility of the metric with this smooth topology is a consequence of the well known regularity theory for complex Monge-Amp\`ere equation.

\subsection{Outlook: towards the SYZ conjecture}

We now explain how this paper fits into a program to prove the metric version of the SYZ conjecture for Calabi-Yau 3-folds (\cf Conjecture \ref{SYZconjecture}). This program runs as follows:
\begin{enumerate}
\item Produce the \textbf{metric models} on the positive and negative vertices. 

\item The metric structure near the \textbf{edges} in the Gross-Ruan picture are expected to be modelled on a fibration by Ooguri-Vafa metrics. The problem is that Ooguri-Vafa metrics transverse to the edge depend on a moduli parameter which can vary along the edge, possibly governed by an adiabatic equation.

\item The SYZ base $B$ as an \textbf{affine manifold with singularity} along a trivalent graph, can be produced from algebraic geometry in some degree of generality \cite{Zharkov2}\cite{KontsevichSoibelman}. The central problem is then to solve the \textbf{real Monge-Amp\`ere equation} with some \textbf{prescribed singularities} along the trivalent graph. This would allow us to produce a semiflat metric which models the generic region of the SYZ fibration.

\item  One then \textbf{glues} together the metric models in various regions to obtain the global Calabi-Yau metric on the Calabi-Yau 3-fold, similar to Gross and Wilson's work on K3 surfaces \cite{GrossWilson}. Some Fourier analysis is needed to prove \textbf{exponential decay estimates} for deviation from the semiflat metric.

\item The existence of the \textbf{SYZ fibration} in the generic region is expected to be a straightforward consequence of the gluing construction. To produce the SYZ fibration near the trivalent graph, one needs to produce models for singular SYZ fibrations on the metric models, and set up a Fredholm deformation theory to ensure the SYZ fibration persists when the metric deforms. 
\end{enumerate}

The principal contribution of this paper is to carry out Step (1), and our linear analysis is likely to be useful in Step (4). Some informal digressions in this paper go some way towards addressing difficulties in the other Steps:

In Step (3), the singularity of the real Monge-Amp\`ere equation near the trivalent graph in $B$ should match up with the asymptotic behaviour of the metric models around the trivalent graph, in order to enable the gluing construction in Step (4). This requires understanding how the Ooguri-Vafa type metrics on the vertices \emph{transition} into the generic region of the SYZ fibration. We propose a mechanism called \textbf{running coupling} for this transition to take place over an exponentially long neck region  (\cf Section \ref{runningcoupling} and \ref{runningcouplingnegativevertex}). Starting from the observation that Ooguri-Vafa type metrics naturally arise in a family parametrised by some positive definite rank 2 matrices referred to as \emph{coupling constants}, we argue semi-heuristically that these coupling constants drift slowly as the logarithmic scale increases, governed by an ODE called the \emph{renormalistion flow equation} which can be solved exactly.

The behaviour of the \textbf{special Lagrangian fibrations} is discussed in Corollary \ref{specialLagrangianfibrationC3}, Corollary \ref{SpeicalLagrangianfibration+vertexCor} and Section \ref{SpecialLagrangiangeometry}. In both the Taub-NUT type $\C^3$ case and the positive vertex case, the $T^2$-symmetry provides two symplectic moment coordinates and another real coordinate $\text{Im}(\eta)$, which define a map to $\R^3$ whose fibres are $T^2$-invariant special Lagrangians with phase zero. However, Joyce's critique suggests the singularity structure of this SYZ fibration is not stable under metric perturbation.

In the negative vertex case (\cf Section \ref{SpecialLagrangiangeometry}), there is a homological constraint for the SYZ fibration to exist, namely the Hermitian matrix $a_{p\bar{q}}$ needs to be symmetric. When this constraint holds, we outline a speculative description of a $U(1)$-invariant SYZ fibration on the model metric, and explain how it fits with Joyce's work on $U(1)$-invariant special Lagrangians. The case where this constraint does not hold is possibly relevant for metric degenerations outside the scope of the SYZ conjecture.

\chapter{Taub-NUT Type Metrics on $\C^3$}\label{TaubNUTtypemetriconC3}

In this Chapter we will construct via the generalised Gibbons-Hawking ansatz a 3-parameter family of new complete Calabi-Yau metrics on $\C^3$ equipped with the usual holomorphic volume form, which can be thought as the analogue of the  Taub-NUT metric in complex dimension 3. This metric is symmetric under the diagonal $T^2$-action,
\[
e^{i\theta_1}\cdot (z_0, z_1, z_2)= (e^{-i\theta_1}z_0, e^{i\theta_1} z_1, z_2), 
\quad 
e^{i\theta_2}\cdot (z_0, z_1, z_2)= (e^{-i\theta_2}z_0, z_1, e^{i\theta_2} z_2),
\]
and its tangent cone at infinity is the flat Euclidean space of dimension 4. On most part of the manifold $\C^3$ the metric is approximated by the constant solution (\cf Example \ref{Constantsolution}). Near the locus where the $T^2$-fibres degenerate and sufficiently far from the origin, the metric is locally modelled on a Taub-NUT fibration.

The basic method is to construct an approximate metric near infinity and then use a modification of the analytic package of H-J. Hein \cite{Heinthesis} to construct the global Calabi-Yau metric. It requires sufficient understanding of the Green operator to correct the error terms near infinity.
This method has a very similar flavour to the recent papers \cite{Li}\cite{Gabor}\cite{Ronan} which construct new Calabi-Yau metrics on $\C^n$ starting from a holomorphic fibration structure (\cf Section \ref{ComparisonwithGabor}).

The organisation is as follows. Section \ref{C3asymptoticmetric} introduces the first order ansatz, which prescribes the asymptote at infinity. Section \ref{MetricbehaviourawayfromDelta} and \ref{Structureneardiscriminantlocus} interprets the construction geometrically in terms of local models. Section \ref{C3complexgeometricperspective} identifies the holomorphic structure by proving the functional equation, and Section \ref{Algebraicgeometricperspective} explains how algebraicity arises from the ring of holomorphic functions with controlled growth. Section \ref{surgery} mollifies the K\"ahler metric and the moment map in a compact region to ensure smoothness, and gives estimates on the initial volume form error. These Sections are written with an overall geometric orientation.

The next few Sections are devoted to analysis. Section \ref{Heinpackage} explains the key points in Hein's analytic packages. Section \ref{Harmonicanalysis} develops the mapping property of the Green operator in weighted H\"older spaces, by a decomposition and patching strategy. This linear analysis is utilized to correct the volume form error asymptotically, leading to the main existence result in Section \ref{PerturbationintoCYC3}, where we also discuss salient geometric features such as the tangent cone at infinity, the decay property of the Riemannian curvature, and the special Lagrangian fibration. In Section \ref{Uniquenessandmoduli} we prove uniqueness within some asymptotic classes defined by decay conditions, using ideas of Conlon and Hein \cite{ConlonHein}. This enables us to determine the moduli of our construction.

The last Section \ref{ComparisonwithGabor} is a panoramic view on exotic complete Calabi-Yau metrics on $\C^n$ for $n\geq 3$, and advocates for the potential of future research in this area.

\section{First order asymptotic metric near infinity}\label{C3asymptoticmetric}

We plan to construct an approximate Calabi-Yau metric using the \textbf{generalised Gibbons-Hawking ansatz}, on a singular $T^2$-bundle $M$ over (the complement of a compact subset of) the real 4-dimensional base $\R^2_{\mu_1, \mu_2}\times \C_\eta$, whose discriminant locus is 
\begin{equation}\label{C3Delta}
\begin{split}
\mathfrak{D}& = \mathfrak{D}_1\cup \mathfrak{D}_2\cup \mathfrak{D}_3\cup \{0\} =
\{ \mu_1=0, \mu_2> 0   \}\cup \{ \mu_2=0, \mu_1> 0    \}\cup \{  \mu_1=\mu_2< 0   \}\cup \{0\} \\
&\subset  \R^2_{\mu_1, \mu_2}\times\{0\}\subset \R^2_{\mu_1, \mu_2}\times \C_\eta\simeq \R^4.
\end{split}
\end{equation}
The topological situation is the same as in the Harvey-Lawson Example \ref{HarveyLawson} in complex dimension 3. Our primary concern is that this metric should be approximately Calabi-Yau near spatial infinity, while on a compact set this approximation is allowed to fail.

The basic heuristic idea is to \textbf{perturb the constant solution} (\cf Example \ref{Constantsolution}) in a way which incorporates the \textbf{topology}. The information in the constant solution is contained in the base metric
\[
g_a= a_{ij} d\mu_i \otimes d\mu_j + A |d\eta|^2,
\]
where $(a_{ij})$ is a real symmetric positive definite $2\times 2$ matrix with inverse matrix $(a^{ij})$, and $A=\det a$. The matrix $a^{ij}$ describes the asymptotic metric on the $T^2$-fibres. The associated volume measure is \[
d\text{Vol}_a= A^{3/2} d\mu_1\wedge d\mu_2 \wedge d\text{Re}\eta \wedge d\text{Im}\eta .\] Now in terms of the local potential $\Phi$ the Calabi-Yau condition (\ref{GibbonsHawkingCY}) reads
\[
\det ( \frac{ \partial^2 \Phi}{\partial \mu_i\partial \mu_j}   )= -4 \frac{ \partial^2 \Phi}{\partial \eta\partial \bar{\eta}} ,
\]
whose linearised equation at the constant solution is the Laplace equation
\[
\Lap_a \phi= a^{ij} \frac{\partial^2\phi}{ \partial \mu_i \partial \mu_j }+ 4A^{-1} \frac{\partial^2 \phi}{\partial \eta \partial \bar{\eta}}=0. 
\]
Here $\Lap_a$ is unsurprisingly the Laplacian of $g_a$. This suggests that at least away from the discriminant locus, the first order correction to $V^{ij}$ and $W$ from the constant solution
\begin{equation}\label{GibbonsHawkinglinearisedC3case1}
v^{ij}= \frac{\partial^2\phi}{ \partial \mu_i \partial \mu_j }, \quad w= -4\frac{\partial^2\phi}{ \partial \eta \partial \bar{\eta} }
\end{equation}
ought to be given by \textbf{$\Lap_a$-harmonic functions},
\begin{equation}\label{GibbonsHawkinglinearisedC3case2}
\Lap_a v^{ij}=0, \quad \Lap_a w=0, \quad Aa^{ij} v^{ij}=w.
\end{equation}
To incorporate the topology we need to recall the distributional equation (\ref{Positivevertexdistribution}) on $V^{ij}$ and $W$. Since $v^{ij}$ and $w$ are linearisations, it makes sense to require the equation on currents
\begin{equation}\label{distributionalequationlinearised}
\frac{\sqrt{-1}}{4\pi} \left(  
\frac{\partial^2 w}{\partial \mu_i \partial \mu_j }
+
4\frac{\partial^2 v^{ij}}{\partial \eta \partial \bar{\eta} }
\right) d\mu_i \wedge d\eta \wedge d\bar{\eta}\otimes e_j= \mathfrak{D}_1\otimes e_1- \mathfrak{D}_2 \otimes e_2+ \mathfrak{D}_3\otimes (e_2-e_1).
\end{equation}
The task is to find a compatible solution to (\ref{GibbonsHawkinglinearisedC3case1})(\ref{GibbonsHawkinglinearisedC3case2})(\ref{distributionalequationlinearised}).
Notice that $v^{ij}$ and $w$ are global quantities while $\phi$ is only local. We view $v^{ij}$ and $w$ as the unknown functions in this system of equations, and the existence of a local $\phi$ solving (\ref{GibbonsHawkinglinearisedC3case1}) is equivalent to some integrability conditions on $v^{ij}$ and $w$ away from $\mathfrak{D}$,
\begin{equation}\label{GibbonsHawkinglinearisedC3case4}
\frac{\partial v^{11}}{\partial \mu_2}= \frac{\partial v^{12}}{\partial \mu_1},\quad  \frac{\partial v^{22}}{\partial \mu_1}= \frac{\partial v^{21}}{\partial \mu_2},
\end{equation}
and 
\[
\frac{\partial^2 w}{\partial \mu_i \partial \mu_j}= -4 \frac{\partial^2 v^{ij}}{\partial \eta \partial \bar{\eta}}.
\]

\begin{rmk}
(Informal discussion on \textbf{singularity})
The $v^{ij}$ and $w$ should have very specific singularities along $\mathfrak{D}_1$, $\mathfrak{D}_2$, $\mathfrak{D}_3$. Let us focus on what happens around $\mathfrak{D}_1$. The delta forcing term appears in the component of (\ref{distributionalequationlinearised}) as
\[
\frac{\sqrt{-1}}{4\pi} \left(  
\frac{\partial^2 w}{\partial \mu_1 \partial \mu_1 }
+
4\frac{\partial^2 v^{11}}{\partial \eta \partial \bar{\eta} }
\right) d\mu_1 \wedge d\eta \wedge d\bar{\eta}= \mathfrak{D}_1.
\]
If we denote the Lebesgue measure $f\mapsto \int_{\mathfrak{D}_1} fd\mu_2$ as $\int_{\mathfrak{D}_1} d\mu_2$, we may rewrite this equation as
\[
\frac{1}{2\pi} \left(  
\frac{\partial^2 w}{\partial \mu_1 \partial \mu_1 }
+
4\frac{\partial^2 v^{11}}{\partial \eta \partial \bar{\eta} }
\right) d\mu_1 \wedge d\mu_2 \wedge d \text{Re}\eta \wedge d\text{Im}{\eta}= -\int_{\mathfrak{D}_1} d\mu_2.
\]
Since $v^{12}, v^{22}$ do not see the forcing term, our best guess is that they are smooth along $\mathfrak{D}_1$. Then modulo smooth terms $w\sim Aa^{11} v^{11}= a_{22} v^{11}$ along $\mathfrak{D}_1$, from which the distributional equation gives the singularity structure along $\mathfrak{D}_1$:
\[
v^{11}\sim  \frac{1}{2 \sqrt{\mu_1^2+ a_{22} |\eta|^2 }   }, \quad w\sim  \frac{a_{22}}{2 \sqrt{\mu_1^2+ a_{22} |\eta|^2 }   }.
\]	
The following base metric encoding the Gibbons-Hawking data has the singularity structure along $\mathfrak{D}_1$
\[
\begin{split}
& g_a+ v^{ij}d\mu_i d\mu_j+ w|d\eta|^2 \sim g_a+  \frac{1}{2 \sqrt{\mu_1^2+ a_{22} |\eta|^2 }   } (d\mu_1^2  + a_{22} |d\eta|^2  ) \\
=& (   \frac{1}{2 \sqrt{\mu_1^2+ a_{22} |\eta|^2 }   }+ \frac{A}{a_{22}}      ) (d\mu_1^2  + a_{22} |d\eta|^2  ) + a_{22}(  d( \mu_2+ \frac{a_{12}}{a_{22}} \mu_1 )       ) ^2  
\end{split}     
\]
from which we recognize the Taub-NUT metric appearing in directions transverse to $\mathfrak{D}_1$. See Section \ref{Structureneardiscriminantlocus} for further details.
\end{rmk}

We now move on to a more formal construction.

\begin{lem}
The functions
\begin{equation}
\begin{cases}
\alpha_1(\mu_1, \mu_2,\eta)=& \frac{1}{2 \sqrt{\mu_1^2+ a_{22}|\eta|^2 }   }\{
\frac{1}{2}+ \frac{1}{\pi} \arctan ( \frac{ a_{22}\mu_2+ a_{12} \mu_1}{ \sqrt{A} \sqrt{\mu_1^2+a_{22}|\eta|^2}   }   )
\},  \\
\alpha_2(\mu_1, \mu_2,\eta)=& \frac{1}{2 \sqrt{\mu_2^2+ a_{11}|\eta|^2 }   }\{
\frac{1}{2}+ \frac{1}{\pi} \arctan ( \frac{ a_{11}\mu_1+ a_{12} \mu_2}{ \sqrt{A} \sqrt{\mu_2^2+a_{11}|\eta|^2}   }   )
\},   \\
\alpha_3(\mu_1, \mu_2,\eta)
=& \frac{1}{2 \sqrt{(\mu_1-\mu_2)^2+ (a_{11}+2a_{12}+a_{22})|\eta|^2 }   } \\
 &\{
\frac{1}{2}+ \frac{1}{\pi} \arctan ( \frac{
	-a_{11}\mu_1-a_{12} \mu_2 -a_{21} \mu_1-a_{22}\mu_2
	 }{ \sqrt{A} \sqrt{(\mu_1-\mu_2)^2+(a_{11}+a_{12}+a_{21}+a_{22})|\eta|^2}   }   )
\} 
\end{cases}
\end{equation}
satisfy the equations on measures
\begin{equation}\label{C3harmonicawayfromDelta}
\begin{cases}
(\Lap_a \alpha_1) d\text{Vol}_a= -2\pi \sqrt{A} \int_{\mathfrak{D}_1}d\mu_2  ,\\
(\Lap_a \alpha_2) d\text{Vol}_a= -2\pi \sqrt{A} \int_{\mathfrak{D}_2}d\mu_1 , \\
(\Lap_a \alpha_3) d\text{Vol}_a= 2\pi \sqrt{A} \int_{\mathfrak{D}_3}d\mu_1
\end{cases}
\end{equation}
where the RHS are signed measures supported on $\mathfrak{D}_1, \mathfrak{D}_2, \mathfrak{D}_3$. Morever,
\begin{equation}\label{C3linearintegrability}
\frac{\partial \alpha_1}{\partial \mu_2}
= \frac{\partial \alpha_2}{\partial \mu_1}
=(-\frac{\partial }{\partial \mu_1}- \frac{\partial }{\partial \mu_2}  )\alpha_3
= \frac{\sqrt{A}}{2\pi |\vec{\mu}|_a^2  } .
\end{equation}
The singularity of $\alpha_i$ occurs along $\mathfrak{D}_i$ and  modulo smooth terms looks like
\[
\begin{cases}
\alpha_1\sim  \frac{1}{2 \sqrt{\mu_1^2+ a_{22}|\eta|^2 }   } , \\
\alpha_2\sim   \frac{1}{2 \sqrt{\mu_2^2+ a_{11}|\eta|^2 }   } , \\
\alpha_3\sim   \frac{1}{2 \sqrt{(\mu_1-\mu_2)^2+ (a_{11}+2a_{12}+ a_{22})|\eta|^2 }   } .
\end{cases}
\]
\end{lem}

\begin{proof}
We denote $\vec{\mu}=(\mu_1, \mu_2, \eta)$ and $|\vec{\mu}|_a=\sqrt{ a_{ij}\mu_i\mu_j+ A|\eta|^2}$.
The Green representation
\[
\begin{split}
&\frac{-1}{4\pi^2}\int_0^\infty \frac{1}{|\vec{\mu}-(0,s,0)|_a^2     } ds  \\
=&\frac{-1}{4\pi^2} \int_0^\infty \frac{1}{  a_{11}\mu_1^2-2a_{12}\mu_1(s-\mu_2)+a_{22} (s-\mu_2)^2+A|\eta|^2         } ds \\
=&\frac{-1}{4\pi^2} \int_{-\mu_2}^\infty \frac{1}{  a_{11}\mu_1^2-2a_{12}\mu_1 s+a_{22} s^2+A|\eta|^2         } ds    \\
=&  \frac{-1}{4\pi} \frac{1}{\sqrt{A}} \frac{1}{ \sqrt{\mu_1^2+ a_{22}|\eta|^2 }   }\{
\frac{1}{2}+ \frac{1}{\pi} \arctan ( \frac{ a_{22}\mu_2+ a_{12} \mu_1}{ \sqrt{A} \sqrt{\mu_1^2+a_{22}|\eta|^2}   }   )
\}
\\
=&  \frac{-1}{2\pi \sqrt{A}  } \alpha_1
\end{split}
\]
shows the equality of the two measures
\[
(\Lap_a \alpha_1) d\text{Vol}_a= -2\pi \sqrt{A} \int_{\mathfrak{D}_1}d\mu_2,
\]
and it is easy to check from this integral calculation
\[
\frac{\partial \alpha_1}{\partial \mu_2}=  \frac{\sqrt{A}}{ 2\pi|\vec{\mu}|_a^2 }.
\]
To see the singularity structure near $\mathfrak{D}_1=\{\mu_1=0, \eta=0, \mu_2>0\}$ explicitly, we can write
\[
\alpha_1= \frac{1}{ 2\sqrt{\mu_1^2+a_{22} |\eta|^2}    } \{   1- \frac{1}{\pi}\arctan ( \frac{ \sqrt{A} \sqrt{\mu_1^2+ a_{22} |\eta|^2 } }{ a_{22}\mu_2+ a_{12}\mu_1  }  )     \},
\]
and Taylor expand the arctan function.

The situations of $\alpha_2, \alpha_3$ are similar. A fast way to derive them by analogy is to remember that $\frac{\partial }{\partial \mu_2}, \frac{\partial }{\partial \mu_1}, -\frac{\partial }{\partial \mu_1}-\frac{\partial }{\partial \mu_2}$ are the directional vectors along $\mathfrak{D}_1, \mathfrak{D}_2, \mathfrak{D}_3$, and notice \[
\iota_{ \frac{\partial }{\partial \mu_2}   } g_a = d( a_{12} \mu_1+ a_{22}\mu_2   ), \quad  g_a(  \frac{\partial }{\partial \mu_2} ,   \frac{\partial }{\partial \mu_2}  )=a_{22}.
\]
\end{proof}

\begin{prop}\label{C3linearisedsolution}
(\textbf{First order linearised solution})
The equations defining $v^{ij}$ and $w$
\begin{equation}
v^{11}=\alpha_1+ \alpha_3, \quad v^{12}=v^{21}=-\alpha_3, \quad v^{22}=\alpha_2+ \alpha_3, \quad w=Aa^{ij} v^{ij}
\end{equation}
or equivalently
\[
v^{ij}d\mu_i\otimes d\mu_j= \alpha_1 d\mu_1^2+ \alpha_2 d\mu_2^2+ \alpha_3 (d(\mu_1-\mu_2))^2, \quad w=Aa^{ij} v^{ij}
\]
provide a solution to the integrability condition (\ref{GibbonsHawkinglinearisedC3case1}) and the harmonicity condition (\ref{GibbonsHawkinglinearisedC3case2}) away from $\mathfrak{D}$, which also satisfies the distributional equation (\ref{distributionalequationlinearised}) globally.
\end{prop}

\begin{proof}
The $\Lap_a$-harmonicity away from $\mathfrak{D}$ follows from (\ref{C3harmonicawayfromDelta}).  To see the distributional equation (\ref{distributionalequationlinearised}), it suffices to notice that both sides are $\Lap_a$-harmonic away from $\mathfrak{D}$, and the singularities along $\mathfrak{D}$ match up by construction.

The integrability condition
(\ref{GibbonsHawkinglinearisedC3case4})
is equivalent to (\ref{C3linearintegrability}). Together with the distributional equation this implies the local existence of the potential required by (\ref{GibbonsHawkinglinearisedC3case1}).
\end{proof}

\begin{rmk}
A Liouville theorem argument shows that
the solution $v^{ij}$ and $w$ to the linear system of equation (\ref{GibbonsHawkinglinearisedC3case1})(\ref{GibbonsHawkinglinearisedC3case2})(\ref{distributionalequationlinearised}) is \textbf{unique}, in the sense that if another solution differs from it by functions with some power law decay at infinity, then the two solutions agree. The key point is to analyse the difference of the two solutions, and observe that now there is no forcing term in the distributional equation, so the $\Lap_a$-harmonicity extend across the discriminant locus. 
\end{rmk}

Now we define the \textbf{K\"ahler ansatz} $(g^{(1)}, \omega^{(1)}, J, \Omega)$ via the generalised Gibbons-Hawking construction, using the functions
\[
V^{ij}_{(1)}=a_{ij}+ v^{ij}, \quad W_{(1)}= A+ w.
\]
Here the script $(1)$ signifies first order approximation. The complex structure and the holomorphic volume form are not scripted because they turn out to agree with the standard structures on $\C^3$ and will not be corrected in a later stage. Notice that the positive definite condition on $V^{ij}$ is implied by $\alpha_i\geq 0, i=1,2,3$, which can be checked from the explicit formula. A grain of salt is that there is no a priori guarantee that the metric is smooth over the discriminant locus, a problem we shall take up in Section \ref{Structureneardiscriminantlocus}.

Let us examine the approximation to the Calabi-Yau condition. This is measured by the \textbf{volume form error} function
\begin{equation}
E^{(1)}= \frac{W_{(1)}}{ \det( V^{ij} _{(1 ) } ) }-1= \frac{A+w}{  A+ A a^{ij} v^{ij} + \det(v^{ij} )   }-1= - \frac{ \det(v^{ij})} {  A+ w + \det(v^{ij} )      },
\end{equation}
where 
$
\det(v^{ij})= \alpha_1 \alpha_2+ \alpha_1\alpha_3+ \alpha_2\alpha_3 
$
and $w= a_{22}\alpha_1+ a_{11}\alpha_2+ (a_{11}+ a_{22}+ 2a_{12}) \alpha_3$. We denote $|\vec{\mu}|_a= \sqrt{ a_{ij}\mu_i\mu_j+ A|\eta|^2}$.
Near spatial infinity
$E^{(1)}=O( \frac{1}{ A^{1/2}|\mu|_a^2    }     )$ sufficiently far away from the discriminant locus, and  $E^{(1)}= O( \frac{1}{ A^{1/4} |\mu|_a    }     )$ near the discriminant locus.
However in standard analytic packages  \cite{Heinthesis} which construct Calabi-Yau metrics from asymptotic approximate solutions, it is essential to have \emph{faster than quadratic} volume error decay rate, which is \emph{not} satisfied by our ansatz, so this error must first be corrected. This issue will be explained more amply in Section \ref{Heinpackage}.

Now we comment on the \textbf{symmetry} of the ansatz. Apart from the $T^2$-symmetry from the construction, there is an additional \textbf{$U(1)$-symmetry} for the K\"ahler metric commuting with the $T^2$-action:
\[
\mu_i\mapsto \mu_i, \quad \eta\mapsto e^{i\theta} \eta.
\]
However, the holomorphic volume form will be rotated by a phase angle under this action. This may be compared to the Taub-NUT metric, which has an $SO(3)$-symmetry acting on the base. Our ansatz has less continuous symmetry  because the base contains a distinguished trivalent graph, which is a new higher dimensional phenomenon. Another analogy to draw from this comparison is that when the Taub-NUT metric glues into the Ooguri-Vafa metric, these additional symmetries are broken, and the same phenomenon shall happen when we construct the  Ooguri-Vafa type metrics on the positive vertex. This is because the Ooguri-Vafa type metrics involve an extra periodicity condition on $\eta$ which is not compatible with rotation; an alternative viewpoint is that the special Lagrangian fibration selects out a preferred phase angle.

In some special cases there can be some \textbf{discrete symmetries} from permuting the 3 edges of the trivalent graph. The most symmetric situation is where
\[
a_{ij}d\mu_i d\mu_j\propto  (d\mu_1)^2+ (d\mu_2)^2+ (d(\mu_1-\mu_2))^2,
\]
or equivalently
\[
\begin{bmatrix}
a_{11} & a_{12} \\
a_{21}  & a_{22}
\end{bmatrix}
\propto
\begin{bmatrix}
2 & -1 \\
-1  & 2
\end{bmatrix}
\]
This choice of parameters has a special significance in the theory.

Morever, the family of ansatzs have a \textbf{scaling symmetry} which will be fundamental when we construct the  Ooguri-Vafa type metrics later. 
This symmetry is prescribed by (\ref{scalingsymmetry}). In our concrete construction, this means replacing
\[
a_{ij}\mapsto \Lambda a_{ij}, \quad A\mapsto \Lambda^2 A.
\]
The region near $( \mu_1, \mu_2, \eta)$ in the $(a_{ij})$-ansatz correspond to the region near the point $(\Lambda^{-1} \mu_1, \Lambda^{-1} \mu_2, \Lambda^{-1.5}\eta)$ in the $(\Lambda a_{ij})$-ansatz.  For example, a useful \textbf{scaling-invariant quantity} is $A^{1/4}|\vec{\mu}|_a$ :
\[
(\Lambda^2 A)^{1/4}   \sqrt{  (\Lambda a_{ij}) (\Lambda^{-1} \mu_i )(\Lambda^{-1} \mu_j) + (\Lambda^2 A) |\Lambda^{-1.5}\eta|  ^2  }=  A^{1/4} |\vec{\mu}|_a.
\]
It defines the \textbf{approximation scale} $A^{1/4} |\vec{\mu}|_a \gtrsim 1$, namely the region where the ansatz is approximately Calabi-Yau. Another scaling-invariant quantity is $A^{1/4}\ell$ where $\ell=A^{-1/4}+ \text{dist}_{g_a}(\cdot, \mathfrak{D})$.
The scaling symmetry enables us to easily extract information about the $(\Lambda a_{ij})$-ansatz by analysing the $(a_{ij})$-ansatz, which is very useful for analytical questions.

\begin{rmk}
The constants appearing in this Chapter depend only on H\"older exponents and the following 
 \textbf{uniform ellipticity bound} on $a_{ij}$:
\[
C^{-1} \delta_{ij} \leq  a_{ij} \leq C\delta_{ij}.
\]
The scaling argument can then be used to relax the uniform ellipticity to
\begin{equation}\label{scaleinvariantellipticity}
C^{-1} A^{1/2} \delta_{ij} \leq  a_{ij} \leq C  A^{1/2}\delta_{ij}.
\end{equation}
In strategic places we will in fact track down the $A$-dependence as well.
\end{rmk}

\section{Metric behaviour away from the discriminant locus}\label{MetricbehaviourawayfromDelta}

This Section uses weighted H\"older norms to quantify the idea that sufficiently away from the discriminant locus the metric $g^{(1)}$ is approximated by the constant solution.

Given a large number $C_1\gg 1$, we consider $M$ over the base region 
\begin{equation}\label{regionawayfromDelta}
\begin{cases}
A^{1/4}|\vec{\mu}|_a \geq C_1,\\
|\vec{\mu}|_a \leq 2C_1 \ell,
\end{cases}
\end{equation}
meaning that the region is far from the origin, and the $g_a$-distance to $\mathfrak{D}$ is comparable to the $g_a$-distance to the origin. Topologically the base region is obtained by removing the apex from a cone over a thrice-punctured 3-sphere. The $A$-dependence is inserted for convenience when we analyse the scaling behaviours.

The \textbf{flat model metric} is simply constructed by applying the generalised Gibbons-Hawking ansatz to $V^{ij}_{\text{flat}}= a_{ij}$ and $W_{\text{flat}}=A$:
\[
g_{\text{flat}}= a_{ij} d\mu_i d\mu_j + A |d\eta|^2+ a^{ij} \vartheta_i^{\text{flat}} \vartheta_j^{\text{flat}},
\]
where $\vartheta_i^{\text{flat}}$ for $i=1,2$ are flat connections. Likewise we define $\omega_{\text{flat}}$ and $\Omega_{\text{flat}}$. A subtlety is that $g_{\text{flat}}$ \emph{cannot} model $g^{(1)}$ globally over the region defined by (\ref{regionawayfromDelta}), because the Chern class of the $T^2$-bundle for $g^{(1)}$ evaluates nontrivially on the $S^2$ cycles wrapping the 3 puncture points in $S^3$, which obstructs the flat connection $\vartheta_i^ {\text{flat} }$. It is thence understood that we are comparing the model metric with $g^{(1)}$ over a finite number of \textbf{contractible conical subregions} which cover (\ref{regionawayfromDelta}).

The deviation of $V^{ij}$ from $a_{ij}$ is measured by $\alpha_1, \alpha_2, \alpha_3$. To estimate these quantities over these regions we introduce some \textbf{weighted H\"older norms} associated to the reference metrics $g_{\text{flat}}$. For any $T^2$-invariant tensor field $T$ defined over the region, we define the normalised H\"older seminorm
\[
[ T  ]_{\alpha}= \sup_{p}  |\vec{\mu}|_a^\alpha\cdot \sup_{|p-p'|_a < \frac{1}{10} \ell    }  \frac{ |T(p)-T(p')|} { d_{\text{flat}}(p,p'  )^\alpha   }
\]
where we compare $T(p)$ and $T(p')$ using parallel transport along minimal geodesics. The weighted norm of  $T$ is then defined by
\[
\norm{T}_{ C^{k,\alpha}_{\tau'  } }=
A^{-\tau'/4}
\sum_{j=0}^k \norm{  |\vec{\mu}|_a^{-\tau'+j } \nabla^j T}_{ L^\infty   }
+
A^{-\tau'/4}
[  |\vec{\mu}|_a^{-\tau'+k} \nabla^k  T    ]_\alpha.
\]
An estimate in this norm is thought as the higher order version of $|T|= O( A^{\tau'/4} |\vec{\mu}|_a^{\tau'} )$.

\begin{lem}
Over each of the finitely many contractible conical subregions $\norm{ \alpha_i}_{C^{k,\alpha}_{-1}  } \leq  {C A^{1/2} }  $.
\end{lem}

\begin{proof}
The absolute value estimate $|\alpha_i|\leq \frac{C A^{1/4} }{ |\vec{\mu}|_a   }$ is clear from the explicit defining formula. The higher order estimates use that $\Lap_a \alpha_i=0$ holds over a $g_a$-ball of radius comparable to $|\vec{\mu}|_a$.
\end{proof}

Next we estimate the deviation of $\vartheta_i$ from $\vartheta_i^{\text{flat}}$.
Since gauge equivalent choices of $\vartheta_i$ give rise to the same K\"ahler structure $(g^{(1)}, \omega^{(1)}, J, \Omega)$ up to holomorphic isometry, we may make any convenient \textbf{gauge choice}. The defining condition on $\vartheta_i$ is
\[
d\vartheta_i= \sqrt{-1} \left(    \frac{1}{2} \frac{\partial W_{(1)}}{\partial \mu_i} d\eta \wedge d\bar{\eta} + \frac{\partial V^{ij}_{(1) }}{\partial \eta} d\mu_j \wedge d\eta -  \frac{\partial V^{ij}_{(1)}}{\partial \bar{\eta}} d\mu_j \wedge d\bar{\eta} \right),
\]
and  $d\vartheta_i^{\text{flat}}=0$. Thus $\norm{ d(\vartheta_i-\vartheta_i^{\text{flat}}) }_{  C^{k,\alpha}_{-2} }\leq CA^{1/2}$
 using the higher derivative estimates on $\alpha_i$. Using the d-Poincar\'e lemma, we can find a gauge fixed choice of the 1-form $\vartheta_i-\vartheta_i^{\text{flat}}$ such that $\norm{\vartheta_i-\vartheta_i^{\text{flat}} }_{ C^{k,\alpha}_{-1}  } \leq C A^{1/4} $. Combining these discussions, and noticing $|d\mu_i|\leq CA^{-1/4}, |d\eta|\leq CA^{-1/2}$, we obtain

\begin{cor}
Over each of the finitely many contractible conical subregions, after suitable gauge fixing, we have the deviation estimates
\[
\norm{ g^{(1) }- g_{\text{flat}}  }_{ C^{k,\alpha}_{-1}   } \leq   C , \quad
\norm{ \omega^{(1) }- \omega_{\text{flat}}  }_{ C^{k,\alpha}_{-1}   } \leq   C ,
\quad 
\norm{ \Omega- \Omega_{\text{flat}}  }_{ C^{k,\alpha}_{-1}   } \leq   C
.
\]
Morever the volume form error function $E^{(1)}$ satisfies $\norm{E^{(1)}}_{ C^{k,\alpha}_{-2}  }\leq C$, namely the higher order version of quadratic decay estimate.
\end{cor}


\section{Structure near discriminant locus}\label{Structureneardiscriminantlocus}

We now study the metric near the discriminant locus but sufficiently far from the origin, which turns out to be locally modelled on a \textbf{fibration by Taub-NUT metrics} over a flat cylinder. A subtlety is that the smooth topology along $\mathfrak{D}_i$ is not a priori prescribed, and needs to be elucidated first.

We focus on the neighbourhood of $\mathfrak{D}_1$ far from the origin, where $\alpha_1\sim \frac{1}{ 2\sqrt{\mu_1^2+ a_{22}|\eta|^2}   }$ and $\alpha_2, \alpha_3$ are smooth. To leading order
\[
V_{(1)}\sim V_{\text{Taub}}=\begin{bmatrix}
\frac{ 1}{ 2\sqrt{\mu_1^2+ a_{22}|\eta|^2}   }+ a_{11} & a_{12} \\
a_{21}  &  a_{22}
\end{bmatrix}, 
 W_{(1)}\sim W_{\text{Taub}}= A+ \frac{ a_{22}}{ 2\sqrt{\mu_1^2+ a_{22}|\eta|^2}   }.
\]
The inverse matrix is
\[
V_{(1)}^{-1}\sim V_{\text{Taub}}^{-1}=
({ A+    \frac{ a_{22}}{ 2\sqrt{\mu_1^2+ a_{22}|\eta|^2}   }    })^{-1}
\begin{bmatrix}
a_{22} & -a_{21}  \\
-a_{12}  & \frac{ 1}{ 2\sqrt{\mu_1^2+ a_{22}|\eta|^2}   }+ a_{11}
\end{bmatrix}.
\]
Now if we apply the generalised Gibbons-Hawking ansatz to $V_{\text{Taub}}$ and $W_{\text{Taub}}$, we obtain a \textbf{model metric}
\[
g_{\text{Taub}}= V^{ij}_{\text{Taub}} d\mu_i d\mu_j + W_{\text{Taub}} |d\eta|^2+ (V^{-1}_{\text{Taub}})^{ij} \vartheta_i' \vartheta_j'
\]
where $\vartheta_1', \vartheta_2'$ are the connections. As $d\vartheta_2'=0$ we may write $\vartheta_2'=d\theta_2'$.
Rewriting the model metric,
\begin{equation}
\begin{split}
g_{\text{Taub}}=
( \frac{ 1}{ 2\sqrt{\mu_1^2+ a_{22}|\eta|^2}   }+ \frac{A}{a_{22}}     )(  d\mu_1^2+ a_{22} |d\eta|^2   ) 
+ a_{22} (  d(\mu_2+ \frac{a_{12}}{ a_{22} }\mu_1 )       )^2 \\
+
(   \frac{ 1}{ 2\sqrt{\mu_1^2+ a_{22}|\eta|^2}   }+ \frac{A}{a_{22}}          )^{-1}
 (\vartheta_1'- \frac{a_{12}}{a_{22}}d\theta_2'  )^2+  \frac{1}{a_{22}} (d\theta_2')^2.
\end{split}
\end{equation}
Notice the dual basis for $ \{\vartheta_1'-\frac{a_{12}}{a_{22}}d\theta_2', \vartheta_2'
\}
$ is given by $\{ \frac{ \partial  }{\partial \theta_1}, \frac{\partial}{\partial \theta_2}+ \frac{a_{12}}{ a_{22} } \frac{ \partial  }{\partial \theta_1} \}$, which corresponds to the moment coordinates $\mu_1$ and $\mu_2+ \frac{a_{12}}{ a_{22} }\mu_1$.

The variables $\mu_2+ \frac{a_{12}}{ a_{22} }\mu_1$ and $\theta_2'$ define a cylinder $\R\times S^1$. Translations in these variables are isometries of the model space. The model space fibres over this cylinder, and restricted to each fibre the metric is recognized as the Taub-NUT metric with parameter $\frac{A}{a_{22}}$.
The fibration is not always a metric product, because for the generators  $\frac{ \partial  }{\partial \theta_1}$ and $\frac{\partial}{\partial \theta_2}+ \frac{a_{12}}{ a_{22} } \frac{ \partial  }{\partial \theta_1}$ to give rise to an integral basis of $H_1(T^2)$ we need 
$ \frac{a_{12}}{ a_{22}}$ to be an integer. On the universal cover the metric becomes the product of Taub-NUT metric with the flat $\R^2$, as $\theta_2$ becomes a real variable instead of a circle variable. In particular the universal cover is topologically $\C^2\times \R^2$, and the model space is a discrete $\Z$-quotient of $\C^2\times \R^2$, so inherits a \textbf{smooth topology}. The \textbf{Riemannian curvature} on the model metric is bounded but does not decay as we move to infinity along $\mathfrak{D}_1$.

\begin{rmk}\label{smoothtopologyissubtle}
We wish to amplify the idea that the \textbf{smooth topology of the $S^1$-fibration map is subtle}. Given a $T^2$-fibration $M\to \mathcal{B}$ say, the $T^2$-invariant smooth functions on $M$ descend into a sheaf of functions on the base, sitting between the sheaf of smooth functions on $\mathcal{B}$ and the sheaf of continuous functions on $\mathcal{B}$. An example of such a function on our model space is $\sqrt{\mu_1^2+ a_{22}|\eta|^2}$. Had we chosen a different $a_{22}$ to begin with, this sheaf would be different. This means assigning a smooth topology on the compactification of a torus bundle across the discriminant locus, is a problem which involves extra data. In general this sheaf depends on functions along $\mathfrak{D}_i$, so carries an infinite amount of information, and is therefore expected to be unstable under deformation. This subtlety is related to Joyce's observation that special Lagrangian fibrations can fail to be given by smooth maps (\cf review Section \ref{Joycecritique} and Section \ref{SpecialLagrangiangeometry}).
\end{rmk}

Our next goal is to quantify the idea that the model $g_{\text{Taub}}$ is a good \textbf{approximation} to the metric ansatz $g^{(1)}$. We view both metrics as defined on the same smooth manifold, fibred over the region
\begin{equation}\label{regionnearDelta1}
|\vec{\mu}|_a \geq	 C_1 A^{-1/4}, \quad  |\vec{\mu}|_a \geq C_1\text{dist}_{g_a} (\cdot, \mathfrak{D}_1) 
\end{equation}
where $C_1$ is a large number as in Section \ref{MetricbehaviourawayfromDelta}. In this region the $g_a$-distance to $\mathfrak{D}_2, \mathfrak{D}_3$ are both $O(|\vec{\mu}|_a)$, and $\text{dist}_{g_a} (\cdot, \mathfrak{D}_1)$ is comparable to $A^{1/4} \sqrt{ \mu_1^2+ a_{22}|\eta|^2}$, so 
\[
\ell = A^{-1/4}+ \text{dist}_{g_a} (\cdot, \mathfrak{D}) \sim A^{-1/4}+ A^{1/4} \sqrt{ \mu_1^2+ a_{22}|\eta|^2}.
\]

We introduce some \textbf{weighted H\"older norms} associated to the reference metric $g_{\text{Taub}}$. The regularity scale of $g_{\text{Taub}}$ is comparable to $\ell$. For any $T^2$-invariant tensor field $T$ over the region (\ref{regionnearDelta1})
, define the normalised H\"older seminorm
\[
[ T  ]_{\alpha}= \sup_{p}  \ell(p)^\alpha\cdot \sup_{p'\in B_{g_{\text{Taub}}(p, \ell/10) }  }  \frac{ |T(p)-T(p')|} { d_{\text{Taub}}(p,p'  )^\alpha   }
\]
where we compare $T(p)$ and $T(p')$ using parallel transport along minimal geodesics. The weighted norm of  $T$ is then defined by
\[
\norm{T}_{ C^{k,\alpha}_{\delta, \tau } }=
A^{-\delta/4-\tau/4}
\sum_{j=0}^k \norm{ \ell^{-\delta+j} |\vec{\mu}|_a^{-\tau} \nabla^j T}_{ L^\infty   }
+
A^{-\delta/4-\tau/4}
[ \ell^{-\delta+k} |\vec{\mu}|_a^{-\tau} \nabla^k  T    ]_\alpha.
\]
An estimate in this norm can be thought as the higher order version of $|T|= O(A^{\tau/4+\delta/4} \ell^\delta |\vec{\mu}|_a^\tau  )$. Similar weighted H\"older norms are defined in the neighbourhood of $\mathfrak{D}_2$ and $\mathfrak{D}_3$.

The deviation between $V^{ij}_{\text{Taub}}$ and $V^{ij}_{(1)} $ near $\mathfrak{D}_1$ is measured by the functions  $\alpha_1- \frac{1}{2\sqrt{ \mu_1^2+ a_{22}|\eta|^2   }}$, $\alpha_2$ and $\alpha_3$.

\begin{lem}
In the above region (\ref{regionnearDelta1}) near $\mathfrak{D}_1$, 
\[
\begin{cases}
\norm{\alpha_2}_{C^{k,\alpha}_{0,-1  } }\leq CA^{1/2}, \\ \norm{\alpha_3}_{C^{k,\alpha}_{0,-1  } }\leq CA^{1/2}, \\ \norm{\alpha_1-\frac{1}{2\sqrt{ \mu_1^2+ a_{22}|\eta|^2   }} }_{C^{k,\alpha}_{0,-1  } }\leq CA^{1/2}.
\end{cases}
\]
Consequently, if $C_1$ is chosen to be large enough, then
\[
|V^{ij}_{(1)}- V^{ij}_{\text{Taub} }| \leq  \frac{C A^{1/4}}{ |\vec{\mu}|_a  }\ll a_{ij}\leq V^{ij}_{\text{Taub} }, \quad 
|W_{(1)}- W_{\text{Taub}}|\leq \frac{ C A^{3/4}} { |\vec{\mu}|_a  } \ll A\leq W_{\text{Taub}}.
\]
\end{lem}

\begin{proof}
The $\Lap_a$-harmonic function $\alpha_2, \alpha_3$ are both of order $O( \frac{A^{1/4}}{|\vec{\mu}|_a}  )$. The function
\[
\alpha_1- \frac{1}{2\sqrt{ \mu_1^2+ a_{22}|\eta|^2   }}= -\frac{1}{2\pi \sqrt{ \mu_1^2+ a_{22}|\eta|^2   }  } \arctan ( \frac{\sqrt{A}\sqrt{ \mu_1^2+ a_{22}|\eta|^2   }  }{ a_{22}\mu_2+ a_{12}\mu_1  }   )
\]
is also $\Lap_a$-harmonic, and by the Taylor expansion of $\arctan$ is seen to be $O( \frac{A^{1/4}}{|\vec{\mu}|_a}  )$ as well. These functions are smooth on the base in the region (\ref{regionnearDelta1}) with regularity scale $O(|\vec{\mu}|_a)$. The $\Lap_a$-harmonicity takes care of all higher order estimates.
\end{proof}

Next we analyse the deviation between the connections $\vartheta_i$ and $\vartheta_i'$ for $i=1,2$, corresponding to the ansatz $g^{(1)}$ and the model $g_{\text{Taub}}$ respectively. This involves the same \textbf{gauge fixing} issue as in Section \ref{MetricbehaviourawayfromDelta}. The defining condition on $\vartheta_i$ is
\[
d\vartheta_i= \sqrt{-1} \left(    \frac{1}{2} \frac{\partial W_{(1)}}{\partial \mu_j} d\eta \wedge d\bar{\eta} + \frac{\partial V^{ij}_{(1) }}{\partial \eta} d\mu_i \wedge d\eta -  \frac{\partial V^{ij}_{(1)}}{\partial \bar{\eta}} d\mu_i \wedge d\bar{\eta} \right),
\]
and similarly for $\vartheta_i'$. Thus $\norm{ d(\vartheta_i-\vartheta_i') }_{  C^{0,-2}} \leq CA^{1/2}$ using the higher derivative estimates on $V^{ij}_{(1)}- V^{ij}_{\text{Taub} }$ etc. Using the d-Poincar\'e lemma, we can find a gauge fixed choice of the smooth 1-form $\vartheta_i-\vartheta_i'$ such that
$
\norm{\vartheta_i- \vartheta_i'}_{C^{k,\alpha}_{0,-1  } }\leq CA^{1/4}.
$
Combining the above, and recalling $|d\mu_i|\leq CA^{-1/4}$, $|d\eta|\leq CA^{-1/2}$, we obtain

\begin{lem}\label{Structureneardiscriminantlocuslemma}
The K\"ahler structure $(g^{(1)}, \omega^{(1)}, J, \Omega)$  \textbf{extends smoothly} over the region (\ref{regionnearDelta1}). The deviation from the model metric admits the estimates
\[
\begin{cases}
\norm{g^{(1)}- g_{\text{Taub}}}_{ C^{k,\alpha}_{0, -1} }\leq C, 
	\quad 
	& \norm{\omega^{(1)}- \omega_{\text{Taub}}}_{ C^{k,\alpha}_{0, -1} }\leq C,
	\\
	\norm{J^{}- J_{\text{Taub}}}_{ C^{k,\alpha}_{0, -1} }\leq C, 
	\quad 
	& \norm{\Omega^{}- \Omega_{\text{Taub}}}_{ C^{k,\alpha}_{0, -1} }\leq C.
	\end{cases}
	\]
In particular, if $C_1$ is chosen large enough, then	
the magnitudes of the deviation
\[
|g^{(1)}- g_{\text{Taub}}|\ll 1, 
\quad 
|\omega^{(1)}- \omega_{\text{Taub}}|\ll 1,
\quad
|\Omega- \Omega_{\text{Taub}}|\ll 1,
\quad 
|J- J_{\text{Taub}}|\ll 1.
\]	
	The volume form error function $E^{(1)}$ satisfies 
	\[
	\norm{ E^{(1)}}_{ C^{k,\alpha}_{-1,-1}   } \leq C.
	\]
\end{lem}

\begin{rmk}
The same arguments show that the K\"ahler ansatz is smooth along the entire $\mathfrak{D}_1$, although 
the smooth topology is \emph{not} yet defined at the origin; this difficulty will later be resolved by shifting to the complex geometric viewpoint and doing a surgery to the K\"ahler ansatz.
\end{rmk}

\begin{rmk}
The metric deviation estimate and the volume form error estimate require $A^{1/4}|\vec{\mu}|_a\gtrsim 1$. Heuristically we may think of the discriminant locus $\mathfrak{D}$ as the source of gravitating force, and for $A^{1/4} |\vec{\mu}|\lesssim 1$ the mutual interactions of $\mathfrak{D}_1, \mathfrak{D}_2, \mathfrak{D}_3$ become too strong, so the perturbative description breaks down.	
\end{rmk}

\begin{rmk}\label{StructurenearDeltaflatmodelclosetoTaubNUT}
Over the subregion of (\ref{regionnearDelta1}) where $\ell \geq 2A^{-1/4}$, namely outside the curvature scale along $\mathfrak{D}_1$, the model metric $g_{\text{Taub}}$ is itself locally approximated by the flat model $g_{\text{flat}}$ (\cf Section \ref{MetricbehaviourawayfromDelta}) over $g_a$-balls of radius $\sim \ell(x)$:
\[
\norm{ g_{\text{Taub}}- g_{\text{flat}} }_{C^{k,\alpha}_{-1,0} }\leq C.
\]
\end{rmk}

\section{Complex geometric perspective}\label{C3complexgeometricperspective}

We now \textbf{identify the complex structure on $M$ with  $\C^3$}. Recall $\zeta_i= V^{ij}_{(1)} d\mu_j + \sqrt{-1} \vartheta_i$ and formula (\ref{GibbonsHawkingholomorphicdifferential}) for their differentials. The main idea is to produce holomorphic differentials by adjusting $\zeta_i$. The reader can refer to the Taub-NUT Example \ref{TaubNUT} for the warm up.

We define the functions $\beta_i(\mu_1, \mu_2, \eta)$ for $i=0,1,2$,
\begin{equation*}
\begin{cases}
\beta_1 =  2\displaystyle \lim_{(\mu_1', \mu_1'-\mu_2')\to (+\infty, +\infty)} 
 \int_{(\mu_1', \mu_2')}^{(\mu_1, \mu_2)} \frac{\partial \alpha_1}{\partial \eta} (s_1, s_2, \eta) ds_1 +   \frac{\partial \alpha_3}{ \partial \eta  } (s_1, s_2, \eta) d(s_1-s_2) 
\\
\beta_2 =  2
\displaystyle
 \lim_{(\mu_2', \mu_2'-\mu_1')\to (+\infty,  +\infty)}  \int_{(\mu_1', \mu_2')}^{(\mu_1, \mu_2)}  \frac{\partial \alpha_2}{\partial \eta} (s_1, s_2, \eta) ds_2 +   \frac{\partial \alpha_3}{ \partial \eta  } (s_1, s_2, \eta) d(s_2-s_1) 
 \\
\beta_0 =  2
\displaystyle
\lim_{(\mu_1', \mu_2')\to (-\infty,  -\infty)}  \int_{(\mu_1', \mu_2')}^{(\mu_1, \mu_2)}  -\frac{\partial \alpha_1}{\partial \eta} (s_1, s_2, \eta) ds_1 -   \frac{\partial \alpha_2}{ \partial \eta  } (s_1, s_2, \eta) ds_2 .
\end{cases}
\end{equation*}
In these improper integrals  $\eta$ is held fixed. Here the integrability condition (\ref{C3linearintegrability}) ensures the integrands are closed differentials, so the integral is path independent. We can take the limit in the definition of $\beta_1$, because 
\[
\begin{cases}
\frac{\partial \alpha_1}{\partial \eta}=O ( \frac{ a_{22} \bar{\eta}}{  (\mu_1^2+ a_{22}|\eta|^2)^{3/2} }    ), \quad &\mu_1\to \infty, 
\\ 
\frac{\partial \alpha_3}{\partial \eta}=O ( \frac{ (a_{11}+2a_{12}+a_{22}) \bar{\eta}}{  ((\mu_1-\mu_2)^2+(a_{11}+2a_{12}+a_{22})|\eta|^2)^{3/2} }    ),& \mu_1-\mu_2\to \infty. 
\end{cases}
\]
Likewise with $\beta_2, \beta_0$.
The domain of definition of $\beta_1, \beta_2, \beta_0$ are respectively $\R^2_{\mu_1, \mu_2}\times \C_\eta \setminus \{ \eta=0, \mu_1\leq 0, \mu_1\leq \mu_2   \}$, $\R^2_{\mu_1, \mu_2}\times \C_\eta \setminus \{ \eta=0, \mu_2\leq 0, \mu_2\leq \mu_1   \}$, and
$\R^2_{\mu_1, \mu_2}\times \C_\eta \setminus \{ \eta=0, \mu_1\geq 0, \mu_2\geq 0   \}$; the singularities in the integrands prevent us from defining $\beta_0, \beta_1, \beta_2$ globally.

\begin{lem}\label{holomorphicdifferentiallemma}
By construction \[
\begin{cases}
\frac{\partial \beta_1}{\partial \mu_1}= 2\frac{\partial}{\partial \eta}(\alpha_1+ \alpha_3)= 2\frac{\partial v^{11}}{\partial \eta}, \quad 
&\frac{\partial \beta_1}{\partial \mu_2}= 2\frac{\partial v^{12}}{\partial \eta}, 
\\
\frac{\partial \beta_2}{\partial \mu_1}= 2\frac{\partial v^{21}}{\partial \eta}, \quad 
&\frac{\partial \beta_2}{\partial \mu_2}= 2\frac{\partial v^{22}}{\partial \eta}
\\
\frac{\partial \beta_0}{\partial \mu_1}= -2\frac{\partial }{\partial \eta}( v^{11}+  v^{21}), \quad 
&\frac{\partial \beta_0}{\partial \mu_2}= -2\frac{\partial }{\partial \eta}( v^{12}+v^{22}  ).
\end{cases}
\quad   
\]	
Morever,
\[
\frac{\partial \beta_1}{\partial \bar{\eta}} =-\frac{1}{2} 
\frac{\partial w}{\partial \mu_1  } ,
\quad  
\frac{\partial \beta_2}{\partial \bar{\eta}} =-\frac{1}{2} 
\frac{\partial w}{\partial \mu_2  } , \quad 
\frac{\partial \beta_0}{\partial \bar{\eta}} =\frac{1}{2} 
(\frac{\partial w}{\partial \mu_1  } + \frac{\partial w}{\partial \mu_2  }).
\]
Therefore the type (1,0) forms
\begin{equation}
\zeta_1'= \zeta_1+ \beta_1 d\eta, \quad 
\zeta_2'= \zeta_2+ \beta_2 d\eta,
\quad
\zeta_0'= -\zeta_1- \zeta_2+ \beta_0 d\eta
\end{equation}
are closed, namely they are \textbf{holomorphic differentials}.
\end{lem}

\begin{proof}
The $\mu_1, \mu_2$ derivatives are clear. For the $\bar{\eta}$ derivative, we can apply the component form of the distributional equation (\ref{distributionalequationlinearised}) away from $\mathfrak{D}$, to see
\[
\begin{split}
\frac{\partial \beta_1}{\partial \bar{\eta}}& = 2\lim \int \frac{\partial^2 v^{11}}{\partial \eta \partial \bar{\eta}} d\mu_1 + \frac{\partial^2 v^{12}}{\partial \eta \partial \bar{\eta}} d\mu_2 \\
&= 
-\frac{1}{2} \lim \int 
\frac{\partial^2 w}{\partial \mu_1 \partial \mu_1 } d\mu_1 + \frac{\partial^2 w}{\partial \mu_1 \partial \mu_2} d\mu_2
=
-\frac{1}{2} 
\frac{\partial w}{\partial \mu_1  },
\end{split}
\]
where in the last equality we compare the asymptotic values at infinity to show there is no constant term depending on $\eta$. Likewise with $\beta_2, \beta_0$.
\end{proof}

\begin{lem}\label{holomorphicdifferentialfunctionalequation}
The sum
$
\beta_1+ \beta_2+ \beta_0= \frac{1}{\eta}.
$ Equivalently,
\[
\zeta_1'+ \zeta_2'+ \zeta_0'= d\log \eta.
\]
\end{lem}

\begin{proof}
By Lemma \ref{holomorphicdifferentiallemma} the sum $\beta_1+ \beta_2+ \beta_0$ is independent of $\mu_1, \mu_2$. Given $\eta\neq 0$, we shall evaluate this sum at the limit point $(\mu_1\to -\infty, \mu_2\to -\infty, \mu_1-\mu_2\to +\infty)$. Then $\beta_0$ has no contribution, while  $\beta_1$ contributes
\[
2\lim_{ \mu_1-\mu_2\to +\infty    } \int_{\mu_1=+\infty , \text{fix $\mu_1-\mu_2$} }^{\mu_1=-\infty} \frac{\partial \alpha_1}{\partial {\eta}} d\mu_1,
\]
and $\beta_2$ contributes 
\[
2\lim_{ \mu_1\to -\infty   } \int_{\mu_2=+\infty, \text{fix $\mu_1$}  }^{\mu_2=-\infty} \frac{\partial \alpha_2}{\partial {\eta}} d\mu_2
\]
plus
\[
-2\lim_{\mu_1\to -\infty  }   \int_{\mu_1-\mu_2=-\infty, \text{fix $\mu_1$}   }^{ \mu_1-\mu_2=+\infty   }   \frac{\partial \alpha_3}{\partial {\eta}} d(\mu_1-\mu_2).
\]
Observe 
\[
\lim_{ \mu_1\to -\infty, \text{fix $\mu_1-\mu_2$}      } \alpha_3 (\mu_1,\mu_2,\eta)=\frac{ 1}{ 2\sqrt{ (\mu_1-\mu_2)^2+ (a_{11}+2a_{12}+a_{22} ) |\eta|^2       }  } ,
\]
\[
\lim_{ \mu_1\to -\infty, \text{fix $\mu_1-\mu_2$}      } \frac{\partial \alpha_3 }{\partial \eta   }= \frac{ -(a_{11}+2a_{12}+a_{22} )\bar{\eta}   }{ 4( (\mu_1-\mu_2)^2+ (a_{11}+2a_{12}+a_{22} ) |\eta|^2        )^{3/2} } 
\]
Using Lebesgue dominated convergence theorem, the third integral contribution is equal to
\[
\frac{ (a_{11}+2a_{12}+a_{22} )\bar{\eta}  }{2} \int_{ -\infty}^{+\infty}  \frac{1}{ (x^2+ (a_{11}+2a_{12}+a_{22} )|\eta|^2  )^{3/2}     } dx= \frac{1}{\eta}.
\]
The other two contributions are zero by similar arguments.
\end{proof}

Now we notice that the multivalued holomorphic functions
\[
\log z_i= \int \zeta_i', \quad i=0,1,2
\]
have periods in $2\pi \sqrt{-1}\Z$, so the \textbf{holomorphic functions} $z_0, z_1, z_2$ are well defined on the domain of definition of $\beta_0, \beta_1, \beta_2$ respectively. Appropriate choices of multiplicative constants ensure the \textbf{functional equation}
\begin{equation}
z_0z_1z_2=\eta,
\end{equation}
which enable us to extend $z_0, z_1, z_2$ over the complement of $\mathfrak{D}$ when $\eta=0$.

\begin{lem}
	The holomorphic functions $z_0, z_1, z_2$ extend smoothly over  $\mathfrak{D}_i\subset \mathfrak{D}\subset \R^2_{ \mu_1, \mu_2}\times \C_\eta$. The function $z_0, z_1, z_2$ vanish over $\mathfrak{D}_1\cup \mathfrak{D}_2$, $\mathfrak{D}_1\cup \mathfrak{D}_3$ and $\mathfrak{D}_2\cup \mathfrak{D}_3$ respectively.
\end{lem}

\begin{proof}
We focus on the neighbourhood of $\mathfrak{D}_1$. Modulo smooth terms
\[
\alpha_1\sim  \frac{1}{2\sqrt{\mu_1^2+ a_{22} |\eta|^2  } }, \quad
\frac{\partial \alpha_1}{\partial \eta}\sim
- \frac{a_{22} \bar{\eta}}{4({\mu_1^2+ a_{22} |\eta|^2  })^{3/2} },
\]
so by the integral definitions, along $\mathfrak{D}_1$ the function $\beta_2$ is non-singular, and
\[
\beta_1\sim  \frac{-1}{2\eta} (\frac{\mu_1}{  \sqrt{ \mu_1^2+ a_{22}|\eta|^2   }  } -1)   , \quad 
\beta_0\sim  \frac{1}{\eta}- \beta_1=   \frac{1}{2\eta} (\frac{\mu_1}{  \sqrt{ \mu_1^2+ a_{22}|\eta|^2   }  } +1)    .
\]
Now \[
\begin{split}
& d\log |z_1|=(a_{11}+ v^{11}) d\mu_1+ (a_{12}+v^{12}) d\mu_2+ \text{Re}(\beta_1 d\eta) \\
&\sim  
\frac{1}{2\sqrt{\mu_1^2+ a_{22} |\eta|^2  } }
d\mu_1+ \frac{1}{2}(1-\frac{\mu_1}{  \sqrt{ \mu_1^2+ a_{22}|\eta|^2   }  } ) d\log |\eta|+ a_{11}d\mu_1+ a_{12}d\mu_2
\\
& = \frac{1}{2}d\log |\eta| + \frac{1}{2} d\sinh^{-1} ( \frac{\mu_1}{\sqrt{a_{22}}|\eta|  } ) + d(a_{11} \mu_1+ a_{12}\mu_2 )  ,
\end{split}
\]
hence up to multiplying by a smooth function 
\[
|z_1|\sim  \text{const}\cdot( \frac{\mu_1}{\sqrt{a_{22}  }   } + \sqrt{ \frac{\mu_1^2+ a_{22} |\eta|^2 }{{a_{22}  }  }      }   )^{1/2} e^{ a_{11}\mu_1+ a_{12}\mu_2 }   \to 0
\]
as the point moves to $\mathfrak{D}_1$. Similarly
\[
|z_0| \sim  \text{const}\cdot( -\frac{\mu_1}{\sqrt{a_{22}  }   } + \sqrt{ \frac{\mu_1^2+ a_{22} |\eta|^2 }{{a_{22}  }  }      }   )^{1/2}  e^{- a_{11}\mu_1- a_{12}\mu_2-a_{21}\mu_1-a_{22}\mu_2 } \to 0.
\]
The function $\log z_2$ encounters no singularity along $\mathfrak{D}_1$. These calculations guarantee the continuous extension of the holomorphic functions $z_0, z_1, z_2$ over $\mathfrak{D}_1$. Since the complex structure is compatible with the smooth topology by Section \ref{Structureneardiscriminantlocus}, these holomorphic functions in fact extend smoothly along $\mathfrak{D}_1$.
We remark that what happens in these calculations is essentially identical to the Taub-NUT metric near the origin.
\end{proof}

Next we show continuous extension of $z_0, z_1, z_2$ at the origin.

\begin{lem}
The functions $z_0, z_1, z_2$ tend to zero as $(\mu_1, \mu_2, \eta)\to 0$.	
\end{lem}

\begin{proof}
To see the main ideas, let us focus on $\eta=0$ and let $\mu_1, \mu_2\to 0$. By construction
$
d\log |z_1|= V^{1j}_{(1)} d\mu_j + \text{Re} ( \beta_1 d\eta  ). 
$
Restricted to $\eta=0$, 
\[
d\log |z_1|= (a_{11}+ v^{11} ) d\mu_1 + (a_{11}+ v^{12} ) d\mu_2= d(a_{11}\mu_1+a_{12}\mu_2)+ \alpha_1 d\mu_1+ \alpha_3 d(\mu_1-\mu_2),
\]
hence
\[
|z_1|= \text{const}\cdot e^{   a_{11}\mu_1+a_{12}\mu_2     } \exp\left(  \int^{(\mu_1, \mu_2)}  \alpha_1 d\mu_1+ \alpha_3 d(\mu_1-\mu_2)         \right).
\]
We need to show $ |z_1|\to 0$ as $\mu_1, \mu_2\to 0$, namely $\int \alpha_1 d\mu_1+ \alpha_3 d(\mu_1-\mu_2)\to -\infty.$
Now because $\alpha_1, \alpha_3$ are positive, this integral viewed as a function of $\mu_1$ and $\mu_1-\mu_2$ is an increasing functions of both variables, so it suffices to show
this integral decreases to $-\infty$ as $(\mu_1, \mu_2)\to 0$ along the ray $\mathfrak{D}_2$:
\[
\begin{split}
 \int  \alpha_1 d\mu_1+ \alpha_3 d(\mu_1-\mu_2)
 =& \log |\mu_1| \{  \frac{1}{2}+ \frac{1}{2\pi } \arctan (\frac{a_{12} }{ \sqrt{A} }  )    +   \frac{1}{2\pi } \arctan (\frac{-a_{11}-a_{12} }{ \sqrt{A} }  )      \}
 \\
 =& \log |\mu_1| \{  \frac{1}{4}+ \frac{1}{2\pi } \arctan (\frac{a_{12}+a_{22} }{ \sqrt{A} }  )  \}\to -\infty,
\end{split}
\]
where the first equality uses that the arctan functions are constant on the ray $\mathfrak{D}_2$, and the second equality is an elementary trignometric identity.

In the more general case of $\eta\neq 0$ 
the $\arctan$ factor would no longer be exactly constant, but one can still make $|z_1|$ arbitrarily small for sufficiently small $|\eta|, \mu_1, \mu_2$. The cases of $z_0$ and $z_2$ are completely analogous.
\end{proof}

We have defined a holomorphic map from $M\setminus \{0\}$ to $\C^3_{z_0, z_1, z_2}\setminus \{0\}$, which extends to a continuous map $M\to \C^3$.

\begin{prop}\label{C3complexstructure}
The map $M\setminus \{0\}\to \C^3_{z_0, z_1, z_2}\setminus \{0\}$ is a \textbf{biholomorphism}.
The \textbf{$T^2$-action} on the holomorphic functions is identified as
	\[
	e^{i\theta_1}\cdot (z_0, z_1, z_2)=( e^{-i\theta_1} z_0, e^{i\theta_1} z_1, z_2    ), \quad
	e^{i\theta_2}\cdot (z_0, z_1, z_2)=( e^{-i\theta_2} z_0,  z_1, e^{i\theta_2}z_2    ).
	\]
	The \textbf{holomorphic volume form}
	$
	\Omega=- dz_0\wedge dz_1\wedge dz_2 .
	$ Henceforth we \textbf{identify $M$ with $\C^3$}.
\end{prop}

\begin{proof}
To identify the $T^2$-action we examine the Hamiltonian vector field action. Recall that $\{\frac{\partial}{\partial \theta_i}\}_{i=1,2}$ is dual to the connection $\{\vartheta_i\}_{i=1,2}$. We compute
	\[
	\mathcal{L}_{ \frac{\partial}{\partial \theta_1}   } z_i= dz_i (\frac{\partial}{\partial \theta_1}  )= z_i \zeta_i'( \frac{\partial}{\partial \theta_1} ),
	\]
in particular \[
	\mathcal{L}_{ \frac{\partial}{\partial \theta_1}   } z_1= z_1 \sqrt{-1}\vartheta_1( \frac{\partial}{\partial \theta_1}  )  =\sqrt{-1} z_1 ,
	\quad 
	\mathcal{L}_{ \frac{\partial}{\partial \theta_1}   } z_0 = -\sqrt{-1} z_0, \quad 
	\mathcal{L}_{ \frac{\partial}{\partial \theta_1}   } z_2=0,
	\]
	from which the first circle action is clear. Likewise with the second circle action.
	
	The holomorphic (3,0)-form $\Omega$ is uniquely determined by the condition that  $\Omega(\frac{\partial}{\partial \theta_1}, \frac{\partial}{\partial \theta_2}  , \cdot )=d\eta$. But
	\[
	-dz_0 \wedge dz_1\wedge dz_2(\frac{\partial}{\partial \theta_1}, \frac{\partial}{\partial \theta_2}  , \cdot )= z_0 z_1dz_2+ z_1z_2 dz_0+ z_0z_2dz_1=d\eta
	\] 
	where in the last step we used the functional equation $z_0z_1z_2=\eta$. This shows $\Omega=- dz_0\wedge dz_1\wedge dz_2$. In particular the map $M\setminus \{0\}\to \C^3\setminus\{0\}$ is locally invertible.
	
To show the map is a homeomorphism, we notice that it is compatible with the fibration structure $M\to \C_\eta$ and $\C^3\to \C_\eta$, so it suffices to show the fibres are identified, which follows from looking at the complexification of the $T^2$-action into a $(\C^*)^2$-action.

Combining the above proves the biholomorphism claim.
\end{proof}

\begin{rmk}
Recall from Section \ref{C3asymptoticmetric} the \textbf{additional $U(1)$-symmetry}
\[
\mu_i\mapsto \mu_i, \quad \eta\mapsto e^{i\theta} \eta,
\]
acting on the base,
which lifts to some $T^2$-equivariant action on $M$ preserving the metric and rotating $\Omega$. Properly speaking, the continuous symmetry group fits naturally into an extension sequence
\[
1\to T^2\to U(1)^3\to U(1)\to 1,
\]
and we are making a non-unique choice to split the extension.
A particular choice can be identified complex geometrically as
\[
(z_1, z_2, z_0) \mapsto ( z_1, z_2, e^{i\theta} z_0    ).
\]
\end{rmk}

It is instructive to understand the $T^2$-invariant \textbf{K\"ahler metric}  in the complex geometric picture near spatial infinity. The reader will not fail to notice the analogy with the Taub-NUT metric. Our $\C^3$ admits a holomorphic fibration $\eta= z_0z_1z_2$. Far away from $\eta=0$, we are in the constant solution regime, so to leading order
\[
\begin{cases}
d\log z_1\sim  a_{11} d\mu_1+ a_{12} d\mu_2+ \sqrt{-1}\vartheta_1,
\\
d\log z_2\sim  a_{21} d\mu_1+ a_{22} d\mu_2+ \sqrt{-1}\vartheta_2,
\\
d\log z_0\sim  -d\log z_1-d\log z_2, 
\\
V^{ij}_{(1)}\sim a_{ij}, \quad W_{(1)}\sim A,
\end{cases}
\]
hence the K\"ahler form is to leading order
\begin{equation}\label{omega1onC3complexgeometry}
\omega^{(1)}\sim  A \frac{\sqrt{-1}}{2} d\eta\wedge d\bar{\eta}+  d\mu_j \wedge \vartheta_j\sim  \frac{\sqrt{-1}}{2}( \sum_{i,j=1,2} a^{ij}d\log z_i \wedge d\overline{\log z_j }  + A d\eta\wedge d\bar{\eta}  ).
\end{equation}
This means in the horizontal direction the dominant term of $\omega^{(1)}$ is the pullback of a Euclidean metric $\frac{\sqrt{-1}}{2} A d\eta\wedge d\bar{\eta}$ on $\C_\eta$, and in the vertical direction $\omega^{(1)}$ is an almost flat metric on the fibre $\{ z_1z_2z_0=\eta  \}\simeq (\C^*)^2$ written in the log coordinates.

When $\eta$ becomes small, the fibre will gradually break up into the union of 3 coordinate planes. Suppose at least two of $|z_0|, |z_1|, |z_2|$ remain large, then we are still far from the discriminant locus $\mathfrak{D}$, and the metric asymptote (\ref{omega1onC3complexgeometry}) still applies. In particular the central fibre $\{\eta=0\}$ has 3 asymptotic branches, exemplified by $\{ z_0=0,  |z_1|\gg 1, |z_2|\gg 1\}$ which is metrically asymptotic to flat $\R^2\times T^2$.

Finally the neighbourhood of $\{z_i=z_j=0\}$ corresponds to the discriminant locus $\mathfrak{D}$. We focus on $\{z_1=z_0=0\}$ corresponding to $\mathfrak{D}_1$. Approximately $d\log z_2 \sim a_{12}d\mu_1+ a_{22}d\mu_2 +\sqrt{-1} \vartheta_2 $, and the function $\log z_2 \in \R\times S^1$ provides a fibration structure over the cylinder $\C^*\simeq \R\times S^1$, where the fibres are approximately $\C^2$ with the Taub-NUT metric (\cf Section \ref{Structureneardiscriminantlocus}).

\section{Algebraic geometric perspective}\label{Algebraicgeometricperspective}

We now take a closer examination of the complex geometry on $\C^3$. We first raise two conceptual puzzles, and then we propose two conceptual explanations which suggest different directions of future investigations.

\begin{itemize}
\item  A priori speaking $M\simeq \C^3$ is only equipped with a \textbf{complex structure}, but the assignment of holomorphic coordinates $z_0, z_1, z_2$ canonically induces an \textbf{algebraic structure}. What is the origin of this algebraicity?

\item  It is well known that $(\C^3, \Omega)$ viewed as a complex manifold or an algebraic variety  has a huge automorphism group preserving the holomorphic volume form. But our construction of coordinate functions are \textbf{canonical} up to multiplying by constants. What is the conceptual explanation?
\end{itemize}

The first explanation is that $\C^3$ has a \textbf{toric structure}. This comes from the holomorphic isometric action of $T^3$, acting diagonally on $z_0, z_1, z_2$. This induces a $(\C^*)^3$-action with an open dense orbit in $\C^3$, making $\C^3$ a toric manifold and in particular algebraic. The canonical coordinates come from the eigenfunctions of this algebraic torus action, and $z_0, z_1, z_2$ up to constant scale factors are special because they have minimal vanishing orders on the toric boundary.

This explanation is simpler, but there are two possible criticisms. First, the $T^3$ has a preferred subgroup $T^2$ whose action has very different nature from the additional $U(1)$-action, so it seems unnatural to put them on the same conceptual footing. Second, the additional $U(1)$-symmetry is accidental to this particular example, which may not survive for other examples generalising our construction. A conjectural example without this $U(1)$-symmetry is described
in subsection \ref{gravitationalinstantons}.

The second and deeper explanation is based on the principle that \textbf{algebraic structures arise from the ring of holomorphic functions with controlled growth} (\cf \cite{DonaldsonSun}).

\begin{lem}\label{algebraicityonC3lemma}
Any algebraic function $f$ on $\C^3$ satisfies the growth estimate
\begin{equation}\label{holomorphicfunctiongrowth}
|f|\leq K_1 e^{K_2(|\mu_1|+|\mu_2|)  } (|\eta|+1)^{K_3}.
\end{equation}
for some constants $K_1, K_2, K_3$ depending on $f$.
\end{lem}

\begin{proof}
It suffices to prove the growth estimate for $z_1, z_2, z_0$.
By elementary calculation $|\frac{\partial\alpha_1}{\partial \eta}| \leq \frac{Ca_{22}|\eta|  }{ (\mu_1^2+ a_{22}|\eta|^2)^{3/2}  }$,
so upon integration
\[
\int_{\mu_1}^\infty |\frac{\partial\alpha_1}{\partial \eta}| d\mu_1 \leq  \frac{C}{|\eta|} ( 
\frac{\mu_1}{ \sqrt{\mu_1^2+ a_{22}|\eta|^2}  }-1
  )\leq \frac{C}{|\eta|}, 
\]
hence $|\beta_i| \leq  \frac{C}{|\eta|}$ by the integral definition of $\beta_i$. From
\[
d\log |z_1|= d(a_{11}\mu_1+a_{12}\mu_2 )+ \alpha_1 d\mu_1+ \alpha_3 d(\mu_1-\mu_2)+ \text{Re}(\beta_1d\eta),
\]
we integrate to obtain the growth bound on $z_1$ for $A^{1/4}|\vec{\mu}|_a \geq 1$. But $|z_1|$ is a continuous function, so the bound holds also near the origin. Similarly we can bound $|z_0|$ and $|z_2|$.
\end{proof}

\begin{prop}
The ring of algebraic functions on $\C^3$ coincides with the holomorphic functions satisfying the growth estimate (\ref{holomorphicfunctiongrowth}) for some $K_1, K_2, K_3$.
\end{prop}

\begin{proof}
We need to prove the converse to Lemma \ref{algebraicityonC3lemma}.
The $T^2$-symmetry acts on functions via
\[
(e^{i\theta_1}, e^{i\theta_2})\cdot f= f( e^{i(\theta_1+\theta_2)} z_0, e^{-i\theta_1} z_1, e^{-i\theta_2} z_2    ).
\]
This action allows us to expand any holomorphic function $f$ as a Fourier series on every $T^2$-fibre:
\[
f=\sum_{n,m\in \Z^2}  f_{n,m},
\]
where $f_{n,m}$ has weight $(n,m)$ with respect to the $T^2$ action. Since the $T^2$ action is holomorphic, the Fourier components $f_{n,m}$ are also holomorphic. Furthermore, these $f_{n,m}$ satisfy the same growth condition as $f$ after perhaps increasing $K_1$.

We claim every $f_{n,m}$ is algebraic. To see this, we can find a suitable monomial of $z_0, z_1, z_2$ which has the same weight as $f_{n,m}$, such that $f_{n,m}$ divided by this monomial has no pole along $\mathfrak{D}_i$. But this quotient function is $T^2$-invariant and holomorphic, so depends only on $\eta$, and in fact has to be a polynomial of $\eta$ by the growth condition.

By applying the Parseval identify to every $T^2$-fibre, we obtain
\[
\frac{1}{4\pi^2}\int_{T^2} |f|^2  \vartheta_1 \wedge \vartheta_2 = \sum_{n,m} |f_{n,m}|^2.
\]
Both LHS and RHS are functions of $\mu_1, \mu_2, \eta$, and LHS has a bound of type (\ref{holomorphicfunctiongrowth}) by assumption. But for any given $K_2, K_3$, only finitely many monomials of $z_0, z_1, z_2$ satisfy the growth bound (\ref{holomorphicfunctiongrowth}) 
globally, so only finitely many $f_{n,m}$ can appear as summands. Hence $f$ is algebraic as required.
\end{proof}

The proof in fact gives a double-index increasing \textbf{filtration}  structure on the ring of algebraic functions:
\[
\mathcal{F}_{K_2, K_3}=\{ f: \text{There exists $K_1$ such that }   |f|\leq K_1 e^{K_2(|\mu_1|+|\mu_2|)  } (|\eta|+1)^{K_3}    \},
\]	
such that every filtered piece is finite. This is the deeper mechanism why the complex automorphism group is cut down to finite size.

The insight from this discussion is that on our $\C^3$ \textbf{ the algebraic structure has a transcendental origin}. The growth of holomorphic functions naturally involve transcendental functions such as $\exp$ and $\log$. The ultimate reason is that torus fibrations are inherently transcendental in nature; this exponential growth behaviour already happened on the flat $\C^*$.

At this moment we still have the freedom to normalise
\[
(z_0, z_1, z_2)\mapsto (\lambda_0 z_0, \lambda_1 z_1, \lambda_2z_2),
\]
where $\lambda_i$ are constants satisfying $\lambda_0\lambda_1\lambda_2=1$. Fixing a \textbf{normalisation} is important for keeping track of how estimates depend on the scaling parameter $A$. We now make a choice so that the region $\{ A^{1/4}|\vec{\mu}|_a\lesssim 1 \}$ resemble a complex ball. Pick a point such that $|\mu_1|, |\mu_2|, |\mu_1-\mu_2|, A^{1/4}|\eta|$ are all comparable to $A^{-1/2}$, so $ A^{1/4}\text{dist}_{g_a}(\cdot, \Delta)\sim  A^{1/4}|\vec{\mu}|_a\sim 1$, and we demand $|z_0|=|z_1|=|z_2|=|\eta|^{1/3}$ at this point. This convention is compatible with both the $A$-scaling and the functional equation $z_0z_1z_2=\eta$. We did not mention the phase of $z_i$  because $T^2$-gauge symmetry renders different phase choices equivalent. Under this convention, on the annulus region $1\lesssim A^{1/4} |\vec{\mu}|_a\lesssim C_1$, the holomorphic functions $A^{1/4}z_0, A^{1/4}z_1, A^{1/4}z_2$ are bounded independent of scaling factor, and the metric $\omega^{(1)}$ is $C^\infty$-equivalent to $\sum_i \sqrt{-1} dz_i \wedge d\bar{z}_i$.

\section{Surgery on the ansatz}
\label{surgery}

We begin with some explanations about our strategy. The generalised Gibbons-Hawking ansatz is convenient for producing the metric ansatz $g^{(1)}$, but very difficult for proving nonlinear existence theorems, due to the singularity issues caused by the distributional equation. So instead we will \textbf{shift to the complex geometric viewpoint} on $\C^3$ and attempt to solve the complex Monge-Amp\`ere equation.

One minor problem is that there is no guarantee for  $g^{(1)}$ to be smooth at the origin. So we do a surgery at the scale $A^{1/4}|\vec{\mu}|_a \lesssim 1$,
namely the scale $|z_1|, |z_2|, |z_0|\lesssim A^{-1/4}$.  In the annulus $\{ 1\leq A^{1/4} \sqrt{ |z_1|^2+ |z_2|^2+ |z_0|^2} \leq 2 \}$, we write $\omega^{(1)}= \sqrt{-1}\partial \bar{\partial } \phi^{(1)}$ which is $C^\infty$-equivalent to $  \sqrt{-1}\sum_i dz_i \wedge d\bar{z}_i$. Using a cutoff function
\[
\chi= \begin{cases}
0    \quad & A^{1/4} \sqrt{ |z_1|^2+ |z_2|^2+ |z_0|^2}\leq 1, \\
1  &   A^{1/4} \sqrt{ |z_1|^2+ |z_2|^2+ |z_0|^2}\geq 1.5,
\end{cases}
\]
we replace $\omega^{(1)}$ in the complex ball $\{ A^{1/4} \sqrt{ |z_1|^2+ |z_2|^2+ |z_0|^2}\leq 2   \}$ by $\sqrt{-1}\partial \bar{\partial } (\chi \phi^{(1)})$, which is now smooth but loses positive definiteness. The remedy is to add to $\sqrt{-1}\partial \bar{\partial } (\chi \phi^{(1)})$ a smooth semipositive closed $(1,1)$-form, which is compactly supported in $\{ A^{1/4} \sqrt{ |z_1|^2+ |z_2|^2+ |z_0|^2}\leq 2  \}$, and larger than $C  \sqrt{-1}\sum_i dz_i \wedge d\bar{z}_i  $ on  $\{ A^{1/4} \sqrt{ |z_1|^2+ |z_2|^2+ |z_0|^2}\leq 1.5  \}$ for some sufficiently large constant $C$. 
Let us call the \textbf{modified K\"ahler form} $\omega^{(2)}$, which clearly agrees with $\omega^{(1)}$
outside $\{ A^{1/4} \sqrt{ |z_1|^2+ |z_2|^2+ |z_0|^2}\leq 2  \}$, and is $C^\infty$-equivalent to $\sqrt{-1}\sum_i dz_i \wedge d\bar{z}_i  $ inside $\{ A^{1/4} \sqrt{ |z_1|^2+ |z_2|^2+ |z_0|^2}\leq 3  \}$. Morever, with a little care in the construction the symmetries of $\omega^{(1)}$ persist on $\omega^{(2)}$. The associated K\"ahler metric is $g^{(2)}$, and we shall refer to both $g^{(2)}$ and $\omega^{(2)}$ interchangably. This metric is clearly complete.

A caveat is that the functions $\mu_1, \mu_2$ on $\C^3$ are no longer the moment coordinates for $\omega^{(2)}$. Furthermore in the smooth structure induced by the complex structure on $\C^3$, the functions $\mu_1, \mu_2$ may not be smooth at the origin. Henceforth in this Chapter we will abuse notation to denote $\mu_1, \mu_2$ as their mollified version. In other words, we perform a \textbf{surgery to the fibration}
$\C^3\xrightarrow{(\mu_1, \mu_2, \eta)} \R^4$ inside the ball $\{  |z_1|^2+|z_2|^2+|z_0|^2\leq A^{-1/2} \}$ to make it defined by a smooth map.

We can now introduce the global \textbf{weighted H\"older norms} $\norm{T}_{C^{k,\alpha}_{\delta, \tau}(\C^3) }$ for $T^2$-invariant tensors $T$ on $\C^3$. 
\begin{itemize}
\item  In the region (\ref{regionawayfromDelta}) away from the discriminant locus, the norm $\norm{T}_{C^{k,\alpha}_{\delta, \tau}(\C^3) }$ is equivalent to 
$\norm{T}_{C^{k,\alpha}_{\delta+\tau }}$ defined in Section \ref{MetricbehaviourawayfromDelta}.
\item
In the region (\ref{regionnearDelta1})   close to $\mathfrak{D}_1$ but far from the origin, the norm  $\norm{T}_{C^{k,\alpha}_{\delta, \tau}(\C^3) }$ is equivalent to $\norm{T}_{C^{k,\alpha}_{\delta,\tau }}$ introduced in Section \ref{Structureneardiscriminantlocus}. Similarly with the regions close to $\mathfrak{D}_2, \mathfrak{D}_3$ and far away from the origin.
\item
In the region $A^{1/4}|\vec{\mu}|_a \lesssim C_1$ where the metric $\omega^{(2)}$ is $C^\infty$-equivalent to $\sqrt{-1}\sum dz_i \wedge d\bar{z}_i$ and $|z_i|\lesssim A^{-1/4}$, the norm $\norm{T}_{C^{k,\alpha}_{\delta, \tau}(\C^3) }$ is equivalent to the normalised $C^{k,\alpha}$-norm on the complex ball of radius $\sim A^{-1/4}$. For example on this ball the mollified functions $\mu_i$ satisfy $\norm{\mu_i}_{  C^{k,\alpha}  }\lesssim A^{-1/2} $ for $i=1,2$.

\end{itemize}

The point is that these regions cover the entire $\C^3$ and the norms are equivalent on overlapping regions, where the equivalence factor is independent of $A$. These norms define the corresponding Banach spaces of $T^2$-invariant functions/tensors on $\C^3$. The spaces which are most relevant for us are $C^{k,\alpha}_{\delta,\tau}(\C^3)$, $C^{k,\alpha}_{\delta,\tau}(\C^3, \Lambda^1)$,
$C^{k,\alpha}_{\delta,\tau}(\C^3, \text{Sym}^2)$ and
$C^{k,\alpha}_{\delta,\tau}(\C^3, \Lambda^{1,1})$, corresponding to functions, 1-forms, real symmetric 2-tensors and real (1,1)-forms.


The \textbf{volume form error} function $E^{(2)}$ is defined by 
\begin{equation}
\frac{3}{4}(1+ E^{(2)}) \sqrt{-1}  \Omega \wedge
\overline{\Omega}   = 
(\omega^{(2)})^3 .
\end{equation}
By construction $E^{(2)}=E^{(1)}$ outside of the compact region where the surgery takes place. It follows from the discussions of Section \ref{MetricbehaviourawayfromDelta} and \ref{Structureneardiscriminantlocus} that

\begin{lem}\label{volumeformerrorE2}
The volume form error has the global estimate
\[
\norm{ E^{(2)} }_{ C^{k,\alpha}_{-1,-1} (\C^3)  }\leq C.
\]
\end{lem}

\section{Hein's package and weighted Sobolev inequality}\label{Heinpackage}

The following few Sections address the analytic problems.  For convenience we assume $C^{-1}\delta_{ij}\leq a_{ij}\leq C\delta_{ij}$, although we will indicate $A$-dependence in strategic places. We rely heavily on the work of Hein (\cf Chapter 3,4 in \cite{Heinthesis}) which sets out a framework for solving the complex Monge-Amp\`ere equation and its linear cousin the Poisson equation on complete noncompact manifolds, building on the seminal paper by Tian and Yau \cite{TianYau}. We explain Hein's results in a variant form which follows from his arguments.
The ambient complete manifold $M$ needs to satisfy the following analytic properties:

\begin{itemize}
\item There is a $C^{k,\alpha}$ quasi-atlas with $k\geq 3$, meaning a collection of charts on which the complex structure and the metric have $C^{k,\alpha}$ bounds, and the injectivity radius/regularity scale in these charts are bounded below. Clearly this condition is satisfied on $(\C^3, \omega^{(2)})$. This assumption allows one to speak of (unweighted) H\"older spaces.

\item 
There is a function $\rho(x)$  uniformly equivalent to the distance function $\text{dist}(0,x)$ outside the unit ball, and satisfies $|\nabla \rho|+ \rho |\nabla^2 \rho| \leq C$. It is easy to check $\rho=\sqrt{|\vec{\mu}|_a^2 + A^{-1/2} }$ works for $\C^3$. This assumption is useful in integration by part arguments.

\item 
We need the \textbf{weighted Sobolev inequality} on functions: assume the power law volume growth $\text{Vol}(B(r))\sim r^{p'}  $ with rate $p'>2$. (In our case of interest $\dim_\R M=6$, $p'=4$.)
For $1\leq p\leq \frac{ \dim_\R M }{\dim_\R M-2}$ and functions $u$ with $L^2$-gradient,
\[
(\int |u|^{2p} (1+ \rho)^{ p(p'-2)-p'  }d\text{Vol} )^{1/p} \leq C\int |\nabla u|^2.
\]
These inequalities differ from the standard Sobolev inequalities in the sense that they do not require the manifold to have Euclidean volume growth, which makes them remarkably flexible.
\end{itemize}

The output of this package is:
\begin{itemize}
\item  (Poisson equation case) Let $f\in C^{0, \alpha}$ satisfy $|f|\leq C \rho^{-q  }$ for given $p'>q>2$. Then there is a unique $C^{2,\alpha}$ solution to $\Lap u=f$ with decay estimate $|u|\leq C \rho^{2-q+\epsilon}$, where $\epsilon\ll 1$ is any fixed small number satisfying $2-q+\epsilon<0$.

\item (Complex Monge-Amp\`ere equation case) Denote $\omega_0$ as the ambient K\"ahler form. Let $f\in C^{2, \alpha}$ satisfy $|f|\leq C \rho^{ -q  }$ for $p'>q>2$. Then there is some $0< \alpha'\leq \alpha$ and $u\in C^{4, \alpha'}$ which solves $(\omega_0+ \sqrt{-1}\partial \bar{\partial} u )^{\dim_\C M }=e^f \omega_0^{ \dim_\C M  }$,
with decay estimate $|u|\leq C \rho^{2-q+\epsilon}$, where $\epsilon\ll 1$ is any fixed small number. 
\end{itemize}

Here we have separated the assumptions on the ambient manifolds from the decay assumptions to emphasize that these are difficulties of distinct nature. The key idea in Hein's package is to obtain a priori $L^\infty$ estimates and power law decay estimates on potentials via the method of weighted Moser iteration, which hinges on the weighted Sobolev inequalities. The estimates from Hein's package are constructive. It is essential to assume \textbf{faster than quadratic decay} on the source function $f$, because the method needs the potential $u= O(\rho^{2-q})$ to be bounded. Another important remark is that Hein's method respects compact group actions.

We give an elementary proof for the following

\begin{prop}
For $1\leq p\leq \frac{3}{2}$,
the weighted Sobolev inequality 
\begin{equation}\label{weightedSobolev}
(\int_M |u|^{2p} (A^{-1/4}+ |\vec{\mu}|_a)^{2p-4} d\text{Vol} )^{1/p} \leq C \int_M |\nabla u|^2 d\text{Vol}
\end{equation}
holds for $T^2$-invariant functions on $(\C^3, g^{(2)})$.	The constant here depends only on the scale invariant ellipticity bound
(\ref{scaleinvariantellipticity})
\end{prop}

\begin{proof}
By scaling analysis we may assume $A\sim 1$.
Let $u$ be a $T^2$-invariant function with $\int_M |\nabla u|^2=1$, so descends to a function on the base $\R^4$. 
Since the weighted Sobolev inequality holds on Euclidean $\R^4$ (by an interpolation of standard Sobolev inequality and Hardy inequality), 
\[
(\int_{\R^4} |u|^{2p} (1+ |\vec{\mu}|_a)^{2p-4} d\text{Vol}_a )^{1/p} \leq C \int_{\R^4} |\nabla_{g_a} u|^2 d\text{Vol}_a \leq C\int_{M} |\nabla_{g^{(2)}} u|^2 d\text{Vol}\leq C,
\]
where the second inequality is easily seen 
 using the model metric in Section \ref{Structureneardiscriminantlocus}. The LHS in this inequality is uniformly equivalent to the LHS in (\ref{weightedSobolev}) except in the region  $\{\text{dist}_{g_a}(\cdot,\mathfrak{D})\leq 1 \}$. So we are left to prove
\[
(\int_{ \text{dist}(\cdot,\mathfrak{D})\lesssim 1    } |u|^{2p} (1+ |\vec{\mu}|_a)^{2p-4} d\text{Vol} )^{1/p} \leq C.
\]

For $x\in \mathfrak{D}_1, \mathfrak{D}_2, \mathfrak{D}_3$, Sobolev inequality on bounded balls imply
\[
(\int_{ B(x, 1)  } |u- \bar{u}(x)|^{2p})^{1/p}  \leq C\int_{ B(x,2)  } |\nabla u|^2, 
\quad \bar{u}(x)= \text{Vol}(B(x,1) )^{-1} \int_{B(x,1)}u
\]
Furthermore we can find a point $x'$ with $\text{dist}(x,x' )\leq 3$, $\text{dist}(x', \mathfrak{D})\gtrsim 2 $, and by Sobolev inequality
\[
(\int_{ B(x', 1)  } |u- \bar{u}(x')|^{2p})^{1/p}  \leq C\int_{ B(x,5)  } |\nabla u|^2, 
\quad \bar{u}(x')= \text{Vol}(B(x',1) )^{-1} \int_{B(x',1)}u.
\]
By Poincar\'e inequality
\[
|\bar{u}(x)- \bar{u}(x') |^2\leq \int_{B(x,5)} |\nabla u|^2.
\]
Combining these,
\[
(\int_{ B(x, 1)  } |u|^{2p})^{1/p} \leq C(\int_{ B(x', 1)  } |u|^{2p})^{1/p} + C\int_{B(x,5)} |\nabla u|^2.
\]
Multiplying this inequality by $(1+ |\vec{\mu}|_a(x))^{(2p-4)/p}$, and summing over $x\in \mathfrak{D}$, we obtain
\[
\begin{split}
&(\int_{ \text{dist}(\cdot,\mathfrak{D})\lesssim 1    } |u|^{2p} (1+ |\vec{\mu}|_a)^{2p-4} d\text{Vol} )^{1/p} 
\\
\leq & C (\int_{ \text{dist}(\cdot,\mathfrak{D})\gtrsim 1    } |u|^{2p} (1+ |\vec{\mu}|_a)^{2p-4} d\text{Vol} )^{1/p} + C\int |\nabla u|^2\leq C
\end{split}
\]
as required.
\end{proof}

Now applying the $T^2$-equivariant version of Hein's result on the Poisson equation,

\begin{cor}\label{harmonicanalysiscompactball}
Let $2< q<4$ and  $0<\epsilon<q-2$.
There is a bounded \textbf{Green operator} for $T^2$-invariant functions
\[
G_{g^{(2)}}:  \{ f\in C^{0,\alpha}| f=O( |\vec{\mu}|_a^{-q} ) \text{ for large $|\vec{\mu}|_a$}  \} \to \{ u\in C^{2,\alpha}| u=O( |\vec{\mu}|_a^{2-q+\epsilon} )  \}.
\]
such that $u=G_{g^{(2)}} f$ satisfies $\Lap_{ g^{(2)} } u=f$.
\end{cor}

This mapping property is rather crude and unsuited for functions with slow decay rates at infinity. Improving our understanding of the Green operator shall be the task of Section \ref{Harmonicanalysis}.

 Recall from Section \ref{Structureneardiscriminantlocus} the model metric $g_{\text{Taub}}$  on a $\Z$-quotient of the space $\text{Taub-NUT}\times \C$. We can view
 $T^2$-invariant functions as pullbacks of functions on the metric product space $\text{Taub-NUT}\times \R$.
A variant of the above discussions leads to weighted Sobolev inequalities and Green's function estimates for   $g_{\text{Taub}}$:

\begin{cor}\label{harmonicanalysisonTaubNUTmodel}
	Let $2< q<4$ and  $0<\epsilon<q-2$.
	There is a bounded Green operator for $T^2$-invariant functions on the model space with the metric $g_{\text{Taub}}$
	\[
	G_{\text{Taub}}:  \{ f\in C^{0,\alpha}| f=O( |\vec{\mu}|_a^{-q} ) \text{ for large $|\vec{\mu}|_a$}  \} \to \{ u\in C^{2,\alpha}| u=O( |\vec{\mu}|_a^{2-q+\epsilon} )  \}.
	\]
	such that $u=G_{\text{Taub}} f$ satisfies $\Lap_{ \text{Taub} } u=f$.
\end{cor}

The gist is that the Green's function for $g_{\text{Taub}}$ decays like $O( |\vec{\mu}|_a^{-2+\epsilon  } )$ at infinity.

\section{Harmonic analysis  }\label{Harmonicanalysis}

This Section develops more precise mapping properties for the Green operator $G_{g^{(2)}}$. Since we are ultimately interested in K\"ahler metrics rather than potentials, we need to bound the zeroth order operator $\sqrt{-1} \partial \bar{\partial } G_{g^{(2)}  }$ for input functions with \textbf{slow decay} such as $E^{(2)}$, a task which requires rather intricate harmonic analysis. Our strategy is to construct a \textbf{parametrix} by divide and conquer. In this Section we shall assume $C^{-1} \delta_{ij} \leq a_{ij} \leq C\delta_{ij}$, and indicate $A$-dependence in strategic places. The main result is Proposition \ref{harmonicanalysismain}.

Recall $\Lap_a$ is the Laplacian for the Euclidean metric $g_a$ on the base $\R^4$. We shall identify $T^2$-invariant functions with functions on the base $\R^2_{\mu_1,\mu_2}\times \C_\eta= \R^4$.

\begin{lem}\label{harmonicanalysislemma1}
Let $-3<\delta<0$ and $\delta+\tau<0$.		
Let $f$ be a $T^2$-invariant function on $\C^3$ supported in $\{ \text{dist}(\cdot, \mathfrak{D}) \gtrsim 1\}$ with $\norm{f}_{ C^{k,\alpha}_{\delta, \tau   }(\C^3) }\leq 1$. Then the second order derivatives of the Euclidean potential $\Lap_a^{-1}f$ satisfies
\[
|\nabla^2_{g_a} \Lap_a^{-1} f|_{g_a} \leq C \ell^{ \delta  } 
(|\vec{\mu}|_a+1)^\tau .
\]
Morever if $\delta<-1$ and $\delta+\tau<-1$, then
\[
\norm{ \nabla^2_{g^{(2)}} \Lap_a^{-1} f}_{ C^{k,\alpha}_{\delta, \tau}(\C^3)   }\leq C.
\]
The constants depend only on $k, \alpha, \delta, \tau$ and the uniform ellipticity bound on $a_{ij}$.
\end{lem}

\begin{proof}
The main task is to estimate the Calderon-Zygmund type operator 
\[
G_{ij}f(x)=  \int_{\R^4} \frac{(x-y)_i (x-y)_j }{ |x-y|_a^6  } f(y) d\text{Vol}_a(y). 
\]
where $(x-y)_i$ denotes the components of $x-y$ viewed as a vector in $\R^4$.
We say $x\in \R^4$ belongs to the dyadic scale $|x|\sim 2^n$ where $n\in \N$, if either $n=0$ and $|x|\leq 1$, or $n>0$ and $2^n\leq |x|\leq  2^{n+1}$. To ensure the Green operator is well defined, we will temporarily assume $f$ to be compactly supported, with no quantitative restriction on the measure of its support.

Since $\delta> -3$ and $|f(y)|\lesssim \ell(y)^{\delta} |\vec{y}|_a^\tau$, we have
$\norm{f}_{ L^1( |y|\sim 2^m  )} \lesssim  2^{m(\delta+\tau+4) }  $. Thus if $|x|\sim 2^n$ does not belong to scale $m$, then the contribution of $|y|\sim 2^m$ to $G_{ij}(x)$ is
bounded by $O( 2^{m(\delta+\tau+4)} \min \{  2^{-4n }, 2^{-4 m}  \}  )$. Adding up all contributions from $m\neq n$, we get
\[
|G_{ij}f(x)- \int_{|y|\sim 2^n} \frac{(x-y)_i (x-y)_j }{ |x-y|_a^6  } f(y) d\text{Vol}_a(y)| \lesssim 2^{ n(\delta+\tau)  } \lesssim (1+|x|_a)^{\delta+\tau},
\]
using $\delta+\tau<0$ for the summability of the series. Since the source is far from the observer, elliptic bootstrap implies higher order H\"older regularity.

Now that we are left with only one scale, it is clear that the claimed bound for second derivatives holds
where $\ell$ is comparable to $|\vec{\mu}|_a$. We now focus on $x$ close to $\mathfrak{D}$. The contribution of $|y-x|_a\gtrsim \ell(x), |y|\sim |x|$ is estimated by
\[
\begin{split}
& C\int_{  |y-x|_a\gtrsim \ell(x), |y|\sim |x|      } \frac{1}{|x-y|_a^4 } \ell(y)^{\delta} |y|^\tau d\text{Vol}_a(y) \\
\leq & C (1+|x|_a)^\tau \int_{ |y-x|_a\gtrsim \ell(x)  } \frac{1}{|x-y|^{4-\delta}  } (\frac{\ell(y-x)}{ |x-y|_a  }  )^\delta d\text{Vol}_a(y)
\\
\leq & C (1+|x|_a)^\tau   \int_{r> \ell(x)} r^{\delta-1} dr \int_{S^3} (\frac{\ell(y')}{ |y'|  }  )^\delta d\text{Area}_{S^3}(y')
\\
\leq & C(1+|x|_a)^\tau \ell(x)^\delta 
\end{split}
\]
where we use $-3<\delta<0$ in the convergence of the integrals. 
Since the contribution comes from sources at distance at least $\ell(x)$ away from the observer, the higher H\"older norms are controlled. Finally, the estimates for the  contribution from $|y-x|_a\lesssim \ell(x)$ follows simply from standard Schauder theory.

At this stage we have proved the second derivative bound
\[
| \nabla^2_{g_a} \Lap_a^{-1} f|_{g_a} \leq C \ell^{\delta}  (|\vec{\mu}|_a+1)^\tau
\] 
together with an implicit weighted $C^{k,\alpha}$-bound in the $g_a$-metric.
Since $f$ is compactly supported by our temporary assumption, qualitatively $\Lap^{-1}_a f$ has quadratic decay at infinity. By integrating the second order derivatives from infinity, we can bound first order derivatives $d\Lap_a^{-1} f$:
\[
| d \Lap_a^{-1} f |_{g_a}
 \leq 
C \ell^{\delta+1} (|\vec{\mu}|_a+1)^\tau,
\]
using $\delta+\tau<-1$ and $\delta<-1$ in the integration. Alternatively the first order derivative bounds can be proved using the same singular integral operator method.

Now the Hessian $\nabla^2_{g^{(2)}} \Lap_a^{-1} f$ can be expanded as a linear combination of second derivatives $\frac{\partial^2}{\partial \mu_i \partial \mu_j} \Lap_a^{-1} f$ etc and first derivatives $\frac{\partial}{\partial \mu_i } \Lap_a^{-1} f$ etc. Hence
\[
|\nabla^2_{g^{(2)}} \Lap_a^{-1} f  | \leq   \sum |\frac{\partial^2}{\partial \mu_i \partial \mu_j} \Lap_a^{-1} f| |\nabla \mu_i| |\nabla \mu_j| + \sum  |\frac{\partial}{\partial \mu_i } \Lap_a^{-1} f| |\nabla_{g^{(2)}}^2 \mu_i|,
\]
where the sum includes also $\eta$-derivatives. Higher order derivatives of the Hessian can be expanded by the Leibniz rule. Using $\norm{d\mu_i}_{C^{k,\alpha}_{0, 0}  }\leq C $ and $\norm{d\eta}_{C^{k,\alpha}_{0, 0}  }\leq C $, we obtain the Hessian bound $\norm{ \nabla^2_{g^{(2)}} \Lap_a^{-1} f  }_{ C^{k,\alpha}_{\delta, \tau}(\C^3)   }\leq C$ as claimed.

Finally, an approximation argument in the weak topology removes the compact support assumption on $f$, so we conclude that $\nabla^2_{g^{(2)}} \Lap_a^{-1}$ extends canonically to a bounded linear operator between the weighted H\"older spaces.

As a delicate side remark, to bound the integral operator $\Lap_a^{-1}$ itself we would need to impose further $\delta+\tau<-2$ and $\delta<-2$, which would not be adequate for our intended applications. It is crucial in the above argument that the integral kernel of $G_{ij}$ decays two orders faster than the Green kernel.
\end{proof}

The Laplacian $\Lap_{g^{(2)}}$ is the trace of the Hessian $\nabla^2_{g^{(2)}  }$. The  idea of the next Lemma is that for $T^2$-invariant functions $G_{g^{(2)}}$ should be well approximated by $\Lap_a^{-1}$ as long as we stay sufficiently away from the discriminant locus $\mathfrak{D}$.

\begin{lem}\label{harmonicanalysislemma2}
In the situation of Lemma \ref{harmonicanalysislemma1},
\[
\norm{ \Lap_{g^{(2)}}\Lap_a^{-1} f -f  }_{ C^{k, \alpha}_{\delta-1, \tau }(\C^3) } \leq C .
\]
In particular for a large enough constant $C_2$,
\[
\norm{ \Lap_{g^{(2)}}\Lap_a^{-1} f -f  }_{ C^{k, \alpha}_{\delta, \tau }(\C^3\cap \{ \text{dist}(   \cdot, \mathfrak{D})>C_2/2 \}  ) } \leq \frac{C}{ C_2 } \norm{f}_{ C^{k,\alpha}_{\delta, \tau}(\C^3)  } \ll \norm{f}_{ C^{k,\alpha}_{\delta, \tau}(\C^3)  }. 
\]
\end{lem}

\begin{proof}
This follows from Lemma \ref{Structureneardiscriminantlocuslemma}, Remark \ref{StructurenearDeltaflatmodelclosetoTaubNUT}, and the fact that 
for the flat model $g_{\text{flat}}$ the Laplacian on $T^2$-invariant functions coincides with the base Laplacian $\Lap_a$.
\end{proof}

Next we study the Green operator $G_{\text{Taub}}$ for the model metric $g_{\text{Taub}}$.

\begin{lem}\label{harmonicanalysislemma3}
Let  $\tau<1 $ and $0<\epsilon\ll1$. 
Let $f$ be a $T^2$-invariant function supported on $\{ \text{dist}_{g_a}(\cdot, \mathfrak{D}_1)< C_2 \}$ inside the model space, with norm $\norm{f}_{ C^{k,\alpha}_{\delta, \tau} }= 1$, so that  $  \norm{f}_{ C^{k,\alpha}_{0, \tau} }\lesssim 1  $. Then 
\[
\begin{cases}
\norm{ G_{ \text{Taub} } f  }_{ C^{k+2,\alpha}_{-1+\epsilon, \tau}  } \leq C , \quad  -1< \tau< 1-\epsilon, \\
\norm{ G_{ \text{Taub} } f  }_{ C^{k+2,\alpha}_{0, -2+\epsilon}  } \leq C, \quad  \tau\leq -1. 
\end{cases}
\]
where the constant only depends on $C_2, \delta, \epsilon,  \tau, k,\alpha$ and the uniform ellipticity bound on $a_{ij}$. In particular if 
\[
\begin{cases}
\text{Either } -1<\tau<1, \quad -3+2\epsilon <\delta\leq 0, \\
\text{or } -2+\epsilon< \tau\leq -1, \quad \delta+\tau >-4 +2\epsilon,
\end{cases}
\]
then
$\norm{ \nabla_{ \text{Taub} }^2 G_{ \text{Taub} } f   }_{ C^{k,\alpha}_{\delta, \tau} }  \leq C.$
\end{lem}

\begin{proof}
As in the proof of Lemma \ref{harmonicanalysislemma1} we may assume $f$ has compact support to ensure a priori the well definition of $G_{\text{Taub}}f$.
We use cutoff functions to decompose $f$ into a sum of functions $f_n$ supported on $\{ n\lesssim \mu_2\lesssim n+1 ,  \text{dist}_{g_a}(\cdot, \mathfrak{D}_1) < C_2 \}$ centred around points $x_n\in\mathfrak{D}$, with 
H\"older bound
 $\norm{f_n}_{C^{k,\alpha}( B(x_n, C_2) )}\lesssim n^\tau$. At a fixed point $x$ bounded away from  $\text{supp}(f_n)$,
the contribution $G_{\text{Taub}}f_n $ is estimated by $| G_{\text{Taub}   } f_n|\lesssim n^\tau (|x- x_n|_a+1)^{\epsilon-2}$, where $\epsilon>0$ is any given small number (\cf Corollary \ref{harmonicanalysisonTaubNUTmodel} and notice the translational symmetry of $g_{\text{Taub}}$ along $\mathfrak{D}_1$). Elliptic bootstrap gives
\[
\norm{ G_{ \text{Taub} } f_n  }_{ C^{k+2,\alpha}_{0,0}( B( x, \ell(x)/10 )  )  } \lesssim n^\tau (|x-x_n|_a+1)^{\epsilon-2}.
\]
Summing over all $n\in \N$,
\[
\begin{split}
& \norm{ G_{ \text{Taub} } f  }_{ C^{k+2,\alpha}_{0,0}( B( x, \ell(x)/10 )  )  } \lesssim  \sum n^\tau (|x-x_n|_a+1)^{\epsilon-2}
\\
\lesssim & \int_1^\infty y^\tau ( \ell(x)^2+ |\mu_2(  x)-y|^2  )^{\epsilon/2-1} dy 
\\
\lesssim & 
\begin{cases}
(|x|_a+1)^{ \tau  } \ell(x)^{\epsilon-1} \quad -1< \tau< 1-\epsilon,
\\
(|x|_a+1)^{\epsilon-2}, \quad \tau\leq -1.
\end{cases}
\end{split}
\]
Thus
\[
\begin{cases}
\norm{ G_{ \text{Taub} } f  }_{ C^{k+2,\alpha}_{-1+\epsilon, \tau}  } \leq C , \quad  -1< \tau< 1-\epsilon, \\
\norm{ G_{ \text{Taub} } f  }_{ C^{k+2,\alpha}_{0, -2+\epsilon}  } \leq C, \quad  \tau\leq -1. 
\end{cases}
\]
which controls
$\norm{ \nabla_{ \text{Taub} }^2 G_{ \text{Taub} } f   }_{ C^{k,\alpha}_{\delta, \tau} }$
under the numerical conditions on weight exponents. 
\end{proof}

Let $C_3\gg \max(C_1, C_2)$ be a large constant to be determined, depending on $k, \alpha, \delta, \tau, C_2$ and the ellipticity constant for $a_{ij}$. Let $\chi_1$ be a cutoff function with regularity scale $\sim C_3$ on $\R^4$,
\[
\chi= \begin{cases}
1 \quad     |\vec{\mu}|_a > 2C_3 > 4 C_1 \text{dist}_{g_a}(\cdot, \mathfrak{D}_1)   ,
\\
0 \quad    |\vec{\mu}|_a < C_3 \text{ or } \text{dist}_{g_a}(\cdot, \mathfrak{D}_1) > C_3 C_1^{-1}.
\end{cases}
\]
Over the support of $\chi_1$ the model metric $g_{\text{Taub}}$ and $g^{(2)}$ coexist, so $\chi_1 G_{\text{Taub}  }f$ can be viewed as a function on $(\C^3, g^{(2)})$. The next Lemma says that outside a neighbourhood of the origin $\chi_1 G_{\text{Taub}  }f$ is a good approximate solution to the Poisson equation.

\begin{lem}\label{harmonicanalysislemma4}
In the situation of Lemma \ref{harmonicanalysislemma3}, if $C_3$ is sufficiently large, then \[
\norm{\Lap_{g^{(2)}  } (\chi_1 G_{\text{Taub} } f)-f  }_{ C^{k,\alpha}_{\delta, \tau}(\C^3\cap \{  |\vec{\mu}|_a> 2C_3   \}  )  } \leq  C C_3^{-\epsilon}  \ll 1.
\]
\end{lem}

\begin{proof}
The error $\Lap_{g^{(2)}  } (\chi_1 G_{\text{Taub} } f)-f$ comes from two sources: the deviation of the metric $g_{\text{Taub}}$ from $g^{(2)}$, and the cutoff error. The metric deviation error is estimated in Lemma \ref{Structureneardiscriminantlocuslemma}
 which we recall as
$
\norm{ g_{\text{Taub}}- g^{(2)} }_{ C^{k,\alpha}_{ 0, -1}  } \leq C.
$
In particular for $|\vec{\mu}|_a> C_3$ and on the support of $\chi_1$, we have
$
\norm{ g_{\text{Taub}}- g^{(2)} }_{ C^{k,\alpha}_{ 0, 0}  } \leq CC_3^{-1} ,
$
so the metric deviation error is $O( C_3^{-1} )$.

We turn to the cutoff error. By Lemma \ref{harmonicanalysislemma3} 
\[
\begin{cases}
\norm{ \chi_1 G_{ \text{Taub} } f  }_{ C^{k+2,\alpha}_{-1+\epsilon, \tau}  } \leq C , \quad  -1< \tau< 1-\epsilon, \\
\norm{ \chi_1 G_{ \text{Taub} } f  }_{ C^{k+2,\alpha}_{0, -2+\epsilon}  } \leq C, \quad  \tau\leq -1. 
\end{cases}
\]
which implies
\[
\begin{cases}
\norm{\nabla^2_{ \text{Taub} }( \chi_1 G_{ \text{Taub} } f )  }_{ C^{k+2,\alpha}_{-3+\epsilon, \tau}  } \leq C , \quad  -1< \tau< 1-\epsilon, \\
\norm{ \nabla^2_{ \text{Taub} }( \chi_1 G_{ \text{Taub} } f ) }_{ C^{k+2,\alpha}_{-2, -2+\epsilon}  } \leq C, \quad  \tau\leq -1. 
\end{cases}
\]
so in particular on  $\text{supp}(d\chi_1)\cap \{ |\vec{\mu}|_a >2 C_3  \}$ where $\ell \sim C_3C_1^{-1}$, we have 
\[
\norm{ \nabla^2_{\text{Taub}  }( \chi_1 G_{ \text{Taub} } f ) }_{ C^{k,\alpha}_{\delta, \tau}  }
\leq C C_3^{-\epsilon}   . 
\]
By the support assumptions $f=0$ on $\text{supp}(d\chi_1)\cap \{ |\vec{\mu}|_a >2 C_3  \}$,
hence the cutoff error is $O( C_3^{-\epsilon})$. Combining the two errors give the claim.
\end{proof}

Clearly completely analogous results apply to the neighbourhood of $\mathfrak{D}_2, \mathfrak{D}_3$.

The  source supported in a bounded region is treated by

\begin{lem}\label{harmonicanalysislemma5}
Assume
\[
\begin{cases}
\text{Either } & -2\leq \delta\leq 0, \quad \tau>-2, \\
\text{Or } & \delta\leq -2,\quad \delta+\tau>-4,
\end{cases}
\]
and let $0<\epsilon\ll 1$ depending on $\delta, \tau$. If $f$ is supported in the ball $\{ |\vec{\mu}|_a < 4C_3 \}$, with bound $\norm{f}_{C^{k,\alpha}_{\delta, \tau  }(\C^3)} \leq 1$ or equivalently $ \norm{f}_{C^{k,\alpha}_{0, 0 }(\C^3)} \lesssim 1  $, then
$
\norm{ G_{g^{(2)} } f} _{ C^{k,\alpha}_{0, -2+\epsilon  }(\C^3)  } \leq C,
$
so in particular \[
\norm{\nabla^2_{g^{(2)}} G_{g^{(2)} } f} _{ C^{k,\alpha}_{\delta, \tau  }(\C^3)  } \leq   \norm{\nabla^2_{g^{(2)}} G_{g^{(2)} } f} _{ C^{k,\alpha}_{-2, -2+\epsilon  }(\C^3)  }          \leq C.
\]
\end{lem}

\begin{proof}
The absolute value is estimated by Corollary \ref{harmonicanalysiscompactball}:
\[
|G_{g^{(2)}} f |\leq C (1+ |\vec{\mu}|_a  )^{ -2+ \epsilon  }.
\]
 The higher order estimate $\norm{\nabla^2_{g^{(2)}} G_{g^{(2)} } f} _{ C^{k,\alpha}_{-2, -2+\epsilon  }(\C^3)  }          \leq C$ follows by bootstrapping, which controls $C^{k,\alpha}_{\delta, \tau}$ norm for the given range of weight exponents $\delta, \tau$.
\end{proof}

We call the polyhedral set
\[
\{ -2\leq \delta<-1, -2< \tau<1, \delta+\tau<-1   \}\cup \{ -3<\delta\leq -2,  \tau<1, -4<\delta+\tau  \}
\]
the \textbf{good range of weight exponents} for $g^{(2)}$, namely the set where all the above Lemmas apply. As long as $(\delta, \tau)$ stays within a compact subset, the estimates in the Lemmas are in fact uniform in $\delta, \tau$. The following Proposition is the main result of this Section.

\begin{prop}\label{harmonicanalysismain}
Suppose $(\delta, \tau)$ stays within a compact subset of the good range of weight exponents. Then the operator
$
\mathcal{R}'= \nabla^2_{g^{(2)}} G_{g^{(2)}}
$
extends to bounded linear operators between the weighted H\"older spaces
\[
\mathcal{R}': C^{k,\alpha}_{\delta, \tau}(\C^3)\to  C^{k,\alpha}_{\delta, \tau}(\C^3, \text{Sym}^2), \quad \norm{\mathcal{R}' }\leq C
\]
where the constant depends only on $k,\alpha$, the compact region of exponents $(\delta, \tau)$, and the scale invariant ellipticity bound
(\ref{scaleinvariantellipticity}). 
The composition with the natural projection 
\[
\mathcal{R}:   C^{k,\alpha}_{\delta, \tau}(\C^3)\xrightarrow{\mathcal{R}'}  C^{k,\alpha}_{\delta, \tau}(\C^3, \text{Sym}^2)\to C^{k,\alpha}_{\delta, \tau}(\C^3, \Lambda^{1,1})
\]
extends the operator 
$ \mathcal{R}= \sqrt{-1} \partial \bar{\partial} G_{g^{(2)}} ,  
$
which takes value in closed real (1,1)-forms and is inverse to taking trace.
\end{prop}

\begin{proof}
The key technique is to construct a \textbf{parametrix} $P_{g^{(2)}}$ for the Green operator semi-explicitly, with precise control on its mapping properties.

Given a function $f$ with $\norm{f}_{C^{k,\alpha}_{\delta, \tau}(\C^3)  }=1$, temporarily assumed to have sufficient decay at infinity, we will construct an approximate solution $u= P_{g^{(2)}} f$ to the Poisson equation as follows. Take a smooth cutoff function $\chi'$ 
\[\chi'=
\begin{cases}
1 \quad &\text{dist}(\cdot ,\mathfrak{D}) \geq 2,\\
0 \quad &\text{dist}(\cdot ,\mathfrak{D}) \leq 1, 
\end{cases}
\]
 then $\chi'f$ has norm $\norm{\chi'f}_{ C^{k,\alpha}_{\delta, \tau}(\C^3)  }\lesssim 1$ and is supported in $\{ \text{dist}(\cdot ,\mathfrak{D}) \geq 1 \}$. Applying Lemma \ref{harmonicanalysislemma1}, the function $u_0= \Lap_a^{-1} (\chi'f)$ satisfies  $\norm{ \nabla^2_{ g^{(2)}  } u_0 }_{ C^{k,\alpha}_{\delta, \tau}(\C^3) }  \leq C$. By Lemma \ref{harmonicanalysislemma2} we can choose $C_2\gg 1$ large enough independent of $f$ to ensure
 \[
 \norm{  \Lap_{g^{(2)}  } u_0-  f }_{ C^{k,\alpha}_{\delta, \tau}(\C^3\cap \{ \text{dist}(\cdot ,\mathfrak{D}) > C_2/2   \}  )   }\ll 1.
 \]

Next we take smooth cutoff functions $\chi_1', \chi_2', \chi_3'$ near $\mathfrak{D}_1, \mathfrak{D}_2, \mathfrak{D}_3$, such that 
\[
\chi_1'= 
\begin{cases}
1 \quad &\text{dist}_{g_a}(\cdot, \mathfrak{D}_1 )\leq C_2/2 \text{ and }  |\vec{\mu}|_a > 2C_2,  \\
0 \quad & \text{dist}_{g_a}(\cdot, \mathfrak{D}_1 ) \geq C_2 \text{ or } |\vec{\mu}|_a < C_2      
\end{cases}
\]
and similarly with $\chi_2', \chi_3'$. The function 
\[
f_1=\chi_1'( f- \Lap_{g^{(2)}  } u_0)=\chi_1'( f- \Tr_{g^{(2)}} \nabla^2_{g^{(2)} } u_0)
\]
is supported in $\{ \text{dist}_{g_a}(\cdot, \mathfrak{D}_1 ) \leq C_2 , |\vec{\mu}|_a \geq C_2\} $  with bound $\norm{f_1}_{C^{k, \alpha}_{\delta, \tau}  }\leq C$. So we can apply Lemma \ref{harmonicanalysislemma3} and Lemma \ref{harmonicanalysislemma4} to find $u_1=\chi_1 G_{\text{Taub}  }f_1$ with bounds
\[
\norm{ \nabla^2_{g^{(2)}  } u_1  }_{ C^{k,\alpha}_{\delta, \tau} (\C^3)    }\leq C,  
\quad 
\norm{ \Lap_{ {g^{(2)}}  } u_1-f_1  }_{ C^{k,\alpha}_{\delta, \tau}( \C^3\cap \{  |\vec{\mu}|_a>2C_3  \} )    }\leq CC_3^{-\epsilon} \ll 1.
\]
Completely analogous constructions are made near $\mathfrak{D}_2$ and $\mathfrak{D}_3$, where we obtain $u_2, u_3$ with similar bounds.

Let $\chi_4'$ be a smooth cutoff function 
\[
\chi_4'=\begin{cases}
1 \quad &|\vec{\mu}|_a \leq 2C_3,\\
0 \quad & |\vec{\mu}|_a\geq 4C_3  ,
\end{cases}
\]
and define $f_4= \chi_4' ( f- \Lap_{ g^{(2)} } (u_0+ u_1+u_2+u_3)   )$, which is supported in the ball $\{ |\vec{\mu}|_a\leq 4C_3  \}$ and admits the bound $\norm{f_4}_{ C^{k,\alpha}_{\delta, \tau} (\C^3)   }\leq C$. Then we can apply Lemma \ref{harmonicanalysislemma5} to obtain $u_4= G_{g^{(2)}  } f_4  $ with bounds
\[
\norm{ \nabla^2_{g^{(2)}  } u_4  }_{ C^{k,\alpha}_{\delta, \tau} (\C^3)    }\leq C,  
\quad 
 \Lap_{ {g^{(2)}}  } u_4=f_4. 
\]
We set $P_{g^{(2)}} f= u= u_0+u_1+u_2+u_3+u_4$. The key point is that by construction
\[
\norm{ \nabla^2_{g^{(2)}  } u  }_{ C^{k,\alpha}_{\delta, \tau} (\C^3)    }\leq C,  
\quad 
\norm{ \Lap_{ {g^{(2)}}  } u-f  }_{ C^{k,\alpha}_{\delta, \tau}( \C^3 )    } \ll 1,
\]
namely $u$ is an \emph{approximate solution to the Poisson equation with bounds}. A subtlety is that $\nabla^2_{ g^{(2)} }u$ is fully controlled while $u$ is only controlled up to an additive constant. In any event, after extending the definition of the operator by removing the fast decay hypotheses on $f$, we have defined a bounded linear operator between weighted H\"older spaces of $T^2$-invariant functions and symmetric 2-tensors on $\C^3$
\[
\nabla^2_{g^{(2)}  } P_{g^{(2)} }: C^{k,\alpha}_{\delta, \tau}(\C^3)\to  C^{k,\alpha}_{\delta, \tau}(\C^3, \text{Sym}^2),
\]
such that the operator $\Tr_{g^{(2)} }\nabla^2_{g^{(2)}} P_{ g^{(2)} }$ is an approximation to the identity. Thus
\[
\mathcal{R}'= \nabla^2_{g^{(2)}  } P_{g^{(2)} }(  \Tr_{g^{(2)} }\nabla^2 P_{ g^{(2)} }   )^{-1}
\]
is a \emph{bounded  right inverse} to $\Tr_{g^{(2)} }$. Composing with the projection to the type (1,1)-forms defines the operator
\[
\mathcal{R}= \sqrt{-1} \partial \bar{\partial } P_{g^{(2)} }( \Tr_{g^{(2)} }\nabla^2 P_{ g^{(2)} }   )^{-1},
\]
which takes value in closed (1,1)-forms and is a bounded inverse to
$\Tr_{g^{(2)}}$. It is worth commenting that the same operators work for different exponents $\delta, \tau$.

It remains to relate $\mathcal{R}$ and $\mathcal{R}'$ to the Green operator $G_{g^{(2)}}$ when $f$ has sufficient decay at infinity.  The point is that for fast decay weights $\delta+\tau<-2$ and $\delta<-2$, the Hessian control 
$ \norm{ \nabla^2_{ g^{(2)} } u }_{ C^{k,\alpha}_{\delta, \tau}(\C^3)  }\leq C$ together with the a priori qualitative decay $u\to 0$ at infinity, imply the quantitative bound 
$
\norm{  u }_{ C^{k+2,\alpha}_{\delta+2, \tau}(\C^3)  }\leq C.
$
This enables us to extend ${P}_{g^{(2)}}$ to a bounded linear operator 
\[
{P}_{g^{(2)}  }: C^{k,\alpha}_{\delta, \tau}(\C^3)\to  C^{k+2,\alpha}_{\delta+2, \tau}(\C^3)
\]
and the operator 
\[
P_{ g^{(2)}  }( \Tr_{g^{(2)} }\nabla^2 P_{ g^{(2)} }   )^{-1}: C^{k,\alpha}_{\delta, \tau}(\C^3)\to  C^{k+2,\alpha}_{\delta+2, \tau}(\C^3)
\]
defines an inverse to the Laplacian $\Lap_{g^{(2)}}$. By the uniqueness of decaying solution to the Poisson equation
  $G_{g^{(2)}}= P_{ g^{(2)}  }( \Tr_{g^{(2)} }\nabla^2 P_{ g^{(2)} }   )^{-1}$. Hence
\[
\mathcal{R}'= \nabla^2_{g^{(2)}} G_{g^{(2)}}, \quad  \mathcal{R}= \sqrt{-1} \partial \bar{\partial} G_{g^{(2)}}
\]
as required.
\end{proof}

\begin{rmk}
	The construction of $\mathcal{P}_{g^{(2)}}$, $\mathcal{R}$, $\mathcal{R}'$ can be made compatible with the symmetries of the ansatz.
\end{rmk}

\begin{rmk}
The moral of this proof is that for slowly decaying sources, it is easier to bound the Hessian of the Green operator than the Green operator itself.
\end{rmk}

\begin{cor}\label{PoissonLemma}
	(\textbf{Solution to the Poisson equation}) Let $(\delta, \tau)$ fall within the good range of weight exponents. Then given $f\in C^{k,\alpha}_{\delta, \tau}(\C^3)$, there exists a function $u$ solving $\Lap_{g^{(2)}} u=f$ with gradient bound
	\[
	\norm{ d u}_{ C^{k+1,\alpha}_{\delta+1,\tau}(\C^3,  \Lambda^1)  } \leq C A^{-1/4} \norm{ f}_{ C^{k,\alpha}_{\delta,\tau}(\C^3)  }.
	\]
\end{cor}

\begin{proof}
For $f$ with sufficient decay at infinity, we can find $u=P_{g^{(2)}} f$ with estimate
$
\norm{ \nabla^2 u}_{ C^{k,\alpha}_{\delta,\tau}(\C^3, \text{ Sym}^2)   }\leq C\norm{ f}_{ C^{k,\alpha}_{\delta,\tau}(\C^3)  }.
$
Using $\delta+\tau<-1$ and $\delta<-1$, we can integrate from spatial infinity to obtain the required gradient bound.

For a general $f$ without fast decay assumption, take a weakly convergent sequence of fast decaying functions $f_k\to f$ bounded  in $C^{k,\alpha}_{\delta, \tau}(\C^3)$, and find $u_k=P_{g^{(2)}}f_k$ with gradient bounds. After adjusting $u_k$ by additive constants to make $u_k(0)=0$, we can extract the subsequential limit $u$ of $u_k$, which solves $\Lap_{g^{(2)}} u=f$ with the gradient bound.
\end{proof}

\section{Perturbation into a Calabi-Yau metric}\label{PerturbationintoCYC3}

In this Section we complete the construction of the promised Taub-NUT type Calabi-Yau metric on $\C^3$.

\begin{lem}\label{perturbationintoCYC3lemma}
Given $0<\epsilon\ll 1$,
there is a K\"ahler metric $\omega^{(3)}= \omega^{(2)}+ \sqrt{-1} \partial \bar{\partial} \phi^{(3)}$ with estimate
\[
\norm{ d \phi^{(3)} }_{ C^{k+1,\alpha}_{ -\epsilon, -1+\epsilon} (\C^3, \Lambda^1)    }\leq C A^{-1/4},
\]
such that the volume form error $E^{(3)}$ defined by
\[
\frac{3}{4}(E^{(3)}+1) \sqrt{-1}\Omega\wedge \overline{\Omega}= (\omega^{(3)}  )^3
\]
satisfies the fast decay estimate
$
\norm{ E^{(3)} }_{ C^{k,\alpha}_{ -4-\epsilon, -4+4\epsilon} }\leq C.
$
Here the constants only depend on $k,\alpha, \epsilon, \kappa$ and the scale invariant uniform ellipticity bound
(\ref{scaleinvariantellipticity}). In particular $\omega^{(3)}$ is close to $\omega^{(2)}$ in the $C^{k,\alpha}$-topology outside a compact set, and the volume form error decay rate is faster than quadratic.
\end{lem}

\begin{proof}
By Lemma \ref{volumeformerrorE2} the initial volume form error is
\[
\norm{ E^{(2)} }_{    C^{k,\alpha}_{ -1-\epsilon, -1+\epsilon}  }\leq \norm{ E^{(2)} }_{    C^{k,\alpha}_{ -1, -1  } } \leq C.
\]
Applying Corollary \ref{PoissonLemma} we can solve the Poisson equation with estimate
\[
\Lap_{g^{(2)} }   u_1= -2E^{(2)}, \quad \norm{ d u_1 }_{ C^{k+1,\alpha}_{ -\epsilon, -1+\epsilon} (\C^3, \Lambda^1)    }\leq C A^{-1/4},
\]
so in particular 
\[
\norm{ \partial \bar{\partial} u_1 }_{ C^{k,\alpha}_{ -1-\epsilon, -1+\epsilon}     }\leq C , \quad 
\norm{ (\partial \bar{\partial} u_1)^2 }_{ C^{k,\alpha}_{ -2-2\epsilon, -2+2\epsilon}    }\leq C .
\]
Now $(\omega^{(2)'})^3=(\omega^{(2)}+ \sqrt{-1}\partial \bar{\partial} u_1  )^3= (\omega^{(2)})^3(1+ \frac{1}{2} \Lap_{g^{(2)}} u_1  + O( |  \partial \bar{\partial} u_1 |^2   ) )  $, so the new volume form error has improved decay:
\[
\frac{3}{4}(E^{(2)'}+1) \sqrt{-1}\Omega\wedge \overline{\Omega}= (\omega^{(2)'}  )^3, \quad 
\norm{ E^{(2)'} }_{ C^{k,\alpha}_{ -2, -2+2\epsilon} }\leq
\norm{ E^{(2)'} }_{ C^{k,\alpha}_{ -2-2\epsilon, -2+2\epsilon} }\leq C.
\]
We notice that the modification to $\omega^{(2)}$ is $C^0$-small outside a compact region, where the positive definite condition for the K\"ahler metric is not affected. Inside the compact set we can add on a locally supported semipositive (1,1)-form to guarantee the K\"ahler condition, as we have done in Section \ref{surgery}. We abuse notation to write this K\"ahler metric after surgery as $\omega^{(2)'}$, which inherits all the analytic properties of $\omega^{(2)}$.

Applying Corollary \ref{PoissonLemma} again to solve the Poisson equation with background metric $g^{(2)'}$,
\[
\Lap_{g^{(2)'} }   u_2= -2E^{(2)'}, \quad \norm{ d u_2 }_{ C^{k+1,\alpha}_{ -1, -2+2\epsilon} (\C^3, \Lambda^1)    }\leq C A^{-1/4},
\]
and using
$
(\omega^{(2)'}+ \sqrt{-1}\partial \bar{\partial} u_2  )^3= (\omega^{(2)'})^3 (1+ \frac{1}{2} \Lap_{g^{(2)'}} u_2 
 + O( | \partial \bar{\partial} u_2|^2 ) ) , 
$
the new volume form error is now bounded in $C^{k,\alpha}_{ -4, -4+4\epsilon} $-norm. Another surgery in the compact region ensures the K\"ahler property.
\end{proof}

Now we can prove the main theorem of this Chapter.

\begin{thm}\label{TaubNUTC3main}
	(\textbf{Taub-NUT type Calabi-Yau metric on $\C^3$})
There exists a complete metric $\omega_{\C^3}= \omega^{(2)}+ \sqrt{-1}\partial \bar{\partial} \phi^{\C^3}  $ on $\C^3$ satisfying 
$
 \omega_{\C^3}   ^3= \frac{3}{4} \sqrt{-1}\Omega\wedge \overline{\Omega}
$,
with metric deviation estimate \[
\norm{ d \phi^{\C^3}}_{ C^{k+1,\alpha}_{ -\epsilon, -1+\epsilon} (\C^3, \Lambda^1)    }\leq C A^{-1/4}, \quad \norm{\sqrt{-1} \partial \bar{\partial}\phi^{\C^3}}_{ C^{k,\alpha}_{-1 -\epsilon, -1+\epsilon} (\C^3, \Lambda^{1,1})    }\leq C .
\]
Here $0<\epsilon\ll 1$ is an arbitrarily small given number, and the constants depend only on 
$k,\alpha, \epsilon$ and the scale invariant uniform ellipticity bound
(\ref{scaleinvariantellipticity}). 
This metric inherits all the symmetries of $\omega^{(2)}$.
\end{thm}

\begin{proof}
We assume $C^{-1}\delta_{ij}\leq a_{ij}\leq C\delta_{ij}$ which can be relaxed by scaling.
It suffices to solve the complex Monge-Amp\`ere equation
\[
(\omega^{(3)}+ \sqrt{-1}\partial \bar{\partial} \phi' )^3= \frac{3}{4} \sqrt{-1}\Omega\wedge \overline{\Omega}.
\]
In Hein's analytic package (\cf Section \ref{Heinpackage}), the conditions on the ambient metric including $C^{k,\alpha}$ quasi-atlas, existence of a distance-like function with Hessian bounds, and the weighted Sobolev inequality, are robust conditions which are inherited by $\omega^{(3)}$ from $\omega^{(2)}$. The volume form error $E^{(3)}$ has faster than quadratic decay by construction:
\[
|E^{(3)}| \leq C(|\vec{\mu}|_a+1)^{-4+\epsilon}.
\]
 Thus Hein's package provides a potential $\phi'$ solving the complex Monge-Amp\`ere equation with decay estimate $|\phi'|\leq C (|\vec \mu|_a+1)^{ -2+2\epsilon } $. Elliptic bootstrap gives the bound
 $
 \norm{\phi'}_{ C^{k+2,\alpha}_{0, -2+2\epsilon  } }\leq C,
 $
so $ \norm{d\phi'}_{ C^{k+1,\alpha}_{ -1, -2+2\epsilon} (\C^3, \Lambda^1)       } \leq C$, 
which combined with Lemma \ref{perturbationintoCYC3lemma} implies the metric deviation estimate.
\end{proof}

Some immediate geometric consequences are

\begin{cor}
The Taub-NUT type Calabi-Yau  metric $g_{\C^3}$ has \textbf{volume growth rate}
\[
C^{-1} \leq  \frac{\text{Vol}( B_{g_{\C^3}}(r)  )}{ r^4 }\leq C, \quad r\geq A^{-1/4}.
\]
and the \textbf{tangent cone at infinity} is the Euclidean $\R^4$.
\end{cor}

\begin{cor}
The \textbf{Riemannian curvature} satisfies the decay estimate \[
|\text{Rm}| \leq C A^{-1/4}\ell^{-3}.
\]

\end{cor}

\begin{proof}
Using the metric deviation estimate, the Riemannian curvature is bounded. Morever if $\ell  > 2 A^{-1/4}$, then we can find a flat model $g_{\text{flat}}$ over a $g_a$-ball  of radius $\sim \ell/10$, where $\norm{g_{\text{flat}}- g_{\C^3} }_{C^{k,\alpha}_{-1, 0}  } \leq C.
$
Using this bound up to second order derivatives, the Christoffel symbols in the local flat coordinates are $O( A^{-1/4}\ell^{-2})$ and the Riemannian curvature is of order $O( A^{-1/4}\ell^{-3}  )$.
\end{proof}

In particular, in the generic region where $|\vec{\mu}|_a$ is comparable to $\ell$, the Riemannian curvature decays as
$
\text{Rm}(x)= O(\text{dist}(x,0)^{-3}), 
$
although the Riemannian curvature does not decay at infinity along $\mathfrak{D}_i$.

\begin{cor}\label{specialLagrangianfibrationC3}
There exist $T^2$-moment coordinates $\tilde{\mu}_1^{\C^3}$, $\tilde{\mu}_2^{\C^3}$ on $(\C^3, \omega_{\C^3})$ with global estimate
\[
|\mu_i - \tilde{\mu}_i^{\C^3}| \leq CA^{-3/4} \ell^{-\epsilon} |\vec{\mu}|_a^{-1+\epsilon} .
\]
The map $\C^3\xrightarrow{( \tilde{\mu}_1^{\C^3}, \tilde{\mu}_2^{\C^3}, \text{Im}(\eta)  )} \R^3$ is a \textbf{special Lagrangian fibration} with phase angle zero, whose critical point set is $\bigcup_{i,j}\{ z_i=z_j=0  \}$ and whose discriminant locus agrees with (\ref{positivevertexDelta}). 
\end{cor}

\begin{rmk}
Moment coordinates for the Taub-NUT type metric $\omega_{\C^3}$ should not be confused with the moment coordinates $\mu_i$ for the K\"ahler ansatz.	
\end{rmk}

\begin{proof}
The existence of moment coordinates follows from $H^1(\C^3)=0$, but for the purpose of estimation we wish to relate $\tilde{\mu}_i^{\C^3}$ to $\mu_i$  outside the ball $\{|\vec{\mu}|_a\lesssim A^{-1/4}  \}$ where the surgery was performed. In this exterior region
\[
\omega_{\C^3}= \omega^{(2)}+\sqrt{-1}\partial \bar{\partial} \phi^{\C^3}= \omega^{(1)}+d d^c \phi^{\C^3}, \quad d^c= \frac{\sqrt{-1}}{2}( \bar{\partial}- \partial  ).
\]
The 1-form $d^c \phi^{\C^3}$ is $T^2$-invariant, so by Cartan's formula
\[
\iota_{ \frac{\partial}{\partial \theta_i}  } dd^c \phi^{\C^3}=-
d \iota_{ \frac{\partial}{\partial \theta_i}  } d^c \phi^{\C^3},
\]
which combined with $d\mu_i = -\iota_{ \frac{\partial}{\partial \theta_i}  }  \omega^{(1)}    $ allow us to find the moment coordinates:
\[
\tilde{\mu}_i^{\C^3}= \mu_i + \iota_{ \frac{\partial}{\partial \theta_i}  } d^c \phi^{\C^3}, \quad d\tilde{\mu}_i^{\C^3} = -\iota_{ \frac{\partial}{\partial \theta_i}  }  \omega_{\C^3}  , \quad i=1,2.
\]
Using the estimates
$
|d\phi^{\C^3}| \leq CA^{-1/2}  \ell^{-\epsilon} |\vec{\mu}|_a^{-1+\epsilon} 
$
and 
$
|\frac{\partial}{\partial \theta_i}  | \leq  CA^{-1/4}  ,
$
we see
\[
|\mu_i - \tilde{\mu}_i^{\C^3}| \leq CA^{-3/4} \ell^{-\epsilon} |\vec{\mu}|_a^{-1+\epsilon} .
\]
Now inside $\{ |\vec{\mu}|_a \lesssim A^{-1/4}  \}$, we have $|\mu_i|\leq CA^{-1/2}$, and 
$|d\tilde{\mu}_i| \leq CA^{-1/4} $ integrates to give $|\tilde{\mu}_i|\leq CA^{-1/2}$. Thus globally on $\C^3$
\[
|\mu_i - \tilde{\mu}_i^{\C^3}| \leq CA^{-3/4} \ell^{-\epsilon}    (  A^{-1/4}+ |\vec{\mu}|_a)^{-1+\epsilon} 
\]
as required. Morever $\tilde{\mu}_1^{\C^3}, \tilde{\mu}_2^{\C^3}, \tilde{\mu}_1^{\C^3}-\tilde{\mu}_2^{\C^3}$ vanish respectively along $\mathfrak{D}_1, \mathfrak{D}_2, \mathfrak{D}_3$, due to the respective vanishing of the circle generators $\frac{\partial}{\partial \theta_1}, \frac{\partial}{\partial \theta_2}, \frac{\partial}{\partial \theta_1}-\frac{\partial}{\partial \theta_2}$.

Now consider the map $\C^3\xrightarrow{( \tilde{\mu}_1^{\C^3}, \tilde{\mu}_2^{\C^3}, \text{Im}(\eta)  )} \R^3$. It is a special Lagrangian fibration by Remark \ref{T2symmetryimpliesSL}.

At a critical point $p\in \C^3$ the Zariski tangent space of the fibre, namely the annihilator of $\text{span}(  d\tilde{\mu}_1^{\C^3}, d\tilde{\mu}_2^{\C^3}, d\text{Im}\eta  ) $, is a linear subspace of $T_p\C^3$ of real dimension at least 4. It contains $\frac{\partial}{\partial \theta_1}, \frac{\partial}{\partial \theta_2}$ and is $g_{\C^3}$-orthogonal to $I\frac{\partial}{\partial \theta_1}, I\frac{\partial}{\partial \theta_2}$. If $d\eta\neq 0$ at $p$, then $\frac{\partial}{\partial \theta_1}, \frac{\partial}{\partial \theta_2}$ are linearly independent, so the Zariski tangent space is the orthogonal complement of $\text{span}(I\frac{\partial}{\partial \theta_1}, I\frac{\partial}{\partial \theta_2})$ by dimension counting. Since $d\text{Im}\eta$ vanishes on the Zariski tangent space, and $d\eta$ vanishes on $\text{span}_{\C}(\frac{\partial}{\partial \theta_1}, \frac{\partial}{\partial \theta_2})$, we deduce $d\eta=0$ on $T_p\C^3$, contradiction. Thus the critical points must satisfy $d\eta=0$, or equivalently $p\in \bigcup_{i,j} \{ z_i=z_j=0  \}$. Conversely all points in $ \bigcup_{i,j} \{ z_i=z_j=0  \}$  are critical. Having identified the critical point set, the discriminant locus follows from the argument in Lemma \ref{discriminantlocusisinsensitive}.
\end{proof}

\section{Uniqueness and moduli}\label{Uniquenessandmoduli}

In this Section we show that under $T^2$ symmetry, there is only one complete Calabi-Yau metric $g_{\C^3}$ on $\C^3$ within a suitably restrictive asymptotic class prescribed by the metric deviation estimate in Theorem \ref{TaubNUTC3main}. The strategy is similar to the one used by Conlon and Hein \cite{ConlonHein}.

\begin{lem}
Let $\delta<-1$ and $\tau<0$. 
If a function $u$ satisfies $\Lap_{g_{\C^3}}u=0$ with bound $\norm{ du}_{ C^{k+1, \alpha}_{\delta+1, \tau}(\C^3, \Lambda^1)   }<\infty$, then $du=0$.
\end{lem}

\begin{proof}
Since $g_{\C^3}$ is a Ricci-flat metric, the Bochner formula implies 
\[
\Lap (\frac{1}{2} |\nabla u|^2)= |\nabla^2 u|^2+  \langle \Lap \nabla u, \nabla u\rangle = |\nabla^2 u|^2+  \langle  \nabla\Lap u, \nabla u\rangle= |\nabla^2 u|^2,
\]
so $|du|^2$ is a non-negative subharmonic function. The decay condition implies it converges to zero at infinity, so maximum principle gives $du=0$.
\end{proof}

\begin{prop}
Let $\delta<-1$ and $\tau<0$. If  a $T^2$-invariant potential $\phi'$ satisfies $(\omega_{\C^3}+ \sqrt{-1}\partial \bar{\partial} \phi')^3= \omega_{\C^3}^3$ with bound $\norm{ d\phi'}_{ C^{k+1, \alpha}_{\delta+1, \tau}(\C^3, \Lambda^1)   }<\infty$, then $d\phi'=0$.
\end{prop}

\begin{proof}
The strategy is to improve the decay rate of $d\phi'$ iteratively, until it becomes sufficiently fast. We rewrite the equation as a Poisson equation
\[
\frac{1}{2}(\Lap_{g_{\C^3}}\phi' )\omega_{\C^3}^3=  -3(\sqrt{-1}\partial \bar{\partial} \phi')^2 \wedge \omega_{\C^3} - (\sqrt{-1}\partial \bar{\partial} \phi')^3.
\]
Notice that $\sqrt{-1}\partial \bar{\partial}\phi'$ lives in $C^{k,\alpha}_{\delta, \tau}$, so its square lives in $C^{k,\alpha}_{2\delta, 2\tau}$.
As long as $(\delta, \tau)$ stays in the good range of weight exponents, Corollary \ref{PoissonLemma} and the above vanishing lemma imply that the solution $\phi'$ to this Poisson equation must satisfy $\norm{ d\phi'}_{ C^{k+1, \alpha}_{2\delta+1, 2\tau}(\C^3, \Lambda^1)   }<\infty$. This is an improved decay estimate because $2\delta+1<\delta$ and $2\tau<\tau$. Since each iteration improves the decay rate by a definite amount, within a finite number of steps we can assume $\delta<-2$ and $\tau<0$. Then $\norm{ d\phi'}_{ C^{k+1, \alpha}_{\delta+1, \tau}(\C^3, \Lambda^1)   }<\infty$ implies that $\norm{ \phi'}_{ C^{k+2, \alpha}_{\delta+2, \tau}(\C^3)   }<\infty$ after adjusting $\phi'$ by a constant. Then we can use the standard integration by part argument for the complex Monge-Amp\`ere equation to see
\[
\int_{\C^3} |\nabla \phi'|^2 \phi'^p \omega_{\C^3}^3 =0, \quad p\gg 1.
\]
Hence $\phi'$ is a constant, and the metric is unique.
\end{proof}

It follows from the uniqueness result that the natural \textbf{parameter space} of our Taub-NUT type metrics is 
the space of positive definite rank 2 matrices $(a_{ij})$, which involves 3 parameters. The \textbf{discrete group} $S_3$ acts on the parameter space by permuting the edges $\mathfrak{D}_1, \mathfrak{D}_2, \mathfrak{D}_3$, or equivalently interchanging the 3 positive numbers
$
a_{11}, a_{22}, a_{11}+a_{12}+a_{21}+a_{22}. 
$
This permutation does not change the holomorphic isometry type of the Taub-NUT type metrics, so the \textbf{moduli space} of our construction is the $S_3$-quotient of the parameter space. The \textbf{scaling transformations} act on the parameter space by
\[
a_{ij}\mapsto \Lambda a_{ij}, \quad A\mapsto \Lambda^2 A.
\]
The size of $A$ is inversely related to the area of the asymptotic $T^2$ in the generic region near infinity, and the inverse matrix $(a^{ij})$ up to scale describes the shape of the asymptotic $T^2$. If we restrict attention to $C^{-1}\delta_{ij}\leq a_{ij}\leq C\delta_{ij}$, then the Taub-NUT type metrics on $\C^3$ are uniformly equivalent.

We mention two interesting problems:

\begin{Question}
What kind of degenerations would happen if the scale invariant uniform ellipticity bound (\ref{scaleinvariantellipticity}) fails?
\end{Question}

\begin{Question}
Can we prove uniqueness under a weaker hypothesis? For instance, if a complete Calabi-Yau metric on $\C^3$ is uniformly equivalent to $g_{\C^3}$, then does it need to be a member of our family of Taub-NUT type metrics? If we are only given the topology of $\C^3$, then is it possible to characterise our Taub-NUT type metrics in terms of its tangent cone at infinity and some extra curvature decay conditions?
\end{Question}

The author feels this uniqueness question would be the beginning of a classification program of higher dimensional gravitational instantons (\cf Section \ref{gravitationalinstantons} for more discussions).

\section{Exotic metrics: past and future  }\label{ComparisonwithGabor}

This informal Section aims to connect the new Taub-NUT type metric on $\C^3$ to a circle of ideas in the literature, and sketch the directions for plausible generalisations and the scope for future research.

\subsection{Exotic metrics on $\C^n$}

We begin with some historical remarks about the fundamental problem:

\begin{Question}\label{exoticCnquestion}
	Given $n\geq 2$, what are the complete Calabi-Yau metrics on $\C^n$ equipped with the standard holomorphic volume form?
\end{Question}

The initial guess was that the only solution is the flat metric. The rationale is that the moduli of compact Calabi-Yau manifolds depends on the cohomology class of the K\"ahler form and the holomorphic volume form, and since $\C^n$ has trivial topology, it seemed that there was no room to admit nontrivial Calabi-Yau metrics. The situation changed when LeBrun first observed that the Taub-NUT metric gives a counterexample on $\C^2$ (\cf Section \ref{TaubNUT}). Hindsight shows that the necessary amount of nontrivial topology comes from an additional \textbf{fibration structure}. In fact the Taub-NUT metric admits two kinds of fibration structures: a holomorphic fibration $\C^2_{z_0, z_1}\xrightarrow{ z_0z_1} \C  $ which gives an algebraic perspective, and a circle fibration coming from the Gibbons-Hawking ansatz which gives a transcendental perspective.

In \cite{Li} the author realised that if we take the \textbf{holomorphic fibration} one step further, namely if we start from the standard Lefschetz fibration $\C^3\xrightarrow{f=z_1^2+z_2^2+z_3^2}\C$ on $\C^3$, then we can construct a nontrivial complete Calabi-Yau metric on $\C^3$, such that near spatial infinity, the restricted metric on the affine quadric fibres are approximately the Eguchi-Hanson metrics on the fibres, and the horizontal part of the metric is approximately the pullback of the Euclidean metric on the base. This work was soon generalised independently by Conlon-Rochon \cite{Ronan} and Sz\'ekelyhidi \cite{Gabor}, who developed more substantial linear analysis to treat more complicated holomorphic fibrations. In the most general known version, we start from a weighted homogeneous polynomial $f: \C^n\to \C$ where $n\geq 3$, such that the only singularities in the fibration $f$ are isolated singularities on the central fibre $f^{-1}(0)$, and we require the weighted cone $f^{-1}(0)$ to admit a conical Calabi-Yau metric whose Reeb vector field action is compatible with the weights. Algebro-geometrically, the singular fibre must have \textbf{klt singularity}, and the requirement for the existence of a conical Calabi-Yau metric imposes a \textbf{stability condition} on the singular fibre. Then by standard results the smoothing fibres $f^{-1}(c)$ are equipped with asymptotically conical Calabi-Yau metrics, which now play the same role as the Eguchi-Hanson metrics played in the $\C^3$ example setting. The final output of their theory is a complete Calabi-Yau metric on $\C^n$ associated to the fibration $f:\C^n\to \C$, equipped with the standard holomorphic volume form.

The most important Riemannian geometric aspect of this infinite class of complete Calabi-Yau metrics is that the volume of metric balls have \textbf{Euclidean volume growth rate}
\[
C^{-1}\leq \frac{  \text{Vol}( B(r) ) }{ \text{Vol}(B_{\text{Euclid}}^{2n}(r))    }\leq 1 , \quad r>0.
\]
Since these manifolds are Ricci-flat, it makes sense to take the \textbf{tangent cone at infinity}, which is identified as the singular variety $f^{-1}(0)\times \C$ with the product metric, and in particular has the same dimension as $\C^n$. This aspect is contrasted with the Taub-NUT metric in complex dimension 2, whose volume growth rate is $\text{Vol}(B(r)\sim r^3$ which is \emph{not} Euclidean. This failure can be traced back to the fact that the singular fibre $z_0z_1=0$ for the Taub-NUT $\C^2$ is not even irreducible, let alone having a Calabi-Yau cone metric.

Furthermore, the metric distance to the origin for these examples on $\C^n$ are bi-H\"older equivalent to the standard Euclidean distance, but not uniformly equivalent. This has the consequence that the ring of algebraic functions on these exotic $\C^n$ coincides with the ring of holomorphic functions with \textbf{polynomial growth}, but the filtration structure on these functions induced by the growth rate is \emph{not} the standard filtration.

Now we turn to the new Taub-NUT type Calabi-Yau metric on $\C^3$. Like the Taub-NUT $\C^2$, it is associated to both a holomorphic fibration structure and a torus fibration structure. The holomorphic fibration is given by $\C^3\xrightarrow{ z_0z_1z_2 }\C$, which may be viewed as a \textbf{degenerate case} where the fibration is allowed to have more severe singularities: here $z_0z_1z_2=0$ is reducible into 3 pieces, and morever its singularity is \textbf{non-isolated}, stretching all the way into spatial infinity. This explains why the \textbf{Riemannian curvature} does not decay at infinity along the locus $\{z_i=z_j=0\}$, a phenomenon similar to Joyce's examples of quasi-ALE Calabi-Yau metrics \cite{JoyceBook}. Another viewpoint is that the generic fibre is stable while the central singular fibre is \textbf{unstable}. Their delicate balance produces a global metric on $\C^3$, but the instability near the singular fibre produces large quantum fluctuation effects.

However, the principal novalty of our Taub-NUT type $\C^3$ metric comes from the \textbf{$T^2$-fibration} structure. An immediate consequence of the fact that 2 spatial dimensions are `compactified', is that the \textbf{volume growth rate is sub-Euclidean}: in fact $ \text{Vol}(B(r)\sim r^4  $ and the tangent cone at infinity is the flat $\R^4$. This sub-Euclidean growth is otherwise known as \textbf{collapsing} in Riemannian geometry.

An important conceptual feature of real tori is that they are inherently \textbf{transcendental} objects, tied up intimately with the fundamental functions $\log $ and $\exp$; we saw the pervasive presence of such transcendental functions in Section \ref{C3complexgeometricperspective} in the metric asymptote.  Another manifestation of this is that the \textbf{ring of algebraic functions} on the Taub-NUT type $\C^3$ is defined by holomorphic functions with an exponential type growth condition, rather than the more familiar polynomial growth which is the expected feature in the Euclidean volume growth situation.

The evidence suggests that the full mystery of complete Calabi-Yau metrics on $\C^n$ involves at least 3 fundamental phenomena:
\begin{itemize}
	\item holomorphic fibrations with a suitable notion of stability, which is associated with Euclidean volume growth rate and polynomial growth rate on holomorphic functions.

	\item torus fibrations, which is associated with collapsing phenomenon and exponential growth rate on holomorphic functions.

	\item an additional layer of combinatorial complexity involving iterative fibrations (\cf subsection \ref{quasiALF} for the flavour).
\end{itemize}

This picture seems to fit well with Kontsevich and Soibelman's conjectural picture for collapsing compact Calabi-Yau manifolds (\cf Chapter 2, 3 in \cite{KontsevichSoibelman}). The relation between the two situations will be further explained in subsection \ref{compactcollapsingmetrics}.

\subsection{Gravitational instantons}\label{gravitationalinstantons}

A \textbf{gravitational instanton} is a complete non-compact hyperK\"ahler 4-manifold with $\int |\text{Rm}|^2 d\text{Vol} <\infty$. The theory of gravitational instantons is very rich, with important contributions from Kronheimer, Atiyah, Hitchin, Hein, and many others. Recent breakthrough made by Chen and Chen \cite{ChenChen} is a decisive step towards a complete classification. A conspicuous feature of this classification program is the crucial role played by the volume growth rate. In the Euclidean volume growth rate case, these are the ALE metrics (`asymptotically locally Euclidean') classified by Kronheimer. In the sub-Euclidean volume growth case, in all known situations the asymptotic geometry near infinity is approximately a flat torus fibration over a flat base.

We will not attempt to review this extensive literature, but limit ourselves to examine a simple class of examples known as \textbf{multi-Taub-NUT} metrics. In the Gibbons-Hawking coordinates (\cf Section \ref{GeneralisedGibbonsHawking}), this is given by the potential
\[
V= A+ \sum_{i=1}^k   \frac{1}{ 2\sqrt{ |\mu-\mu_i|^2+ |\eta-\eta_i|^2  }  },
\]
where $(\mu_i, \eta_i)$ are disjoint given points on the base $\R^3=\R_\mu\oplus \C_\eta$, and $k\geq 1$. The asymptotic geometry is given by a degree $k$ circle bundle over the complement of a compact region in $\R^3$, whose circle fibres have approximate length $2\pi A^{-1/2}$. This behaviour is known as \textbf{asymptotically locally flat}, or ALF for short. The $k=1$ case is the usual Taub-NUT metric.

Let's assume for convenience that the $\eta_i$ are all distinct, which is the generic situation.
From the holomorphic perspective, the multi-Taub-NUT metrics lives on the smooth algebraic varieties
\[
\{ z_0z_1=F(\eta)= (\eta-\eta_1)(\eta-\eta_2)\ldots (\eta-\eta_k)\} \subset \C^3_{z_0, z_1, \eta},
\]
with nowhere vanishing holomorphic volume form $\Omega= \frac{1}{ F'(\eta)} \sqrt{-1} dz_0 \wedge dz_1  $,
and the circle action is
\[
e^{i \theta}\cdot (z_0, z_1)=  ( e^{-i\theta} z_0, e^{i\theta} z_1  ),
\]
which ensures $\iota_{ \frac{\partial}{\partial \theta}} \Omega=d\eta $.

A crucial aspect of multi-Taub-NUT metrics is that they come in a \textbf{moduli space}, determined by the positions of $(\mu_i, \eta_i)$. In general, the fact that a family of geometric objects has natural moduli indicates the possibility that in some degenerate limit they decompose into more \textbf{primary} objects, and the parameters in the moduli comes from the parameters in these building blocks and the combinatorics of the gluing construction. This is the case when the spatial separation distance of the monopole points $(\mu_i, \eta_i)\in \R^3$ is far larger than the circle length parameter $A^{-1/2}$. Then we can view the multi-Taub-NUT metric as obtained from \textbf{gluing} $k$ copies of the Taub-NUT metrics, whose curvature centres are far separated and therefore whose mutual interaction is weak.

Now we can try to push this story to higher dimensions. The natural generalisation of complete hyperK\"ahler 4-folds is \textbf{complete Calabi-Yau manifolds}. Since in our Taub-NUT type $\C^3$ example the Riemannian curvature does not decay at infinity along $\mathfrak{D}_i$, the total $L^2$-curvature integral is infinite.  Finding the correct generalised notion of \textbf{finite curvature condition} is clearly fundamental to any \text{classification program}. We do not fully understand what this notion is. 
 A tentative idea compatible with Chen and Chen's work \cite{ChenChen} and our Taub-NUT type $\C^3$ example is to require that $|\text{Rm}|\leq C$ globally and $|\text{Rm}|(x)\leq C \text{dist}(x,0)^{-2-\epsilon}$ for some $\epsilon>0$ in the \textbf{generic region}.

In the direction of constructing more examples, we comment that Hein's existence package is by no means limited to the case of $\C^n$. Focusing on complex dimension 3,  the distinguished role of our Taub-NUT type metrics on $\C^3$  is instead that they are more \textbf{primary} objects, and in particular ought to have a more rigid moduli space, than most of the other 3-dimensional complete Calabi-Yau metrics with similar behaviours. It is perhaps best to illustrate this by a conjectural example which generalises the multi-Taub-NUT metrics.

Let $\eta_i$ be all distinct and take the smooth algebraic varieties
\[
\{ z_0z_1z_2 =F(\eta)= (\eta-\eta_1)(\eta-\eta_2)\ldots (\eta-\eta_k)\} \subset \C^4_{z_0, z_1, z_2, \eta},
\]
with nowhere vanishing holomorphic volume form $\Omega= \frac{-1}{ F'(\eta)}  dz_0 \wedge dz_1 \wedge dz_2 $. These
admit a $T^2$-action 
\[
e^{i \theta_1}\cdot (z_0, z_1, z_2)=  ( e^{-i\theta_1} z_0, e^{i\theta_1} z_1, z_2  ),  \quad e^{i \theta_2}\cdot (z_0, z_1, z_2)=  ( e^{-i\theta_2} z_0,z_1,  e^{i\theta_2}  z_2  )
\]
which ensures $\Omega(  \frac{\partial}{\partial \theta_1}, \frac{\partial}{\partial \theta_2}  )=d\eta $. It seems likely that Hein's package can be made to provide a multi-parameter family of complete Calabi-Yau metrics on these varieties. Morever, when the $T^2$-fibres have much smaller lengths compared to the spatial separation of $\eta_i$, then the author expects such metrics to have a gluing description in terms of our Taub-NUT type metric on $\C^3$. On the other extreme, if we allow $\eta_i$ to collide, then we may see new metric behaviours not yet understood in the literature.

\begin{rmk}
	Another conjectural example of this flavour can be found in the final Section of the author's paper \cite{Ligluing}.
\end{rmk}

\subsection{Generalisation of ALF geometry}\label{quasiALF}

We now discuss the problems of generalising the Taub-NUT type $\C^3$ to higher dimensional exotic metrics on $\C^n$. The key issue seems to be an extra layer of \textbf{combinatorial complexity} of \textbf{recursive} nature. This calls for a theory which deals with linear analysis on \textbf{quasi-ALF} geometry. Roughly put, a quasi-ALF geometry of complexity 1 asymptotically looks like a flat torus fibration over a flat base. A quasi-ALF geometry of complexity $k$ is a singular torus fibration, whose asymptotic behaviour away from the neighbourhood of a lower dimensional stratified singular set looks ALF, and whose behaviour transverse to the singular locus is modelled on a quasi-ALF geometry of complexity $k-1$. We shall not attempt to make a formal definition, but merely point out that theories of a very similar flavour are much studied, such as QALE spaces by Joyce \cite{JoyceBook},  and QAC spaces by Degeratu and Mazzeo \cite{DegeratuMazzeo}.

A \textbf{conjectural example} which illustrates the main ideas is the direct generalisation of our Taub-NUT type metric to $\C^n$ with $n\geq 4$. We take the holomorphic fibration
\[
\C^n\xrightarrow{ \eta=z_0 z_1\ldots z_{n-1} } \C_\eta, \quad \Omega=\sqrt{-1}^{n-1} dz_0\wedge dz_1\ldots \wedge dz_{n-1}.
\] 
which admits the action by the diagonal torus $T^{n-1}\subset \text{SL}(n, \C)$. The asymptotic geometry is as follows:

\begin{itemize}
\item Far away from $\cup \{ z_i=z_j=0\}$, the metric looks like a flat $T^{n-1}$-fibration over a flat base. In the holomorphic persepcitive, the fibres of $\eta=z_0\ldots z_{n-1}$ have a almost flat cylindrical metric on $(\C^*)^{n-1}$, and the horizontal part of the metric looks like the pullback of a Euclidean metric on $\C_\eta$.

\item Near $\cup \{z_i=z_j=0\}$ but far from $\cup \{ z_i=z_j=z_k  \}$, we see the Taub-NUT metric appearing in the transverse direction to $\cup \{z_i=z_j=0\}$.

\item Near $\cup\{ z_i=z_j=z_k  \}$ but far from the intersection of 4 coordinate hyperplanes, we see the Taub-NUT type $\C^3$ appearing in the transverse direction to $\cup\{ z_i=z_j=z_k  \}$.

\ldots

\item Near $\{ z_1=\ldots z_{n-1}=0 \}$ but far from $\{ z_0=0\}$, we see the conjectural metric on $\C^{n-1}$ appearing in the transverse direction.
\end{itemize}

The point is that if one has a sufficiently powerful linear theory which could correct the initial volume form errors to have faster than quadratic decay near infinity, then one can invoke Hein's package to produce a global Calabi-Yau metric. The whole construction follows a clearly inductive pattern.

\subsection{Connection to collapsing compact Calabi-Yau metrics}\label{compactcollapsingmetrics}

A family of Calabi-Yau metrics $(X_t, \omega_t)$ living on a flat family of compact Calabi-Yau manifolds is said to be \textbf{collapsing} if there is no uniform estimate \[
\text{Vol}_{\omega_t}(B(x_t, r))\geq \kappa r^{\dim_\R X_t}, \quad \forall x_t\in X_t, \forall 0<r<\text{diam}(X_t), \quad \kappa>0.
\]
Two well-studied basic mechanisms for collapsing are:
\begin{itemize}
\item  Fix the complex structure of $X_t=X$ and a reference K\"ahler class $[\omega_X]$ on $X$. Assume there is a holomorphic fibration $f:X\to Y$ to a lower dimensional K\"ahler manifold $Y$ with K\"ahler class $[\omega_Y]$. Then we take $\omega_t$ to be the Calabi-Yau metric in the class $[t\omega_X+f^*\omega_Y ]$, where $t\ll 1$. Crucially the fibre volume is cohomologically determined, and the fibre length scale is much smaller compared to the diameter of the base (\cf \cite{Tosatti}).

\item Fix a polarisation on a 1-parameter flat family $X_t$, which prescribes the K\"ahler class, and assume there is a holomorphic volume form $\Omega_{\mathcal{X}}$ on the total space, so there are induced holomorphic volume forms $\Omega_t$ on $X_t$ depending on $t$ in a holomorphic way. Then we study the Calabi-Yau metrics $\omega_t$ as we allow the complex structure to degenerate, in such a way that the central fibre $X_0$ has \emph{worse than klt singularities}.
\end{itemize}

Kontsevich and Soibelman observe that in the polarised collapsing situation, the resolution of singularity implies
\[
\int_{X_t} \Omega_t \wedge \overline{\Omega}_t = C (\log |t|)^m |t|^{k} (1+o(1)),
\]
where $C$ is some constant, $k$ is an integer which can be taken as zero by adjusting $\Omega_t$, and $0< m\leq \dim_\C X_t$ if the central fibre has \emph{worse than klt singularities}
. The integer $m$ is determined by  \emph{Hodge theory} for the degeneration. The curious presence of the \emph{transcendental factor} $(\log |t|)^m$ is interpreted by Kontsevich and Soilbelman as indicating the presence of an $m$-dimensional \emph{torus fibration}; in the special case of the large complex structure limit $\dim_\C X_t=m$ they predict a $T^m$-fibration, which is compatible with the SYZ proposal (\cf Section 3.1 \cite{KontsevichSoibelman}). Transcendental phenomenon is captured by \emph{non-archimdean analysis}. They also suggest that collapsing phenomenon in general involves an iterative fibration structure, based on motivations from conformal field theory (\cf Section 2.3 in \cite{KontsevichSoibelman}).

There is a simple conceptual relation between collapsing families of Calabi-Yau metrics $(X_t, \omega_t)$ on compact manifolds, and non-compact complete Calabi-Yau metrics. If we scale the metrics such that $\sup |\text{Rm}|=1$ inside a region of interest, then there is a dichotomy:

\begin{itemize}
\item  If the injectivity radius is bounded below, then the pointed Gromov-Hausdorff limit is a smooth complete Calabi-Yau manifold (a `\textbf{complete bubble}').

\item If the injectivity is not bounded below, then we are in the situation of collapsing with bounded curvature, and we should instead look at the covering geometry.
\end{itemize}

It often happens that the original $X_t$ has a natural fibration structure, which would strongly motivate a complete Calabi-Yau manifold with the same kind of fibration structure.

To explain the role of the  \textbf{Euclidean volume growth condition} for the complete Calabi-Yau manifolds, we recall a basic fact in Riemannian geometry called  Bishop-Gromov monotonicity, which implies that for Ricci-flat manifolds of real dimension $N$, the normalised volume
\[
\frac{ \text{Vol}(B(x,r) )}{ \text{Vol}(B^N_{ \text{ Euclid }}(0,r)  )  }
\]
is a decreasing function of the radius $r$. Thus if one has a geometric reason for the non-collapsing bound $\text{Vol}(B(x,R) )\geq \kappa R^N$ at a particular distance scale $R$, then in all smaller scales $r$ we have also $\text{Vol}(B(x,r) )\geq \kappa r^N$. In particular, even though a family of Calabi-Yau metric is collapsing globally, it can happen that in a local region of interest the non-collapsing bound holds, so the complete bubble inherits the Euclidean volume growth condition. The reader is referred to the author's papers \cite{Li2}\cite{Ligluing} for concrete examples where this phenomenon happens.

Finally, focusing on complex dimension 3, 
 recall from subsection \ref{gravitationalinstantons} that the Taub-NUT type metrics on $\C^3$ are expected to be \textbf{primary} objects, while the conjectural multi-Taub-NUT type metrics are \textbf{composite} objects which naturally arise in high dimensional families.
We suggest that this means
the Taub-NUT type metric on $\C^3$ typically occurs as a complete bubble in a suitably \textbf{generic 1-parameter collapsing family} of compact Calabi-Yau metrics when the Euclidean volume growth condition fails, while most other complete bubbles are relevant for \textbf{multi-parameter degenerations}.

\chapter{The Positive Vertex}\label{Positivevertexchapter}

In this Chapter we will construct using the generalised Gibbons-Hawking ansatz a family of incomplete Calabi-Yau metrics describing the positive vertex, which we advocate as an analogue of the Ooguri-Vafa metric in complex dimension 3. This metric has $T^2$-symmetry and admits a special Lagrangian $T^3$ fibration. The discriminant locus $\mathfrak{D}$ is a trivalent graph with one vertex, living inside $\R^2_{\mu_1,\mu_2}\times (S^1\times \R)$ . Suitably away from $\mathfrak{D}$ the metric is approximately a flat $T^2$-bundle over an open subset of $\R^2_{\mu_1,\mu_2}\times S^1\times \R$. Along the 3 edges of $\mathfrak{D}$ but a little away from the trivalent vertex, the metric is modelled on a fibration by Taub-NUT metrics. Finally, a tiny region near the vertex is modelled on the Taub-NUT type metric on $\C^3$ we constructed in Chapter \ref{TaubNUTtypemetriconC3}. The topological setup and the holomorphic structures agree with the Gross-Ruan-Joyce-Zharkov picture (\cf review Section \ref{Positivevertices}, \ref{Degeneratingtorichypersurface}).

The Ooguri-Vafa type metric on the positive vertex space is best thought as the periodic version of the Taub-NUT type metric on $\C^3$. The fundamental mechanism is that the periodicity condition breaks down the scaling invariance and results in a gluing construction. The same periodicity condition also gives rise to exponential decay of higher Fourier modes, so that the Ooguri-Vafa type metric looks semiflat at large distance.

The organization is as follows.
Section \ref{Firstorderapproximatemetricpositivevertex}, \ref{PositivevertexasymptototeSection}, \ref{Complexgeometricperspectivepositivevertex} describe a K\"ahler ansatz and identify its holomorphic structure explicitly, and are written with an overall geometric orientation. Section \ref{WeightedHoldernormsandinitialerror}
to \ref{HarmonicanalysisII} develop the analysis to glue this ansatz to the Taub-NUT type metric on $\C^3$ and perturb the metric to be Calabi-Yau. This linear analysis is  an extension of ideas in Chapter \ref{TaubNUTtypemetriconC3}, and the only new input addressing exponential decay of higher Fourier modes appear in Section \ref{HarmonicanalysisI}. More technically, we first improve the decay of the volume form error in the generic region using the Gibbons-Hawking framework, and then treat the error elsewhere by shifting to the complex geometric framework. Section \ref{PerturbationintoaCYmetricpositivevertex} discuss geometric properties, notably the exponential decay of higher Fourier modes and the existence of specical Lagrangian fibration.
Section \ref{runningcoupling} is a semi-heuristic discussion on how to partially go beyond perturbation theory using an idea inspired by QFT, which we call the renormalisation flow.

\section{First order approximate metric}\label{Firstorderapproximatemetricpositivevertex}

We plan to construct an approximate Calabi-Yau metric using the \textbf{generalised Gibbons-Hawking ansatz}, on a singular $T^2$-bundle $M^+$ over an open neighbourhood of the origin inside the real 4-dimensional base $\R^2_{\mu_1, \mu_2}\times (S^1\times \R)_\eta$, whose discriminant locus is 
\begin{equation*}
\begin{split}
\mathfrak{D}& = \mathfrak{D}_1\cup \mathfrak{D}_2\cup \mathfrak{D}_3\cup \{0\} =
\{ \mu_1=0, \mu_2> 0   \}\cup \{ \mu_2=0, \mu_1> 0    \}\cup \{  \mu_1=\mu_2< 0   \}\cup \{0\} \\
&\subset  \R^2_{\mu_1, \mu_2}\times\{0\}\subset \R^2_{\mu_1, \mu_2}\times (S^1\times \R)_\eta.
\end{split}
\end{equation*}
Here $\eta=x+ \sqrt{-1}y$ is a complex variable with period 1. The topological situation is described in Section \ref{Positivevertices}, Example \ref{compactificationpositivevertex} and the expected complex structure can be found in Section \ref{Degeneratingtorichypersurface}.

This situation has very strong similarity with the Taub-NUT type metric on $\C^3$ in Chapter \ref{TaubNUTtypemetriconC3}, the only difference being the \textbf{periodicity condition} on $\eta$. The basic heuristic idea is to \textbf{perturb the constant solution} (\cf Example \ref{Constantsolution})  after incorporating the \textbf{topology}. The information in the constant solution is encoded by the base metric
\begin{equation}
g_a= a_{ij} d\mu_i \otimes d\mu_j +A|d\eta|^2
\end{equation}
with $a_{ij}$ being a real symmetric positive definite matrix and $A=\det a$,
analogous to Section \ref{C3asymptoticmetric}. We call $a_{ij}$ the \textbf{coupling constants} and emphasize that $a_{ij}$ are parameters we would like to vary. We impose the scale invariant ellipticity bound
\begin{equation}
C^{-1}A^{1/2} \delta_{ij}\leq a_{ij}\leq CA^{1/2} \delta_{ij} , \quad A\gg 1.
\end{equation}
The $A\gg 1$ assumption is essential for the perturbative way of thinking to be effective; this assumption was absent in the $\C^3$ case because there was no intrinsic \textbf{scale} provided by periodicity. The appearance of the gluing parameter $A$ means we need to carefully track down $A$-dependence in our estimates; in this Chapter all constants in estimates depend on $a_{ij}$ only through the above scale invariant ellipticity constant unless stated otherwise.

\begin{notation}
We denote $\vec{\mu}=(\mu_1,\mu_2,\eta)$ and $|\vec{\mu}|_a= \sqrt{ a_{ij}\mu_i\mu_j+ A|\eta|^2}$ is the $g_a$-distance to the origin. A variant $\varrho=|(\mu_1, \mu_2, y)|_a'= \sqrt{ a_{ij}\mu_i\mu_j +Ay^2}$ stands for the distance in the $g_a'$-metric on $\R^2_{\mu_1,\mu_2}\times \R_y$
\begin{equation}
g_a'= a_{ij} d\mu_i d\mu_j + A dy^2= a_{ij} d\mu_i d\mu_j + A |d\text{Im}(\eta)|^2.
\end{equation}
Another useful length parameter is $\ell= \text{dist}_{g_a}(\cdot, \mathfrak{D})+ A^{-1/4}$ which is relevant for regularity scales.
\end{notation}

Exactly the same discussions as in Section \ref{C3asymptoticmetric} lead us to consider the \textbf{linearised equations} (\ref{GibbonsHawkinglinearisedC3case1})(\ref{GibbonsHawkinglinearisedC3case2})(\ref{distributionalequationlinearised}), which describe the first order corrections we need to make to the constant solution. The key difference is the periodicity requirement. The principle of superposition allows us to immediately produce the solution from Proposition \ref{C3linearisedsolution}. We recall from there the functions $\alpha_1, \alpha_2, \alpha_3$.

\begin{prop}
We define the functions $\tilde{\alpha}_i(\mu_1,\mu_2,\eta)$ by
\begin{equation}\label{alphaitildedefinition}
\begin{cases}
\tilde{\alpha}_1= \alpha_1(\mu_1, \mu_2, \eta)+ \sum_{n\in \Z\setminus \{0\}} \{  \alpha_1(\mu_1, \mu_2, \eta+n) -\frac{1}{  4|n|\sqrt{a_{22} }      }          \} 
\\
\tilde{\alpha}_2= \alpha_2(\mu_1, \mu_2, \eta)+ \sum_{n\in \Z\setminus \{0\}} \{  \alpha_2(\mu_1, \mu_2, \eta+n) -\frac{1}{  4|n|\sqrt{a_{11} }      }          \}
\\
\tilde{\alpha}_3= \alpha_3(\mu_1, \mu_2, \eta)+ \sum_{n\in \Z\setminus \{0\}} \{  \alpha_3(\mu_1, \mu_2, \eta+n) -\frac{1}{  4|n| \sqrt{a_{11}+2a_{12}+a_{22} }      }          \}
\end{cases}
\end{equation}
Then $\tilde{\alpha}_1, \tilde{\alpha}_2, \tilde{\alpha}_3$ are convergent away from $\mathfrak{D}$, 1-periodic in $\eta$, and $\Lap_a$-harmonic away from $\mathfrak{D}$. Morever the functions
\[
\tilde{v}^{11}= \tilde{\alpha}_1+ \tilde{\alpha}_3, \quad \tilde{v}^{22}= \tilde{\alpha}_2+ \tilde{\alpha}_3 ,\quad  \tilde{v}^{12}=\tilde{v}^{21}= -\tilde{\alpha}_3,
\quad
\tilde{w}= Aa^{ij} \tilde{v}^{ij}
\]
provide a \textbf{solution} in the periodic setting to (\ref{GibbonsHawkinglinearisedC3case1})(\ref{GibbonsHawkinglinearisedC3case2}) away from $\mathfrak{D}$, which also solves the distributional equation (\ref{distributionalequationlinearised}) globally.

\end{prop}

\begin{proof}
The only issue worth checking is convergence, which follows from the fact that $|\alpha_1(\mu_1, \mu_2, \eta+n)- \frac{1}{ 4|n|\sqrt{a_{22}  } }|\leq  \frac{C(a_{ij}, \mu_1, \mu_2, \eta)}{n^2}  $, and likewise for $\alpha_2, \alpha_3$.
\end{proof}

We obtain by the generalised Gibbons-Hawking construction a \textbf{K\"ahler ansatz} $(\tilde{g}^{(1)}, \tilde{\omega}^{(1)}, \tilde{J}^{(1)}, \tilde{\Omega}^{(1)}  )$ associated to 
\[
\tilde{V}^{ij}_{(1)}= a_{ij}+ \tilde{v}^{ij}, \quad \tilde{W}_{(1)}= A+ \tilde{w}.
\]
A subtlety here is that the connection $\vartheta=(\vartheta_1, \vartheta_2)$ can be twisted by a \textbf{flat connection}. This choice is parametrised by $H^1( \R^2_{\mu_1,\mu_2}\times S^1\times \R\setminus \mathfrak{D}, T^2 )= H^1( \R^2\times S^1\times \R, T^2 )=T^2$, since the codimension 3 subset $\mathfrak{D}$ inside the base does not affect the fundamental group. We sometimes suppress mentioning this choice since it does not have a strong impact on the geometry, especially because we will exclusively work with $T^2$-invariant tensors, which are rarely sensitive to the flat connection.

The K\"ahler structure is well defined over the region where the matrix $\tilde{V}_{(1)}^{ij}$ is positive definite and $\tilde{W}_{(1)}$ is positive, except at the singular point $(\mu_1, \mu_2, \eta)=0$. A sufficient condition for positive definiteness will be given in (\ref{M+region}). The K\"ahler structure extends smoothly across  $\mathfrak{D}_i$, where the local structure is modelled on the Taub-NUT fibration described by $g_{\text{Taub}}$ for $|\vec{\mu}|_a\gtrsim A^{-1/4}$ (\cf Section \ref{Structureneardiscriminantlocus}).

\begin{rmk}
The series definition of $\tilde{\alpha}_i$ involves `subtracting a logarithmic infinity from a logarithmic infinity', as in the usual Ooguri-Vafa metric. 
\end{rmk}

\begin{rmk}
	Compared to the Taub-NUT type $\C^3$ case in Chapter \ref{TaubNUTtypemetriconC3}, the $T^2$-symmetry and the discrete symmetry persist, while the scaling symmetry and the additional $U(1)$-symmetry are now broken.
\end{rmk}

\begin{rmk}\label{freedomofconstant}
There is some freedom to add some additive constants to the definition of $\tilde{\alpha}_1, \tilde{\alpha}_2, \tilde{\alpha}_3$, which does not affect the validity of the linearised equations. Our choice ensures that $\tilde{\alpha}_i- \alpha_i$ vanishes at the origin, which is need later for \text{gluing in the Taub-NUT type metric on $\C^3$}. A more quantitative statement is:
\end{rmk}

\begin{lem}\label{positivevertexloggrowthestimate}
Let $|x|=|\text{Re}(\eta)|\leq \frac{1}{2}$.
The difference $\tilde{\alpha}_i- \alpha_i$ satisfies the estimate
\[
|\tilde{\alpha}_i- \alpha_i| \leq C A^{-1/4} \log (  1+ A^{-1/2}|\vec{\mu}|_a      ).
\]
\end{lem}

\begin{proof}
In the series  (\ref{alphaitildedefinition}) defining $\tilde{\alpha}_1$, we can separate the sum into two ranges
$|n| \gtrsim A^{-1/2}|\vec{\mu}|_a +1$ and
 $1\leq |n|\lesssim  A^{-1/2}|\vec{\mu}|_a $. In the first range, using elementary Taylor expansion of arctan,
\[
|\alpha_1(\mu_1, \mu_2, \eta+n) - \frac{1}{ 4|n| \sqrt{a_{22}} }|
\leq    
\frac{ C|\vec{\mu}|_a }{ A^{3/4} |n|^2 },
\]
which implies after summation
\[
\begin{split}
\sum_{ |n|\gtrsim  A^{-1/2}|\vec{\mu}|_a +1} |\alpha_1(\mu_1, \mu_2, \eta+n) - \frac{1}{ 4|n| \sqrt{a_{22}} }|
\leq C A^{-1/4} \min \{ 1, A^{-3/4}  |\vec{\mu}|_a    \} .
\end{split}
\]
The second range only appears if $1\lesssim  A^{-1/2}|\vec{\mu}|_a $. This sum is crudely estimated by
\[
\begin{split}
&\sum_{1\leq |n|\lesssim A^{-1/2}|\vec{\mu}|_a } |\alpha_1(\mu_1, \mu_2, \eta+n) - \frac{1}{ 4|n| \sqrt{a_{22}} }|
 \leq  CA^{-1/4} \sum_{1\leq |n|\lesssim A^{-1/2}|\vec{\mu}|_a } \frac{1}{n} \\
& \leq CA^{-1/4} \log (  A^{-3/4}|\vec{\mu}|_a     ).
\end{split}
\]
Combining the discussions gives the result.
\end{proof}

\section{Asymptotes for the first order ansatz}\label{PositivevertexasymptototeSection}

This Section is concerned with obtaining \textbf{refined asymptotes} of $\tilde{\alpha}_1$, $\tilde{\alpha}_2$, $\tilde{\alpha}_3$; a summary can be found at the end of the Section. We define the \textbf{average functions}
\begin{equation}
\begin{cases}
\bar{\alpha}_1(\mu_1, \mu_2, y)= \int_0^1 \tilde{\alpha}_1(\mu_1,\mu_2, x+\sqrt{-1}y ) dx, \\
\bar{\alpha}_2(\mu_1, \mu_2, y)= \int_0^1 \tilde{\alpha}_2(\mu_1,\mu_2, x+\sqrt{-1}y ) dx,
\\
\bar{\alpha}_3(\mu_1, \mu_2, y)= \int_0^1 \tilde{\alpha}_3(\mu_1,\mu_2, x+\sqrt{-1}y ) dx.
\end{cases}
\end{equation}
The main goal in this Section is to prove \textbf{exponential decay estimate} for $\tilde{\alpha}_i- \bar{\alpha}_i$ outside a tubular neighbourhood of $\mathfrak{D}_i$.

\begin{prop}
	(\textbf{Leading order asymptote})
The formulae for $\bar{\alpha}_i$ are explicitly given as
\begin{equation}\label{alphabar}
\begin{cases}
\bar{\alpha}_1=& \frac{1}{ 2\sqrt{a_{22} } } \{
\log 2- \gamma_E-\log \left(  \frac{1}{ \sqrt{A } } |(\mu_1, \mu_2, y)|_a' - \frac{\sqrt{a_{22}}  }{\sqrt{A}  }(\mu_2+\frac{a_{12}}{ a_{22} }\mu_1 )        \right)   \} \\
\bar{\alpha}_2=& \frac{1}{ 2\sqrt{a_{11} } } \{ \log 2- \gamma_E 
-\log \left(  \frac{1}{ \sqrt{A} } |(\mu_1, \mu_2, y)|_a' -  \frac{ \sqrt{a_{11} }}{ \sqrt{A}} (\mu_1+ \frac{a_{12}}{ a_{11} }\mu_2   )      \right)  \} 
\\
\bar{\alpha}_3= &\frac{1}{ 2\sqrt{a_{11}+2a_{12}+a_{22} } }\{  \log 2- \gamma_E \\ &
-\log \left(  \frac{1}{ \sqrt{A} } |(\mu_1, \mu_2, y)|_a' + \frac{ a_{11}\mu_1+a_{12}\mu_2+a_{21}\mu_1+a_{22}\mu_2   }
{ \sqrt{A} \sqrt{ a_{11}+ 2a_{12}+ a_{22}  }  }
        \right) \} 
\end{cases}
\end{equation} 
where $\gamma_E= \lim_{n\to \infty} \sum_{k=1}^n \frac{1}{k} -\log n$ is the Euler constant.
\end{prop}

\begin{proof}
We will focus on $\bar{\alpha}_1$. The periodic version of equation (\ref{C3harmonicawayfromDelta}) on $\R^2_{\mu_1,\mu_2}\times (S^1\times \R)_\eta$ is the measure equation
\[
(\Lap_a  \tilde{\alpha}_1) A^{3/2} d\mu_1 \wedge d\mu_2\wedge dx \wedge dy= -2\pi \sqrt{A} \int_{\mathfrak{D}_1} d\mu_2.
\]
Integrating in the periodic $x$-variable from 0 to 1,
\begin{equation}\label{alphabardistributioanlequation}
(\Lap_a' \bar{\alpha}_1) d\text{Vol}_a'= -2\pi \int_{\mathfrak{D}_1} d\mu_2,
\end{equation}
where $\Lap_a'$ is the Laplacian of the metric  $g_a'= a_{ij} d\mu_i d\mu_j+ Ady^2$  on $\R^2_{\mu_1, \mu_2}\times \R$, whose volume form is $d\text{Vol}_a'= A  d\mu_1 \wedge d\mu_2 \wedge dy$.

Now the basic strategy is to build a function satisfying the same measure equation and then compare. For a large positive cutoff $\Lambda$, we calculate the Green representation
\[
\begin{split}
& -\frac{1}{4\pi} \int_0^\Lambda \frac{-2\pi}{ |(\mu_1, \mu_2-s, y)  |_a'} ds 
=   \frac{1}{2\sqrt{a_{22} } } \sinh^{-1} \left( 
\frac{s}{  \sqrt{ \frac{A}{a_{22}^2 }\mu_1^2+ \frac{A}{a_{22} } y^2   }  }
\right)  \vert _{ s=- \mu_2- \frac{a_{12} }{ a_{22} }\mu_1 }^{ s=- \mu_2- \frac{a_{12} }{ a_{22} }\mu_1+ \Lambda   }.
\end{split}
\]
If we subtract $\frac{1}{ 2\sqrt{a_{22} }  }\log (2\Lambda)  $ and take the limit $\Lambda\to \infty$,
 we obtain the function
\[
-\frac{1}{ 2\sqrt{a_{22} } } 
\log \left(  \frac{1}{ \sqrt{a_{22} } } |(\mu_1, \mu_2, y)|_a' - \mu_2- \frac{a_{12}}{ a_{22} }\mu_1         \right)
\]
which by construction satisfies the same measure equation as (\ref{alphabardistributioanlequation}).

We claim that this function differs from $\bar{\alpha}_1$ by a constant. By the Liouville theorem, it suffices to show that the function $\bar{\alpha}_1$ on $\R^2\times \R$ has the logarithmic growth estimate
\[
\bar{\alpha}_1 \leq C A^{-1/4} \{ \log (1+ A^{-3/4}\varrho     ) + |\log ( \frac{1}{ a_{22}^{-1}\mu_1^2 + y^2   }  )| + 1 \}
\]
which is easy to deduce from Lemma \ref{positivevertexloggrowthestimate}.

Now to pin down the constant, we can evaluate $\bar{\alpha}_1$ for $\mu_1=\mu_2=0, y \neq 0$. Then the $\arctan $ term drops out, and
\[
\begin{split}
\bar{\alpha}_1(0, 0, y)&= \frac{1}{2\sqrt{a_{22}}}\lim_{\Lambda\to +\infty}\{  \int_0^\Lambda \frac{1}{ \sqrt{x^2+y^2 }   }  dx - \sum_{n=1}^{\Lambda} \frac{1}{n} \} \\
&= \frac{1}{2\sqrt{a_{22}}}
\lim_{\Lambda\to \infty}\{
	\sinh^{-1} ( 
	\frac{\Lambda}{ |y|  }   ) - \log \Lambda- \gamma_E 
\} \\
&=\frac{1}{2\sqrt{a_{22}}}( \log (\frac{2}{|y|}) -\gamma_E  )
\end{split}
\]
Comparing the expressions give the formula for $\bar{\alpha}_1$.
\end{proof}

\begin{lem}\label{Positivevertexexpdecaylemma}
The difference $\tilde{\alpha}_1- \bar{\alpha}_1$ satisfies the following estimate: if either $y^2 + a_{22}^{-1} \mu_1^2 \gtrsim 1 $ or $ \mu_2\lesssim - A^{1/4}$, namely if $\text{dist}_{g_a'}(\cdot, \mathfrak{D}_1) \gtrsim A^{1/2}$, 
 then $|\tilde{\alpha}_1- \bar{\alpha}_1| \leq CA^{-1/4} $. Similar bounds hold for $\tilde{\alpha}_i- \bar{\alpha}_i$ for $i=1,2,3$.
\end{lem}

\begin{proof}
We notice in advance that $\tilde{\alpha}_1$ and $\bar{\alpha}_1$ are periodic in $\eta$, so it suffices to assume $|x|\leq 1/2$.	The main idea of a variant of Cauchy's integral test for convergence.

Using the fact that $|\frac{\partial^2 \alpha_1 }{\partial x^2  }| \leq C\frac{ a_{22} }{ 
	( \mu_1^2+ a_{22}|\eta|^2    )^{3/2}    }
$, and the mean value type inequality
\[
f(0)- \int_{-1/2}^{1/2} f(s) ds \leq C\sup_{|s|\leq 1/2  } |f''(s)|,
\]
we deduce that for $|\eta|^2+ \frac{\mu_1^2}{a_{22}}\gtrsim 1$,
\[
|\alpha_1( \mu_1, \mu_2, \eta   )- \int_{ x -1/2}^{ x +1/2   } \alpha_1( \mu_1, \mu_2, s+ \sqrt{-1} y) ds |\leq CA^{-1/4}  ( |\eta|^2+ \frac{\mu_1^2}{a_{22}}   )^{-3/2}
\]
Thus for $|x|\leq 1/2$, 
\[
\begin{split}
& |\sum_{n\in \Z\setminus \{0\}} \{ \alpha_1( \mu_1, \mu_2, \eta +n   ) - \int_{ x -1/2}^{ x +1/2   } \alpha_1( \mu_1, \mu_2, s+ n+ \sqrt{-1} y) ds \}|
\\
&  
\leq C A^{-1/4} \sum_{|n|\neq 0  } ( |\eta+n |^2+ \frac{\mu_1^2}{a_{22}}   )^{-3/2}
\leq  CA^{-1/4},
\end{split}
\]
and the sum converges to zero as $\mu_1^2+ |\eta|^2\to \infty$.

In particular if $y^2 + a_{22}^{-1} \mu_1^2 \gtrsim 1 $, then adding the above two inequalities already implies the bound
\[
\begin{split}
|\tilde{\alpha}_1- \bar{\alpha}_1| &=
|\sum_{n\in \Z} \{ \alpha_1( \mu_1, \mu_2, \eta +n   ) - \int_{ x -1/2}^{ x +1/2   } \alpha_1( \mu_1, \mu_2, s+ n+ \sqrt{-1} y) ds \}| 
\leq CA^{-1/4}, 
\end{split}
\]
and that $|\tilde{\alpha}_1- \bar{\alpha}_1| $ converges to zero as $\mu_1^2+ |\eta|^2\to \infty$.

If however $y^2 + a_{22}^{-1} \mu_1^2 \ll 1 $ but $\mu_2\lesssim -A^{1/4}$, then we can make
\[
\frac{ a_{22}\mu_2+ a_{12}\mu_1 }{  A^{1/2} \sqrt{\mu_1^2+ a_{22} |\eta|^2 }   } \lesssim -1,
\] 
and the Taylor expansion of $\arctan$ will ensure $|\alpha_1(\mu_1, \mu_2, \eta)|\leq \frac{C}{ -\mu_2 }$, so
\[
|\alpha_1( \mu_1, \mu_2, \eta   )- \int_{ x -1/2}^{ x +1/2   } \alpha_1( \mu_1, \mu_2, s+ \sqrt{-1} y) ds |\leq C\mu_2^{-1} \leq C  A^{-1/4},
\]
from which we again deduce $|\tilde{\alpha}_1- \bar{\alpha}_1|\leq CA^{-1/4}$.
\end{proof}

\begin{prop}\label{Exponentialdecayforthefirstorderansatz}
(\textbf{Exponential decay for higher Fourier modes in the first order ansatz})
If $\text{dist}_{g_a'}(\cdot, \mathfrak{D}_1) \gtrsim A^{1/2}$, then 
\begin{equation}
|\tilde{\alpha}_1- \bar{\alpha}_1|
 \leq C A^{-3/4} \text{dist}_{g_a'} (\cdot, \mathfrak{D}_1)\exp(  - 2\pi A^{-1/2}  \text{dist}_{g_a'} (\cdot, \mathfrak{D}_1)  ).
\end{equation} 
Similar bounds hold for $\tilde{\alpha}_i- \bar{\alpha}_i$ for $i=1,2,3$.

\end{prop}

\begin{proof}
We focus on the region $\{ \text{dist}_{g_a'}( \cdot, \mathfrak{D}_1 )\gtrsim A^{1/2}  \}$.
The key idea is that $\tilde{\alpha}_1- \bar{\alpha}_1$ is $\Lap_a$-harmonic , bounded and has no zero Fourier mode in the $S^1$ direction defined by the $x$-variable, so the exponential decay follows from \textbf{Fourier analysis}. We remark that similar ideas have appeared in the recent paper 
\cite{HeinSunViaclovsky}.

We perform Fourier decomposition in the $S^1$ direction 
\[
\tilde{\alpha}_1- \bar{\alpha}_1=\sum_{n\neq 0} h_n( \mu_1, \mu_2, y   )e^{2\pi in x},
\]
Parseval identity combined with Lemma \ref{Positivevertexexpdecaylemma} shows
\[
\sum_n |h_n|^2 = \int_{0}^1 |\tilde{\alpha}_1- \bar{\alpha}_1|^2 dx\leq CA^{-1/2}.
\]
Now $\Lap_a$-harmonicity translates into the 3-dimensional Helmholtz equations:
\[
\Lap_a' h_n - 4\pi^2 n^2 A^{-1} h_n=0.
\]
The remaining task is conceptually speaking to estimate the Dirichlet Green's function for the Helmholtz equation on the  noncompact 3-dimensional domain $\{ \text{dist}_{g_a'}( \cdot, \mathfrak{D}_1 )\gtrsim A^{1/2}  \}$. In practice, building an upper barrier for the Green's function suffices for our purpose.

Recall $\varrho=|(\mu_1,\mu_2,y)|_a'$ is the distance function for the Euclidean metric $g_a'$ on $\R^2_{\mu_1, \mu_2}\times \R_y$.
By simple direct computation, for any $\kappa>0$,
\[
(\Lap_a'- 4\pi^2 n^2 A^{-1})    e^{ - \kappa \varrho  } \leq ( \kappa^2-  4\pi^2 n^2 A^{-1}) e^{ - \kappa \varrho  }, 
\]
so for $k_n= 2\pi |n| A^{-1/2}$, the function $e^{- k_n \varrho  }$ is a supersolution of the Helmholtz equation. Now we build a barrier function
\[
h_n'(\mu_1, \mu_2, y)=A^{-1/2} \int_{0}^{\infty} e^{ - k_n |(\mu_1, \mu_2-s, y)|_a'  } ds,
\]
whose singularity lies on $\mathfrak{D}_1$. Since $h_n'$ is a positive superposition of supersolutions, it must be itself a supersolution. Other basic properties are: 
\begin{itemize}
\item 
On  $\{ \text{dist}_{g_a'}( \cdot, \mathfrak{D}_1 )\geq A^{1/2}  \}$,  using the saddle point method for Laplace type integrals
\[
0\leq h_n'\leq C A^{-3/4} \text{dist}_{g_a'} (\cdot, \mathfrak{D}_1)\exp(  - k_n  \text{dist}_{g_a'} (\cdot, \mathfrak{D}_1)  ).
\]

\item
On the boundary of $\{ \text{dist}_{g_a'}( \cdot, \mathfrak{D}_1 )\geq A^{1/2}  \}$, we have $h_n' \geq  \frac{ A^{-1/4} }{C|n|}  $.
\end{itemize}

Since $|h_n| \leq CA^{-1/4}$ by the Parseval identity, the comparison principle implies
 \[
 \begin{split}
 h_n\leq C |n| h_n'
\leq C A^{-3/4} |n| \text{dist}_{g_a'} (\cdot, \mathfrak{D}_1)\exp( - k_n  \text{dist}_{g_a'} (\cdot, \mathfrak{D}_1)  ).
\end{split}
\]
Thus on $\{ \text{dist}_{g_a'}( \cdot, \mathfrak{D}_1 )\geq A^{1/2}  \}$, the desired bound on $|\tilde{\alpha}_1- \bar{\alpha}_1|$ follows 
by summing over these estimates over $n$. 
It is worth commenting that we expect the exponential decay rate to be sharp.
\end{proof}

\begin{rmk}
The periodicity condition is responsible for the exponential decay. Its effect becomes significant when $\eta\sim 1$, which is compatible with the length scale $|\vec{\mu}|_a\sim A^{1/2}$, or $\text{dist}_{g_a'}(\cdot, \mathfrak{D})\sim A^{1/2}$. The geometric significance of exponential decay is that the ansatz models the transition from fully quantum into semiflat behaviour (\cf review Section \ref{OoguriVafa}).
\end{rmk}

Next we ask for fine asymptote as we move far along $\mathfrak{D}_i$.

\begin{lem}
We have the identity
\[
\begin{split}
& \tilde{\alpha}_1(\mu_1, \mu_2, \eta)+ \tilde{\alpha}_1(-\mu_1, -\mu_2, -\eta)\\
=& \frac{1}{2 \sqrt{ \mu_1^2+ a_{22}|\eta|^2  }   }+ \sum_{n\in \Z\setminus \{0\}} \{ \frac{1}{2 \sqrt{ \mu_1^2+ a_{22}|\eta+n |^2  }   } - \frac{1}{ 2\sqrt{a_{22}}|n|   } \}
\end{split}
\]
where the RHS is recognized as the main part of the \emph{complex 2-dimensional Ooguri-Vafa potential}. Similarly with $\mathfrak{D}_i$ for $i=1,2,3$.
\end{lem}

\begin{proof}
Clear from $\alpha_1(\mu_1, \mu_2, \eta)+ {\alpha}_1(-\mu_1, -\mu_2, -\eta)
= \frac{1}{2 \sqrt{ \mu_1^2+ a_{22}|\eta|^2  }   }.
$
\end{proof}

The utility of this Lemma is that for $\text{dist}_{g_a}(\cdot, \mathfrak{D}_1) \lesssim A^{1/2},   |\vec{\mu}|_a \gtrsim A^{1/2} $, 
namely if we move far from the origin along $\mathfrak{D}_1$, then up to exponentially small errors
\[
\begin{split}
&\tilde{\alpha}_1(-\mu_1, -\mu_2, -\eta)\sim \bar{\alpha}_1(-\mu_1, -\mu_2, -\eta) \\
&= 
\frac{1}{ 2\sqrt{a_{22} } } \{
\log 2- \gamma_E-\log \left(  \frac{1}{ \sqrt{A } } |(\mu_1, \mu_2, y)|_a' + \frac{\sqrt{a_{22}}  }{\sqrt{A}  }(\mu_2+\frac{a_{12}}{ a_{22} }\mu_1 )        \right)   \}
\end{split}
\]
by Proposition \ref{Exponentialdecayforthefirstorderansatz}, so the Lemma provides very precise asymptote for $\tilde{\alpha}_1(\mu_1, \mu_2, \eta)$ along $\mathfrak{D}_1$.

The \textbf{refined asymptotic behaviour} of $\tilde{\alpha}_i$ is summarised as
\begin{itemize}
\item  Near the origin $\tilde{\alpha}_i \sim \alpha_i$. This is designed to match the asymptote of the Taub-NUT type metric on $\C^3$ from Chapter \ref{TaubNUTtypemetriconC3}.
\item  Sufficiently far from $\mathfrak{D}_i$, the $\tilde{\alpha}_i$ is modelled by an elementary logarithmic function $ \bar{\alpha}_i $ up to exponentially small fluctuation.
\item  Near $\mathfrak{D}_i$ and far from the origin, the $\tilde{\alpha}_i$ agrees with the 2-dimensional Ooguri-Vafa potential, up to an elementary logarithmic function and some exponentially small fluctuation.
\end{itemize}

We comment that although $\tilde{\alpha}_i$ is defined globally over $\R^2_{\mu_1, \mu_2}\times (S^1\times \R)_\eta$, the K\"ahler ansatz is only defined over a finite region  and is \textbf{incomplete}, because $A+\tilde{w}$ becomes negative when $\tilde{\alpha}_i\sim \bar{\alpha}_i\gtrsim  A^{1/2}$, which happens when
$
\log (A^{-1/2} \varrho )\gtrsim  A^{3/4} .
$
Conversely for a fixed $0< \epsilon_0 \ll 1$ independent of $A$, the K\"ahler ansatz is positive definite on
\begin{equation}\label{M+region}
M^+= \{  \log (A^{-1/2} \varrho )<     \epsilon_0 A^{3/4} \}.
\end{equation}

\section{Complex geometric perspective}\label{Complexgeometricperspectivepositivevertex}

We now proceed to identify the complex structure $(\tilde{J}^{(1)}, \tilde{\Omega}^{(1)})$ on the K\"ahler ansatz. Our technique is to find a periodic version of the constructions made in Section \ref{C3complexgeometricperspective} about the Taub-NUT type metric on $\C^3$, in the same way that the Ooguri-Vafa metric is seen as a periodic version of the Taub-NUT metric. The reader is encouraged to warm up by refering to Section \ref{OoguriVafa}  and  \ref{C3complexgeometricperspective}. In this approach algebraic structures will emerge from relations between transcendental integrals of geometric origin. For the converse viewpoint which starts with the algebra, see the review Section \ref{Degeneratingtorichypersurface}.

The generalised Gibbons-Hawking construction provides the $(1,0)$-forms $\tilde{\zeta}_i=\tilde{V}_{(1)}^{ij}d\mu_j+ \sqrt{-1}\vartheta_i$, and the formula (\ref{GibbonsHawkingholomorphicdifferential}) computes their differentials. The main idea is to produce holomorphic differentials by adjusting $\tilde{\zeta}_i$. We define the functions
\[
\tilde{\beta}_i(\mu_1,\mu_2, \eta)=   \lim_{N\to \infty} \sum_{n=-N}^N \beta_i(\mu_1, \mu_2, \eta+n), \quad i=0,1,2.
\]

\begin{lem}\label{betatildeconverges}
The series defining $\tilde{\beta}_i$ converge for $\eta\notin\Z$, and $\tilde{\beta}_i$ are 1-periodic in $\eta$. Morever if $|x|\leq \frac{1}{2}$, then
\[
|\tilde{\beta}_i- \beta_i|\leq C(|\eta|+\frac{ |\mu_1|+|\mu_2| }{  A^{1/4} }).
\]
\end{lem}

\begin{proof}
Let $\mu_1, \mu_2$ be fixed. The essential task is to understand the asymptotic behaviour of $\beta_i(\mu_1, \mu_2, \eta+n)$ as $|n|$ becomes large. We focus on $\beta_1$.

 Using the homogeneity property of $\alpha_1,\alpha_2, \alpha_3$ in the $\mu_1, \mu_2$ and $\eta$ variables, it is easy to see from the integral definition of $\beta_1$ that
 \[
 \beta_1 (0, 0, \eta)= \frac{1}{\eta} \beta_1(0,0, 1), \quad  |\beta_1(0,0,1)|\leq C.
 \]
By elementary properties of arctan
\[
\alpha_1(\mu_1, \mu_2, \eta)= \frac{1}{4 \sqrt{ \mu_1^2+ a_{22} |\eta|^2  }  } + O( \frac{ |\mu_1|+|\mu_2|  }{ \mu_1^2+ a_{22}|\eta|^2 }  ),
\]
\[
\frac{\partial \alpha_1}{ \partial \eta}= \frac{ -a_{22} \bar{\eta}}{ 8(\mu_1^2+a_{22}|\eta|^2)^{3/2} }
	\left(  
	1+ O(  \frac{  |\mu_1|+|\mu_2|}{   
	( \mu_1^2+ a_{22}|\eta|^2)^{1/2}	} )    
	 \right) ,
\]
and similarly
\[
\frac{\partial \alpha_3}{ \partial \eta}= \frac{ -(a_{11}+2a_{22}+a_{22}) \bar{\eta}    }{ 8((\mu_1-\mu_2)^2+    
	(a_{11}+2a_{22}+a_{22})
	|\eta|^2)^{3/2}  } 	\left(  
1+ O(  \frac{  |\mu_1|+|\mu_2|}{   
	( (\mu_1-\mu_2)^2+ A^{1/2}|\eta|^2)^{1/2}	} )    
\right)    .
\]
After integration
\[
\begin{split}
|\beta_1(\mu_1, \mu_2, \eta)-\beta_1(0,0, \eta)|\leq  \frac{C (|\mu_1|+|\mu_2|)  }{ |\eta|  }( \frac{1}{  \sqrt{\mu_1^2+ a_{22}|\eta|^2}  } +    \frac{1}{  \sqrt{(\mu_1-\mu_2)^2+ a_{22}|\eta|^2}  }        ).
\end{split}
\]
This shows the series 
\[
\sum_{n\in \Z} \beta_1(\mu_1, \mu_2, \eta+n)-\beta_1(0,0, \eta+n)
\]
is absolutely convergent if $\eta\notin \Z$, and if morever $|x|\leq \frac{1}{2}$ then we have the bound
\[
\sum_{n\in \Z\setminus \{0\} } |\beta_1(\mu_1, \mu_2, \eta+n)-\beta_1(0,0, \eta+n)|\leq \frac{C(|\mu_1|+|\mu_2|)}{ A^{1/4}       }.
\]

Thus the convergence of the series $\tilde{\beta}_1$ is equivalent to the convergence of 
\[
\lim_{N\to \infty} \sum_{n=-N}^N \beta_1(0, 0, \eta+n)= \beta_1(0,0,1) \lim_{N\to \infty} \sum_{n=-N}^N \frac{1}{\eta+n}=\beta_1(0,0,1) \pi \cot( \pi \eta  ),
\]
and similarly for $\tilde{\beta}_2$ and $\tilde{\beta}_0$.
The periodicity claim follows from standard rearranging theorems for series. The estimate on $\tilde{\beta}_i-\beta_i$ follows by combining the above discussions.
\end{proof}

\begin{lem}
Let $0\leq \theta_1^\infty, \theta_2^\infty\leq 2\pi$ be two real numbers to be determined.
The holomorphic 1-forms
\[
\begin{cases}
\tilde{\zeta}_1'= \tilde{\zeta}_1+ (\tilde{\beta}_1+ \beta_1(0,0,1)\pi \sqrt{-1}- \sqrt{-1}\theta_1^\infty) d\eta,
\\
\tilde{\zeta}_2'= \tilde{\zeta}_2+ (\tilde{\beta}_2+ \beta_2(0,0,1)\pi \sqrt{-1}- \sqrt{-1}\theta_2^\infty) d\eta,
\\
\tilde{\zeta}_0'= \tilde{\zeta}_0+ (\tilde{\beta}_0+ \beta_0(0,0,1)\pi \sqrt{-1}+ \sqrt{-1}\theta_1^\infty+ \sqrt{-1}\theta_2^\infty  ) d\eta
\end{cases}
\]
 are closed, namely they are \textbf{holomorphic differentials}.
\end{lem}

\begin{proof}
This is the periodic version of Lemma \ref{holomorphicdifferentiallemma}. The terms $\beta_i(0,0,1)\pi \sqrt{-1}$ and $\theta_i^\infty$ are added for later convenience.
\end{proof}

\begin{lem}\label{functionalequationpostivevertexlemma}
The sum $\tilde{\beta}_1+ \tilde{\beta}_2+ \tilde{\beta}_0= \pi \cot (\pi \eta)$. Equivalently,
\[
\tilde{\zeta}_1+ \tilde{\zeta}_2+ \tilde{\zeta}_0= d\log (1- e^{2\pi \sqrt{-1} \eta}).
\]
\end{lem}

\begin{proof}
This is the periodic version of Lemma \ref{holomorphicdifferentialfunctionalequation}, using Euler's series identity for $\cot (\pi\eta)$:
\[
\lim_{N\to \infty} \sum_{n=-N}^N \frac{1}{\eta+n}= \pi \cot( \pi \eta  ),
\]
and $\sum_{i} \beta_i(0,0,1)=1$ from Lemma \ref{holomorphicdifferentialfunctionalequation}.
\end{proof}

To compute the periods of the integrals $\int \tilde{\zeta}_i'$, we recall from the topological description (\cf review Section \ref{Positivevertices}) that there are three $S^1$-cycles generating $H_1(T^3)$, two of which come from the $T^2$-fibres, and the third comes from lifting the $S^1$ on the base $\R^2_{\mu_1, \mu_2}\times (S^1\times \R)_\eta$ to the total space, which involves monodromy issues.

\begin{lem}\label{functionalequationpositivevertex}
	For appropriate choices of $\theta^\infty_1, \theta^\infty_2$,
the $T^3$-periods of the holomorphic differentials $\tilde{\zeta}_i$ take values in $2\pi \sqrt{-1}\Z$. In particular, the \textbf{holomorphic functions}
\[
Z_i= \exp(\int \tilde{\zeta}_i) , \quad i=0,1,2
\]
are defined without multivalue issues. For a suitable choice of multiplicative normalisation on $Z_i$, we have the \textbf{functional equation}
\[
Z_0Z_1Z_2= 1- e^{2\pi \sqrt{-1}\eta}.
\] 
\end{lem}

\begin{proof}
The periods along the generating cycles in the $T^2$-fibres are straightforward: \[
\int_{S^1} \tilde{\zeta}_i'= \int_{S^1} \sqrt{-1}\vartheta_i=2\pi \sqrt{-1}, \quad i=1,2,
\]
and $\int_{S^1} \tilde{\zeta}_0'=  -\sqrt{-1}\int_{S^1}\vartheta_1+\vartheta_2=-4\pi \sqrt{-1}$.

Computing the period along the other $S^1$ requires a special trick. As a preparatory subtle remark, the K\"ahler metric is not globally defined over the base $\R^2_{\mu_1, \mu_2}\times (S^1\times \R)_\eta$ due to incompleteness issues, but the quantities $\tilde{v}^{ij}, \tilde{w}, \vartheta_i$ make sense globally. Consider the $S^1$ on the base defined by $\{\mu_1=\mu_2=0, y=\text{const}\}$. If we attempt to lift this $S^1$ by parallel transport, in general we cannot get a closed loop, and this failure is measured by the holonomy of the $T^2$-connection $\vartheta=(\vartheta_1, \vartheta_2)$ along the $S^1$. When $y\to +\infty$, due to the exponential decay of the $x$-dependent part of $\tilde{v}^{ij}, \tilde{w}, \vartheta_i$,
this holonomy converges to two real numbers $(\theta_1^\infty, \theta_2^\infty)$ modulo $2\pi \Z$. In particular, if we twist $\vartheta$ by a flat $T^2$-connection, then $\theta_1^\infty, \theta_2^\infty$ receive a corresponding twist so that $\tilde{\zeta}_i$ is unaffected. Thus we can assume without loss of generality that $\theta_i^\infty=0$, namely the asymptotic holonomy of $\vartheta$ is zero, so  in the limit the $S^1$ cycle lifts to a closed loop, on which we can evaluate the period asymptotically.

By construction $\int_{S^1} \vartheta_i=0$, and using 
$
\tilde{\beta}_i(0,0,\eta)= \beta_i(0,0,1) \pi \cot(\pi\eta)
$
from the proof of Lemma \ref{betatildeconverges}, we compute
\[
\begin{split}
\int_{S^1} \tilde{\zeta}_i'&= \int_{S^1} (\tilde{\beta}_i + \beta_i(0,0,1)\pi \sqrt{-1}) d\zeta
\\
&= \lim_{ y\to \infty } \int_{S^1} \beta_i(0,0,1) (\pi \cot(\pi\eta)+ \pi \sqrt{-1}) d\eta
\\
&=\lim_{ y\to \infty } \beta_i(0,0,1)\int_{S^1} d\log (1- e^{2\pi \sqrt{-1}\eta})=0.
\end{split}
\]
From this we see the integrality condition on the periods, so the holomorphic functions $Z_i$ are well defined without multivalue issues.

 Notice the definition of $Z_i$ for $i=0,1,2$ involve three unspecified multiplicative constants; by prescribing their product appropriately, the functional equation follows from Lemma \ref{functionalequationpostivevertexlemma}. The remaining two free multiplicative constants will be fixed in later Sections.
\end{proof}

We denote $Z_3= \exp( 2\pi \sqrt{-1} \eta )\in \C^*$. The functional equation gives a map 
\[
M^+\to \{  Z_0Z_1Z_2=1-Z_3      \} \subset \C^3_{Z_1, Z_2, Z_0 }\times \C^*_{Z_3}.
\]
By the same argument as Section \ref{C3complexgeometricperspective}, this is a holomorphic map on $M^+\setminus \{0\}$ and extends continuously at the origin.

\begin{prop}\label{positivevertexcomplexstructure}
The map $M^+\to \{  Z_0Z_1Z_2=1-Z_3      \}$ is a \textbf{holomorphic open embedding}. The \textbf{$T^2$-action} is identified as
\[
e^{i\theta_1}\cdot (Z_0, Z_1, Z_2)= (e^{-i\theta_1}Z_0, e^{i\theta_1}Z_1, Z_2 ), \quad
 e^{i\theta_2}\cdot (Z_0, Z_1, Z_2)= (e^{-i\theta_2}Z_0, Z_1,  e^{i\theta_2}Z_2 ),
\] 
and the \textbf{holomorphic volume form} is 
$\tilde{\Omega}^{(1)}=  -\frac{\sqrt{-1}}{2\pi Z_3} dZ_0\wedge  dZ_1\wedge dZ_2  $. 
\end{prop}

\begin{proof}
The $T^2$-action follows the same argument as Proposition \ref{C3complexstructure}. The holomorphic volume form is characterised by
$
\tilde{\Omega}^{(1)}( \frac{\partial}{\partial \theta_1}, \frac{\partial}{\partial \theta_2}, \cdot   )=d\eta.
$
Notice also
\[
-dZ_0\wedge  dZ_1\wedge dZ_2(  \frac{\partial}{\partial \theta_1}, \frac{\partial}{\partial \theta_2}, \cdot  )= d(Z_0Z_1Z_2)= -dZ_3= -2\pi\sqrt{-1}Z_3 d\eta,
\]
so $\tilde{\Omega}^{(1)}= -\frac{\sqrt{-1}}{2\pi Z_3} dZ_0\wedge  dZ_1\wedge dZ_2  $. This formula in particular implies the map $M^+\to \{  Z_0Z_1Z_2=1-Z_3      \}$ is a local biholomorphism. We finally need to show this map is injective. Since both $M^+$ and $\{  Z_0Z_1Z_2=1-Z_3      \}$ fibre over the $\C^*$ coordinate $\eta$ in a compatible way, it suffices to compare the fibres, which have compatible $T^2$-actions, so boils down to the injectivity of $(\mu_1,\mu_2)\mapsto (\log |Z_0|, \log |Z_1|, \log |Z_2|)$ for fixed $\eta$. 
\end{proof}

\begin{rmk}
Section \ref{Algebraicgeometricperspective} shows that the algebraic structure on Taub-NUT type $\C^3$ emerges from \textbf{holomorphic functions with controlled growth} at infinity. Since our K\"ahler ansatz is incomplete, it makes no literal sense to speak of spatial infinity. Instead growth rate is thought in terms of \textbf{effective} estimates. For a holomorphic function $f$ on $M^+$  normalised to $\norm{f}_{L^2}=1$, if we decompose $f$ according to the weights of the $T^2$-action, then
in a smaller metric ball around the origin only Fourier components with small $T^2$-weights contribute significantly to $|f|$. The intuition is that $T^2$-weights are related to an effective filtration of local holomorphic functions.
\end{rmk}

\section{Weighted H\"older norms and initial error estimates}\label{WeightedHoldernormsandinitialerror}

The central analytic difficulty comes from three sources:
\begin{itemize}
\item The metric ansatz behaves very differently in various characteristic regions, and for different Fourier modes. In short, the geometry is \emph{multi-scaled}.
\item
The volume form error becomes larger at large distance, a problem closely related to the \emph{incompleteness} of the metric. 
\item
We wish to treat the error estimates with relatively \emph{high precision}, incorporating features such as exponential decay of higher Fourier modes.
\end{itemize}

These difficulties require us to introduce some weighted H\"older norms which are more complicated than the ones used in a standard gluing problem. The purpose of this Section is to give precise estimates on the volume form errors of the ansatz, in the complement of a small ball near the origin in $M^+$; the small ball itself will be later replaced in our gluing construction by a region in $\C^3$ equipped with the Taub-NUT type metric. The task of developing the requisite linear analysis will be deferred to later Sections.

There are 3 useful \textbf{weight parameters} or \textbf{characteristic length scales}:
\begin{itemize}
\item The $g_a'$-distance to the origin is $\varrho=|(\mu_1, \mu_2, y)|_a'$.

\item The regularity scale is controlled by the parameter 
$
\ell\sim  A^{-1/4} + \text{dist}_{g_a}(\cdot, \mathfrak{D}   ).
$

\item The parameter $\tilde{\ell}= 2\pi A^{-1/2} \text{dist}_{g_a'}(\cdot, \mathfrak{D})$ is useful for measuring the rate of exponential decay of higher Fourier modes.
\end{itemize}

The key quantity to understand is the \textbf{volume form error}:
\[
\tilde{E}^{(1)}= \frac{ \tilde{W}_{(1)}}{ \det( \tilde{V}^{ij} _{(1 ) } ) }-1= \frac{A+\tilde{w}}{  A+ A a^{ij} \tilde{v}^{ij} + \det(\tilde{v}^{ij} )   }-1= - \frac{ \det(\tilde{v}^{ij})} {  A+ \tilde{w} + \det(\tilde{v}^{ij} )      },
\]
where $\det(\tilde{v}^{ij})= \tilde{\alpha}_1 \tilde{\alpha}_2+  \tilde{\alpha}_1 \tilde{\alpha}_3+ \tilde{\alpha}_2 \tilde{\alpha}_3$ and $\tilde{w}=a_{22} \tilde{\alpha}_1+ a_{11} \tilde{\alpha}_2+(a_{11}+2a_{12}+ a_{22}) \tilde{\alpha}_3$.
The weighted H\"older norms will be taylor made for the volume form error. Familiarity with Section \ref{MetricbehaviourawayfromDelta}, \ref{Structureneardiscriminantlocus} and \ref{surgery} will be assumed.

Let $\delta\leq 0$. We shall define the \textbf{weighted H\"older norms} $\norm{T}_{C^{k,\alpha}_{\delta,0} }$  for $T^2$-invariant tensor fields $T$
on $M^+\cap \{ |\vec{\mu}|_a \gtrsim A^{-1/4}    \}$, by prescribing the norm on a number of overlapping regions up to uniform equivalence. 

\begin{itemize}
\item In the region close to  $\mathfrak{D}$ characterised by $\{ \ell\lesssim A^{1/2} \}$, the ansatz metric is approximated by $g_{\text{Taub}}$ (\cf Section \ref{Structureneardiscriminantlocus}), and $\norm{T}_{C^{k,\alpha}_{\delta,0}}$ is uniformly equivalent to the norm $\norm{T}_{C^{k,\alpha}_{ \delta,0 }}$ for $g_{\text{Taub}}$ introduced in Section \ref{Structureneardiscriminantlocus}.

\item The region $\{\ell\gtrsim A^{1/2}\}$ can be covered by subregions of diameter $\sim \ell$, where the $T^2$-bundle is topologically trivial. Over each subregion the metric is approximated by the periodic version of the constant solution $g_{\text{flat}}$  (\cf Section \ref{MetricbehaviourawayfromDelta}). The $x$-variable defines an $S^1$ direction.
We decompose $T$ into the part $\bar{T}$ independent of $x$ (the `zeroth Fourier mode') and the oscillatory part $T-\bar{T}$ (the `higher Fourier mode'), and define the weighted H\"older norm separately on the two parts.
\item
On the zeroth Fourier mode, the norm $\norm{\bar{T}}_{C^{k,\alpha}_{\delta,0}}$ is equivalent to
\[
A^{-3\delta/4}(\sum_{j=0}^k \norm{  \ell^j \nabla^j \bar{T}}_{ L^\infty   }
+
[  \ell^k \nabla^k \bar{T}    ]_\alpha),
\]
where $[]_\alpha$ denotes the appropriately normalised H\"older seminorm.
Here the $\ell$-dependence is inserted to reflect the regularity scale.
\item
On the higher Fourier modes we build in the exponential decay. Fix a parameter $0<\kappa< 1$. The norm $\norm{T-\bar{T}}_{C^{k,\alpha}_{\delta,0}}$ in this region is equivalent to
\[
A^{-3\delta/4}
\sup_{ \ell(p)\gtrsim A^{1/2}  } e^{\kappa \tilde{\ell}   } (   \sum_{j=0}^k \norm{  A^{j/2} \nabla^j (T-\bar{T}) }_{ L^\infty   }
+ A^{k/2}
[   \nabla^k ( T-\bar{T})    ]_\alpha    ).
\]
 An estimate in this norm is the higher order version of
$
|T-\bar{T}| \leq CA^{3\delta/4} e^{-\kappa \tilde{\ell}}.
$
\end{itemize}

\begin{notation}
The norm $\norm{\cdot }_{C^{k,\alpha}_{\delta,0} }$ can refer to any type of tensors depending on the context, such as functions, 1-forms, symmetric 2-tensors, and in some cases can refer to the norm computed in a subregion. Strictly speaking this norm depends on $\kappa$, but we suppress this to avoid cluttering the notation.
\end{notation}

We will also need a \textbf{variant} weighted H\"older norm $\norm{T}_{C^{k,\alpha}_\delta}$. The only difference from $\norm{T}_{C^{k,\alpha}_{\delta,0} }$ is that in the region $\{\ell\gtrsim A^{1/2}  \}$ on the zeroth Fourier mode, 
$\norm{\bar{T}}_{C^{k,\alpha}_{\delta}}$ is equivalent to
\[
A^{-\delta/4}(\sum_{j=0}^k \norm{  \ell^{j-\delta} \nabla^j \bar{T}}_{ L^\infty   }
+
[  \ell^{k-\delta} \nabla^k \bar{T}    ]_\alpha),
\]
so an estimate in this norm is the higher order version of $|\bar{T}|= O( A^{3\delta/4}(A^{-1/2}\ell)^{\delta}  )$.
We have inserted an extra decay factor $(A^{-1/2}\ell)^{\delta}$.

\begin{notation}
For a parameter $\nu$ with $1\ll \nu< \epsilon_0 A^{3/4}$, define the subregion of $M^+$
\[
M_\nu^+= \{   A^{-1/2}\varrho < e^\nu          \} \subset M^+.
\]
Its base is $\mathcal{B}^+_\nu= \{   A^{-1/2}\varrho < e^\nu          \} \subset \R^2_{\mu_1,\mu_2}\times (S^1\times \R)_\eta$.
\end{notation}

The following Lemmas are simple consequences of asymptotes in Section \ref{Firstorderapproximatemetricpositivevertex} and \ref{PositivevertexasymptototeSection}. The higher order estimates are taken care by $\Lap_a$-harmonicity of $\tilde{\alpha}_i$.

\begin{lem}\label{tildealphailemma}
In the region $M^+_\nu \setminus  \{  \vec{\mu}|_a\lesssim A^{1/2} \}$,
\[
\norm{\tilde{\alpha}_i}_{ C^{k,\alpha}_{-1,0}  } \leq  CA^{1/2}\nu , \quad i=1,2,3.
\]
\end{lem}

\begin{lem}
In the region $\{ \text{dist}_{g_a}(\cdot, \mathfrak{D}_1) \lesssim A^{1/2} \lesssim |\vec{\mu}|_a \}\subset M^+_\nu$, which is far away from $\mathfrak{D}_2$ and $\mathfrak{D}_3$,
\[
\begin{cases}
\norm{\tilde{\alpha}_1 - \frac{ 1   }{2 \sqrt{\mu_1^2+a_{22}|\eta|^2} } }_{ C^{k,\alpha}_{0,0}  } \leq  CA^{-1/4} \nu , 
\\
\norm{\tilde{\alpha}_2 }_{ C^{k,\alpha}_{0,0}  } \leq  CA^{-1/4} \nu,
\\
\norm{\tilde{\alpha}_3 }_{ C^{k,\alpha}_{0,0}  } \leq  CA^{-1/4}\nu.
\end{cases}
\]
Likewise with the neighbourhood of $\mathfrak{D}_2$ and $\mathfrak{D}_3$.
\end{lem}

\begin{lem}\label{WeightedHolderspacemetricdeviationestimateTaubNUTregion}
In the region $\{ \text{dist}_{g_a}(\cdot, \mathfrak{D}_1) \lesssim A^{1/2} \lesssim |\vec{\mu}|_a \}\subset M^+_\nu$,
the K\"ahler ansatz is approximated by the suitably gauge fixed metric model $g_{\text{Taub}}$, with \textbf{metric deviation estimate}
\[
\norm{ \tilde{g}^{(1)}- g_{\text{Taub} }}_{ C^{k,\alpha}_{0,0}  } \leq  CA^{-3/4}\nu , \quad \norm{ \tilde{\Omega}^{(1)}- \Omega_{\text{Taub} }}_{ C^{k,\alpha}_{0,0}  } \leq  CA^{-3/4}\nu.
\] 
Similarly with the neighbourhood of $\mathfrak{D}_2$ and $\mathfrak{D}_3$.
\end{lem}

\begin{lem}\label{WeightedHolderspacemetricdeviationestimateflatregion}
The region $\{\ell \gtrsim A^{1/2}  \}\subset M^+_\nu$ is covered by subregions of diameter $\sim \ell$ where the K\"ahler ansatz is approximated by suitably gauge fixed flat models $g_{\text{flat}}$, with \textbf{metric deviation estimate}
\[
\norm{ \tilde{g}^{(1)}- g_{\text{flat} }}_{ C^{k,\alpha}_{0,0}  } \leq  CA^{-3/4}\nu , \quad \norm{ \tilde{\Omega}^{(1)}- \Omega_{\text{flat} }}_{ C^{k,\alpha}_{0,0}  } \leq  CA^{-3/4}\nu.
\]
\end{lem}

Finally, multiplication property for the weighted H\"older norms implies
\begin{lem}\label{volumeformerrortildeE1}
In the region $M^+_\nu \setminus \{    |\vec{\mu}|_a \lesssim A^{1/2} \}  $, the \textbf{volume form error} is estimated by
\[
\norm{ \tilde{E}^{(1)} }_{ C^{k,\alpha}_{-1,0}  } \leq CA^{-3/4}\nu^2  .
\]
\end{lem}

\section{Harmonic analysis I: periodic Euclidean region}\label{HarmonicanalysisI}

This Section obtains \textbf{refined mapping properties of the Euclidean Green operator} $\Lap_a^{-1}$ on $\R_{\mu_1, \mu_2}\times (S^1\times \R)_\eta$, which will be used to correct volume form error away from $\mathfrak{D}$. The method is similar to Lemma \ref{harmonicanalysislemma1}, and the new technical difficulties are the exponential decay estimate and the growth of the error at large distance. We shall identify $T^2$-invariant functions with functions on the base.

As a preliminary observation,
the \textbf{periodic Newtonian potential} on  $\R_{\mu_1, \mu_2}\times (S^1\times \R)_\eta$ equipped with the Euclidean metric $g_a$ is given by
\[
G_a( \mu_1, \mu_2, \eta )= \sum_{n\in \Z} \frac{-1}{ 4\pi^2  |(\mu_1, \mu_2, \eta+n)|_a^2     } .
\]
Its zeroth Fourier mode is
\[
\bar{G}_a (\mu_1, \mu_2, y)= \int_0^1 G_a dx=  - \frac{1}{4\pi A^{1/2} \varrho }.
\]
Up to a factor $A^{1/2}$ this agrees with the Newtonian potential for $g_a'$ on $\R^2_{\mu_1, \mu_2}\times\R_{y}$. Our real emphasis will be on the \textbf{second derivatives}  $\nabla^2_{g_a} G_a$.
Since higher order estimates follow from easy bootstrap arguments, we will focus on absolute estimates.

\begin{lem}\label{harmonicanalysisI1}
	For $\varrho\gtrsim A^{1/2}$,
	\[
	|    \nabla^2_{g_a} G_a- \nabla^2_{g_a} \bar{G}_a    |_{g_a}\leq CA^{1/2}\varrho^{-5} .
	\]	
\end{lem}

\begin{proof}
	By the mean value inequality
	\[
	|\int_{\eta-\frac{1}{2}}^{ \eta+ \frac{1}{2} } |(\mu_1, \mu_2, s+ \sqrt{-1}y)|_a^{-2} ds - |(\mu_1, \mu_2, \eta)|_a^{-2} |\leq CA |(\mu_1, \mu_2, \eta)|_a^{-4},
	\]
	changing $\eta$ to $\eta+n$ and summing over $n\in \Z$, we obtain for $\varrho\gtrsim A^{1/2}$ that
	\[
	\begin{split}
	&|G_a(\mu_1, \mu_2, \eta)- \bar{G}_a(\mu_1, \mu_2, y)| \leq \sum_n CA |(\mu_1, \mu_2, \eta+n)|_a^{-4}
	\\
	\leq & CA \int_{-\infty}^\infty |(\mu_1, \mu_2, s+ \sqrt{-1} y)|_a^{-4} ds \\
	\leq & CA^{1/2} \varrho^{-3} .
	\end{split}
	\]
	The claim follows from $\Lap_a$-harmonicity and bootstrap arguments.
\end{proof}

We can improve this to an \textbf{exponential decay estimate}:

\begin{lem}\label{harmonicanalysisI1prime}
	For $\varrho\gtrsim A^{1/2}$,
	\[
|    \nabla^2_{g_a} G_a- \nabla^2_{g_a} \bar{G}_a    |_{g_a}\leq C A^{-2} e^{-2\pi A^{-1/2} \varrho  }  . 
	\]	
\end{lem}

\begin{proof}
	The basic idea is Fourier analysis in the $x$-variable combined with $\Lap_a$-harmonicity. The argument is a simpler version of Proposition \ref{Exponentialdecayforthefirstorderansatz}, using the barrier method. 
\end{proof}

\begin{lem}\label{harmonicanalysisI2}
	Let $-3<\delta<0$.
	Let a function $f$ be compactly supported in $
	\{  \ell> 2A^{-1/4}   \}\subset \mathcal{B}^+_\nu
	$ with $\norm{f}_{ C^{k,\alpha}_{\delta,0} }\leq 1$.  
	Then $\nabla^2_{g_a}\Lap_a^{-1}f$ is estimated on $\mathcal{B}^+_\nu$ by
	\[
	|\nabla^2_{g_a}\Lap_a^{-1}f  |_{g_a} \leq C\nu \begin{cases}
	A^{\delta/4}\ell^\delta \quad & \ell \lesssim A^{1/2},
	\\
	A^{3\delta/4}\quad & \ell \gtrsim A^{1/2}.
	\end{cases}
	\]
\end{lem}

\begin{rmk}
The support cutoff condition $\varrho< A^{1/2} e^\nu$ is needed because sources located at exponentially large distance drives up the elliptic constants; this suggests the metric ansatz destabilizes at exponentially large distance (\cf Section \ref{runningcoupling}).
\end{rmk}

\begin{rmk}
The Green operator will propagate the effects out of $\text{supp}(f)$
 into the tail region $\{ \varrho\geq A^{1/2}e^\nu  \}$ and the neighbourhood $\{ \ell\leq 2A^{-1/4} \}$ of $\mathfrak{D}$. 
\end{rmk}

\begin{proof}
	The basic idea is similar to Proposition \ref{harmonicanalysislemma1}. We analyse the contribution of the source located at $q$ to the convolution integral $\nabla^2_{g_a}G_a * f(p) $, depending on the spatial separation between $p$ and $q$. We write $|q|_a'= \varrho(q), |p|_a'= \varrho(p)$.

	Suppose $p$ and $q$ do not belong to the same dyadic scale, namely $|p|_a' \geq A^{1/2}+2|q|_a'$ or $|q|_a'\geq A^{1/2}+ 2|p|_a'$. From Lemma \ref{harmonicanalysisI1} we easily deduce
	\[ 
	|\nabla^2_{g_a}G_a |_{g_a} 
	 \leq A^{-1/2}\min( |q|_a'^{-3}, |p|_a'^{-3}  ) ,
	\]
	so the contribution from all such dyadic scales on $\text{supp}(f)$ is bounded by
	\[
	\begin{split}
	& CA^{3\delta/4-1/2} (\int_{2|p|_a< \varrho<  A^{1/2} e^\nu }   \varrho^{-3} d\text{Vol}_a  +\int_{ A^{1/2} \lesssim \varrho< |p|_a'/2}   |p|_a'^{-3} d\text{Vol}_a)
	\leq   C A^{3\delta/4} \nu ,
	\end{split}
	\]
	where we use $\delta>-3$ to control the source $f= O(A^{\delta/4} \ell^\delta )$ in $L^1$.

	We are left with one dyadic scale $|q|_a'\sim |p|_a'\lesssim A^{1/2}e^\nu$. By a similar argument, the contribution from sources at $A^{1/2}\lesssim |p-q|_a \lesssim 2|p|_a'$ is bounded by
	$
	CA^{3\delta/4} \nu  .
	$
	If $\ell(p)>\frac{1}{2} A^{1/2}$, then the contribution from sources at $|p-q|_a \leq \frac{1}{4} A^{1/2}$ is controlled by
	$
	CA^{3\delta/4}
	$
	using standard Schauder theory.

	If $\ell ( p)\leq\frac{1}{2}A^{1/2}$ and $|p-q|_a'\lesssim A^{1/2}$, then for the purpose of estimating the convolution integral we can simply replace the Green kernel $ G_a$ by $\frac{-1}{4\pi^2 |(\mu_1, \mu_2, \eta)|_a^2} $, and correspondingly for their second derivatives. The point is that at this length scale the periodicity effect is secondary, and we are essentially in the same situation as Lemma \ref{harmonicanalysislemma1} with $-3<\delta<0$ and $\tau=0$. A careful examination of that argument there, restoring the $A$-dependence, shows that the contribution of sources inside this region towards $\nabla^2_{g_a}\Lap_a^{-1}f$ is bounded by
	$
	C(A^{1/4} \ell)^\delta .
	$
	
	Combining the above shows the claim.
\end{proof}

\begin{lem}\label{harmonicanalysisI4}
	(\text{Exponential decay of higher Fourier modes})
	In the situation of Lemma \ref{harmonicanalysisI2}, the higher Fourier modes of $\nabla_{g_a}^2\Lap_a^{-1}f$ admit estimate in the region $\{ \ell>A^{1/2} \}\subset \mathcal{B}^+_\nu$,
	\[
	| \nabla^2_{g_a}\Lap_a^{-1}(f-\bar{f}) |_{g_a} \leq C A^{3\delta/4}  e^{-\kappa \tilde{\ell}}.
	\]
\end{lem}

\begin{proof}
	The key observation is that if without loss of generality $f$ has no zeroth Fourier modes, then the convolution integral 
	\[
	\nabla^2_{g_a}G_a* f= ( \nabla^2_{g_a}G_a- \nabla^2_{g_a}\bar{G}_a   )* f,
	\]
	but Lemma \ref{harmonicanalysisI1prime} says the integral kernel $ \nabla^2_{g_a}G_a- \nabla^2_{g_a}\bar{G}_a $ has exponential decay, at a rate \emph{faster} than the exponential decay rate of $f$ itself. Thus at any point $p$ in the region $\{  \ell> A^{1/2}  \}\cap \mathcal{B}^+_\nu $, the contribution to $\nabla^2_{g_a}\Lap_a^{-1}f|_p$ from sources outside the ball $\{ |p-q|_a' \lesssim  A^{1/2}  \}$ is negligible. The contribution from sources inside the ball is treated by standard Schauder theory, and inherits the same exponential decay factor $e^{-\kappa \tilde{\ell}}$ as $f$ itself.
\end{proof}

Combining the Lemmas shows the main  result of this Section after bootstrap.

\begin{prop}\label{harmonicanalysis1main}
	(\textbf{Periodic Euclidean region)}
	In the situation of Lemma \ref{harmonicanalysisI2}, in the region $\{ |\vec{\mu}|_a\gtrsim A^{-1/4}  \}\cap \mathcal{B}^+_\nu $
\[
\norm{ \nabla^2_{g_a} \Lap_a^{-1}f }_{ C^{k,\alpha}_{\delta,0} } \leq C\nu.
\]	
The constant only depends on $k,\alpha, \delta, \kappa$ and the scale-invariant uniform ellipticity bound on $a_{ij}$.
\end{prop}

We also record the following variant (\cf Section \ref{WeightedHoldernormsandinitialerror} for definition of norm).

\begin{prop}\label{harmonicanalysis1main2}
Let $-3<\delta<0$.
Let a function $f$ be compactly supported in $
\{  \ell> 2A^{-1/4}   \}\subset \mathcal{B}^+_\nu
$ with $\norm{f}_{ C^{k,\alpha}_{\delta} }\leq 1$. Then in the region $\{ |\vec{\mu}|_a\gtrsim A^{-1/4}  \}\cap \mathcal{B}^+_\nu $
\[
\norm{ \nabla^2_{g_a} \Lap_a^{-1}f }_{ C^{k,\alpha}_{\delta} } \leq C.
\]		
\end{prop}

We do not need the extra log factor $\nu$ in the RHS because $\norm{f}_{ C^{k,\alpha}_{\delta}}\leq 1 $ implies power law decay on $f$ in the generic region, wheras $\norm{f}_{ C^{k,\alpha}_{\delta,0}}\leq 1 $ implies no decay.

\begin{rmk}
	In the small ball $\{ |\vec{\mu}|_a< A^{-1/4}  \}$ the norms are not defined yet, but the regularity of $\nabla^2_{g_a}\Lap_a^{-1}f$ is well controlled by $\Lap_a$-harmonicity, since here $f=0$ by assumption.
\end{rmk}

\section{Perturbation in the Euclidean region}\label{PerturbationintheEuclideanregionSection}

This Section corrects the volume form error sufficiently away from $\mathfrak{D}$, by perturbatively solving the generalised Gibbons-Hawking equation. We will circumvent the problem caused by metric incompleteness by a trick from \cite{Gabor} called extension norm. From now on $1\ll \nu\ll A^{3/8}$.

\begin{prop}\label{PerturbationintheEulideanregionprop}
	Let $1\ll \nu \ll A^{3/8}$. Then there is a real valued function $\varphi_1$ on $\mathcal{B}^+_\nu$, solving the \textbf{generalised Gibbons-Hawking equation} on $ \mathcal{B}^+_{\nu}\cap \{ \ell> 2A^{1/2}   \} $
	\[
	\tilde{V}^{ij}_{(2)}= \tilde {V}^{ij}_{(1)}+  \frac{\partial^2 \varphi_1}{\partial \mu_i \partial \mu_j }, \quad \tilde{W}_{(2)} =  \tilde{W}_{(1)}-4 \frac{\partial^2 \varphi_1}{\partial \eta \partial \bar{\eta} } ,\quad 
	\det(   \tilde{V}^{ij}_{(2)}   )= \tilde{W}_{(2)}.
	\]
Morever $\varphi_1$ is $\Lap_a$-harmonic on $\mathcal{B}^+_\nu\cap \{ \ell<A^{1/2}  \}$, and
\[
\norm{ \nabla^2_{g_a} \varphi_1 }_{ C^{k,\alpha}_{-1,0}(\mathcal{B}^+_\nu\cap \{  \ell\gtrsim A^{1/2} \} )  } \leq  C \nu^3 A^{-3/4},
\]
and $ |\nabla^2_{g_a} \varphi_1|_{g_a} \leq C\nu^3A^{-3/2}   $ on $\mathcal{B}^+_\nu$. In particular the matrix $(  \tilde{V}^{ij}_{(2)})$ is positive definite and $\tilde{W}_{(2)}$ is positive on $\mathcal{B}^+_\nu$.

\end{prop}

\begin{proof}
The method is to set up a \textbf{Banach iteration scheme} to correct the volume form error. The generalised Gibbons-Hawking equation can be rewritten in the linearised form
\[
\mathcal{L} \varphi_1 + \frac{1}{\tilde{W}_{(1)}} \det( \frac{\partial^2 \varphi_1 }{\partial \mu_i \partial \mu_j }   )= \tilde{E}^{(1)}.
	\]
where the linearised operator
\[
\mathcal{L}= \frac{1}{ \tilde{W}_{(1)}  }( \tilde{V}_{(1)}^{11} \frac{\partial^2}{\partial \mu_2\partial \mu_2}+ \tilde{V}_{(1)}^{22} \frac{\partial^2}{\partial \mu_1\partial \mu_1} - 2\tilde{V}^{12}_{(1)} \frac{\partial^2}{\partial \mu_1\partial \mu_2}  + 4  \frac{\partial^2}{\partial \eta\partial \bar{\eta}}        ).
\]	
The key point below is that in $\mathcal{B}^+_\nu\cap \{  \ell>2A^{1/2} \}$ the quadratic term is small while $\mathcal{L}$ is approximately $\Lap_a$.

	\begin{itemize}
		\item Start with the initial volume form error $\tilde{E}^{(1)}$ on $\mathcal{B}^+_{\nu+1} \cap \{ \ell \geq  A^{1/2}  \}$, where $\norm{ \tilde{E}^{(1)} }_{ C^{k,\alpha}_{-1,0} } \leq CA^{-3/4}\nu^2$ according to Lemma \ref{volumeformerrortildeE1}.
		We will only need the precise value of $\tilde{E}^{(1)}$ in the shrinked region $\mathcal{B}^+_\nu \cap \{ \ell> 2A^{1/2} \} $.

\item
Define the \textbf{extension norm} for a function $f$ on $\mathcal{B}^+_\nu\cap \{ \ell> 2A^{1/2} \} $ as the infimum of the $C^{k,\alpha}_{-1,0}$-norms for all functions $f'$ extending $f$ with compact support inside $\mathcal{B}^+_{\nu+1} \cap \{ \ell> A^{1/2} \} $. The extension norm of $\tilde{E}^{(1)}$ is bounded by $CA^{-3/4}\nu^2$, since we can find an appropriate cutoff function $\chi$ such that $\chi \tilde{E}^{(1)}$ provides a required extension.

\item 
Apply Proposition \ref{harmonicanalysis1main} to produce
$
u_1= \Lap_a^{-1} (   \chi \tilde{E}^{(1)}      ),
$
with second derivative bound on $\mathcal{B}^+_{\nu+1}\cap \{  \ell>A^{1/2} \}$,
\[
\norm{ \nabla^2_{g_a} u_1}_{ C^{k,\alpha}_{-1,0}  } \leq C\nu \norm{  \tilde{E}^{(1)}   }_{ C^{k,\alpha}_{-1,0}  } \leq C\nu^3 A^{-3/4}. 
\]
In particular on $\mathcal{B}^+_{\nu+1}\cap \{  \ell>A^{1/2} \}$,
\[
|\nabla^2_{g_a} u_1|_{ g_a} \leq C\nu^3 A^{ -3/2 },
\]
which in fact holds on the entire $\mathcal{B}^+_{\nu+1}$ using $\Lap_a$-harmonicity in $\{ \ell< A^{1/2} \}$.
Whence the quadratic term is bounded on $ \mathcal{B}^+_{\nu+1}\cap \{  \ell>A^{1/2} \} $ by
	\[
	\begin{split}
	\norm{ \frac{1}{\tilde{W}_{(1)}} \det( \frac{\partial^2 u_1 }{\partial \mu_i \partial \mu_j }   )   }_{ C^{k,\alpha}_{-1,0} } \leq & C\nu^3 A^{-3/2} \norm{ \nabla^2_{g_a} u_1}_{ C^{k,\alpha}_{-1,0}  }
	\\
	\leq & C\nu^4 A^{-3/2} \norm{ \tilde{E}^{(1)} }_{ C^{k,\alpha}_{-1,0}  }
	\\
	\ll & \norm{ \tilde{E}^{(1)} }_{ C^{k,\alpha}_{-1, 0}  }.
		\end{split}
		\]
		The last inequality uses the condition $\nu\ll A^{3/8}$.

The linearised equation is approximately satisfied on $\mathcal{B}^+_{\nu+1}\cap \{  \ell>A^{1/2} \} $:
\[
\begin{split}
& \norm{ \mathcal{L} u_1- \chi \tilde{E}^{(1)} }_{ C^{k,\alpha}_{-1,0}  }
= \norm{ \mathcal{L} u_1- \Lap_a u_1 }_{ C^{k,\alpha}_{-1,0}  }\\
\leq &  A^{-3/4}\nu \norm{ \nabla^2_{g_a}u_1}_{ C^{k,\alpha}_{-1,0}   }
\ll \norm{ \tilde{E}^{(1)} }_{ C^{k,\alpha}_{-1,0}  }.
\end{split}
\]
where we used the metric deviation estimate in Lemma \ref{tildealphailemma}.

Elementary algebra shows that inside $\mathcal{B}^+_\nu\cap \{  \ell>A^{1/2 } \}$, the volume form error is improved:
		\[
		\tilde{E}^{(1)}_1=  \det( W^{p\bar{q}}_{(1)} - 4\frac{\partial^2 u_1 }{\partial \eta_p  \partial \bar{\eta}_q} ) (V_{(1)} + \frac{\partial^2 u_1 }{\partial \mu \partial \mu} )^{-1} -1,
		\]
		
		\[
		\norm{ \tilde{E}^{(1)}_1 }_{ C^{k,\alpha}_{-1}( \mathcal{B}^+_\nu\cap \{\ell >2A^{1/2} \} )   } \ll  \norm{ \tilde{E}^{(1)} }_{ C^{k,\alpha}_{-1,0}( \mathcal{B}^+_{\nu+1}\cap \{ \ell>A^{1/2}  \} )   }.
		\]
		More formally the extension norm of $E^{(1)}$ is far smaller than that of $E^{(1)}$, after taking into account the cutoff procedures.
		
		\item
		Iterate this procedure to produce $u_1, u_2, \ldots$, each time improving the extension norm by a factor say $10^{-1}$. The second derivative estimate
		\[
		\norm{ \nabla^2_{g_a} u_j}_{ C^{k,\alpha}_{-1,0}   } \leq C 10^{-j} \nu^3 A^{-3/4}
		\]
		implies that the series
		$
		\sum_j  \nabla^2_{g_a} u_j 
		$ converges. The series $\varphi_1=\sum_j u_j$ also converges after possibly adjusting $u_j$ by some affine linear functions, and satisfies the Hessian estimate
		$
		\norm{ \nabla^2_{g_a} \varphi_1 }_{ C^{k,\alpha}_{-1,0}   } \leq  C \nu^3 A^{-3/4}.
		$
		By construction  the generalised Gibbons-Hawking equation holds on $\mathcal{B}^+_\nu\cap \{ \ell>2A^{1/2} \}$.
	\end{itemize}
\end{proof}

Applying the generalised Gibbons-Hawking ansatz, we obtain a second K\"ahler ansatz $(\tilde{g}^{(2)}, \tilde{\omega}^{(2)}, \tilde{J}^{(2)}, \tilde{\Omega}^{(2)}    )$ associated to the data $\tilde{V}^{ij}_{(2)}$ and $\tilde{W}_{(2)}$. The new $T^2$-connection is (\cf (\ref{connectionformula}))
\[
\vartheta_i^{(2)}= \vartheta_i+ \sqrt{-1} \frac{\partial^2\varphi_1}{\partial \eta \partial \mu_i }   d\eta- \sqrt{-1} \frac{\partial^2 \varphi_1 }{\partial \bar{\eta} \partial \mu_i }  d\bar{\eta}.
\]

\begin{cor}
\label{volumeformerrortildeg2}
The volume form error $\tilde{E}^{(2)}$ of $(\tilde{g}^{(2)}, \tilde{\omega}^{(2)}, \tilde{J}^{(2)}, \tilde{\Omega}^{(2)}    )$ is zero on $M^+_\nu\cap \{ \ell>2A^{1/2}  \}$ and
satisfies the bound on $M^+_\nu\cap \{ |\vec{\mu}|_a\gtrsim A^{1/2}  \}$
\[
\norm{ \tilde{E}^{(2)} }_{ C^{k,\alpha}_{-1,0}  } \leq CA^{-3/4}\nu^2.
\]
\end{cor}

\section{Glue in the Taub-NUT type metric on $\C^3$}\label{GlueinTaubNUTC3}

The Ooguri-Vafa type K\"ahler metric ansatz is designed as a periodic version of the Taub-NUT type metric on $\C^3$, the latter having the correct topology and metric asymptote to glue in as a metric bubble inside the former. We shall produce the gluing ansatz while maintaining control on the complex structure. This will be divided into a number of steps.

\subsection{Relative Gibbons-Hawking potential}

We plan to exhibit a $T^2$-bundle preserving \text{diffeomorphism} $\Psi_1$ between the Taub-NUT type $\C^3$ and the positive vertex space $M^+_\nu$ over the common base $\{  0<|\vec{\mu} |_a \leq \frac{1}{3}A^{1/2}       \}$, with good estimates on the deviations between both K\"ahler structures.
Since $\frac{1}{3}A^{1/2} < \frac{1}{2}A^{1/2}$, the $\eta$-periodic copies of such punctured discs do not overlap. 
The topology of the $T^2$-bundle structures on both spaces agree by construction. The remaining degrees of freedom in defining $\Psi_1$ amounts to a \textbf{gauge choice}, which is the same as a prescription of $\vartheta_i^{\C^3}-\Psi_1^*\vartheta_i$.

As a general guideline, the corresponding quantities on $\C^3$ and $M^+_\nu$ have the same singularity, so their difference are smooth quantities. We use superscripts for quantities on $\C^3$ to disambiguate from quantities on $M^+_\nu$.

\begin{lem}
Over the region $ \{ |\vec{\mu} |_a \leq \frac{1}{3}A^{1/2}       \}$,
\[
| \Psi_1^{-1*}\alpha_i- \tilde{\alpha}_i  |\leq CA^{-3/4}|\vec{\mu}|_a , \quad |\nabla_{g_a}^k ( \Psi_1^{-1*}\alpha_i- \tilde{\alpha}_i ) |_{g_a}\leq CA^{-1/4-k/2}, \quad i=1,2,3.
\]
\end{lem}

\begin{proof}
The absolute estimate follows from Lemma \ref{positivevertexloggrowthestimate}. The higher order estimate follows from $\Lap_a$-harmonicity.
\end{proof}

\begin{lem}
	Over the  disc $\{ |\vec{\mu}|_a\leq \frac{1}{3}A^{1/2}  \}$
	\[
	|\tilde{\beta}_i- \Psi_1^{-1*}\beta_i|\leq CA^{-1/2}|\vec{\mu}|_a, \quad |\nabla^k_{g_a}(\tilde{\beta}_i- \Psi_1^{-1*}\beta_i)|_{g_a}\leq CA^{-k/2}.
	\]
\end{lem}

\begin{proof}
	The absolute estimate is contained in Lemma \ref{betatildeconverges}. The higher order estimates follow from the differential relations between $\tilde{\beta}_i$ and $\tilde{\alpha}_i$, vis-a-vis ${\beta}_i$ and ${\alpha}_i$ (\cf Lemma \ref{holomorphicdifferentiallemma}).
\end{proof}

\begin{cor}\label{relativeGibbonsHawkingpotential}
	There is a real-valued relative Gibbons-Hawking potential $\varphi_2$ on the disc $\{ |\vec{\mu}|_a\leq \frac{1}{3}A^{1/2}  \}$, such that its second derivatives are given by
	\[
	\begin{cases}
	\frac{\partial^2 \varphi_2 }{\partial \mu_i \partial \mu_j}= V_{(1)}^{ij}- \tilde{V}_{(1)}^{ij}- \frac{\partial^2 \varphi_1}{\partial \mu_i\partial \mu_j}, \quad i,j=1,2,
	\\
	\frac{\partial^2 \varphi_2 }{\partial \eta \partial \bar{\eta}}=-\frac{1}{4}( W_{(1)}- \tilde{W}_{(1)} )- \frac{\partial^2 \varphi_1 }{\partial \eta \partial \bar{\eta}} ,
	\\
	\frac{\partial^2 \varphi_2 }{\partial \eta \partial \mu_i}= \frac{1}{2} (\beta_i- \tilde{\beta}_i)-  \frac{\partial^2 \varphi_1 }{\partial \eta \partial \mu_i},  \quad i=1,2.
	\end{cases}
	\]
	We can demand the estimates in $\{ |\vec{\mu}|_a\leq \frac{1}{3}A^{1/2} \}$:
	\[
	 |\nabla^{k}_{g_a} \varphi_2|_{g_a} \leq C\nu    A^{1/4-k/2}, \quad k\geq 0.
	\]
\end{cor}

\begin{proof}
	The existence of $\varphi_2$ with presecribed second order derivatives is a consequence of integrability, notably Lemma \ref{holomorphicdifferentiallemma} and its counterpart for $M^+_\nu$. If we impose that $\varphi$ and its first order derivatives vanish at the origin, then the estimates follow immediately from the Lemmas above and Proposition \ref{PerturbationintheEulideanregionprop}. 
\end{proof}

\subsection{Modifying the K\"ahler ansatz I}

We now modify  $(\tilde{g}^{(2)}, \tilde{\omega}^{(2)}, \tilde{\Omega}^{(2)})$ to an intermediate K\"ahler ansatz $(\tilde{g}^{(3)}, \tilde{\omega}^{(3)}, \Omega)$  designed to match up \emph{exactly} with $(g^{(1)}, \omega^{(1)}, \Omega_{\C^3})$ over  $\{ |\vec{\mu}|_a\leq \frac{1}{6}A^{1/2}  \}$. This will be constructed using the generalised Gibbons-Hawking ansatz.

Take a standard cutoff function $\chi$ on $\R$ with
\[
\chi(s)= \begin{cases}
1  \quad s\leq 1, \\
0 \quad  s \geq 2,
\end{cases}
\]
and let $\varphi_3= \chi( \frac{|\vec{\mu}|_a}{\frac{1}{6} A^{1/2}} ) \varphi_2$ with $\varphi_2$ from Corollary \ref{relativeGibbonsHawkingpotential},
\[
\tilde{V}_{(3)}^{ij}= \tilde{V}_{(2)}^{ij}+ \frac{\partial^2\varphi_3}{\partial \mu_i \partial \mu_j}  , \quad \tilde{W}_{(3)}^{ij}= \tilde{W}_{(2)}^{ij}- 4\frac{\partial^2\varphi_3}{\partial \eta \partial \bar{\eta}} .
\]
 The perturbations are sufficiently small so that positive definiteness is not affected. The generalised Gibbons-Hawking construction produces the \textbf{intermediate K\"ahler ansatz} $(\tilde{g}^{(3)}, \tilde{\omega}^{(3)}, \Omega)$. We identify $M^+_\nu$ with the underlying space of $(\tilde{g}^{(3)}, \tilde{\omega}^{(3)}, \Omega)$. The $T^2$-connection for $(\tilde{g}^{(3)}, \tilde{\omega}^{(3)}, \Omega)$ is identified as  (\cf (\ref{connectionformula}))
\[
\vartheta_i^{(3)}= \vartheta_i^{(2)}+ \sqrt{-1} \frac{\partial^2\varphi_3}{\partial \eta \partial \mu_i }   d\eta- \sqrt{-1} \frac{\partial^2 \varphi_3}{\partial \bar{\eta} \partial \mu_i }  d\bar{\eta}.
\]
This amounts to making a \textbf{gauge choice}.

By construction $(\tilde{g}^{(3)}, \tilde{\omega}^{(3)}, \Omega)$ agrees identically with $(\tilde{g}^{(2)}, \tilde{\omega}^{(2)}, \tilde{\Omega}^{(2)})$ over $\{ |\vec{\mu}|_a\geq \frac{1}{3}A^{1/2} \}$, and modulo diffeomorphism agrees identically with $({g}^{(1)}, {\omega}^{(1)}, \Omega_{\C^3})$ over $\{ |\vec{\mu}|_a\leq \frac{1}{6}A^{1/2}\}$. 
By Corollary
\ref{relativeGibbonsHawkingpotential},

\begin{lem}\label{gluingTaubNUTmetricdeviation1}
Over the region $\{ \frac{1}{6}A^{1/2} \leq|\vec{\mu}|_a \leq \frac{1}{3}A^{1/2}  \}$, 
	\[
	\norm{ \tilde{g}^{(3)} - g^{(1)} }_{ C^{k,\alpha}_{0,0}  } \leq C\nu A^{-3/4},
	\quad 
	\norm{ \Omega - \Omega_{\C^3} }_{ C^{k,\alpha}_{0,0}  } \leq C\nu A^{-3/4} ,
	\]
and the volume form error $\tilde{E}^{(3)}$ of $\tilde{g}^{(3)}$ satisfies 
\[
\norm{ \tilde{E}^{(3)}- \Lap_a \varphi_3  }_{ C^{k,\alpha}_{-1,0}  } \leq C\nu^2 A^{-3/4}.
\]

\end{lem}

Henceforth the complex structure will be fixed, and can be identified as follows. The new \textbf{holomorphic differentials} are
\begin{equation}\label{holomorphicdifferentialtildeZi}
\begin{cases}
d\log \tilde{Z}_i= d\log Z_i + d  (\frac{\partial (\varphi_3+\varphi_1) }{\partial  \mu_i} ) , \quad i=1,2,
\\
d\log \tilde{Z}_0= d\log Z_0 - d  (\frac{\partial (\varphi_3+\varphi_1) }{\partial  \mu_1}+  \frac{\partial (\varphi_3+\varphi_1) }{\partial  \mu_2}).
\end{cases}
\end{equation}
These have the same $T^3$-periods as $d\log Z_i$, which lie inside $2\pi \sqrt{-1}\Z$, so the new holomorphic functions $\tilde{Z}_0, \tilde{Z}_1, \tilde{Z}_2$ are defined without multivalue issues. The functional equation
\[
\tilde{Z}_0 \tilde{Z}_1 \tilde{Z}_2= 1- e^{2\pi \sqrt{-1}\eta}=1- Z_3
\]
persists from Lemma \ref{functionalequationpositivevertex}.
The results in Proposition \ref{positivevertexcomplexstructure} hold verbatim:

\begin{prop}
(\textbf{complex structure}) The map $M^+\to \{  \tilde{Z}_0\tilde{Z}_1\tilde{Z}_2=1-Z_3      \}$ is a \textbf{holomorphic open embedding}. The \textbf{$T^2$-action} is identified as
\[
e^{i\theta_1}\cdot (\tilde{Z}_0, \tilde{Z}_1, \tilde{Z}_2)= (e^{-i\theta_1}\tilde{Z}_0, e^{i\theta_1}\tilde{Z}_1, \tilde{Z}_2 ), \quad
e^{i\theta_2}\cdot (\tilde{Z}_0, \tilde{Z}_1, \tilde{Z}_2)= (e^{-i\theta_2}\tilde{Z}_0, \tilde{Z}_1,  e^{i\theta_2}\tilde{Z}_2 ),
\] 
and the \textbf{holomorphic volume form} is 
$\Omega=  -\frac{\sqrt{-1}}{2\pi Z_3} d\tilde{Z}_0\wedge  d\tilde{Z}_1\wedge d\tilde{Z}_2  $. We shall identify $M^+$ with its image.
\end{prop}

Over $\{ |\vec{\mu}|_a\leq \frac{1}{6}A^{1/2} \}$ the ansatz $(\tilde{g}^{(3)}, \tilde{\omega}^{(3)}, \Omega)$ is identified with $({g}^{(1)}, \omega^{(1)}, \Omega_{\C^3} )$ after suitable diffeomorphism. An identification of complex coordinates compatible with the holomorphic differential formula (\ref{holomorphicdifferentialtildeZi}) is
\[
\begin{cases}
\tilde{Z}_1= z_1 \exp\{(\pi\sqrt{-1}\beta_1(0,0,1)- \sqrt{-1}\theta_1^\infty)\eta \},\\ 
\tilde{Z}_2=z_2 \exp\{(\pi\sqrt{-1}\beta_2(0,0,1)-\sqrt{-1}\theta_2^\infty)\eta \}, \\
\tilde{Z}_0=\frac{-2\sqrt{-1}\sin(\pi\eta)}{\eta} z_0 \exp\{(\pi\sqrt{-1}\beta_0(0,0,1)+ \sqrt{-1} \theta_1^\infty+ \sqrt{-1}\theta_2^\infty )\eta \}    .
\end{cases}
\]
This fixes the normalisation for the multiplicative constants of $\tilde{Z}_i$.

\subsection{Modifying the K\"ahler ansatz II}

We make a second modification from $(\tilde{g}^{(3)}, \tilde{\omega}^{(3)}, \Omega)$ to another new K\"ahler ansatz $(\tilde{g}^{(4)}, \tilde{\omega}^{(4)}, \Omega)$ designed to match up  with the Taub-NUT type metric $(g_{\C^3}, \omega_{\C^3}, \Omega_{\C^3}) $ in Chapter \ref{TaubNUTtypemetriconC3}.

Recall from Theorem \ref{TaubNUTC3main} that there is a K\"ahler potential $\phi^{\C^3}$ such that
\[
\omega_{\C^3}= \omega^{(2)}+ \sqrt{-1}\partial \bar{\partial} \phi^{\C^3}= \omega^{(1)}+ \sqrt{-1}\partial \bar{\partial} \phi^{\C^3}, \quad \text{for } |\vec{\mu}|_a\gtrsim A^{-1/4},
\] with bound
$
\norm{ d\phi^{\C^3}}_{   C^{k+1, \alpha}_{ -\epsilon, -1+\epsilon}(\C^3)      }\leq CA^{-1/4}.
$
We can impose a normalisation such that
$
|\phi^{\C^3}| \leq CA^{-1/2+ \frac{3}{4}\epsilon }$ for $ \frac{1}{100}A^{1/2}\lesssim |\vec{\mu}|_a\leq \frac{1}{3}A^{1/2} .
$

We then define a \textbf{modified K\"ahler metric ansatz} $\tilde{\omega}^{(4)}$ on $M^+_\nu$. Take a standard cutoff function
\[
\chi(s)=\begin{cases}
1 \quad  s\leq 1, \\
0 \quad  s\geq 2,
\end{cases}
\]
and define 
\[
\tilde{\omega}^{(4)}= \tilde{\omega}^{(3) }+ \sqrt{-1}\partial \bar{\partial} \phi_4,  
\quad 
\phi_4=  \chi( \frac{|\vec{\mu}|_a}{\frac{1}{12}A^{1/2}  }  )  \phi^{\C^3}- 2\varphi_3
.
\]
In particular
\[
\begin{cases}
\tilde{\omega}^{(4)}= \omega_{\C^3} - 2\sqrt{-1} \partial \bar{\partial} \varphi_3, \quad  &|\vec{\mu}|_a \leq \frac{1}{12} A^{1/2},\\
\tilde{\omega}^{(4)}= \tilde{\omega}^{(3) }, \quad & |\vec{\mu}|_a \geq \frac{1}{3} A^{1/2}.
\end{cases}
\]
The positive definiteness of $\tilde{\omega}^{(4)}$ follows from the  metric deviation estimate:
\[
\begin{cases}
\norm{ \partial \bar{\partial} \varphi_3 }_{ C^{k,\alpha}_{0,0}( \C^3 \cap \{  |\vec{\mu}|_a\leq    \frac{1}{3}  A^{1/2} \}  )  } \leq C\nu A^{-3/4}, 
\\
\norm{ \partial \bar{\partial} \{ \chi( \frac{|\vec{\mu}|_a}{\frac{1}{12}A^{1/2}  }  )  \phi^{\C^3} \}   }_{ C^{k,\alpha}_{-1-\epsilon,0}( |\vec{\mu}|_a\sim A^{1/2} ) } \leq CA^{3/4(-1+\epsilon)} .
\end{cases} 
\]

Here $\varphi_3$ is inserted to approximately cancel the cutoff error $\Lap_a \varphi_3$ in the volume form error $\tilde{E}^{(3)}$ (\cf Lemma \ref{gluingTaubNUTmetricdeviation1}).

\begin{lem}\label{glueTaubNUTC3metricdeviation}
The \text{volume form error} for $\tilde{g}^{(4)}$ admits bound in $\{ |\vec{\mu}|_a \leq \frac{1}{3} A^{1/2}  \}$:
\[
\tilde{E}^{(4)}= \frac{4}{3}\frac{ (\tilde{\omega}^{(4)})^ 3 }{  \sqrt{-1}\Omega\wedge \overline{\Omega}  }-1, \quad \norm{ \tilde{E}^{(4)} }_{ C^{k,\alpha}_{-1-\epsilon,0}( \C^3\cap \{ |\vec{\mu}|_a \leq \frac{1}{3} A^{1/2}  \} ) } \leq C\nu^2 A^{3/4(-1+\epsilon)} .
\] 
	
\end{lem}

\subsection{Global weighted H\"older norms and error estimates}

Now we introduce the global weighted H\"older norms $\norm{\cdot}_{C^{k,\alpha}_{\delta} }$ on $M^+_\nu$ by demanding that up to uniform equivalence the norm is
\begin{itemize}
	\item 
 $\norm{\cdot}_{ C^{k,\alpha}_{\delta}   }$ on $M^+\cap \{ |\vec{\mu}|_a\gtrsim A^{-1/4}  \}$, 	as defined in Section \ref{WeightedHoldernormsandinitialerror}.
	\item 
$ \norm{\cdot}_{C^{k,\alpha}_{ \delta, 0  }(\C^3\cap \{  |\vec{\mu}|_a\leq  \frac{1}{3} A^{1/2}   \}   ) }$ on $\C^3$ for $ |\vec{\mu}|_a\leq \frac{1}{3}A^{1/2}$. 
\end{itemize}
On overlapping regions the definitions are equivalent.

\begin{prop}
On $M^+_\nu$
the \textbf{volume form error} satisfies the estimate
\begin{equation}\label{positivevertexvolumeformerror}
\norm{ \tilde{E}^{(4)} }_{ C^{k,\alpha}_{-1-\epsilon } } \leq CA^{3/4(-1+\epsilon) } \nu^2 .
\end{equation}
\end{prop}

\begin{proof}
Combine Lemma \ref{glueTaubNUTC3metricdeviation} with Corollary \ref{volumeformerrortildeg2}.
\end{proof}

\section{Harmonic analysis II: perturbation to Calabi-Yau metric}\label{HarmonicanalysisII}

We now \emph{shift to the complex geometric viewpoint} and solve the complex Monge-Amp\`ere equation by perturbative methods.
The main result of the linear theory is (Compare Proposition \ref{harmonicanalysismain}):

\begin{prop}\label{harmonicanalysisIImain}
Let $-3<\delta<-1$ and $1\ll \nu \ll A^{3/8}$. Let $f$ be a  $T^2$-invariant function compactly supported in $M^+_\nu$ with $\norm{f}_{ C^{k,\alpha}_{\delta } }=1$. Then there is a $T^2$-invariant function $u$ such that the Poisson equation is approximately solved on $M^+_\nu$:
 \[
\norm{ \Lap_{\tilde{g}^{(4)} } u-f }_{ C^{k,\alpha}_{\delta} } \ll 1,
\]
with the Hessian bound
\[
\norm{ \nabla^2_{\tilde{g}^{(2)}} u  }_{ C^{k,\alpha}_{\delta}  } \leq C, \quad \norm{ d u   }_{ C^{k+1,\alpha}_{\delta+1}(M^+_\nu)  } \leq CA^{-1/4}.
\]
The constants depend only on $k,\alpha, \delta,\kappa$ and the scale invariant uniform ellipticity constant of $a_{ij}$.
\end{prop}

\begin{proof}(Sketch)
The method is the decomposition and patching
argument of Section \ref{Harmonicanalysis}, using Proposition \ref{harmonicanalysis1main2}, Lemma \ref{harmonicanalysislemma3} and Proposition \ref{harmonicanalysismain} as ingredients to provide local parametrices. 
\end{proof}

\begin{thm}\label{OoguriVafatypemetriconpositivevertextheorem}
(\textbf{Ooguri-Vafa type metric on the positive vertex})
Fix $k,\alpha, \kappa$ and $0<\epsilon\ll 1$, and let $1\ll \nu \ll A^{3/8}$.
	Then there is a $T^2$-invariant Calabi-Yau metric on $M^+_\nu$ given by a $T^2$-invariant  K\"ahler potential $\phi^+$,
	\[
	\omega_+= \tilde{\omega}^{(4)}+ \sqrt{-1}\partial \bar{\partial}\phi^+,\quad \omega_+^3= \frac{3}{4}\sqrt{-1}\Omega\wedge \overline{\Omega},
	\]
	satisfying the metric deviation estimate
	\begin{equation}\label{OoguriVafatypemetricon+vertextheoremmetricdeviation}
	\norm{\omega_+- \tilde{\omega}^{(4)} }_{  C^{k,\alpha}_{-1-\epsilon}(M^+_\nu ) }   \leq C\nu^2  A^{3/4(-1+\epsilon)}, \quad \norm{ d \phi^+}  _{ C^{k+1,\alpha}_{-\epsilon}(M^+_\nu) } \leq C\nu^2  A^{-1+3\epsilon/4}.
	\end{equation}
The constants depend only on $k,\alpha,\epsilon, \kappa$ and the scale invariant ellipticity bound on $a_{ij}$.
\end{thm}

\begin{proof}(Sketch)
Given Proposition \ref{harmonicanalysisIImain}, one can set up a Banach iteration scheme to correct the volume form error. A subtlety caused by metric incompleteness is that the parametrix $P_\nu$ can only invert sources with compact supports. This problem can be circumvented using the extension norm trick as in Proposition \ref{PerturbationintheEulideanregionprop}, and we obtain a Calabi-Yau metric on a shrinked domain $M^+_{\nu-1}$. Changing $\nu$ to $\nu+1$ gives the statement.
\end{proof}

\begin{rmk}\label{exponentialnecklength}
	The metric lives over a region $M^+_\nu$ with \textbf{exponential neck length}, but the exponent $\nu$ is \emph{not} expected to be optimal. For existence results over longer necks one should allow the coupling constants $a_{ij}$ to drift according to the renormalisation flow equation (\cf Section \ref{runningcoupling} and \ref{runningcouplingnegativevertex} for discussions).
\end{rmk}

\section{Ooguri-Vafa type metric on the positive vertex}\label{PerturbationintoaCYmetricpositivevertex}

We discuss geometric aspects of the Ooguri-Vafa type metric $\omega_+$.

\begin{cor}
(\textbf{Exponential decay to semiflat metric away from $\mathfrak{D}$}) Assume the setup of Theorem \ref{OoguriVafatypemetriconpositivevertextheorem}. In the subregion $\{\tilde{\ell}\gtrsim 1 \}\subset M^+_\nu$, the deviation of $\omega_+$ from its zeroth Fourier mode $\bar{\omega}_+$ decays exponentially:
\begin{equation}
|\omega_+   - \bar{\omega}_+|\leq C  A^{-3/4}\nu e^{- \kappa\tilde{ \ell}}  .
\end{equation}
The constant depends only on $\epsilon, \kappa$ and the scale invariant  ellipticity bound on $a_{ij}$. The decay rate $0<\kappa<1$ can be chosen arbitrarily close to 1.
\end{cor}

\begin{cor}\label{SpeicalLagrangianfibration+vertexCor}
(\textbf{Special Lagrangian fibration})
There exist moment coordinates $\tilde{\mu}_1, \tilde{\mu}_2$ for the $T^2$ action on $\omega^+$. The special Lagrangian fibration
\begin{equation}\label{specialLagrangianfibration+vertex}
M^+_\nu \xrightarrow{ ( \tilde{\mu}_1, \tilde{\mu}_2, \text{Im}\eta   )   } \R^3
\end{equation}
is proper over $\{|(\tilde{\mu}_1, \tilde{\mu}_2, \text{Im}\eta)|_a' \leq \frac{1}{2}A^{1/2}e^\nu \} \subset \R^3$ where the generic fibre is topologically $T^3$. The critical point set is $\bigcup_{i,j\in \{0,1,2\}  } \{ \tilde{Z}_i= \tilde{Z}_j=0    \} $ and the discriminant locus is contained in $\mathfrak{D}$.  The monodromy of the fibration and the topology of the central singular fibre agrees with the Gross-Ruan prediction in Section \ref{Positivevertices}.
\end{cor}

\begin{proof}
By similar calculations as in Corollary \ref{specialLagrangianfibrationC3}, the moment map on $M^+_\nu$ is expressed as
\[
\tilde{\mu}_i=\begin{cases}
 \mu_i + \iota_{ \frac{\partial}{\partial \theta_i}  } d^c \phi^+, \quad  & |\vec{\mu}|_a > \frac{1}{3}A^{1/2}
 \\
 \mu_i + \iota_{\frac{\partial}{\partial \theta_i}} d^c ( \phi^+ + \phi_4       ) , \quad &   \frac{1}{12}A^{1/2}< |\vec{\mu}|_a \leq \frac{1}{3}A^{1/2},    \\
 \tilde{\mu}_i^{\C^3}+   \iota_{\frac{\partial}{\partial \theta_i}} d^c  (\phi^+ - 2\varphi_3)   , \quad & |\vec{\mu}|_a \leq \frac{1}{12}A^{1/2}.
\end{cases}
\]
where $\phi_4 $ is the K\"ahler potential between $\tilde{\omega}^{(4)}$ and $\tilde{\omega}^{(3)}$ (\cf Section Section \ref{GlueinTaubNUTC3}), and  $\tilde{\mu}_i^{\C^3}$ are the moment coordinates for $\omega_{\C^3}$ (\cf Corollary \ref{specialLagrangianfibrationC3}). By construction $\tilde{\mu}_1, \tilde{\mu}_2, \tilde{\mu}_1-\tilde{\mu}_2$ vanish respectively along $\mathfrak{D}_1, \mathfrak{D}_2, \mathfrak{D}_3$, due to the respective vanishing of the circle generators $\frac{\partial}{\partial \theta_1}, \frac{\partial}{\partial \theta_2}, \frac{\partial}{\partial \theta_1}-\frac{\partial}{\partial \theta_2}$. This fixes the additive normalisation on the moment coordinates.

The gradient estimates on K\"ahler potentials and  Corollary \ref{specialLagrangianfibrationC3}  imply on $M^+_\nu$
\begin{equation}\label{momentmapspositivevertex}
|\mu_i- \tilde{\mu}_i| \leq
\begin{cases}
 C\nu^2 A^{-5/4+3\epsilon/4}, \quad & |\vec{\mu}|_a \geq\frac{1}{3}A^{1/2} ,
 \\
 CA^{-3/4} \ell^{-\epsilon} |\vec{\mu}|_a^{-1+\epsilon}+ C\nu  A^{-1/2} , \quad & |\vec{\mu}|_a < \frac{1}{12} A^{1/2}.
\end{cases}
\end{equation}
In particular, if $|(\tilde{\mu}_1, \tilde{\mu}_2, y)|_a' \leq \frac{1}{2}A^{1/2} e^\nu$, then $\varrho\leq \frac{3}{4} A^{1/2} e^\nu$, so the map (\ref{specialLagrangianfibration+vertex}) is \emph{proper} over $\{|(\tilde{\mu}_1, \tilde{\mu}_2, y)|_a' \leq \frac{1}{2}A^{1/2}e^\nu \} \subset \R^3$. 
By the same argument in Corollary \ref{specialLagrangianfibrationC3}, the fibres of (\ref{specialLagrangianfibration+vertex}) are \emph{special Lagrangians} of phase angle zero, the critical point set is $\bigcup_{i,j\in \{0,1,2\} } \{ \tilde{Z}_i= \tilde{Z}_j=0    \} $ and the discriminant locus is contained in $\mathfrak{D}$.

Next we consider the map $(\mu_1, \mu_2, \eta)\mapsto (\tilde{\mu}_1, \tilde{\mu}_2, \eta)$ on the region $\{  \varrho< A^{1/2}e^\nu  \}$. Using (\ref{momentmapspositivevertex}) and the implicit function theorem, this map restricted to the region $\{   \ell\geq A^{1/2}, \varrho < \frac{3}{4} A^{1/2} e^\nu    \}$ 
is an approximate identity, and in particular a diffeomorphism onto its image. 
Morever by (\ref{momentmapspositivevertex}) no  points elsewhere can map into 
 $\text{Image}( \{   \ell\geq A^{1/2}, \varrho < \frac{3}{4} A^{1/2} e^\nu    \}       )$. Interpreted geometrically, this implies that the special Lagrangian fibres of (\ref{specialLagrangianfibration+vertex})
 lying over the region
 $\{ |(\tilde{\mu}_1, \tilde{\mu}_2, y)|_a'\leq \frac{1}{2}A^{1/2}e^\nu \}$ and suitably away from $\mathfrak{D}$,  must be small perturbations of the $T^3$-fibres of the map $ M^+_{\nu}\xrightarrow{(\mu_1, \mu_2, y)} \R^3$. This shows the generic fibre of (\ref{specialLagrangianfibration+vertex}) is \textbf{topologically} $T^3$, and the \textbf{monodromy data} of (\ref{specialLagrangianfibration+vertex}) is the same as for  $ M^+\xrightarrow{(\mu_1, \mu_2, y)} \R^3$, which by construction agrees with the Gross-Ruan prediction in Section \ref{Positivevertices}.

Finally we need to determine the \textbf{topology of the central singular fibre}, defined as the set $X_0=\{ \tilde{\mu}_1=\tilde{\mu}_2=0, y=0\}$, which is invariant under the $T^2$-action. From our knowledge of the critical point set, the only singular point on the central fibre is $\tilde{Z}_0=\tilde{Z}_1=\tilde{Z}_2=0$. Thus the quotient $X_0/T^2$ must be a compact 1-dimensional manifold with possibly one singular point. But there is also a homological constraint
\[
\text{Vol}_{ g_+ }( X_0 )= \int_{ X_0 } \Omega= 4\pi^2 \int_{X_0/T^2 } d\eta= \int_{ T^3} \Omega= 4\pi^2, 
\]
so $X_0/T^2$ is connected and must in fact be a circle. Therefore $X_0$ has the topology of $T^3$ with a copy of $T^2$ collapsed to a point, in accordance with the Gross-Ruan prediction on the positive vertex.
\end{proof}

\begin{rmk}
The singular fibres over $\mathfrak{D} \setminus\{0\}$ have non-isolated singularities by $T^2$-invariance.
This does not contradict Joyce's critique, since the Ooguri-Vafa type metric on the positive vertex is not a generic metric (\cf review Section \ref{Joycecritique}). But when we glue the Ooguri-Vafa type metric  into the global Calabi-Yau metric on a degenerating 3-fold (\cf review Section \ref{Degeneratingtorichypersurface}), the exponentially small corrections to the complex structure will destroy $T^2$-invariance. We then expect the singularity structure of the SYZ fibration to be drastically changed, and in particular its discriminant locus thickens into a ribbon around $\mathfrak{D}$ as predicted by Joyce \cite{Joyce}. 
\end{rmk}

\section{Incompleteness and running coupling}\label{runningcoupling}

We now give a deeper perspective on the \textbf{incompleteness} of the metric, and a semi-heuristic discussion about how to partially overcome one of the main limitations of the perturbation method: the final metric one constructs is by necessity $C^0$-close to the metric ansatz one starts with.

The main insights are as follows. The Ooguri-Vafa type metric is intended as an \textbf{effective local description below a certain distance scale} for the collapsing family of Calabi-Yau metrics on compact manifolds near the large complex structure limit. Our starting assumption is that the metric is a \textbf{perturbation of a constant solution} after incorporating topology.
These constant solutions come naturally in a family parametrised by the \textbf{coupling constants} $a_{ij}$, which have a geometric meaning in terms of the size and shape of the generic $T^2$-fibres in the local region. The nontrivial topology manifests itself in a distributional equation which dictates the first order corrections $\tilde{\alpha}_i$ to the constant solutions, and after Fourier analysis we see the dominant correction terms $\bar{\alpha}_i$ depend \textbf{logarithmically} on  $\mu_i, \eta$. The slow growth of $\log$ means it can be treated as a perturbation term in an exponentially long region, but once we attempt to go beyond, the correction will have a perceptible effect on the size and shape of the average $T^2$-fibres, which would break down our initial effective description via the original constant solution. This suggests that \emph{the coupling constants in the effective description drift slowly as we move up the logarithmic scale}, a phenomenon we call \textbf{running coupling}. Morever, the precise formula of these log corrections dictate how these coupling constants change as a function of the logarithmic scale, which we will discuss under the name of  \textbf{renormalisation flow equation}.
Geometrically, the ansatz metrics naturally come in families, and each time we move up a log scale, we really should glue a different ansatz with slightly changed coupling constants to the previous ansatz. The fact that at very large distance scales the original ansatz should be replaced by another ansatz within the same family, is the deep reason why the ansatz metric is \textbf{incomplete}.

The terminologies are based on the following analogy. According to my rudimentary understanding of high energy physics, Quantum Electrodynamics (QED) is intended as an effective description below a certain energy scale for some more sophisticated theories. The starting assumption of Feynman diagram calculations in QED is that the scattering amplitudes are perturbations of the free field theory, after adding new interaction terms in the Lagrangian. These interaction terms come naturally in a family parametrised by the coupling constants, whose physical meaning is related to the observed charges in low energy experiments. Loop calculations in Feynmann diagrams suggest that the coupling constants depend on the energy scale at which one conducts the experiments, a phenomenon known as running coupling. The equation which governs how the coupling constants change as a function of the cutoff energy scale is known as the renormalisation flow equation.

We now flesh out the ideas in the setting of our Ooguri-Vafa type metrics on the positive vertices. We begin by recalling some main features about the \emph{family} of ansatz metrics. The construction begins with the choice of parameters $a_{ij}$, and outputs a generalised Gibbons-Hawking metric associated to the data
\[
\begin{cases}
V^{11}_a= a_{11}+ \tilde{\alpha}_{1, a}+ \tilde{\alpha}_{3,a} + \text{constant},
\\
V^{12}_a= a_{12}-\tilde{\alpha}_{3,a} + \text{constant},
\\
V^{22}_a= a_{22}+ \tilde{\alpha}_{2, a}+ \tilde{\alpha}_{3,a} + \text{constant},
\\
W_a= A+ a_{22} \tilde{\alpha}_1+ a_{11} \tilde{\alpha}_2+ ( a_{11}+2a_{12}+a_{22}  )\tilde{\alpha}_3+ \text{constant}.
\end{cases}
\]
Here the subscript is to emphasize the dependence on $a_{ij}$. The ambiguity of twisting by a flat connection is not important for the discussions below. The functions $\tilde{\alpha}_i$ generically behave like the logarithmic functions $\bar{\alpha}_i$. The constants above refer to numbers independent of $\mu_1, \mu_2, \eta$ which are up to our choice (\cf Remark \ref{freedomofconstant}). The significance of this extra freedom is that if we are interested only in the ansatz at one particular logarithmic scale, then we can always adjust the constants to cancel some log factors in $\tilde{\alpha}_i$ so that $a_{ij}$ is the \textbf{average value} of $V^{ij}_a$ over this log scale. The QFT analogue of these constants are called \emph{counterterms}. This step is needed to back up the idea that the constant solution defined by $a_{ij}$ really offers an \textbf{effective description at the given log scale} of the metric, suitably away from the discriminant locus $\mathfrak{D}$. This issue did not appear previously, because when $A^{-1/2} \varrho\lesssim 1$ these constants were essentially zero.

The central question is how these effective coupling constants $a_{ij}$ vary as a function of the log scale $\lambda$. Moving up to the next log scale means 
\[
\lambda\mapsto \lambda+1, \quad \mu_i\mapsto e\mu_i, \quad y\mapsto ey.
\]
Since $a_{ij}$ drifts very slowly, to zeroth order we can treat them as constants. Now $\bar{\alpha}_i$ are explicit functions given by formula (\ref{alphabar}), whose values receive a small increment as we move up the log scale:
\[
\begin{cases}
\bar{\alpha}_1\mapsto \bar{\alpha}_1 - \frac{1}{2 \sqrt{a_{22}} }, 
\\
\bar{\alpha}_2\mapsto \bar{\alpha}_2 - \frac{1}{2 \sqrt{a_{11}} }, 
\\
\bar{\alpha}_3\mapsto \bar{\alpha}_3 - \frac{1}{2 \sqrt{a_{11}+2a_{12}+a_{22} } }.
\end{cases}
\]
Since $\tilde{\alpha}_i$ are generically almost the same as $\bar{\alpha}_i$, this means as we move up a log scale, the average value of $V^{ij}_a$ drift by
\[
\begin{cases}
V^{11}_a \mapsto V^{11}_a - \frac{1}{2 \sqrt{a_{22}} }- \frac{1}{2 \sqrt{a_{11}+2a_{12}+a_{22} } }, 
\\
V^{12}_a \mapsto V^{12}_a + \frac{1}{2 \sqrt{a_{11}+2a_{12}+a_{22} } }, 
\\
V^{22}_a \mapsto V^{22}_a - \frac{1}{2 \sqrt{a_{11}} }- \frac{1}{2 \sqrt{a_{11}+2a_{12}+a_{22} } }.
\end{cases}
\]
In our viewpoint, it means the first order change of the coupling constants when we move up a log scale is
\[
\begin{cases}
a_{11} \mapsto a_{11} - \frac{1}{2 \sqrt{a_{22}} }- \frac{1}{2 \sqrt{a_{11}+2a_{12}+a_{22} } }, 
\\
a_{12} \mapsto a_{12} + \frac{1}{2 \sqrt{a_{11}+2a_{12}+a_{22} } }, 
\\
a_{22} \mapsto a_{22} - \frac{1}{2 \sqrt{a_{11}} }- \frac{1}{2 \sqrt{a_{11}+2a_{12}+a_{22} } }.
\end{cases}
\]
We now denote \[
p_1=\sqrt{a_{22}}, \quad p_2= \sqrt{a_{11}}, \quad p_3= \sqrt{a_{11}+2a_{12}+a_{22}}.
\]
As we move up a log scale,
\[
\begin{cases}
\lambda\mapsto \lambda+1,\\
p_1^2 \mapsto p_1^2 - \frac{1}{2 p_2 }- \frac{1}{2 p_3 },
\\
p_2^2 \mapsto p_2^2 - \frac{1}{2 p_1 }- \frac{1}{2 p_3 }, 
\\
p_3^2 \mapsto p_3^2 - \frac{1}{2 p_1 }- \frac{1}{2 p_2 }.
\end{cases}
\]

We have presented this discussion from a discretized viewpoint, which the author thinks is conceptually simpler. The continuum version is the \textbf{renormalisation flow equation}
\begin{equation}\label{renormalisationflowequation}
\begin{cases}
 \frac{d}{d\lambda} p_1^2 =  - \frac{1}{2 p_2 }- \frac{1}{2 p_3 },
\\
 \frac{d}{d\lambda} p_2^2 = - \frac{1}{2 p_1 }- \frac{1}{2 p_3 }, 
\\
 \frac{d}{d\lambda}  p_3^2  =- \frac{1}{2 p_1 }- \frac{1}{2 p_2 }.
\end{cases}
\end{equation}
The remarkable fact is that this ODE system is \textbf{exactly solvable}.

\begin{prop}\label{renormalisationflowsolution}
There exist constants $K_1, K_2, K_3$ such that the solution to the renormalisation flow equation admits the parametrised representation
\begin{equation*}
\begin{cases}
p_1= t^{2/3}+ \frac{1}{3} (K_1+K_2) t^{-1/3} ,\\
p_2= t^{2/3}+  \frac{1}{3} (K_2-2K_1  )t^{-1/3} ,\\
p_3= t^{2/3} + \frac{1}{3} (K_1- 2K_2) t^{-1/3},\\
\lambda= - \frac{2}{3} t^2+ \frac{4}{9}(K_1^2+K_2^2-K_1K_2)\log t+ \frac{4}{81} (K_1+K_2)(K_2-2K_1)(K_1-2K_2)\frac{1}{t}+K_3.
\end{cases}
\end{equation*}
\end{prop}

\begin{proof}
The renormalisation flow equation is equivalent to
\[
\begin{cases}
\frac{d}{d\lambda} p_1 =  - \frac{p_2+p_3}{4 p_1p_2p_3  },
\\
\frac{d}{d\lambda} p_2 =  - \frac{p_1+p_3}{4 p_1p_2p_3  },
\\
\frac{d}{d\lambda} p_3 =  - \frac{p_1+p_2}{4 p_1p_2p_3  }.
\end{cases}
\]
Summing over the three equations,
\[
\frac{d}{d\lambda} (p_1+ p_2+p_3)= - \frac{p_1+p_2+p_3}{2 p_1p_2p_3  }, 
\]
and taking the differences give
\[
\frac{d}{d\lambda} (p_1-p_2)=  \frac{p_1-p_2}{4 p_1p_2p_3  }, \quad \frac{d}{d\lambda} (p_1-p_3)=  \frac{p_1-p_3}{4 p_1p_2p_3  }.
\]
Without loss of generality $p_1\geq p_2\geq p_3$, then
\[
d\log (p_1+p_2+p_3  )=-2d\log (p_1-p_2)=-2 \log (p_1-p_3)= - \frac{d\lambda}{2p_1p_2p_3}.
\]
In the degenerate case where $p_1=p_2$ say, it is understood that $p_1=p_2$ identically. Denote 
$
p_1+ p_2+ p_3= 3 t^{2/3}
$
for some new parameter $t$, then after integration
\[
p_1-p_2= K_1 t^{-1/3}, \quad p_1-p_3= K_2 t^{-1/3},
\]
for some constants $0\leq K_1\leq K_2 $. Rewriting these equations give
\[
\begin{cases}
p_1= t^{2/3}+ \frac{1}{3} (K_1+K_2) t^{-1/3} ,\\
p_2= t^{2/3}+  \frac{1}{3} (K_2-2K_1  )t^{-1/3} ,\\
p_3= t^{2/3} + \frac{1}{3} (K_1- 2K_2) t^{-1/3}.
\end{cases}
\]
Now 
\[
d\lambda= - \frac{4}{3} p_1p_2p_3 \frac{dt}{t} =  \frac{-4}{3t^2} (t+ \frac{1}{3} (K_1+K_2)  ) (t+ \frac{1}{3} (K_1-2K_2)  ) (t+ \frac{1}{3} (K_2-2K_1)  ) dt.
\]
Increasing $\lambda$ corresponds to decreasing $t$. We integrate to obtain
\[
\lambda= - \frac{2}{3} t^2+ \frac{4}{9}(K_1^2+K_2^2-K_1K_2)\log t+ \frac{4}{81} (K_1+K_2)(K_2-2K_1)(K_1-2K_2)\frac{1}{t}+K_3,
\]
where $K_3$ is an integration constant.
\end{proof}

The rest of the Section offers a heuristic interpretation of the renormalisation flow, whose power is to predict \textbf{effective} metric behvaiour up to a very large distance scale. It is helpful to keep in mind the Gross-Wilson K3 metric   \cite{GrossWilson}. The positive vertex is best understood as part of a global SYZ $T^3$-fibration on a Calabi-Yau 3-fold near the large complex structure limit consisting of a finite number of overlapping pieces with simple complex geometric descriptions (\cf review Section \ref{Degeneratingtorichypersurface}). 
The renormalisation flow \textbf{breaks down} when $\min_i p_i$ becomes negative, which indicates a \textbf{metric transition} into a different piece in the 3-fold.

In our normalisation convention $ \text{Vol}(T^3)=4\pi^2\sim 1$. In the generic region of the 3-fold, it is reasonable to expect the 3 circle factors of the SYZ $T^3$-fibres to have comparable length scales, so $\text{diam}(T^3)\sim 1$. In contrast, our starting point for constructing the Ooguri-Vafa type metric on the positive vertex is that a $T^2$ factor inside $T^3$ has much smaller diameter compared to $\text{diam}(T^3/T^2)$. In order for the Ooguri-Vafa type metric to smoothly transition into the generic region of the SYZ fibration, we require an \textbf{exponentially long neck region}, modelled by the renormalisation flow.

The renormalisation flow has the curious feature that at smaller distance scales $\sum p_i$ becomes larger but $|p_i-p_j|$ becomes smaller, so the scale invariant ellipticity bound $CA^{1/2}\delta_{ij}\leq a_{ij}\leq CA^{1/2}\delta_{ij}$ works better at smaller distance scales. Suppose this bound holds throughout the renormalisation flow until $A$ decreases to $A\sim 1$ where the metric transitions into the generic region, then in Proposition \ref{renormalisationflowsolution} the constants $K_1, K_2=O(1)$. Consequently at smaller distance scales, where $t$ is large, the terms $K_1, K_2$ are neglegible, and the solution of the renormalisation flow is approximated by the special solution
\[
\begin{cases}
p_1=p_2=p_3= t^{2/3} , \\
\lambda= - \frac{2}{3} t^2+K_3, \quad t\gg 1.
\end{cases}
\]
This special solution is invariant under the \textbf{$S_3$-discrete symmetry} interchanging the 3 edges $\mathfrak{D}_i$. The insight is that the most symmetric configuration of $a_{ij}$ is the \textbf{attractive fixed point} of the renormalisation flow.

We can also use the special solution to approximately count the number of log scales involved in the neck region. At the innermost log scale 
\[
\lambda\sim 0, \quad t\sim \sqrt{ \frac{3}{2}K_3}, \quad p_1\sim p_2\sim p_3\sim t^{2/3}\sim (\frac{3}{2}K_3  )^{1/3}, \quad A_{\text{max}}\sim \frac{3}{4}p_1^4\sim \frac{3}{4}(\frac{3}{2}K_3  )^{4/3},
\]
and at the outermost log scale
\[
p_1, p_2, p_3\sim 1, \quad A\sim 1,  \quad  t\sim 1, \quad \lambda\sim K_3.
\]
The total \textbf{number of log scales} is roughly $K_3\sim \frac{2}{3} (\frac{4A_{\text{max}}}{3}  )^{3/4}$. The \textbf{diameter of the neck region} is of the order
\[
\int_0^{K_3} p_1(\lambda) e^\lambda d\lambda \sim e^{K_3} \sim \exp(\frac{2}{3} (\frac{4A_{\text{max}}}{3}  )^{3/4}  ).
\]

\chapter{The Negative Vertex}\label{NegativervertexChapter}

In this Chapter we will construct a family of incomplete Calabi-Yau metrics describing the negative vertex, which we advocate as an analogue of the Ooguri-Vafa metric in complex dimension 3. These metrics have $S^1$-symmetry, inducing an 
$S^1$-fibration over an open subset inside $(\C^*)^2_{z_1, z_2}\times \R_\mu$, branched along the real codimension 3 discriminant locus $S=\{   z_1+z_2=1 \}\times \{ 0 \}$. Suitably away from $S$ the metric is approximately a flat $S^1$-bundle over a Euclidean region with coordinates $\log z_1, \log z_2, \mu$. Transverse to $S$ the metric is modelled on a fibration by Taub-NUT metrics. The topological description of the total space agrees with the predictions in Section \ref{Joycecritique}, and the holomorphic structures agree with the Zharkov picture (\cf review Section \ref{Degeneratingtorichypersurface}).

The Ooguri-Vafa type metric on the negative vertex is constructed in the generalised Gibbons-Hawking framework by perturbing from the periodic constant solution after incorporating topology. The 5-dimensional base has two periodic directions, which give rise to  exponential decay of higher Fourier modes, so that the metric looks semiflat at large distance from $S$.

The organization is as follows. Section \ref{FirstorderapproximationNegativevertex}, \ref{Negativevertexasymptote1}, \ref{Negativevertexasymptote2} introduce the first order ansatz and extract its leading order asymptote away from $S$ using Fourier analysis. Section \ref{StructurenearDeltanegativevertex}, \ref{StructurenearDeltanegativevertexII} extract the asymptote near $S$ and interpret this metrically in terms of Taub-NUT metrics transverse to $S$; a recurrent subtlety is the absence of an a priori given smooth structure along $S$. Section \ref{ComplexgeometricperspectiveNegativevertex} identifies the holomorphic structure explicitly by constructing holomorphic differentials. These Sections are written with an overall geometric orientation, with a flavour resembling classical complex analysis. The main difficulties here involve handling series sums akin to the Weierstrass function, and extracting  finite limits out of delicate divergent integrals.

Section \ref{WeightedHoldernormsandinitialerrorsNegativevertex} measures the volume form error and the metric deviation error, using weighted H\"older type norms with low regularity. Section \ref{HarmonicanalysisInegativevertex} improves the approximation in the generic region by working in the generalised Gibbons-Hawking framework, while Section \ref{HarmonicanalysisIInegativevertex}, 
\ref{HarmonicanalysisIII} solve the complex Monge-Amp\`ere equation perturbatively by constructing a parametrix for the right inverse to the Laplacian. The main idea is decomposition and patching as in the counterparts of Chapter \ref{TaubNUTtypemetriconC3} and \ref{Positivevertexchapter}, and the principal new difficulty is to construct a parametrix near the curved discriminant locus $S$ (\cf Section \ref{HarmonicanalysisIInegativevertex}).  Section \ref{OoguriVafatypemetriconNegativevertecSection} summarize up a number of salient features, notably the a posteriori emergence of a smooth topology.

The last two Sections are more informal in style. The purpose of Section \ref{SpecialLagrangiangeometry} is to speculate on the special Lagrangian torus fibration on the Ooguri-Vafa type metric, and explain the intimate relation to Joyce's work on $U(1)$-invariant special Lagrangians in $\C^3$. A curious feature is that not every Ooguri-Vafa type metric within the parameter space will admit special Lagrangian tori; a homological constraint is required. Section \ref{runningcouplingnegativevertex} observes that the renormalisation flow equations controlling large scale behaviours on the positive and the negative vertex are formally identical, and then explains this in terms of semiflat mirror symmetry.


\section{First order approximate metric}\label{FirstorderapproximationNegativevertex}

We plan to construct an approximate Calabi-Yau metric using the \textbf{generalised Gibbons-Hawking ansatz}, on a singular $S^1$-bundle $M^-$ over an open neighbourhood of the origin inside the real 5-dimensional base $\C^*_{z_1}\times \C^*_{z_2}\times \R_\mu$, whose discriminant locus is 
\begin{equation*}
S=\{ z_1+z_2=1     \}\subset \C^*_{z_1}\times \C^*_{z_2}\times\{0\}\subset 
\C^*_{z_1}\times \C^*_{z_2}\times
 \R_\mu.
\end{equation*}
Let 
\[
\eta_1= \frac{1}{2\pi \sqrt{-1}}\log z_1, \quad \eta_2= \frac{1}{2\pi \sqrt{-1}}\log z_2
\] be  complex variables with period 1, and denote $\eta_p=x_p+\sqrt{-1} y_p$ for $p=1,2$. The topological situation is described in Section \ref{Joycecritique} and the expected complex structure is discussed in Section \ref{Degeneratingtorichypersurface}. A more historical view can be found in Section \ref{Negativevertices}.

The basic heuristic idea is again to \textbf{perturb the constant solution} (\cf Example \ref{Constantsolution}) while incorporating the \textbf{topology}. The information of the constant solution is contained in the base metric
\begin{equation}
g_a= \text{Re}( a_{p\bar{q} } d\eta_p \otimes d\bar{\eta}_q) +A |d\mu|^2,
\end{equation}
where $(a_{p \bar{q}})$ is a \textbf{Hermitian} $2\times 2$ matrix referred to as \emph{coupling constants}, with determinant $A= \det a$, and $(a^{p\bar{q}})$ is the transposed inverse matrix such that $a_{j\bar{q}}a^{p\bar{q}}=\delta_j^p$.
The associated volume measure is
\[
d\text{Vol}_a= A^{3/2} dx_1 \wedge dy_1\wedge dx_2
\wedge dy_2 \wedge d\mu.
\]
In order for the perturbative way of thinking to be effective, we impose
\begin{equation}
C^{-1} A^{1/2}\delta_{p\bar{q}} \leq a_{p\bar{q}} \leq CA^{1/2}\delta_{p\bar{q}}
, \quad A\gg 1.
\end{equation}
In this Chapter all constants in estimates depend on $a_{p\bar{q}}$ only through the above  scale-invariant uniform ellipticity constant.

\begin{notation}
The $g_a$-distance to the origin is $|(\eta_1, \eta_2, \mu)|_a= \sqrt{ a_{p\bar{q}} \eta_p \bar{\eta}_q+ A\mu^2}$. A variant
\[
\varrho=|(y_1, y_2, \mu)|_a'= ( \frac{A}{    \mathbb{A}} a_{p\bar{q}} y_py_q + A\mu^2     )^{1/2}, \quad \mathbb{A}= A+|\text{Im}(a_{1\bar{2}})|^2.
\]
stands for the distance function for the Euclidean metric $g_a'$ on $\R^2_{y_1,y_2}\times \R_y$
\begin{equation}\label{ga'negativevertex}
\begin{split}
g_a'
= \frac{A}{ \mathbb{A}   }( a_{1\bar{1}} dy_1^2+ 2\text{Re}(a_{1\bar{2}}) dy_1dy_2+ a_{2\bar{2}} dy_2^2  )+ A|d\mu|^2.
\end{split}
\end{equation}
Let $S$ is $R= \text{dist}_{g_a}(\cdot, S)$. The parameter $R+A^{-1/2}$ is relevant for regularity scales.
\end{notation}

Now in terms of the local potential $\Phi$ the Calabi-Yau condition (\ref{GibbonsHawkingCY}) reads
\[
\det ( -4 \frac{\partial^2 \Phi}{\partial \eta_p \partial \bar{\eta}_q}     )= \frac{\partial^2 \Phi}{\partial \mu \partial \mu},
\]
whose linearised equation at the constant solution is the Laplace equation
\[
\Lap_a \phi= A^{-1} \frac{\partial^2 \phi}{\partial \mu \partial \mu} + 4 a^{p\bar{q}} \frac{\partial^2 \phi}{\partial \eta_p \partial \bar{\eta}_q}=0.
\]
Here $\Lap_a$ is unsurprisingly the Laplacian of $g_a$. This suggests that at least away from the discriminant locus, the first order correction to $V$ and $W^{p\bar{q}}$ from the constant solution
\begin{equation}\label{Negativevertexlinearised1}
v= \frac{\partial^2 \phi}{\partial \mu \partial \mu}, \quad w^{p\bar{q}}= -4  \frac{\partial^2 \phi}{\partial \eta_p \partial \bar{\eta}_q}
\end{equation}
ought to be given by $\Lap_a$-harmonic functions,
\begin{equation}\label{Negativevertexlinearised2}
\Lap_a w^{p\bar{q}}=0, \quad \Lap_a v=0, \quad v= Aa^{p\bar{q}  } w^{p\bar{q}}.
\end{equation}
To incorporate the topology we recall the distributional equation (\ref{distributioanlequationnegativevertex}). Since $v$ and $w^{p\bar{q}}$ are linearisations, it makes sense to require the equation on currents
\begin{equation}\label{Negativevertexdistributionalequationlinearised}
 -\frac{\sqrt{-1}}{4\pi} \left(  
\frac{\partial^2 w^{p\bar{q}}}{\partial \mu \partial \mu }
+
4\frac{\partial^2 v}{\partial \eta_p \partial \bar{\eta}_q }
\right) d\mu \wedge d\eta_p \wedge d\bar{\eta}_q=S.
\end{equation}
The task is to find a compatible solution to (\ref{Negativevertexlinearised1})(\ref{Negativevertexlinearised2})(\ref{Negativevertexdistributionalequationlinearised}). As in the last two Chapters, the functions $v$ and $w^{p\bar{q}}$ are global quantities while $\phi$ is only locally defined. The existence of the local potential $\phi$ in (\ref{Negativevertexlinearised1}) should be read as imposing some integrability on $v$ and $w^{p\bar{q}}$ (\cf (\ref{GibbonsHawkingintegrability1})(\ref{GibbonsHawkingintegrability2})).

\begin{rmk}
(Motivational Discussion on \textbf{singularities}) We denote \[
f_S=1-z_1-z_2=1- e^{2\pi i\eta_1}-e^{2\pi i \eta_2}
\] and write the 3-current $S$ as 
\[
S= \delta(f_S) \frac{\sqrt{-1}}{4\pi^2} df_S \wedge d\bar{f}_S \wedge d\mu, 
\]
which defines a generalised function $\delta(f_S)$ satisfying the measures identities:
\[
\int \frac{\sqrt{-1}}{4\pi^2}d\eta_p \wedge d\bar{\eta}_q \wedge \delta(f_S) \wedge df_S\wedge d\bar{f}_S \wedge d\mu= \int_S  d\eta_p\wedge d\bar{\eta}_q,
\]
where the notation $\int_S  d\eta_p\wedge d\bar{\eta}_q$ is the shorthand for the complex measure $f\mapsto \int_S fd\eta_p\wedge d\bar{\eta}_q$, and similarly for the LHS. Now $df_S= -2\pi \sqrt{-1} (z_1 d\eta_1+ z_2d\eta_2  )$, so 
\[
\begin{cases}
-  \int_S |z_2|^2 \delta(f_S)
 d\eta_1 \wedge d\bar{\eta}_1 \wedge d\eta_2 \wedge d\bar{\eta}_2\wedge d\mu  = \int_S \sqrt{-1} d\eta_1\wedge d\bar{\eta}_1 , \\
-  \int_S |z_1|^2 \delta(f_S) d\eta_1 \wedge d\bar{\eta}_1 \wedge d\eta_2 \wedge d\bar{\eta}_2 \wedge d\mu = \int_S \sqrt{-1} d\eta_2\wedge d\bar{\eta}_2  ,\\
  \int_S \bar{z}_1 z_2 \delta(f_S) d\eta_1 \wedge d\bar{\eta}_1 \wedge d\eta_2 \wedge d\bar{\eta}_2 \wedge d\mu  = \int_S \sqrt{-1} d\eta_1\wedge d\bar{\eta}_2 , \\
 \int_S \bar{z}_2 z_1 \delta(f_S) d\eta_1 \wedge d\bar{\eta}_1 \wedge d\eta_2 \wedge d\bar{\eta}_2  \wedge d\mu= \int_S \sqrt{-1} d\eta_2\wedge d\bar{\eta}_1  .\\
\end{cases}
\]
The distributional equation (\ref{Negativevertexdistributionalequationlinearised}) is written in components as
\[
 -\frac{1}{4\pi} \left(  
 \frac{\partial^2 w^{p\bar{q}}}{\partial \mu \partial \mu }
 +
 4\frac{\partial^2 v}{\partial \eta_p \partial \bar{\eta}_q }
 \right)= 
 \delta (f_S) z_p \bar{z}_q
 .
\]
Multiplying these equations by $a^{p\bar{q}}$ and summing up, we obtain
\begin{equation}\label{Laplaceequationforv}
- \frac{1}{4\pi} \Lap_a v= \delta(f_S) a^{p \bar{q} }  z_p \bar{z}_q, 
\end{equation}
or equivalently the measure equality
\[
 (\Lap_a v) d\text{Vol}_a   = -\int_S \pi\sqrt{-1}A^{1/2} a_{p\bar{q}} d\eta_p \wedge d\bar{\eta}_q= -\int_S 2\pi A^{1/2} d\mathcal{A},
\]
where $d\mathcal{A}= \frac{\sqrt{-1}}{ 2  } a_{p\bar{q}} d\eta_p \wedge d\bar{\eta}_q$ is the natural area form on $S$.
A natural guess for  $w^{p\bar{q}}$ is then
\begin{equation}\label{Laplaceequationforw}
- \frac{1}{4\pi} \Lap_a w^{p\bar{q}}= A^{-1} \delta(f_S)   z_p \bar{z}_q,
\end{equation}
or equivalently
\[
(\Lap_a w^{p\bar{q}}) d\text{Vol}_a   = -\int_S 2\pi   \frac{ A^{-1/2} z_p \bar{z}_q} {    a^{i\bar{j}} z_i \bar{z}_j    } d\mathcal{A},
\]
where summation convention is used.
The singularity around $S$ to leading order looks like (\cf Section \ref{StructurenearDeltanegativevertex} below)
\[
v\sim \frac{A^{1/2}} {2R}, \quad w^{p\bar{q}}\sim   \frac{ A^{-1/2} z_p \bar{z}_q} {2 R    a^{i\bar{j}} z_i \bar{z}_j    },
\quad R\sim ( \frac{|f_S|^2}{ 4\pi^2  a^{i\bar{j}} z_i \bar{z}_j      }  + A\mu^2 )^{1/2}
,
\]
which is compatible with the singularity in the distributional equation (\ref{Negativevertexdistributionalequationlinearised}).

\end{rmk}

Now we move on to a more formal construction. The main idea is to write down the solution via a \textbf{periodic} version of  \textbf{Green's representation}.
The series
\begin{equation}\label{gammadefinition}
\gamma(\eta_1, \eta_2, \mu)= -\frac{1}{8\pi^2}\sum_{(n_1, n_2)\in \Z^2} \frac{1}{ |(\eta_1+n_1, \eta_2+n_2, \mu)  |_a^{3}  },
\end{equation}
converges absolutely away from $\Z^2\times \{0\}\subset \C_{\eta_1}\times \C_{\eta_2}\times \R_\mu$ and is $\Z^2$-periodic, so descends to a function on $\C^*_{z_1}\times \C^*_{z_2}\times \R_\mu$ which is the \textbf{periodic Newtonian potential}. We shall extract the \textbf{asymptote} for $\gamma(\eta_1, \eta_2, \mu)$:

\begin{lem}\label{gammaasymptote1}
For $\varrho \gtrsim A^{1/4}$, we have
\[
|\gamma(\eta_1, \eta_2, \mu)+ \frac{ 1}{4\pi \varrho \sqrt{ \mathbb{A}}    }  | \leq  C\varrho^{-3}   .
\]
\end{lem}

\begin{proof}
We consider the closely related integral
\[
\bar{\gamma}(y_1, y_2, \mu)=- \frac{1}{8\pi^2} \int \frac{1}{    (a_{p\bar{q} }  \eta_p \bar{\eta}_q +A\mu^2 )^{3/2}       } dx_1 dx_2.
\]
After substituting the variables
\[
\begin{cases}
x_1'= x_1 - \frac{ \text{Im} (a_{2\bar{1}}) }{  a_{1\bar{1}}  } y_2+ \frac{ \text{Re}( a_{2\bar{1}}) }{  a_{1\bar{1}}  } x_2, 
\\
x_2'= x_2+ \frac{ \text{Im} (a_{2\bar{1}}) \text{Re} (a_{2\bar{1}})   }{  \mathbb{A}  } y_2+ \frac{ a_{1\bar{1}} \text{Im}( a_{2\bar{1}}) }{   \mathbb{A}  } y_1,
\end{cases}
\]
we complete the square
\[
a_{p\bar{q} }  \eta_p \bar{\eta}_q +A\mu^2= a_{1\bar{1}} x_1'^2 + \frac{\mathbb{A}}{a_{1\bar{1}}} x_2'^2 + |(y_1, y_2, \mu)|_a'^2= a_{1\bar{1}} x_1'^2 + \frac{\mathbb{A}}{a_{1\bar{1}}} x_2'^2 + \varrho^2  .
\]
This allows us to evaluate using polar coordinates
\[
\bar{\gamma}=  - \frac{ 1}{4\pi\varrho \sqrt{ \mathbb{A}}    }  .
\]

For fixed $y_1, y_2, \mu$, 
we can compare the integral $\bar{\gamma}$ with the series $\gamma$, by estimating the difference using the mean value inequality
\[
\begin{split}
& \frac{1}{    |(\eta_1, \eta_2, \mu)|_a^3       }
-
\int_{[ x_1-\frac{1}{2}, x_1+\frac{1}{2} ] \times [x_2-\frac{1}{2}, x_2+\frac{1}{2} ]    } \frac{1}{    |(s_1+\sqrt{-1}y_1, s_2+\sqrt{-1} y_2, \mu)|_ a^3      } ds_1 ds_2
\\
\leq
& \frac{C A^{1/2}}{  |(\eta_1, \eta_2, \mu)|_a^5   }.
\end{split}
\]
Summing over all square regions, and applying Cauchy integral test,
\[
\begin{split}
|\gamma- \bar{\gamma}|& \leq CA^{1/2} \sum_{n,m} \frac{1}{  |(\eta_1+n, \eta_2+m, \mu)|_a^5   } \\
&\leq   CA^{1/2} \int \frac{1}{  |(s_1+\sqrt{-1}y_1, s_2+\sqrt{-1}y_2, \mu)|_a^5   }ds_1ds_2
\\
&\leq   C  \varrho^{-3}   ,
\end{split}
\]
as required.
\end{proof}

\begin{lem}\label{gammaasymptote2}
For $\varrho \lesssim A^{1/4}$ and $|x_1|, |x_2|\leq \frac{1}{2}$ we have
\[
|\gamma(\eta_1, \eta_2, \mu)+ \frac{ 1}{8\pi^2|(\eta_1, \eta_2, \mu)|_a^3 }| \leq CA^{-3/4}.
\]
\end{lem}

\begin{proof}
Modify the above proof to control the series for $(n_1,n_2)\in \Z^2\setminus \{0\}$.
\end{proof}

Before proceeding further we recall that topologically $S$ is a thrice punctured 2-sphere. The 3 punctures correspond to 3 \textbf{ends} of $S$:
\[
\begin{cases}
y_2>1, \quad \eta_1= \frac{1}{2\pi i} \log ( 1- e^{2\pi i \eta_2}  ),
\\
y_1>1, \quad \eta_2= \frac{1}{2\pi i} \log ( 1- e^{2\pi i \eta_1}  ),
\\
y_1<-1, \quad y_2<-1,  \quad \eta_2 -\eta_1 = \frac{1}{2\pi i}\log (-1+ e^{ -2\pi i \eta_1  } ).
\end{cases}
\]
At infinity these are respectively asymptotic to $\mathfrak{D}_1\times S^1$, $\mathfrak{D}_2\times S^1$, $\mathfrak{D}_3\times S^1$ where
\[
\mathfrak{D}_1= \{  y_1=0, y_2>0 ,\mu=0      \}, \quad \mathfrak{D}_2= \{  y_2=0, y_1>0  , \mu=0        \}, \quad \mathfrak{D}_3= \{  y_1= y_2<0   , \mu=0       \}.
\]
The image of $S$ under the log map $\C^2_{z_1, z_2}\to \R^2_{y_1, y_2}$ (called the `\textbf{amoeba}') is \begin{equation}
\text{Image}(S)=
\{  e^{-2\pi y_1}+ e^{-2\pi y_2} \geq 1,  e^{-2\pi y_1} +1\geq e^{-2\pi y_2},  e^{-2\pi y_2}+ 1\geq e^{-2\pi y_1} \},
\end{equation}
which is a thickening of the trivalent graph
$
\mathfrak{D}= \mathfrak{D}_1\cup \mathfrak{D}_2\cup \mathfrak{D}_3\cup \{0 \}.
$
This is the simplest case of a general picture for amoebas of algebraic varieties \cite{Zharkov1}.

We can now make the following definitions, involving a cutoff and limiting procedure for logarithmically divergent integrals.
\begin{equation}\label{gammai}
\begin{cases}
\gamma_1(\eta_1, \eta_2,\mu)=&  -\text{Re}  \lim_{\Lambda\to \infty} \{ \pi A^{1/2}\int_{ S\cap \{  y_2'<\Lambda \} } \gamma(\eta_1-\eta_1', \eta_2-\eta_2', \mu  ) 
\\
& \sqrt{-1} d\eta_2' \wedge (d\bar{\eta}_2'- d\bar{\eta}_1')+\frac{1}{ 2\sqrt{ a_{2\bar{2}} } } \log 2\Lambda  \}
\\
\gamma_2(\eta_1, \eta_2,\mu)=&  -\text{Re}  \lim_{\Lambda\to \infty} \{ \pi A^{1/2}\int_{ S\cap \{  y_1'<\Lambda \} } \gamma(\eta_1-\eta_1', \eta_2-\eta_2', \mu  ) 
\\
& \sqrt{-1} d\eta_1' \wedge (d\bar{\eta}_1'- d\bar{\eta}_2') +\frac{1}{ 2\sqrt{ a_{1\bar{1}} } } \log 2\Lambda  \}
\\
\gamma_3(\eta_1, \eta_2,\mu)=&  -\text{Re}  \lim_{\Lambda\to \infty} \{ \pi A^{1/2}\int_{ S\cap \{  y_1'>-\Lambda \} } \gamma(\eta_1-\eta_1', \eta_2-\eta_2', \mu  ) 
\\
& \sqrt{-1} d\eta_2' \wedge d\bar{\eta}_1'  +\frac{1}{ 2\sqrt{  a_{1\bar{1}}+  a_{1\bar{2}}+ a_{2\bar{1}} +a_{2\bar{2}}  } } \log 2\Lambda   \}
\\
\gamma_4(\eta_1, \eta_2,\mu)=&  -\text{Im}   \{   \pi A^{1/2}\int_{ S } \gamma(\eta_1-\eta_1', \eta_2-\eta_2', \mu ) 
 \sqrt{-1} d\eta_2' \wedge d\bar{\eta}_1'    \}.
\end{cases}
\end{equation}
The desired \textbf{first order corrections} $v$ and $w^{p\bar{q}}$ are constructed as linear combinations:
\begin{equation}
\begin{cases}
v= Aa^{p\bar{q}} w^{p\bar{q}}   ,    \\
w^{1\bar{1}}=  \gamma_1+ \gamma_3, \\
w^{1\bar{2}}=  - (\gamma_3+ \sqrt{-1} \gamma_4 )  ,  \\
w^{2\bar{1}}=   -(\gamma_3- \sqrt{-1} \gamma_4 ) ,   \\
w^{2\bar{2}}=   \gamma_2+ \gamma_3.
\end{cases}
\end{equation}
The advantage of $\gamma_i$ is that they only involve divergence issues at one end. This is because the measures $\text{Re}( \sqrt{-1} d\eta_2' \wedge (d\bar{\eta}_2'- d\bar{\eta}_1'))$ etc decay exponentially along all but one end, with respect to the Lebesgue measure on the three asymptotic cylinders.

\begin{lem}
The limits defining $\gamma_i$ converge as $\Lambda\to +\infty$.
\end{lem}

\begin{proof}
We focus on $\gamma_1$. The 2-form $\sqrt{-1} d\eta_2 \wedge (d\bar{\eta}_2- d\bar{\eta}_1)$ on $S$  is exponentially small along the $\mathfrak{D}_2, \mathfrak{D}_3$ ends, so the only divergence problem happens at infinity along the $\mathfrak{D}_1$ end.

Applying Lemma \ref{gammaasymptote1} allows us to replace $\gamma$ by the much simpler function
\[
-\frac{ 1}{4\pi \sqrt{ \mathbb{A}}    }  |(y_1- y_1', y_2-y_2', \mu)|_a'^{-1}.
\]
The integral
\[\int_{ S\cap \{  y_2'<\Lambda \} } \gamma(\eta_1-\eta_1', \eta_2-\eta_2', \mu  ) 
 \sqrt{-1} d\eta_2' \wedge (d\bar{\eta}_2'- d\bar{\eta}_1')\]
has the same divergence behaviour as
\[
\int^\Lambda 
-\frac{ 1}{ 2\pi\sqrt{ \mathbb{A}}    }  |(y_1, y_2-y_2', \mu)|_a'^{-1} 
 dy_2'\sim -\frac{1}{ 2\pi\sqrt{A a_{2\bar{2}} } } \log \Lambda,
\]
which is cancelled by the log term we put in the limit.
\end{proof}

By the construction of the Green representations,

\begin{prop}\label{NegativevertexLaplace}
The functions $v$ and $w^{p\bar{q}}$ satisfy the decoupled \textbf{Laplace equations} with distributional terms (\ref{Laplaceequationforv}) and (\ref{Laplaceequationforw}).
\end{prop}

However the original linearised equations we set off to solve is an overdetermined coupled system, not just the decoupled Laplace equations. We still need to check the integrability equation (\ref{Negativevertexlinearised1}) and the distributional equation (\ref{Negativevertexdistributionalequationlinearised}).

\begin{lem}\label{Negativevertexintegrability1}
The following \textbf{integrability condition} is satisfied globally
\[
\frac{\partial w^{p\bar{q}}}{\partial \eta_r}= \frac{\partial w^{rq}}{\partial \eta_p} , \quad \frac{\partial w^{p\bar{q}}}{\partial \bar{\eta}_r}= \frac{\partial w^{pr}}{\partial \bar{\eta}_q}, \quad p, q, r=1,2.
\]
\end{lem}

\begin{proof}
We consider the Laplacian
\[
\begin{split}
\Lap_a (  \frac{\partial w^{p\bar{q}}}{\partial \eta_r}- \frac{\partial w^{rq}}{\partial \eta_p}   )= & \frac{\partial }{\partial \eta_r} \Lap_a w^{p\bar{q}}- \frac{\partial }{\partial \eta_p} \Lap_a w^{rq} 
\\
= &  -4\pi A^{-1} \{ \frac{\partial }{\partial \eta_r} (\delta(f_S) z_p \bar{z}_q )- \frac{\partial }{\partial \eta_p}  (\delta(f_S) z_r \bar{z}_q )   \}
\\
= & -4\pi A^{-1}\delta(f_S) \{ \frac{\partial }{\partial \eta_r} ( z_p \bar{z}_q )- \frac{\partial }{\partial \eta_p}  ( z_r \bar{z}_q )   \}=0,
\end{split}
\]
where we have crucially used that $S$ is an algebraic cycle to deduce $\partial \delta(f_S)=0$. Thus 
$\frac{\partial w^{p\bar{q}}}{\partial \eta_r}- \frac{\partial w^{rq}}{\partial \eta_p}=0$ would follow from a Liouville theorem argument, by checking some a priori growth condition
\[
|\frac{\partial w^{p\bar{q}}}{\partial \eta_r}| \lesssim \begin{cases}   R^{-1}  
,  \quad   &\varrho\gtrsim A^{1/4},
\\
A^{1/4}R^{-2}  
,  \quad   &\varrho\lesssim A^{1/4},
\end{cases} 
\]
which is easy to derive using the techniques in the previous lemmas in this Section. The $\bar{\eta}$ derivatives can be treated similarly.
\end{proof}

\begin{cor}\label{Negativevertexintegrability2}
The \textbf{distributional equation} (\ref{Negativevertexdistributionalequationlinearised}) is satisfied. In component form,
\[
-\frac{1}{4\pi}  (\frac{\partial^2 w^{p\bar{q}}}{\partial \mu\partial \mu}+ 4 \frac{\partial^2 v}{\partial \eta_p \partial \bar{\eta}_q  })=\delta(f_S) z_p \bar{z}_q.
\]
\end{cor}

\begin{proof}
Let's focus on $p=q=1$. By Proposition \ref{NegativevertexLaplace},
\[
\frac{1}{4\pi} \frac{\partial^2 w^{1\bar{1}}}{\partial \mu\partial \mu}+ \frac{1}{\pi}  Aa^{i\bar{j}} \frac{\partial^2 w^{1\bar{1}}}{\partial \eta_i \partial \bar{\eta}_j}= - \delta(f_S) |z_1|^2.
\]
But by Lemma \ref{Negativevertexintegrability2} we have $
\frac{\partial^2 w^{i\bar{j}}}{\partial \eta_1 \partial \bar{\eta}_1}= \frac{\partial^2 w^{1\bar{1}}}{\partial \eta_i \partial \bar{\eta}_j},
$ so $
Aa^{i\bar{j}} \frac{\partial^2 w^{1\bar{1}}}{\partial \eta_i \partial \bar{\eta}_j}=\frac{\partial^2 v}{\partial \eta_1 \partial \bar{\eta}_1  },
$ hence the claim.
\end{proof}

Lemma \ref{Negativevertexintegrability1} and Corollary \ref{Negativevertexintegrability2} combine to imply the local existence of the potential away from $S$ as is required in (\ref{Negativevertexlinearised1}). Taking stock of our progress,

\begin{prop}\label{Negativevertexfirstordersolution}
(\textbf{First order linearised solution}) The functions $v$ and $w^{p\bar{q}}$ solve the integrability condition (\ref{Negativevertexlinearised1}) and the harmonicity condition (\ref{Negativevertexlinearised2}) away from $S$, and the distributional equation (\ref{Negativevertexdistributionalequationlinearised}) globally.
\end{prop}

\begin{rmk}\label{Negativevertexuniqueness}
It will turn out in the next few Sections that $v$ and $w^{p\bar{q}}$ have logarithmic growth at infinity bounded away from $S$. If we restrict to solutions to (\ref{Negativevertexlinearised1})(\ref{Negativevertexlinearised2})(\ref{Negativevertexdistributionalequationlinearised}) with the same growth properties, then $v$ and $w^{p\bar{q}}$ are unique up to additive constants. The choices of these constants are not completely canonical, related to the philosophy that the Ooguri-Vafa type metrics are only effective descriptions admitting a certain amount of small fluctuation.
\end{rmk}

We obtain by the generalised Gibbons-Hawking construction a \textbf{K\"ahler ansatz} $(g^{(1)}, \omega^{(1)}, J^{(1)}, \Omega^{(1)}  )$ associated to 
\begin{equation}
V_{(1)}= A+ v, \quad W^{p\bar{q}}_{(1)}= a_{p\bar{q}} + w^{p\bar{q}}.
\end{equation}
A subtlety here is that the $S^1$-connection $\vartheta$ can be twisted by a \textbf{flat connection}. This choice is parametrised by $H^1( (\C^*)^2\times \R\setminus S, \R/\Z )= H^1( (\C^*)^2\times \R, \R/\Z )=T^2$, since the codimension 3 subset $S$ inside the base does not affect the fundamental group. We sometimes suppress mentioning this choice since it does not have a strong impact on the geometry, especially because we will exclusively work with $S^1$-invariant tensors, which are rarely sensitive to the flat connection. The K\"ahler structure is well defined away from $S$, over a bounded region where $V_{(1)}>0$ and $W^{p\bar{q}}_{(1)}$ is positive definite; the metric is incomplete. We will specify more precisely the ambient space $M^-$ of the K\"ahler ansatz once we obtain sufficiently accurate asymptotes on $V_{(1)}$ and $W^{p\bar{q}}_{(1)}$ to check positive definiteness (\cf Corollary \ref{M-definition}).

The family of ansatzs admit $S_3$-\textbf{discrete symmetries}, generated by
\[
\eta_1\leftrightarrow \eta_2, \quad \eta_2\leftrightarrow \eta_2-\eta_1+\frac{1}{2}, \quad \eta_1\leftrightarrow \eta_1-\eta_2+\frac{1}{2}, \quad \mu\to \mu.
\]
These actions on $(\C^*)^2\times \R_\mu$ preserve $S$, and respectively interchange $\mathfrak{D}_1$ with $\mathfrak{D}_2$,  $\mathfrak{D}_1$ with $\mathfrak{D}_3$, and $\mathfrak{D}_2$ with $\mathfrak{D}_3$. The induced action on coupling constants permute $a_{1\bar{1}}, a_{2\bar{2}}, a_{1\bar{1}}+ a_{1\bar{2}}+ a_{2\bar{1}}+ a_{2\bar{2}}$, and act on $\text{Im}(a_{1\bar{2}})$ by $\pm 1$ depending on the sign of the permutation.

\section{Asymptotic for the first order ansatz I}\label{Negativevertexasymptote1}

The following two Sections study the leading order behaviour of the K\"ahler ansatz away from $S$ at large distance.
The region under consideration lies over
\begin{equation}\label{regionawayfromS}
\begin{split}
\{ y_2> 1, A^{1/4}|\mu|+   |y_1|> e^{-2\pi y_2}   \}\cup \{  y_1>1, A^{1/4}|\mu|+|y_2|> e^{-2\pi y_1}       \}
\\
\cup \{  y_1<-1, A^{1/4}|\mu|+ |y_2-y_1|> e^{ 2\pi y_1  }   \} \cup \{   |\mu|>A^{-1/4} \} .
\end{split}
\end{equation}
This is a quantitative way of asserting boundedness away from $S$.

We define the \textbf{average functions} of $\gamma_i$ (\cf (\ref{gammai})) on $\R^2_{y_1, y_2}\times \R_\mu$ by
\begin{equation}
\bar{\gamma}_i(y_1, y_2, \mu  )= \int_0^1\int_0^1 \gamma_i(x_1+\sqrt{-1} y_1, x_2+\sqrt{-1}y_2, \mu) dx_1 dx_2.
\end{equation}
This Section is concerned with describing the behaviour of $\bar{\gamma}_i$, and next Section proves exponential decay estimate for $|\gamma_i-\bar{\gamma}_i|$.

Recall from (\ref{ga'negativevertex}) the {Euclidean metric} $g_a'$ on $\R^2_{y_1, y_2}\times \R_\mu$. Its volume measure is 
$
d\text{Vol}_a'= \frac{A^{3/2}} { \sqrt{ \mathbb{A} } } dy_1dy_2 d\mu,
$
and the associated Laplacian is $\Lap_a'=  a^{p\bar{q}} \frac{\partial^2}{\partial y_p \partial y_q}+ A^{-1}  \frac{\partial^2}{\partial \mu \partial \mu}  $. Here some care is needed in the calculations regarding the difference between Hermitian and symmetric matrices.

\begin{lem}(\textbf{Harmonicity})
In the region (\ref{regionawayfromS}) inside $\R^2_{y_1, y_2}\times \R_\mu$ the functions $\bar{\gamma}_i$ satisfy $\Lap_a' \bar{\gamma}_i=0$, or equivalently their pullbacks to $(\C^*)^2_{\eta_1, \eta_2}\times \R_\mu$ satisfy $\Lap_a \bar{\gamma}_i=0$.
\end{lem}

\begin{proof}
The amoeba $\text{Im}(S)$ is disjoint from the region (\ref{regionawayfromS}), so Proposition \ref{NegativevertexLaplace} asserts the $\Lap_a$-harmonicity of $\gamma_i$, whence the harmonicty of $\bar{\gamma}_i$.
\end{proof}

Next we wish to write $\bar{\gamma}_i$ also in terms of a Green's representation. From the calculation in Lemma \ref{gammaasymptote1},
\begin{equation}\label{gammaibarharmonic}
\bar{\gamma}=\int_0^1 \int_0^1 \gamma(x_1+ \sqrt{-1}y_1, x_2+ \sqrt{-1}y_2, \mu  )dx_1 dx_2= - \frac{ 1}{4\pi \varrho\sqrt{ \mathbb{A}}    }  .
\end{equation}
Thus by integrating (\ref{gammai}) in the $x_1, x_2$ variables,

\begin{cor}(\textbf{Green's representation formula for $\bar{\gamma}_i$})
	\[
\begin{cases}
	\bar{\gamma}_1(y_1, y_2,\mu)=&  \text{Re}  \lim_{\Lambda\to \infty} \{  \frac{  A^{1/2} }{4 \sqrt{ \mathbb{A}}    }  \int_{ S\cap \{  y_2'<\Lambda \} }  |(y_1-y_1', y_2-y_2', \mu)|_a'^{-1} 
	\\
	& \sqrt{-1} d\eta_2' \wedge (d\bar{\eta}_2'- d\bar{\eta}_1')-\frac{1}{ 2\sqrt{ a_{2\bar{2}} } } \log 2\Lambda  \}
	\\
	\bar{\gamma}_2(y_1, y_2,\mu)=&  \text{Re}  \lim_{\Lambda\to \infty} \{  \frac{  A^{1/2} }{4 \sqrt{ \mathbb{A}}    } \int_{ S\cap \{  y_1'<\Lambda \} }  |(y_1-y_1', y_2-y_2', \mu)|_a'^{-1} 
	\\
	& \sqrt{-1} d\eta_1' \wedge (d\bar{\eta}_1'- d\bar{\eta}_2') -\frac{1}{ 2\sqrt{ a_{1\bar{1}} } } \log 2\Lambda  \}
	\\
	\bar{\gamma}_3(y_1, y_2,\mu)=&  \text{Re}  \lim_{\Lambda\to \infty} \{  \frac{  A^{1/2} }{4 \sqrt{ \mathbb{A}}    } \int_{ S\cap \{  y_1'>-\Lambda \} }  |(y_1-y_1', y_2-y_2', \mu)|_a'^{-1}  
	\\
	& \sqrt{-1} d\eta_2' \wedge d\bar{\eta}_1'  -\frac{1}{ 2\sqrt{  a_{1\bar{1}}+  a_{1\bar{2}}+ a_{2\bar{1}} +a_{2\bar{2}}  } } \log 2\Lambda   \}
	\\
	\bar{\gamma}_4(y_1, y_2,\mu)=&  \text{Im}   \{   \frac{  A^{1/2} }{4 \sqrt{ \mathbb{A}}    }  \int_{ S }  |(y_1-y_1', y_2-y_2', \mu)|_a'^{-1} 
	\sqrt{-1} d\eta_2' \wedge d\bar{\eta}_1'    \}.
\end{cases}
\]
\end{cor}

Our goal is to extract the leading order behvaiour in terms of an explicit elementary formula. For this purpose we essentially replace $S$ by its asymptotic cylinders $\mathfrak{D}_i\times S^1$. Define

\[
\begin{cases}
\bar{\bar{\gamma}}_1(y_1, y_2,\mu)= &   \lim_{\Lambda\to \infty} \{  \frac{  A^{1/2} }{2 \sqrt{ \mathbb{A}}    }  \int_0^\Lambda    |(y_1, y_2-s, \mu)|_a'^{-1} ds
 -\frac{1}{ 2\sqrt{ a_{2\bar{2}} } } \log 2\Lambda  \}
\\
\bar{\bar{\gamma}}_2(y_1, y_2,\mu)= &  \lim_{\Lambda\to \infty} \{  \frac{  A^{1/2} }{2 \sqrt{ \mathbb{A}}    } \int_0^\Lambda  |(y_1-s, y_2, \mu)|_a'^{-1} ds 
 -\frac{1}{ 2\sqrt{ a_{1\bar{1}} } } \log 2\Lambda  \}
\\
\bar{\bar{\gamma}}_3(y_1, y_2,\mu)= &  \lim_{\Lambda\to \infty} \{  \frac{  A^{1/2} }{2 \sqrt{ \mathbb{A}}    } \int_{-\Lambda}^0  |(y_1-s, y_2-s, \mu)|_a'^{-1}  ds\\
& -\frac{1}{ 2\sqrt{  a_{1\bar{1}}+  a_{1\bar{2}}+ a_{2\bar{1}} +a_{2\bar{2}}  } } \log 2\Lambda   \}.
\end{cases}
\]
By construction $\bar{\bar{\gamma}}_i$ is $\Lap_a'$-harmonic in the region (\ref{regionawayfromS}). Morever,

\begin{lem}\label{gamma4bardecays}
(\textbf{Estimate of remainder terms})
In the region  (\ref{regionawayfromS}) we have
$|\bar{\gamma}_4 |\leq C \varrho^{-1}$. For $i=1,2, 3$, in the subset of (\ref{regionawayfromS}) where $\text{dist}_{g_a'}(\cdot, \mathfrak{D}_i)\gtrsim A^{1/4}$ we have $|\bar{\gamma}_i- \bar{\bar{\gamma}}_i |\leq C\varrho^{-1}$.
\end{lem}

\begin{proof}
We use the Green representation of $\bar{\gamma}_4$.
The total measure \[
\int_S  \text{Im} ( \sqrt{-1} d\eta_2' \wedge d\bar{\eta}_1'   ) \leq C,
\]
 so the contribution to $\bar{\gamma}_4$ from the ball $\{ |(\eta_1', \eta_2',0)|_a \lesssim A^{1/4}  \}\subset S$  is bounded by $C\varrho^{-1}$. The contributions from the 3 ends are neglegible unless the point  $(y_1, y_2, \mu)$ inside the region (\ref{regionawayfromS}) is  close to $S$ along some $\mathfrak{D}_i$; we focus on the case of $\mathfrak{D}_1$. The key fact is the exponential decay of the measure: along $\mathfrak{D}_1$ we have
 \[
 \text{Im} ( \sqrt{-1} d\eta_2' \wedge d\bar{\eta}_1'   )\leq C e^{-2\pi y_2'} \sqrt{-1} d\eta_2' \wedge d\bar{\eta}_2'.
 \]
 Thus the contribution from the end $\{y_2'>1\}\cap S$ is controlled by
 \[
 \begin{split}
 C \int_0^\infty e^{-2\pi y_2'}   |(y_1, y_2-y_2', \mu)|_a'^{-1}   dy_2'
 \leq  C \varrho^{-1}.
 \end{split}
 \]
 which implies the estimates on $\bar{\gamma}_4$.
 
 For $\bar{\gamma}_i- \bar{\bar{\gamma}}_i$, the main point is that $S$ approaches its asymptotic cylinder at an exponentially fast rate. The rest of the arguments are similar.
\end{proof}

Elementary integration gives

\begin{lem}(\textbf{Leading order asymptote})
The formulae for $\bar{\bar{\gamma}}_i$ are
given explicitly as
\begin{equation}\label{gammaibarbarformula}
\begin{cases}
\bar{\bar{\gamma}}_1=&-\frac{1}{ 2 \sqrt{a_{2\bar{2}} } } \log ( \frac{\mathbb{A}^{1/2}}{ \sqrt{ A a_{2\bar{2}}} } |(y_1, y_2, \mu)|_a'- y_2- \frac{ \text{Re}(a_{1\bar{2}} )}{  a_{2\bar{2} }}   y_1     )
\\
\bar{\bar{\gamma}}_2=&-\frac{1}{ 2 \sqrt{a_{1\bar{1}} } } \log ( \frac{\mathbb{A}^{1/2}}{ \sqrt{A a_{1\bar{1}}}} |(y_1, y_2, \mu)|_a'- y_1- \frac{ \text{Re}(a_{1\bar{2}} )}{  a_{1\bar{1} }}   y_2     )
\\
\bar{\bar{\gamma}}_3=&-\frac{1}{ 2 \sqrt{a_{1\bar{1}}+ a_{1\bar{2} }+ a_{2\bar{1} }+ a_{2\bar{2} }  } }\\ &\log ( \frac{\mathbb{A}^{1/2}}{ \sqrt{ A (a_{1\bar{1}}+ a_{1\bar{2} }+ a_{2\bar{1} }+ a_{2\bar{2} } ) }  } |(y_1, y_2, \mu)|_a'+ 
 \frac{ a_{1\bar{1 }} y_1+ \text{Re}( a_{1\bar{2 }}) y_2 +\text{Re}( a_{2\bar{1 }}) y_1 +  a_{2\bar{2 }} y_2    } {  a_{1\bar{1 }}+ 2\text{Re}( a_{1\bar{2 }}) +  a_{2\bar{2 }}     }
 ).
\end{cases}
\end{equation}
\end{lem}

\begin{rmk}
These formulae have strong similarity with $\bar{\alpha}_i$ in (\ref{alphabar})
except for the absence of an additive constant as in (\ref{alphabar}), which is an artefact of a non-canonical choice of constant in our definition of $\gamma_i$ (\cf Remark \ref{Negativevertexuniqueness}).
\end{rmk}

\section{Asymptotic for the first order ansatz II}\label{Negativevertexasymptote2}

This Section proves exponential decay estimate for higher Fourier modes $\gamma_i- \bar{\gamma}_i$ in the region bounded away from $S$. The main idea is that the Laplace equation together with the vanishing of the zeroth Fourier modes imply exponential decay through \textbf{Fourier analysis}; this discussion is parallel to Section \ref{PositivevertexasymptototeSection}.

\begin{lem}\label{Exponentialdecayboundednesslemmanegativevertex}
In the region defined by (\ref{regionawayfromS}) we have $|\gamma_4- \bar{\gamma}_4|\leq CA^{-1/4}$. For $i=1,2,3$,
in the subset of the region (\ref{regionawayfromS}) where $\text{dist}_{g_a'}( \cdot, \mathfrak{D}_i  )\gtrsim A^{1/4}$ we have $|\gamma_i- \bar{\gamma}_i|\leq CA^{-1/4}$.

\end{lem}

\begin{proof}
Consider $\gamma_4$ at a given point in the region (\ref{regionawayfromS}).
Its integral formula (\ref{gammai}) can be split into two parts, corresponding to far away sources $|(y_1-y_1', y_2-y_2', \mu)|_a'\gtrsim A^{1/4}$ and nearby sources $|(y_1-y_1', y_2-y_2', \mu)|_a'\lesssim A^{1/4}$.

For far away sources, we use Lemma \ref{gammaasymptote1} to write the integrand $\gamma$ as a dominant term $ -\frac{ 1}{4\pi \sqrt{ \mathbb{A}   } |(y_1-y_1', y_2-y_2', \mu)|_a'   }  $ plus a remainder term estimated by $\frac{C}{  |(y_1-y_1', y_2-y_2', \mu)|_a'^3 }$. The dominant term does not contribute to $\gamma_4-\bar{\gamma}_4$ because it is constant in the $x_1, x_2$ direction. The remainder term contribution to $\gamma_4$ is bounded by
\[
C A^{1/2} \text{Im} \int_{ S\cap \{ \text{dist}\gtrsim A^{1/4}  \} }  \frac{1}{  |(y_1-y_1', y_2-y_2', \mu)|_a'^3 }   \sqrt{-1}d\eta_2' \wedge d\bar{\eta}_1' \leq C A^{-1/4}.
\]

The contribution from nearby sources only arises if our given point of interest is too close to $S$ along one of $\mathfrak{D}_1$, $\mathfrak{D}_2$ or $\mathfrak{D}_3$ directions; we focus on $\mathfrak{D}_1$. Lemma \ref{gammaasymptote2} allows us to write $\gamma$ as $-\frac{ 1}{8\pi^2|(\eta_1-\eta_1', \eta_2-\eta_2', \mu)|_a^3}$ plus a well controlled remainder term. By the exponential decay property of the measure $\text{Im} \sqrt{-1}d\eta_2' \wedge d\bar{\eta}_1'$,
\[
\begin{split}
& CA^{1/2} e^{-2\pi y_2}  \int_{ S\cap \{ \text{dist}\lesssim A^{1/4}  \} }  
\frac{  \sqrt{-1}d\eta_1' \wedge d\bar{\eta}_1'  }{|(\eta_1-\eta_1', \eta_2-\eta_2', \mu)|_a^3}   
\leq  C e^{-2\pi y_2} R^{-1}  \leq CA^{-1/4}.
\end{split}
\]
Combining the above shows $|\gamma_4-\bar{\gamma}_4|\leq CA^{-1/4}$.

All these arguments carry through to $\gamma_1, \gamma_2, \gamma_3$ except the exponential decay of the measure. This is compensated by staying sufficiently far from $\mathfrak{D}_i$.	
\end{proof}

\begin{prop}\label{Exponentialdecayfor higherFouriermodesinthefirstorderansatzNegativevertex}
(\textbf{Exponential decay for higher Fourier modes in the first order ansatz})
In the region where $\text{dist}_{g_a' }(\cdot, \text{Im}(S))\gtrsim A^{1/4}$,
\begin{equation}
|\gamma_i-\bar{\gamma}_i|\leq CA^{-1/4} \tilde{\ell} e^{-\tilde{\ell}},\quad \tilde{\ell}=\kappa_a \text{dist}_{g_a'}(\cdot, \text{Im}(S) ) , \quad i=1,2,3,4,
\end{equation}
where $\kappa_a$ is the minimum of $k_{n_1, n_2}= 2\pi ( \frac{Aa^{p\bar{q}} n_p n_q  }{ \mathbb{A}  }  )^{1/2} $ for all $(n_1, n_2)\in \Z^2\setminus \{0\}$. 
\end{prop}

\begin{proof}
This proof is parallel to Proposition \ref{Exponentialdecayforthefirstorderansatz}, so will be sketchy. We  focus on $\gamma_4- \bar{\gamma}_4$ as the same arguments work for $\gamma_i-\bar{\gamma}_i$.


 We  perform Fourier decomposition in the periodic variables $x_1, x_2$,
\[
\gamma_4- \bar{\gamma}_4= \sum_{(n_1,n_2)\in \Z^2\setminus\{0\} } h_{n_1,n_2} (y_1, y_2, \mu) e^{2\pi in_1 x_1+ 2\pi i n_2 x_2}.
\]
The zeroth Fourier mode vanishes by construction. Parseval identity combined with Lemma \ref{Exponentialdecayboundednesslemmanegativevertex} shows
\[
\sum_{n_1,n_2} |h_{n_1,n_2}|^2= \int_0^1\int_0^1 | \gamma_4- \bar{\gamma}_4|^2 dx_1dx_2 \leq CA^{-1/2}.
\]
Over the region (\ref{regionawayfromS}), 
according to Proposition \ref{NegativevertexLaplace} and  (\ref{gammaibarharmonic})
\[
\Lap_a ( \gamma_4- \bar{\gamma}_4   )=0,
\]
which translates into the Helmholtz type equations
\[
\Lap_a' h_{n_1,n_2} -4\pi \sqrt{-1}\text{Im}(a^{1\bar{2}}) (  n_1 \frac{\partial h_{n_1,n_2}}{ \partial y_2  } -   n_2 \frac{\partial h_{n_1,n_2}}{ \partial y_1  }   )    - 4\pi^2 ( a^{p\bar{q}} n_p n_q ) h_{n_1, n_2}=0.
\]
After the variable substitution 
\[
\tilde{h}_{n_1,n_2}=h_{n_1,n_2} \exp\left(  \frac{  2 \pi i \text{Im}(a_{1\bar{2}})}{ \mathbb{A}  }        (  a_{1\bar{1}} n_2 y_1- \text{Re}( a_{1\bar{2}}) n_1 y_1+ \text{Re}( a_{1\bar{2}}  ) n_2 y_2- a_{2\bar{2}} n_1 y_2          ) \right) ,
\]
these equations become the Helmholtz equations on $\R^2_{y_1,y_2}\times \R_\mu$,
\[
\Lap_a' \tilde{h}_{n_1,n_2}= k_{n_1,n_2}^2 \tilde{h}_{n_1,n_2},
\]
where we denote $k_{n_1, n_2}= 2\pi ( a^{p\bar{q}} n_p n_q   )^{1/2} ( \frac{A}{ \mathbb{A}  }   ) ^{1/2} $.

The rest of the argument is substantially similar to Proposition \ref{Exponentialdecayforthefirstorderansatz}. Observe that $|\text{dist}_{g_a' }(\cdot, \text{Im}(S))-R|\leq CA^{1/4}$. An upper barrier supersolution to the Helmholtz equation is directly constructed as
\[
\begin{split}
h_{n_1,n_2}'= A^{-1/4} & \{ \int_{\mathfrak{D}_1} e^{-k_{n_1,n_2} |(y_1, y_2-s, \mu)|_a'  } ds + \int_{\mathfrak{D}_2} e^{-k_{n_1,n_2} |(y_1-s, y_2, \mu)|_a'  } ds \\
&+ \int_{\mathfrak{D}_3} e^{-k_{n_1,n_2} |(y_1-s, y_2-s, \mu)|_a'  } ds \},
\end{split}
\] 
where $ds= dy_2, dy_1, -dy_1$ on $\mathfrak{D}_1, \mathfrak{D}_2, \mathfrak{D}_3$ respectively. Comparing the real and imaginary parts of $\tilde{h}_{n_1,n_2}$ with $C (|n_1|+|n_2|)h_{n_1,n_2}'$ yields the result.
\end{proof}

By the topological description (\cf Section \ref{Negativevertices} and \ref{Joycecritique}), we can lift the $T^2\subset (\C^*)^2\times \R_\mu$ to a $T^2$-cycle on the total space of the singular $S^1$-bundle, namely the monodromy invariant $T^2$-cycle for the topological $T^3$-fibration. As an application of the asymptotes above, we shall evaluate (up to sign) the integral of the closed 2-form $\omega^{(1)}$ on this $T^2$-cycle.

\begin{lem}\label{symplecticarea1}
The integral $\int_{T^2} \omega^{(1)}=\text{Im}(a_{2\bar{1}}) $.
\end{lem}

\begin{proof}
As a preliminary remark, although $\omega^{(1)}$ is a K\"ahler form only in a bounded region, it makes sense as a closed 2-form over the entire $(\C^*)^2\times \R_\mu$. The integral $\int_{T^2} \omega^{(1)}$ is a cohomological invariant, which can be evaluated asymptotically on a $T^2$-cycle as $\eta_1, \eta_2$ stay bounded and $\mu\to +\infty$. By choosing the $T^2$-cycle on which $\mu, y_1, y_2$ are constants,
\[
\begin{split}
\int_{T^2} \omega^{(1)}= & \frac{\sqrt{-1}}{2} \int_{T^2} W^{p\bar{q}}_{(1)} d\eta_p\wedge d\bar{\eta}_q = \frac{\sqrt{-1}}{2} \int_{T^2} ( W^{1\bar{2}}_{(1)}- W^{2\bar{1}}_{(1)} ) dx_1\wedge dx_2
\\
= &  \int_{T^2} \text{Im} ( W^{2\bar{1}}_{(1)} ) dx_1\wedge dx_2 = \text{Im}(a_{2\bar{1}})+  \int_{T^2} \gamma_4 dx_1\wedge dx_2.
\end{split}
\]
But by Lemma \ref{gamma4bardecays} and Proposition \ref{Exponentialdecayfor higherFouriermodesinthefirstorderansatzNegativevertex}, the quantity $\gamma_4\to 0$ as $\mu\to +\infty$, so the only contribution is
$
\int_{T^2} \omega^{(1)}=\text{Im}(a_{2\bar{1}}).
$
\end{proof}

\section{Structure near the singular locus I}\label{StructurenearDeltanegativevertex}

There exists a constant $0<C\ll 1$ such that discs of $g_a$-radius $CA^{1/4}$ centred at points in $S$ do not intersect each other; their union defines a \textbf{disc bundle} over $S$:  $\{ R\lesssim A^{1/4}  \}\subset (\C^*)^2\times \R_\mu$. We need to understand the local singularity structure of $v$ and $w^{p\bar{q}}$ in this disc bundle.

The following general setting is a variant of the classical Green's function asymptote for  submanifolds. Take the Euclidean space $(\R^5, \sum_{i=1}^5 ds_i^2 )$, containing the codimension 3 graphical submanifold  
\[ 
\Gamma=\{ (s_1,s_2,s_3,s_4,s_5)|    s_i=f_i( s_1, s_2    ), i=3,4,5 ,|s_1|^2+ |s_2|^2<1    \}
\] such that $|f_i|_{C^2}\leq C
$
for $i=3,4,5$. Let $\vec{n}_1, \vec{n}_2, \vec{n}_3$ be an orthonormal frame on $\Gamma$ and define a local parametrisation of a tubular neighbourhood of $\Gamma$:
\[
\Psi( s_1, s_2, t_1, t_2,t_3  )=\vec{h}(s_1,s_2) + \sum_1^3 t_i \vec{n}_i, \quad \vec{h}(s_1,s_2)=(s_1, s_2, f_i(s_1,s_2))\in \Gamma.
\]
such that $\Gamma$ is $\{ t_1=t_2=t_3=0  \}$ in the coordinates $s_1, s_2, t_1, t_2, t_3$. Denote $R=\sqrt{\sum t_i^2}$ as the Euclidean distance to $\Gamma$. Let $f=f(\vec{h}(s_1,s_2))$ be a function on $\Gamma$ with bound $\norm{f}_{C^2}\leq C$. We need asymptotes for the Green integral
\[
I_{\Gamma, f}(s)=   \int_{\Gamma}  - \frac{1}{ 8\pi^2 |s-s'|^3  } f(s')  d\text{Area}_{ \Gamma }(s'), \quad s\in \R^5.
\]
 We view $I_{\Gamma,f}$ as a function of $s_1,s_2,t_1,t_2,t_3$.

\begin{lem}\label{graphicalGreenfunction}
For $\sum s_i^2 \lesssim 1$, 
\[
\begin{cases}
|I_{\Gamma, f}(s_1,s_2, t_1,t_2,t_3)+\frac{f(s_1,s_2) }{ 4\pi R  }| \leq C    , \\
| \frac{\partial I_{\Gamma, f}}{\partial t_i}- \frac{f(s_1,s_2)  t_i }{ 4\pi R^3  } ( 1 + \frac{1}{5} (\sum_j t_j \vec{n}_j )\cdot \vec{H}  )  | \leq C    ,  \quad & i=3,4,5, \\
| \frac{\partial I_{\Gamma, f}}{\partial s_i}  + \frac{1}{4\pi R} \frac{\partial f}{\partial s_i}(s_1,s_2) | \leq C ,  \quad & i=1,2.
\end{cases}
\]
where $\vec{H}$ is the mean curvature vector of $\Gamma$ at $\vec{h}(s_1,s_2)$. If morever $\norm{f}_{C^3}\leq C$, $ \norm{f_i}_{C^3}\leq C$, then
\[
| \frac{\partial^2 I_{\Gamma, f}}{\partial t_i\partial t_j }-\frac{ (\delta_{ij}R^2-3 t_it_j ) f}{ 4\pi R^5  }   |\leq \frac{C}{R^2}, \quad 
| \frac{\partial^2 I_{\Gamma, f}}{\partial t_i\partial s_j }- \frac{ t_i }{ 4\pi R^3  } \frac{\partial f}{\partial s_j} |\leq \frac{C}{R}, \quad
 | \frac{\partial^2 I_{\Gamma, f}}{\partial s_i\partial s_j } |\leq \frac{C}{R}.
\]
If morever $f(0)=0$, then at $s_1=s_2=0$,
\[
| \frac{\partial^2 I_{\Gamma, f}}{\partial t_i\partial t_j }   |\leq \frac{C}{R}, 
\quad
| \frac{\partial^2 I_{\Gamma, f}}{\partial s_i\partial s_j }+ \frac{1}{4\pi R} \frac{\partial^2 f}{\partial s_i\partial s_j } |\leq C.
\]
If morever $df(0)=0$, then $| \frac{\partial^2 I_{\Gamma, f}}{\partial t_i\partial s_j }|\leq C$ for $s_1=s_2=0$.
\end{lem}

\begin{proof}
(Sketch) Consider $s_1=s_2=0$.
The leading order asymptote of $I_{\Gamma,f}$ is obtained by replacing $f$ with the constant $f(0)$ and replacing $\Gamma$ with $\R^2_{s_1, s_2}$. At $s=\sum t_i \vec{n}_i \in \R^5$,
\[
 \int_{\R^2}  - \frac{1}{ 8\pi^2 |s-s'|^3  } f(0)  ds_1' ds_2'= - \frac{f(0) }{ 4\pi R  }.
\]
We then need to estimate the deviation of $I_{\Gamma, f}$ from this leading asymptote.
After writing the surface integral as an integral over $s_1,s_2$ plane, we reduce to  the flat graph case $f_3=f_4=f_5=0$. Writing 
\[
f(s')= f(0)+ \sum_{i=1}^2\frac{\partial f}{\partial s_i'}(0) s_i' +O(|s'|^2),
\]
we observe that the linear term does not contribute to $I_{\Gamma,f}(s)$ by parity, and the $O(|s'|^2)$ contribution is bounded by
\[
C\int  \frac{ |s'|^2}{ |s-s'|^3  } ds_1' ds_2' \leq C \int  \frac{ r^2}{ (r^2+R^2)^{3/2}  } rdr \leq C.
\]

We now consider the normal first derivative $\frac{\partial I_{\Gamma,f} }{\partial t_i}$ for $s_1=s_2=0$ assuming without loss of generality that $df_i(0)=0$. After using the Taylor expansion and parity trick above, modulo bounded terms
\[
\begin{split}
\frac{\partial I_{\Gamma,f} }{\partial t_i}
\sim & f(0) \frac{3t_i}{8\pi^2} \int  \frac{1}{ (s_1'^2+s_2'^2+\sum_j (t_j-f_j)^2  )^{5/2} } ds_1'ds_2' 
\\
\sim & f(0) \frac{3t_i}{8\pi^2} \int  \frac{1}{ (s_1'^2+s_2'^2+R^2 )^{5/2} }( 1+ \frac{2\sum t_j f_j }{ s_1'^2+s_2'^2+ R^2 }  ) ds_1'ds_2' 
\\
\sim & f(0) \frac{ t_i}{ 4\pi R^3  } ( 1 + \frac{1}{5} \sum_j t_j (  \frac{\partial^2 f_i}{\partial s_1'^2}+ \frac{\partial^2 f_i}{\partial s_2'^2}  )  )
\\
= & f(0) \frac{ t_i}{ 4\pi R^3  } ( 1 + \frac{1}{5} (\sum_j t_j \vec{n}_j )\cdot \vec{H}  ),
\end{split}
\]
where $\vec{H}$ is the mean curvature of $\Gamma$ at the origin.

In the same setup, the tangential first derivative $\frac{\partial I_{\Gamma,f} }{\partial s_i}$ is modulo bounded terms
\[
\begin{split}
\frac{\partial I_{\Gamma,f} }{\partial s_i}\sim & \int  f  \frac{\partial}{\partial s_i'} \frac{1}{ 8\pi^2 |s-s'|^3 } ds_1'ds_2' \\
\sim  &
\int \frac{-1}{ 8\pi^2 |s-s'|^3 } \frac{\partial f}{\partial s_i'} ds_1'ds_2'
\sim   -\frac{1}{4\pi R} \frac{\partial f}{\partial s_i'}(0).
\end{split}
\]

The argument for second derivatives are similar.
\end{proof}

Around a  point $P\in S$ of interest, we introduce linear change of coordinates,
\[
\begin{cases}
\xi'_1= \frac{ z_1(P)( \eta_1- \eta_1(P))+ z_2(P) (\eta_2- \eta_2(P)  )  }{   A^{1/2} | a^{i\bar{j}} z_i(P)\bar{z}_j(P) |^{1/2}       } ,
\\
\xi_2= \frac{  (a^{1\bar{1}}  \bar{z}_1(P)+ a^{1\bar{2}} \bar{z}_2(P)  ) ( \eta_2- \eta_2(P)) -    (a^{2\bar{1}}  \bar{z}_1(P)+ a^{2\bar{2}} \bar{z}_2(P)  )    (\eta_1- \eta_1(P) ) )    } {   |a^{i\bar{j}} z_i(P)\bar{z}_j(P)  |^{1/2}      },
\end{cases}
\]
such that $T_P S=\{ \xi'_1=0, \mu=0 \}$, and $\xi_2=0$ on the normal 3-plane to $T_P S$. 
In these new linear coordinates,
\[
g_a= A (d\mu^2+ |d\xi'_1|^2+ |d\xi_2|^2),
\quad d\eta_1\wedge d\eta_2= A^{1/2} d\xi_1'\wedge d\xi_2.
\]
A problem is that the tangent planes tilts as $P$ moves along $S$. We find a local smooth vector valued function $\vec{h}$ to represent $S$ locally as a graph
\[
S\cap  \{ |\xi_2|\lesssim A^{-1/4}  \}  = \{ (\xi_1, \xi_2, \mu)= \vec{h}( \xi_2  ),   |\xi_2|\lesssim A^{-1/4}     \} \subset (\C^*)^2\times \R_\mu  .
\] 
The $g_a$-normal (1,0)-type vector to $S\subset (\C^*)_{\eta_1,\eta_2}^2$ at the point $\vec{h}(\xi_2)\in S$ is 
\[
\vec{n}( \vec{h}(\xi_2)  )= \frac{ a^{p\bar{q}} \bar{z}_q} { | a^{i\bar{j}} z_i \bar{z}_j|^{1/2}       }( \vec{h}(\xi_2) ) \frac{\partial}{\partial \eta_p}.
\] 
Denote $r=\sqrt{ A(|\xi_1|^2+|\xi_2|^2+\mu^2)}$. 
Define a \emph{local diffeomorphism} $\Psi$ on the local chart $\{ r\lesssim A^{1/4} \}$,
\[
\Psi( \xi_1, \xi_2, \mu )= \vec{h}( \xi_2  )+ A^{1/2}\xi_1 \vec{n}( \vec{h}(\xi_2)  )+ \mu \vec{e}_\mu \in (S^1\times \R)_{\eta_1,\eta_2}^2\times \R_\mu ,
\]
where $\vec{e}_\mu$ denotes the basis vector with $g_a$-magnitude $A^{1/2}$ corresponding to the $\mu$-variable. The image of $\Psi$ is a tubular neighbourhood of a graphical subset of $S$, and the straight degeneracy locus $\{ \mu=\xi_1=0  \}$ is identified with the curved degeneracy locus $S$. The pullback function $\Psi^*R= A^{1/2} \sqrt{\mu^2+ |\xi_1|^2} $, and  $\Psi^* g_a$ satisfies
\begin{equation}\label{diffeomorphismapproximateisometrygaterm}
\begin{cases}
| \Psi^* g_a- A( |d\xi_1|^2+|d\xi_2|^2+ d\mu^2  )|_{g_a} \leq CA^{1/4}|\xi_2|,
\\
|\nabla^k_{g_a} \{ \Psi^* g_a- A( |d\xi_1|^2+|d\xi_2|^2+ d\mu^2  ) \}|_{g_a} \leq CA^{-k/4   }  , \quad k\geq 1.
\end{cases}
\end{equation}

\begin{prop}\label{NegativevertexleadingorderasymptotenearDelta1}
(\textbf{leading order asymptote near $S$})
Via the local diffeomorphism $\Psi$, on the chart $\{  r\lesssim A^{1/4} \}$,
\[
\begin{cases}
|\Psi^*( w^{p\bar{q}}d\eta_p\otimes d\bar{\eta}_q) -  \frac{1}{ 2\sqrt{\mu^2+|\xi_1|^2 }  } d\xi_1\otimes d\bar{\xi}_1 |_{g_a} \leq CA^{-3/4} \max( 1,   \log ( A^{-1/4}\varrho ),
\\
|\Psi^*v -  \frac{1}{ 2\sqrt{\mu^2+|\xi_1|^2 }  } | \leq CA^{1/4} \max( 1,   \log ( A^{-1/4}\varrho ),
\end{cases}
\]

\end{prop}

\begin{rmk}
In the original coordinates, for $R	\lesssim A^{1/4}$,
\[
\begin{cases}
|w^{p\bar{q}}- \frac{ z_p \bar{z}_q }{ 2R A^{1/2}a^{i\bar{j}}z_i\bar{z}_j  }|\leq CA^{-1/4}\max( 1,   \log ( A^{-1/4}\varrho ) ),
\\
|v- \frac{ A^{1/2} }{ 2R  }|\leq CA^{1/4}\max( 1,   \log ( A^{-1/4}\varrho ) ). 
\end{cases}
\]
\end{rmk}

\begin{proof}
We focus on $w^{1\bar{1}}=\gamma_1+ \gamma_3$, where $\gamma_1, \gamma_3$ admit the Green's representation (\ref{gammai}). We split the integral on $S$ into the short distance contribution from
$
S\cap\{ |\eta_1- \eta'_1|, |\eta_2- \eta'_2|\lesssim \frac{1}{2}  \}
$  and the long distance contribution from $S\cap \{   |\eta_1- \eta'_1|\gtrsim \frac{1}{2} \text{ or } |\eta_2- \eta'_2|\gtrsim \frac{1}{2}       \}$.

The short distance contribution to the integral $w^{1\bar{1}}$ is
\[
-\pi A^{1/2} \int_{S\cap \{ |\eta_1- \eta'_1|, |\eta_2- \eta'_2|\lesssim \frac{1}{2}  \} } \gamma(\eta_1-\eta_1', \eta_2-\eta_2',\mu) \sqrt{-1} d\eta_2'\wedge d\bar{\eta}_2'.
\]
Applying Lemma \ref{gammaasymptote2}, we can replace the periodic Newtontian potential $\gamma$ by the ordinary Newtonian potential, so the short distance contribution is replaced by
\[
-\pi A^{1/2} \int_{S\cap \{ |\eta_1- \eta'_1|, |\eta_2- \eta'_2|\lesssim \frac{1}{2}  \} } - \frac{1}{8\pi^2|(\eta_1-\eta_1', \eta_2-\eta_2',\mu) |_a^3  } \sqrt{-1} d\eta_2'\wedge d\bar{\eta}_2',
\]
at a cost of a smooth error of order $O(A^{-1/4})$. The measure $\sqrt{-1} d\eta_2'\wedge d\bar{\eta}_2'$ is equal to $2 \frac{ |z_1'|^2}{ Aa^{i\bar{j}}z_i'\bar{z}_j'  }d\mathcal{A}(\eta_1', \eta_2')$, so the above expression is
\[
 \int_{S\cap \{ |\eta_1- \eta'_1|, |\eta_2- \eta'_2|\lesssim \frac{1}{2}  \} } - \frac{1}{8\pi^2|(\eta_1-\eta_1', \eta_2-\eta_2',\mu) |_a^3  }  \frac{-2\pi  |z_1'|^2}{ A^{1/2}a^{i\bar{j}}z_i'\bar{z}_j'  }d\mathcal{A}(\eta_1', \eta_2') .
\]

We may assume the submanifold $S\cap \{ |\eta_1- \eta'_1|, |\eta_2- \eta'_2|\lesssim \frac{1}{2}  \}$ is graphical, so Lemma \ref{graphicalGreenfunction} applies after scaling. Thus the short distance contribution to $w^{1\bar{1}}$ is
\[
 -\frac{1}{4\pi R }   \frac{-2\pi  |z_1'|^2}{ A^{1/2}a^{i\bar{j}}z_i'\bar{z}_j'  }+ O(A^{-1/4})=     \frac{ |z_1'|^2}{ 2R A^{1/2}a^{i\bar{j}}z_i'\bar{z}_j'  }+ O(A^{-1/4})     ,
\]  
where the complex coordinates $z_i'$ are computed at $\vec{h}(\xi_2)\in S$. But the factor $\frac{ |z_1|^2}{ A^{1/2}a^{i\bar{j}}z_i\bar{z}_j  }$ varies slowly, so we may as well compute it at $(\eta_1,\eta_2,\mu)$.

The long distance contribution to $w^{1\bar{1}}$ is $O( A^{-1/4}\max(1, \log (A^{-1/4}\varrho)   )    )$ by following the same steps as in Section \ref{Negativevertexasymptote1}, \ref{Negativevertexasymptote2}, using  
Lemma \ref{gammaasymptote1}.
Combining the two contributions, 
\[
|w^{1\bar{1}}- \frac{ |z_1|^2}{ 2R A^{1/2}a^{i\bar{j}}z_i\bar{z}_j  }| \leq C A^{-1/4}\max(1, \log (A^{-1/4}\varrho)   )    ).
\]
The cases of $w^{p\bar{q}}$ and $v=Aa^{p\bar{q}} w^{p\bar{q}}$ are similar.
\end{proof}

\begin{cor}\label{M-definition}
Fix $0<\epsilon_0 \ll 1 $, then on the total space
\[
M^-=\{  A^{-1/4} \varrho < \exp(\epsilon_0 A^{3/4})    \},
\]
the function $V_{(1)}$ is positive and the matrix $W^{p\bar{q}}_{(1)}$ is positive definite.
\end{cor}

\section{Structure near the singular locus II}\label{StructurenearDeltanegativevertexII}

This Section interprets the metric structure transverse to the singular locus $S$ in terms of Taub-NUT metrics. First we emphasize that \textbf{smooth topology of the $S^1$-fibration is subtle}. 
\begin{itemize}
\item 
The K\"ahler ansatz is only defined a priori on the complement of $S$, and there are no transparent choices of smooth local coordinates near $S$ to exhibit the smooth extension of both the complex structure and the metric.
\item 
By the reasons stated in Remark \ref{smoothtopologyissubtle}, smooth topology of the singular $S^1$-bundle  is expected to be unstable under deformation. 
\end{itemize}
Because of these difficulties, in the region $\{ R\lesssim A^{-1/2}\}$ near the singular locus $S\cap M^-$, we will be forced to work with tensor fields of \textbf{low regularity}.

We now define a \textbf{model metric}
\begin{equation}\label{tildegTaubdefinition}
\begin{cases}
g_{\text{NUT}}= (A+ \frac{1}{ 2\sqrt{\mu^2+ |\xi_1|^2} }  ) (d\mu^2+ |d\xi_1|^2  )+ (A+ \frac{1}{ 2\sqrt{\mu^2+ |\xi_1|^2}  }  )^{-1} \vartheta_{\text{NUT}}^2+ A |d\xi_2|^2,
\\
\omega_{\text{NUT}}= (A+ \frac{1}{ 2\sqrt{\mu^2+ |\xi_1|^2} }  ) (d\mu\wedge \vartheta_{\text{NUT}} + \frac{  \sqrt{-1}}{2} d\xi_1\wedge d\bar{\xi}_1  )+ \frac{A  \sqrt{-1}}{2} d\xi_2\wedge d\bar{\xi}_2,
\\
\Omega_{\text{NUT}}=A^{1/2} (\vartheta_{ \text{NUT}}-\sqrt{-1} (A+ \frac{1}{ 2\sqrt{\mu^2+ |\xi_1|^2} }  ) d\mu )\wedge d\xi_1\wedge d\xi_2,
\end{cases}
\end{equation}
namely the product of the  Taub-NUT metric with flat $\R^2$. The weighted H\"older norm $\norm{T}_{C^{k,\alpha}_{\delta}( g_{\text{NUT}}  )}$ for $S^1$-invariant tensors $T$ on this model space is defined by
\[
\norm{T}_{C^{k,\alpha}_{\delta}( g_{\text{NUT}}  )}  = \sup_{p} (A^{1/2}\ell)^{-\delta} \{ \sum_{j=0}^k|\ell^j\nabla^j T(p)|+   \sup_{|p-p'|_a \leq \ell  }    \frac{ |\nabla^k T(p)-\nabla^k T(p')|} { d_{g_{\text{NUT}}}(p,p'  )^\alpha \ell^{-\alpha-k}  }\}.
\]
where $\ell\sim A^{1/2}\sqrt{\mu_1^2+ |\xi_1|^2}+A^{-1/2}$ reflects the regularity scale of the model metric, and we compare $T(p)$ and $T(p')$ using parallel transport along minimal geodesics. An estimate in this norm is thought as the H\"older version of $
|T|= O((A^{1/2}\ell)^{\delta}  ).
$

\begin{prop}\label{NegativevertexleadingorderasymptotenearDelta2}
	Via the local diffeomorphism $\Psi$, on the chart $\{  r\lesssim A^{1/4} \}$,
	\[
	\begin{cases}
	| \nabla_{g_a} \{\Psi^*( w^{p\bar{q}}d\eta_p\otimes d\bar{\eta}_q) -  \frac{1}{ 2\sqrt{\mu^2+|\xi_1|^2 }  } d\xi_1\otimes d\bar{\xi}_1\}  |_{g_a} \leq CA^{-1} ,
	\\
	| \nabla_{g_a} \{ \Psi^*v -  \frac{1}{ 2\sqrt{\mu^2+|\xi_1|^2 }  } \} |_{g_a} \leq C .
	\end{cases}
	\]
	Furthermore
	\[
	\begin{cases}
	| \nabla^2_{g_a} \{\Psi^*( w^{p\bar{q}}d\eta_p\otimes d\bar{\eta}_q) -  \frac{1}{ 2\sqrt{\mu^2+|\xi_1|^2 }  } d\xi_1\otimes d\bar{\xi}_1\}  |_{g_{\text{NUT}}} \leq CA^{-3/4}\ell^{-2},
	\\
	|\nabla^2_{g_a} \{\Psi^*v -  \frac{1}{ 2\sqrt{\mu^2+|\xi_1|^2 }  } \} |_{g_{\text{NUT} } } \leq  CA^{1/4}\ell^{-2}.
	\end{cases}
	\]
\end{prop}

\begin{proof}
	(Sketch) The method is the same as in Proposition \ref{NegativevertexleadingorderasymptotenearDelta1}, so we only mention the key points. The long distance contribution to the integral has improved decay, so there is no need for the log factor. The short distance contribution appeals to Lemma \ref{graphicalGreenfunction}. In the first derivative estimates, notice the mean curvature vector vanishes because $S$ an algebraic curve.

	For second derivative estimates, notice for $R\lesssim A^{-1/2}$, the magnitudes of tensors are inhomogeneous:
	\[
	|d\mu|_{g_{\text{NUT}}}\leq A^{-1/4}R^{1/2}    ,  \quad |d\xi_1|_{ g_{\text{NUT}} } \leq A^{-1/4}R^{1/2}  , \quad |d\xi_2|_{ g_{\text{NUT}} } \leq CA^{-1/2}.
	\] 
	These $R$ factors make the $g_{\text{NUT}}$-magnitudes bounded even though the $g_a$-magnitudes can be unbounded.
	Inside the tensor $\Psi^*( w^{p\bar{q}}d\eta_p\otimes d\bar{\eta}_q)$, the coeffient of $d\xi_1\otimes d\bar{\xi}_2$ and  $d\xi_2\otimes d\bar{\xi}_1$ correspond to imposing $f(0)=0$, and the coeffient of $d\xi_2\otimes d\bar{\xi}_2$ correspond to imposing $f(0)=0$ and $df(0)=0$.
\end{proof}

\begin{prop}\label{TransverseTaubNUTmetric}
(\textbf{Transverse Taub-NUT metric})  Fix $0<\alpha<1$. Under suitable gauge choices for the $S^1$-connection, the ansatz metric $g^{(1)}$ over the local chart $\{  r\lesssim A^{1/4}  \} $ is approximated by the model metric:
\[
\begin{cases}
|\Psi^* g^{(1)}- g_{\text{NUT}}  |_{g_{\text{NUT}}} \leq  CA^{-3/4} \max(    \log (A^{-1/4}\varrho) , 1 ) + CA^{1/4}|\xi_2| ,  \\
| \nabla_{ g_{\text{NUT}} } (\Psi^*g^{(1)}- g_{\text{NUT}})  |_{ C^{\alpha}_{-1} ( r\lesssim A^{1/4}    )  }    \leq
CA^{-1/4}   \max(    \log (A^{-1/4}\varrho) , 1 )  .
\end{cases}
\]
and
\[
\begin{cases}
|\Psi^* \Omega^{(1)}- \Omega_{\text{NUT}}  |_{g_{\text{NUT}}} \leq  CA^{-3/4} \max(    \log (A^{-1/4}\varrho) , 1 ) + CA^{1/4}|\xi_2| ,  \\
| \nabla_{ \Omega_{\text{NUT}} } (\Psi^*g^{(1)}- \Omega_{\text{NUT}})  |_{ C^{\alpha}_{-1} ( r\lesssim A^{1/4}    )  }    \leq
CA^{-1/4}   \max(    \log (A^{-1/4}\varrho) , 1 )  .
\end{cases}
\]
\end{prop}

\begin{proof}
Applying the asymptotes in Proposition \ref{NegativevertexleadingorderasymptotenearDelta1} and
\ref{NegativevertexleadingorderasymptotenearDelta2}, on the local chart $\{r\lesssim A^{1/4} \}$, up to an error of order $O(  A^{-3/4} \max( 1, \log ( A^{-1/4}\varrho)   )    )$, the ansatz metric admits asymptote
\begin{equation}\label{TaubNUTfibrationasymptote1}
\begin{cases}
\begin{split}
g^{(1)}= & V_{(1)}d\mu^2+  V_{(1)}^{-1} \vartheta^2+ \text{Re}( W^{p\bar{q}}_{(1)} d\eta_p \otimes d\bar{\eta}_q     )
\\
\sim & \frac{1}{ 2\sqrt{\mu^2+|\xi_1|^2} }   (d\mu^2+ |d\xi_1|^2)  + (A+ \frac{1}{ 2\sqrt{\mu^2+|\xi|^2} }  )^{-1} \vartheta^2 + g_a  ,
\end{split}
\\
\omega^{(1)}\sim d\mu\wedge \vartheta+ \frac{\sqrt{-1} }{2} ( a_{p\bar{q}}d\eta_p\wedge d\bar{\eta}_q+  \frac{1}{ 2\sqrt{\mu^2+|\xi_1|^2} } d\xi_1\wedge d\bar{\xi}_1   ),
\\
\Omega^{(1)}\sim  A^{1/2} \{\vartheta -\sqrt{-1} ( A+ \frac{1}{ 2\sqrt{\mu^2+|\xi_1|^2} } ) d\mu ) \} \wedge d\eta_1\wedge d\eta_2        .
\end{cases}
\end{equation}

The main issue then is to compare the connection $\vartheta$ with $\vartheta_{\text{NUT}}$. The curvature of the Taub-NUT metric $d\vartheta_{\text{NUT}}$ has the explicit formula
\[
\sqrt{-1} \{   \frac{ -\mu}{ 4( \mu^2+ |\xi_1|^2  )^{3/2}   } d\xi_1\wedge d\bar{\xi}_1 +   
\frac{ -\bar{\xi}_1 }{ 4( \mu^2+ |\xi_1|^2  )^{3/2}   } d\mu\wedge d\xi_1   -  \frac{ -\xi_1}{ 4( \mu^2+ |\xi_1|^2  )^{3/2}   } d\mu\wedge d\bar{\xi}_1  \}.
\]
The curvature $d\vartheta$ is prescribed by formula (\ref{GibbonsHawkingcurvature}), involving the first derivatives of $v$ and $w^{p\bar{q}}$. Applying the asymptotic from Proposition \ref{NegativevertexleadingorderasymptotenearDelta2},
\[
|(d\vartheta- d\vartheta_{\text{NUT}})|_{g_a} \leq CA^{-1/2} , \quad  |\nabla_{g_a}(d\vartheta- d\vartheta_{\text{NUT}})|_{g_{\text{NUT}}} \leq CA^{-1/4} \ell^{-2}.
\] 
In particular $\norm{ (d\vartheta- d\vartheta_{\text{NUT}}) }_{ C^1_{-1}(g_{  \text{NUT}  })  } \leq CA^{1/4}$. After suitable gauge fixing, we can find a 1-form $\vartheta- \vartheta_{ \text{NUT} }$ on the base $\{ r\lesssim A^{1/4}  \}$, with norm estimate upstairs $\norm{ \vartheta- \vartheta_{\text{NUT}} }_{ C^{1,\alpha}_{0}(g_{  \text{NUT}  })  } \leq CA^{-1/4}$ for any fixed $0<\alpha<1$. This specifies a gauge choice of $\vartheta$.

Thus up to an admissible amount of error we can replace $\vartheta$ by $\vartheta_{ \text{NUT} }$ in the asymptote (\ref{TaubNUTfibrationasymptote1}). The deviation between RHS of (\ref{TaubNUTfibrationasymptote1}) and $g_{ \text{NUT} }$ is an elementary term $\Psi^* g_a- A(d\mu^2+ |d\xi_1|^2+ |d\xi_2|^2)$ controlled by (\ref{diffeomorphismapproximateisometrygaterm}).
\end{proof}

\begin{rmk}
For $R\gtrsim A^{-1/2}$, one can use $\Lap_a$-harmonicity to obtain weighted $C^{k,\alpha}$-estimates for any large $k$. For $R\lesssim A^{-1/2}$, we exhibited a collection of local charts corresponding to a choice of $P\in S$, such that the K\"ahler ansatz has $C^{1,\alpha}$-regularity. Thus if $u$ is a $C^{2,\alpha}$ function on a chart, then its $g^{(1)}$-Laplacian is $C^\alpha$. This regularity is sufficient for setting up weighted H\"older analysis. 
\end{rmk}

\section{Complex geometric perspective}\label{ComplexgeometricperspectiveNegativevertex}

The goal of this Section is to identify the holomorphic structure of the K\"ahler ansatz $(J^{(1)}, \Omega^{(1)})$. Recall the (1,0)-form $\zeta= V_{(1)}d\mu+ \sqrt{-1}\vartheta$ and formula (\ref{GibbonsHawkingholomorphicdifferential}) for its differential. The main idea is to produce holomorphic differentials
\begin{equation}\label{holomorphicdifferentialsnegativevertex}
\begin{cases}
\zeta_3= \zeta+ \beta_{13} d\eta_1+ \beta_{23} d\eta_2,
\\
\zeta_4= -\zeta+ \beta_{14} d\eta_1+ \beta_{24} d\eta_2,
\end{cases}
\end{equation}
by solving for the unknown functions $\beta_{13}, \beta_{23}, \beta_{14}, \beta_{24}$. The requirement for $d\zeta_3=d\zeta_4=0$ translates into an overdetermined and underdetermined system of equations 
\begin{equation}\label{overdeterminedsystennegativevertex}
\begin{cases}
\frac{\partial \beta_{p3}}{\partial \bar{\eta}_q}= - \frac{1}{2} \frac{\partial w^{p\bar{q}}}{\partial \mu}   ,  
\quad   
 &\frac{\partial \beta_{p3}}{\partial \mu} = 2 \frac{\partial v}{\partial \eta_p},
 \\
 \frac{\partial \beta_{p4}}{\partial \bar{\eta}_q}=  \frac{1}{2} \frac{\partial w^{p\bar{q}}}{\partial \mu}   ,  
 \quad   
 &\frac{\partial \beta_{p4}}{\partial \mu} = -2 \frac{\partial v}{\partial \eta_p},
 \\
\frac{\partial \beta_{p3}}{\partial \eta_q}= \frac{\partial \beta_{q3}}{\partial \eta_p},  \quad & \frac{\partial \beta_{p4}}{\partial \eta_q}= \frac{\partial \beta_{q4}}{\partial \eta_p},  \quad p,q=1,2.
\end{cases}
\end{equation}
The overdetermined nature is closely related to the integrability of the complex structure. The underdetermined nature is related to the fact that we can add certain holomorphic functions of $\eta_1, \eta_2$ to $\beta_{p3}, \beta_{p4}$ and solve the same equations; to eliminate this ambiguity one has to impose more growth conditions. Our strategy for solving this system is a direct construction using \textbf{integral representations}, and the main technical difficulty is to extract finite expressions out of divergent integrals.

We use the shorthand notation $|\eta|_a= ( a_{p\bar{q}} \eta_p\bar{\eta}_q  )^{1/2}$ and $|y|_a=( a_{p\bar{q}} y_py_q  )^{1/2}$.
We introduce two \textbf{auxiliary functions}
 \[
\gamma_{\pm}(\eta_1, \eta_2, \mu)=  \frac{1}{3} \frac{ A^{1/2} \mu }{ |(\eta_1,\eta_2,\mu)|_a^3 |\eta|_a^2 }+ \frac{2}{3} \frac{1}{  |\eta|_a^4  } ( \frac{A^{1/2}\mu }{ |(\eta_1, \eta_2, \mu)|_a   } \mp 1      ),
\]
and define for $p=1,2$ the \textbf{series}
\begin{equation}
\begin{cases}
\gamma_{p3}(\eta_1, \eta_2,\mu)= \sum_{(n_1, n_2)\in \Z^2} 
\frac{3 a_{p\bar{q}} (\bar{\eta}_q+n_q ) }{8\pi^2 A^{1/2} } \gamma_+(\eta_1+n_1, \eta_2+n_2, \mu),
\\
\gamma_{p4}(\eta_1,\eta_2,\mu)= -\sum_{(n_1, n_2)\in \Z^2} 
\frac{3 a_{p\bar{q}} (\bar{\eta}_q+n_q ) }{8\pi^2 A^{1/2} } \gamma_-(\eta_1+n_1, \eta_2+n_2, \mu).
\end{cases}
\end{equation}
These series converge absolutely for $(\eta_1,\eta_2)\notin \Z^2$, and  are 1-periodic in $\eta_1,\eta_2$ variables. When $\eta_1=\eta_2=0$, the series are designed so that $\gamma_+$ and $\gamma_{p3}$ \textbf{extend smoothly} over $\{\mu>0\}$, while $\gamma_-$ and $\gamma_{p4}$ extend smoothly over $\{ \mu<0 \}$. We will later use $\gamma_{p3}$ and $\gamma_{p4}$ as \textbf{integrands} to construct $\beta_{p3}$ and $\beta_{p4}$.

\begin{lem}\label{gammap34differentialrelations}
(Differential identities) 
\[
\frac{\partial \gamma_{p3}}{\partial \mu}= -  \frac{\partial\gamma_{p4} }{\partial \mu}= 2  \frac{\partial\gamma}{\partial \eta_p}, \quad p=1,2,
\] 
and morever
\[
\frac{\partial \gamma_{p3}}{\partial \eta_q}=   \frac{\partial\gamma_{q3} }{\partial \eta_p}, \quad \frac{\partial \gamma_{p4}}{\partial \eta_q}=   \frac{\partial\gamma_{q4} }{\partial \eta_p}, \quad p,q=1,2.
\]
\end{lem}

\begin{proof}
We differentiate the series definition (\ref{gammadefinition}) of $\gamma$ to get
\[
\frac{\partial \gamma}{\partial \eta_p}= \frac{3  }{16\pi^2} \sum_{(n_1,n_2)\in \Z^2}   \frac{ a_{p\bar{q}} (\bar{\eta}_q +n_q) }{  |(\eta_1+n_1, \eta_2+n_2, \mu)|_a^5     }.
\]
Using the elementary formula for indefinite integrals
\[
\int \frac{1}{ (s^2+ t^2)^{5/2}   }ds= \frac{1}{3} \frac{s}{ (s^2+ t^2)^{3/2} t^2 }+ \frac{2}{3} \frac{1}{t^4} ( \frac{s}{\sqrt{s^2+t^2}  } \mp 1      ),
\]
we see
\[
\int \frac{1}{  |(\eta_1, \eta_2, \mu)|_a^5     } d\mu = A^{-1/2} \gamma_{\pm},
\]
or equivalently
\[
\frac{\partial}{\partial \eta_p} \left(-\frac{1}{8\pi^2 } \frac{1}{  |(\eta_1, \eta_2, \mu)|_a^3     }\right)  
=\frac{\partial }{\partial \mu}\left(
\frac{3 a_{p\bar{q}} \bar{\eta}_q }{16\pi^2 A^{1/2} }  \gamma_{\pm}\right).
\]
Thus after summation
\[
\frac{\partial\gamma}{\partial \eta_p} =
\frac{\partial}{\partial \eta_p}\sum_{(n_1,n_2)\in\Z^2} \left(-\frac{1}{8\pi^2 } \frac{1}{  |(\eta_1+n_1, \eta_2+n_2, \mu)|_a^3     }\right)  
=\frac{1}{2}  \frac{\partial \gamma_{p3}}{\partial \mu}= -\frac{1}{2}  \frac{\partial\gamma_{p4} }{\partial \mu}
.
\]

The `morever' statement follows from summing over the elementary differential relations
\[
(a_{p\bar{r}} \bar{\eta}_r)   \frac{\partial \gamma_{\pm}}{\partial \eta_q}(\eta_1, \eta_2, \mu)=  (a_{q\bar{r}} \bar{\eta}_r)  \frac{\partial\gamma_{\pm} }{\partial \eta_p}(\eta_1, \eta_2, \mu), \quad p,q=1,2.
\]
\end{proof}

By the periodicity of $\gamma_{p3}, \gamma_{p4}$, we may assume $|x_1|, |x_2|\leq \frac{1}{2}$. In order to integrate $\gamma_{p3}$ and $\gamma_{p4}$ we need to bound these functions. It is convenient to introduce some closely related integrals:
\begin{equation}
\begin{cases}
I_{01}(y_1,y_2,\mu)= \int_{\R^2} \frac{1}{ |(s_1+ \sqrt{-1}y_1, s_2+ \sqrt{-1}y_2, \mu  )|_a^3 a_{p\bar{q}} (s_p+ \sqrt{-1}y_p    ) (s_q- \sqrt{-1}y_q   )  } ds_1 ds_2,
\\
I_{02}(y_1,y_2,\mu)= \int_{\R^2} \frac{1}{  |a_{p\bar{q}}(s_p+\sqrt{-1}y_p)(s_q-\sqrt{-1}y_q ) |^2 |(s_1+ \sqrt{-1}y_1, s_2+ \sqrt{-1}y_2, \mu  )|_a  } ds_1 ds_2,
\\
I_{03}(y_1,y_2,\mu)= \int_{\R^2} \frac{1}{  |a_{p\bar{q}}(s_p+\sqrt{-1}y_p)(s_q-\sqrt{-1}y_q ) |^2  } ds_1 ds_2,
\end{cases}
\end{equation}
and we can express
\[
\begin{split}
I_{p3} &= \frac{-3a_{p\bar{q}}\sqrt{-1}y_q  }{8\pi^2 A^{1/2}  } \int_{\R^2} \gamma_+( s_1+ \sqrt{-1}y_1, s_2+ \sqrt{-1}y_2, \mu  ) ds_1 ds_2 \\
& =  \frac{-a_{p\bar{q}}\sqrt{-1}y_q  }{8\pi^2   }(  \mu I_{01} + 2\mu I_{02}  - 2A^{-1/2} I_{03}  )  , 
\end{split}
\]
and
\[
\begin{split}
I_{p4}& = \frac{3a_{p\bar{q}}\sqrt{-1}y_q  }{8\pi^2 A^{1/2}  } \int_{\R^2} \gamma_-( s_1+ \sqrt{-1}y_1, s_2+ \sqrt{-1}y_2, \mu  ) ds_1 ds_2
\\
& =\frac{ a_{p\bar{q}}\sqrt{-1}y_q  }{8\pi^2   }(  \mu I_{01} + 2\mu I_{02}  +2A^{-1/2} I_{03}  ).
\end{split}
\]

\begin{lem}\label{Integrandestimate0}
These integrals admit the simplified formulae:
\[
\begin{cases}
I_{01}
 =  \frac{\pi}{ \sqrt{ \mathbb{A} }  } \int_{ \frac{A}{ \mathbb{A}} |y|_a^2     }^{\infty} \frac{1}{ s(s+ A\mu^2)^{3/2}   } ds,
 \\
I_{02}= -\frac{1}{2}I_{01} + \frac{\pi\sqrt{ \mathbb{A}} }{  A\varrho |y|_a^2      }
\\
I_{03}= \frac{\pi \sqrt{  \mathbb{A}   } }{ A   |y|_a^2   }.
\end{cases}
\]
Consequently
\[
\begin{cases}
I_{p3}= \frac{ a_{p\bar{q}}\sqrt{-1}y_q  }{4\pi A^{1/2}\sqrt{ \mathbb{A}}   }   \frac{1 }{  \varrho (  \varrho+ A^{1/2} \mu  )     }   ,
\\
I_{p4}= \frac{ a_{p\bar{q}}\sqrt{-1}y_q  }{4\pi A^{1/2}\sqrt{ \mathbb{A}}   } \frac{1 }{  \varrho (  \varrho- A^{1/2} \mu  )     }   .
\end{cases}
\]

\end{lem}

\begin{proof}
To evaluate these integrals, we introduce a radial variable \[
s=a_{p\bar{q}}(s_p+\sqrt{-1}y_p)(s-\sqrt{-1}y_q  ),
\]
and then elementary calculations in polar coordinates give
\[
\begin{split}
 I_{01}
=  \frac{\pi}{ \sqrt{ \mathbb{A} }  } \int_{ \frac{A}{ \mathbb{A}} |y|_a^2      }^{\infty} \frac{1}{ s(s+ A\mu^2)^{3/2}   } ds,
\end{split}
\]	
and similarly
\[
\begin{split}
 I_{02}
= \frac{\pi}{ \sqrt{ \mathbb{A} }  } \int_{ \frac{A}{ \mathbb{A}} |y|_a^2      }^{\infty} \frac{1}{ s^2(s+ A\mu^2)^{1/2}   } ds
= -\frac{1}{2}I_{01} + \frac{\pi\sqrt{ \mathbb{A}} }{  A\varrho |y|_a^2      },
\end{split}
\]
together with the formula for $I_{03}$.
The formulae for $I_{p3}$ and $I_{p4}$ follow from taking linear combinations.
\end{proof}

\begin{lem}\label{IntegrandestimateI}
	(\text{Estimating integrands I})
For $|y_1|+ |y_2|\gtrsim 1$ and $|x_1|, |x_2|\leq \frac{1}{2}$, we have the estimate
\[
|\gamma_{p3}-\gamma_{p4}|\leq   \frac{ C  A^{-1/4}|\mu|   }{ |y|_a \varrho   } .
\]
Morever there are improved estimates for $\gamma_{p3}, \gamma_{p4}$ depending on the sign of $\mu$:
\[
\begin{cases}
|\gamma_{p3}| \leq \frac{CA^{-3/4} |y|_a    }{ \varrho^2    } \max( 1, \log( \frac{A\mu^2}{|y|_a^2} )    ) , \quad & \mu\geq 0, 
\\
|\gamma_{p4}| \leq \frac{CA^{-3/4} |y|_a  }{ \varrho^2    } \max( 1, \log( \frac{A\mu^2}{|y|_a^2} )    ) , \quad & \mu\leq 0.
\end{cases}
\]
\end{lem}

\begin{proof}
Consider first the special case where $x_1=x_2=0, \eta_p=\sqrt{-1}y_p$. By pairing $(n_1, n_2)$ with $(-n_1, -n_2)$ in the summation, we obtain
\begin{equation*}
\begin{cases}
\gamma_{p3}(\eta_1, \eta_2, \mu)= \frac{-3a_{p\bar{q}}\sqrt{-1}y_q  }{8\pi^2 A^{1/2}  } \sum_{(n_1,n_2)\in \Z^2} \gamma_+( n_1+ \sqrt{-1}y_1, n_2+ \sqrt{-1}y_2, \mu  ) , 
\\
\gamma_{p4}(\eta_1, \eta_2, \mu)= \frac{3a_{p\bar{q}}\sqrt{-1}y_q  }{8\pi^2 A^{1/2}  } \sum_{(n_1,n_2)\in \Z^2} \gamma_-( n_1+ \sqrt{-1}y_1, n_2+ \sqrt{-1}y_2, \mu  ) .
\end{cases}
\end{equation*}
By the Cauchy integral test,
\[
\begin{split}
&\sum_{(n_1,n_2)\in \Z^2} \frac{1}{ |(n_1+ \sqrt{-1}y_1, n_2+ \sqrt{-1}y_2, \mu  )|_a^3 a_{p\bar{q}} (n_p+ \sqrt{-1}y_p    ) (n_q- \sqrt{-1}y_q   )  }\\
\leq & CI_{01}	
\leq  CA^{-1/2} \frac{1}{ \varrho^3    } \max( 1, \log( \frac{A\mu^2}{|y|_a^2} )    )   .
\end{split}
\]	
Similarly,
\[
\begin{split}
& |\sum_{n_1, n_2} \frac{1}{  (a_{p\bar{q}}(n_p+\sqrt{-1}y_p)(n_q-\sqrt{-1}y_q ) )^2  |( n_1+\sqrt{-1}y_1 ,  n_2+\sqrt{-1}y_2 , \mu)|_a   } |
\\
\leq & CI_{02}
\leq  CA^{-1/2} \frac{1}{ |y|_a^2 \varrho   }.
\end{split}
\]
Combining these two estimates,
\[
|\gamma_{p3}-\gamma_{p4}|(\eta_1,\eta_2, \mu)\leq CA^{-1/2}|a_{pq}y_q||\mu|  \frac{1}{ |y|_a^2 \varrho   } \leq    \frac{ CA^{-1/4}|\mu|}{ |y|_a \varrho   } .
\]
Morever, when $\mu\geq 0$,
\[
\begin{split}
&\frac{1}{ a_{p\bar{q}}(n_p+\sqrt{-1}y_p)(n_q-\sqrt{-1}y_q )   }| \frac{A^{1/2}\mu }{ |( n_1+\sqrt{-1}y_1 ,  n_2+\sqrt{-1}y_2 , \mu)|_a   } - 1 | 
\\
\leq & C \frac{ 1}{  |( n_1+\sqrt{-1}y_1 ,  n_2+\sqrt{-1}y_2 , \mu)|_a^2    },
\end{split}  
\] 	
whence
\[
\begin{split}
& |\sum_{n_1, n_2} \frac{1}{  |a_{p\bar{q}}(n_p+\sqrt{-1}y_p)(n_q-\sqrt{-1}y_q ) |^2  } \left( \frac{A^{1/2}\mu }{ |( n_1+\sqrt{-1}y_1 ,  n_2+\sqrt{-1}y_2 , \mu)|_a   } - 1      \right)|\\
\leq & C\sum_{n_1, n_2} \frac{1}{  a_{p\bar{q}}(n_p+\sqrt{-1}y_p)(n_q-\sqrt{-1}y_q )   |( n_1+\sqrt{-1}y_1 ,  n_2+\sqrt{-1}y_2 , \mu)|_a^2   } \\
\leq & C\int_{\R^2} \frac{1}{  a_{p\bar{q}}(s_p+\sqrt{-1}y_p)(s_q-\sqrt{-1}y_q )   |( s_1+\sqrt{-1}y_1 ,  s_2+\sqrt{-1}y_2 , \mu)|_a^2     } ds_1 ds_2 \\
\leq & CA^{-1/2}   \frac{1}{ \varrho^2    } \max( 1, \log( \frac{A\mu^2}{|y|_a^2 } )    )  .
\end{split}
\]
This leads to 
\[
|\gamma_{p3}(\eta_1,\eta_2,\mu)| \leq \frac{CA^{-1} |a_{p\bar{q}}y_q|  }{ \varrho^2    } \max( 1, \log( \frac{A\mu^2}{|y|_a^2} )    ) , \quad  \mu\geq 0. 
\]
Similarly
\[
|\gamma_{p4}(\eta_1,\eta_2,\mu)| \leq \frac{CA^{-1}|a_{p\bar{q}}y_q| }{ \varrho^2    } \max( 1, \log( \frac{A\mu^2}{|y|_a^2} )    ) , \quad  \mu\leq 0. 
\]

For general $|x_1|,|x_2|\leq \frac{1}{2}$, the difference $\gamma_{p3}(\eta_1, \eta_2, \mu)- \gamma_{p3}(\sqrt{-1}y_1,\sqrt{-1}y_2,\mu)$, respectively $\gamma_{p4}(\eta_1, \eta_2, \mu)- \gamma_{p4}(\sqrt{-1}y_1,\sqrt{-1}y_2,\mu)$, can be estimated by termwise comparing the two series using the methods above. The result is
\[
\begin{cases}
|\gamma_{p3}(\eta_1, \eta_2, \mu)- \gamma_{p3}(\sqrt{-1}y_1,\sqrt{-1}y_2,\mu)|\leq \frac{ C (|x_1|+|x_2|)  }{ A^{1/2}\varrho^2    } \max( 1, \log( \frac{A\mu^2}{|y|_a^2} )    ),\mu\geq 0,
\\
|\gamma_{p4}(\eta_1, \eta_2, \mu)- \gamma_{p4}(\sqrt{-1}y_1, \sqrt{-1}y_2,\mu)|\leq \frac{ C (|x_1|+|x_2|)  }{ A^{1/2}\varrho^2    } \max( 1, \log( \frac{A\mu^2}{|y|_a^2} )    ), \mu\leq 0,
\end{cases}
\]
and
\[
\begin{split}
&|(\gamma_{p3}-\gamma_{p4} )(\eta_1, \eta_2, \mu)- (\gamma_{p3}-\gamma_{p4} )(\sqrt{-1}y_1, \sqrt{-1}y_2, \mu)| 
\leq \frac{ C|\mu| (|x_1|+|x_2|)  }{ |y|_a^2 \varrho  } ,
\end{split}
\]
so the claims in the Lemma reduces to the special case $x_1=x_2=0$ above.
\end{proof}

Next we examine
\begin{equation}\label{gammap3+gammap4}
\begin{split}
\gamma_{p3}+\gamma_{p4}= & -\sum_{(n_1,n_2)\in \Z^2} 
\frac{ a_{p\bar{q}} (\bar{\eta}_q+n_q ) }{2\pi^2 A^{1/2}  |a_{i\bar{j}}(\eta_i+n_i)(\bar{\eta}_j+n_j) |^2  } \\
= &  \sum_{(n_1,n_2)\in \Z^2}   \frac{\partial}{\partial \eta_p} \frac{ 1 }{2\pi^2 A^{1/2}  a_{i\bar{j}}(\eta_i+n_i)(\bar{\eta}_j+n_j)   } .
\end{split}
\end{equation}

\begin{lem}\label{IntegrandestimateII}
(Estimating integrands II) For $|y_1|+|y_2|\gtrsim 1$ and $|x_1|, |x_2|\leq \frac{1}{2}$, we have the estimate
\[
|\gamma_{p3}+\gamma_{p4}- \frac{ \sqrt{-1} a_{p\bar{q}} y_q\sqrt{  \mathbb{A}   } }{2\pi A^{3/2}   |y|_a^2    }|\leq \frac{C}{ A^{1/2} |y|_a^2  }.
\]
Morever,
\[
|\gamma_{p3}-I_{p3}|+ |\gamma_{p4}-I_{p4}|\leq \frac{ C|\mu|  }{|y|_a^2 \varrho   } .
\]
\end{lem}

\begin{proof}
Using the same strategy as in Lemma \ref{IntegrandestimateI}, we reduce to the special case $x_1=x_2=0, \eta_p=\sqrt{-1}y_p$. Pairing $(n_1, n_2)$ with $(-n_1, -n_2)$ in the series  (\ref{gammap3+gammap4}),
\[
(\gamma_{p3}+\gamma_{p4})(\eta_1,\eta_2, \mu)=\sum_{(n_1,n_2)\in \Z^2} 
\frac{ \sqrt{-1} a_{p\bar{q}} y_q  }{2\pi^2 A^{1/2}  |a_{i\bar{j}}(n_i+\sqrt{-1}y_i)(n_j-\sqrt{-1}y_j) |^2  }.
\]
We compare this series expression of $\gamma_{p3}+\gamma_{p4}$ to the closely related integral (\cf Lemma \ref{Integrandestimate0})
\[
\begin{split}
\frac{ \sqrt{-1} a_{p\bar{q}} y_q }{2\pi^2A^{1/2}  }I_{03}
=  \frac{ \sqrt{-1} a_{p\bar{q}} y_q\sqrt{  \mathbb{A}   } }{2\pi A^{3/2}  |y|_a^2    }.
\end{split}
\]
The deviation between the series and the integral is bounded by
\[
C| a_{p\bar{q}} y_q|A^{-1/4}  \int_{\R^2}  \frac{ 1  }{  |a_{i\bar{j}}(s_i+\sqrt{-1}y_i)(s_j-\sqrt{-1}y_j) |^{5/2}  } ds_1 ds_2 \leq  \frac{C}{ A^{1/2} |y|_a^2   },
\]
using the same type of Cauchy integral test argument as Lemma \ref{gammaasymptote1}.

The `morever' statement is a minor variant of the proof of Lemma \ref{IntegrandestimateI}, where in the application of the Cauchy integral test we use the mean value inequality to estimate the difference between the series and the integral, similar to the argument in Lemma \ref{gammaasymptote1}.
\end{proof}

We would like to use Lemma \ref{IntegrandestimateI}, \ref{IntegrandestimateII} to \textbf{construct functions $\beta_{p3},\beta_{p4}$ as integrals}:
\begin{equation}
\begin{cases}
\beta_{p3}(\eta_1, \eta_2, \mu)=-2\pi A^{1/2} \int_S \gamma_{p3}(  \eta_1-\eta_1', \eta_2-\eta_2', \mu     ) d\mathcal{A}(\eta_1',\eta_2'),
\\
\beta_{p4}(\eta_1, \eta_2, \mu)=-2\pi A^{1/2} \int_S \gamma_{p4}(  \eta_1-\eta_1', \eta_2-\eta_2', \mu     ) d\mathcal{A}(\eta_1',\eta_2').
\end{cases}
\end{equation}
where we recall $d\mathcal{A}$ is the area form on $S$.
The problem is that these integrals \textbf{diverge} at the three ends of $S$, and we need to \textbf{extract some convergent limit} to make sense of $\beta_{p3}, \beta_{p4}$, in a fashion rather similar to (\ref{gammai}).

The ends of $S$ are up to exponentially small errors approximately $\mathfrak{D}_i\times S^1$ for $i=1,2,3$. By Lemma \ref{IntegrandestimateI}, the expression $\beta_{p3}-\beta_{p4}$ makes sense as an ordinary integral with integrand $\gamma_{p3}-\gamma_{p4}$ thanks to the convergence of $\int^\infty \frac{1}{y^2}dy$. It suffices to makes sense of $\beta_{p3}+\beta_{p4}$. We consider the integral over large bounded regions with a cutoff scale $\Lambda$,
\[
\int_{ S\cap \{ |y'|_a <\Lambda \}} (\gamma_{p3}+\gamma_{p4})( \eta_1-\eta_1', \eta_2-\eta_2', \mu  ) d\mathcal{A}(\eta_1', \eta_2').
\]
Lemma \ref{IntegrandestimateII} tells us the exact nature of divergence. At the end $\mathfrak{D}_1\times S^1$, 
\[
\gamma_{p3}+\gamma_{p4}\sim  -\frac{ \sqrt{-1} a_{p\bar{q}} y_q'\sqrt{  \mathbb{A}   } }{2\pi A^{3/2}   |y'|_a^2   }\sim -\frac{\sqrt{-1} a_{p\bar{2}} \sqrt{  \mathbb{A}   } }{2\pi A^{3/2}   a_{2\bar{2}} y_2'    }, \quad d\mathcal{A}(\eta_1', \eta_2')\sim a_{2\bar{2}} dx_2'\wedge dy_2' ,
\]
so the divergence behaviour of the integral  is $\sim
-\frac{\sqrt{-1} a_{p\bar{2}} \sqrt{  \mathbb{A}   } }{2\pi A^{3/2}      }\log \frac{\Lambda}{\sqrt{a_{2\bar{2}}}   }
$ at  $\mathfrak{D}_1\times S^1$.
Similarly, the divergence behaviour   is
$\sim -\frac{ \sqrt{-1} a_{p\bar{1}} \sqrt{  \mathbb{A}   } }{2\pi A^{3/2}      }\log \frac{\Lambda}{\sqrt{a_{1\bar{1}}} }  $ at $\mathfrak{D}_2\times S^1$, and is $\sim \frac{ \sqrt{-1} (a_{p\bar{1}}+ a_{p\bar{2}}) \sqrt{  \mathbb{A}   } }{2\pi A^{3/2}      }\log \frac{\Lambda}{   \sqrt{a_{1\bar{1}}+ a_{1\bar{2}}+ a_{2\bar{1}}+ a_{2\bar{2}}  } }$ at $\mathfrak{D}_3\times S^1$. The remarkable fact is that the divergent parts cancel out so that 
\[
\lim_{\Lambda\to \infty}\int_{ S\cap \{ |y'|_a<\Lambda \}}
 (\gamma_{p3}+\gamma_{p4})( \eta_1-\eta_1', \eta_2-\eta_2', \mu  ) d\mathcal{A}(\eta_1', \eta_2')
\]
converges; geometrically this cancellation comes from some balancing condition on the 3 directional vectors along $\mathfrak{D}_1, \mathfrak{D}_2, \mathfrak{D}_3$.  The upshot is that $\beta_{p3}$ and $\beta_{p4}$ make sense as \textbf{improper integrals}. The \textbf{domain of definition} for $\beta_{p3}$ is $(\C^*)^2\times \R_\mu\setminus \{ f_S=0, \mu\leq 0     \}$, and for $\beta_{p4}$ it is $(\C^*)^2\times \R_\mu\setminus \{ f_S=0, \mu\geq 0     \}$.

\begin{lem}\label{gammap34asymptoteforlargemu}
(Asymptotes as $\mu\to \pm\infty$) For any fixed $\eta_1, \eta_2$,
\[
\begin{cases}
\lim_{\mu\to +\infty} \beta_{p3}(\eta_1, \eta_2, \mu)=0,
\\
\lim_{\mu\to -\infty} \beta_{p4}(\eta_1, \eta_2, \mu)=0.
\end{cases}
\]
Morever
\[
\begin{cases}
\lim_{\mu\to +\infty} \frac{\partial \beta_{p3}}{ \partial \bar{\eta}_p  }( \eta_1, \eta_2, \mu )=0,
\\
\lim_{\mu\to -\infty} \frac{\partial \beta_{p4}}{ \partial \bar{\eta}_p  }( \eta_1, \eta_2, \mu )=0.
\end{cases}
\]
\end{lem}

\begin{proof}
We focus on the $\beta_{p3}$ case, and consider $A^{1/4}\mu \gg |\eta_1|+|\eta_2|+1$. Using Lemma \ref{IntegrandestimateI}, the contribution to $\int \gamma_{p3} d\mathcal{A}$ from the region $\{ |\eta_1-\eta_1'|+|\eta_2-\eta_2'| \leq (A^{1/4}\mu)^{1-\epsilon}  \}\cap S$ is negligible, where $0<\epsilon\ll 1$ is any small given number. Outside this region $S$ is asymptotic to $\mathfrak{D}_i\times S^1$ along the three ends up to exponentially small error, and furthermore Lemma \ref{IntegrandestimateII} allows us to replace $\gamma_{p3}$ by $I_{p3}$ without affecting the $\mu\to +\infty$ limit.

We are now left to consider the improper integral
\[
\int_{ \cup (\mathfrak{D}_i\times S^1)  }  I_{p3}(y_1-y_1', y_2-y_2', \mu) d\mathcal{A}(\eta_1', \eta_2').
\]
Using the formula of $I_{p3}$ in Lemma \ref{Integrandestimate0}, we can simplify further by setting $y_1=y_2=0$ without affecting the $\mu\to +\infty$ limit.
Along the $\mathfrak{D}_1\times S^1$ end, 
\[
\int_{  \mathfrak{D}_1\times S^1\cap \{ y_2'< \frac{\Lambda}{ \sqrt{a_{2\bar{2}}}} \}  }  I_{p3}(-y_1', -y_2', \mu) d\mathcal{A}(\eta_1', \eta_2')= \int_{  0 }^\frac{\Lambda}{ \sqrt{a_{2\bar{2}}}}  I_{p3}(0, -y_2', \mu) a_{2\bar{2}} dy_2',
\]
which we compute as
\[
\begin{split}
& \frac{ -a_{p\bar{2}}\sqrt{-1}  }{4\pi A^{1/2}\sqrt{ \mathbb{A}}   }
\int_0^\frac{\Lambda}{ \sqrt{a_{2\bar{2}}}} a_{2\bar{2}} y_2' dy_2'  \frac{1}{  |(0, -y_2', \mu)|_a' (  |(0, -y_2', \mu)|_a'+ A^{1/2} \mu  )     }   \\
=&  \frac{- a_{p\bar{2}}\sqrt{-1} \sqrt{ \mathbb{A}}  }{8\pi A^{3/2}  } \int_{A\mu^2}^{   \frac{A}{  \mathbb{A}      } \Lambda^2+ A\mu^2}    \frac{ds}{ s^{1/2} (s^{1/2}+ A^{1/2}\mu)    }
\\
=& \frac{- a_{p\bar{2}}\sqrt{-1} \sqrt{ \mathbb{A}}  }{4\pi A^{3/2}  } \log \left(   \frac{   (    \mathbb{A}^{-1}       \Lambda^2+ \mu^2  )^{1/2}   + \mu  }{ 2 \mu }      \right).
\end{split}
\]
Similarly, the integrals from $\mathfrak{D}_2\times S^1$ and $\mathfrak{D}_3\times S^1$ are respectively
\[
 \frac{ -a_{p\bar{1}}\sqrt{-1} \sqrt{ \mathbb{A}}  }{4\pi A^{3/2}  }  \log \left(   \frac{   (    \mathbb{A}^{-1}       \Lambda^2+ \mu^2  )^{1/2}   + \mu  }{ 2 \mu }      \right)
\]
and
\[
\frac{ (a_{p\bar{1}}+a_{p\bar{2}})\sqrt{-1} \sqrt{ \mathbb{A}}  }{4\pi A^{3/2}  }  
\log \left(   \frac{   (    \mathbb{A}^{-1}       \Lambda^2+ \mu^2  )^{1/2}   + \mu  }{ 2 \mu }      \right)
 .
\]
Summing over the three contributions and take the limit $\Lambda\to +\infty$,
\[
\int_{ \cup (\mathfrak{D}_i\times S^1)  }  I_{p3}(-y_1', -y_2', \mu) d\mathcal{A}(\eta_1', \eta_2')= 0.
\]
This proves 
$
\lim_{\mu\to +\infty } \beta_{p3}(\eta_1, \eta_2, \mu)=0
$
. Likewise with the $\beta_{p4}$ case.

The `morever' statement follows from a simpler argument. The key is that higher derivatives of the integrand have faster decay at large distance, so that the divergence issues do not arise.
\end{proof}

\begin{lem}\label{betap3+betap4formula}
The explicit formula for $\beta_{p3}+\beta_{p4}$ is
\[
\beta_{p3}+ \beta_{p4}= \frac{ -2\pi i e^{2\pi i \eta_p} }{ 1- e^{2\pi i \eta_1}- e^{2\pi i \eta_2}     }+K_p(a)= \frac{ -2\pi i e^{2\pi i \eta_p} }{ f_S    }+K_p(a), \quad p=1,2.
\]
where the constant $K_p(a)$ is
\[
\begin{split}
K_p(a)= &  \frac{\sqrt{-1} (  a_{p\bar{2} } \text{Re}(a_{1\bar{2}})- a_{p\bar{1}} a_{2\bar{2}}) }{ A  } (  \frac{\pi}{2} + \arctan (  \frac{a_{2\bar{2}} + \text{Re} (a_{1\bar{2}} ) }{ \sqrt{\mathbb{A} }  }    )     ) \\ + &
\frac{\sqrt{-1} (  a_{p\bar{1} }\text{Re}(a_{1\bar{2}} )- a_{p\bar{2}} a_{1\bar{1}}) }{ A  } (  \frac{\pi}{2} + \arctan (  \frac{a_{1\bar{1}} + \text{Re} (a_{1\bar{2}}) }{ \sqrt{\mathbb{A} }  }    )     ).
\end{split}
\]
\end{lem}

\begin{proof}
The basic strategy is a Liouville theorem argument: we will construct a function with the same distributional $\Lap_a$-Laplacian as $\beta_{p3}+\beta_{p4}$, and then argue they must be equal.

We start with the Poincar\'e-Lelong formula
\[
S= \frac{\sqrt{-1}}{2\pi } \partial \bar{\partial } \log | 1- e^{2\pi i \eta_1}- e^{2\pi i \eta_1}  |^2= \frac{\sqrt{-1}}{2\pi } \partial \bar{\partial } \log |f_S|^2 ,
\]
from which we obtain the equality of  measures
\[
\begin{split}
\int_S d\mathcal{A}= & \int_S \frac{\sqrt{-1}}{2} a_{p\bar{q}} d\eta_p\wedge d\bar{\eta}_q
\\
= & \frac{\sqrt{-1}}{2\pi } \partial \bar{\partial } \log |f_S|^2\wedge \frac{\sqrt{-1}}{2} a_{p\bar{q}} d\eta_p\wedge d\bar{\eta}_q\\
= &  \frac{1}{4\pi} (\Lap_a \log |f_S|^2) d\text{Vol}_a.
\end{split}
\]
The periodic Newtonian potential on $(\C^*)^2$ with the $g_a$-metric is
\[
\gamma_5(\eta_1, \eta_2, \mu)= -\frac{1}{4\pi^2} \sum_{(n_1, n_2) \in \Z^2} \{ \frac{1}{ a_{p\bar{q}} ( \eta_p+n_p   )(\bar{\eta}_q+n_q)       } -   \frac{1}{ a_{p\bar{q}} n_p   n_q}   \}.
\]
Thus for any large cutoff scale $\Lambda$, the Green's representation 
\[
\int_{S\cap \{ |y'|_a <\Lambda  \} } \gamma_5(\eta_1-\eta_1', \eta_2-\eta_2', \mu) d\mathcal{A}(\eta_1', \eta_2')
\]
has the same distributional $\Lap_a$-Laplacian as that of $\frac{1}{4\pi}\log |f_S|^2 $ in the large compact region. Taking the $\eta_p$ derivative and taking the $\Lambda\to \infty$ limit shows that the $\Lap_a$-Laplacian of  $\frac{ -i e^{2\pi i\eta_p}}{ 2f_S} $ agrees with that of the improper integral
\[
\int_{S }\frac{\partial}{\partial \eta_p}\gamma_5(\eta_1-\eta_1', \eta_2-\eta_2', \mu) d\mathcal{A}(\eta_1', \eta_2'),
\]
which by formula (\ref{gammap3+gammap4}) is the same as the improper integral
\[
-\frac{1}{2}A^{1/2}\int_{S }(\gamma_{p3}+\gamma_{p4})(\eta_1-\eta_1', \eta_2-\eta_2', \mu) d\mathcal{A}(\eta_1', \eta_2')=  \frac{1}{4\pi} (\beta_{p3}+ \beta_{p4}  ).
\]

The upshot is that $\beta_{p3}+\beta_{p4}$ differs from $\frac{ -2\pi i e^{2\pi i \eta_p}}{ f_S }$ by a globally smooth $\Lap_a$-harmonic function on $(\C^*)^2$. It is also easy to show using techniques in this Section that this difference can have at most log growth in $y_1, y_2$ variables. Thus it has to be a constant.

The rest of this proof is to pin down precisely this constant, by considering the limit $(\eta_1, \eta_2)=(\sqrt{-1}b,\sqrt{-1}b)$ for  $b\to +\infty$. This uses techniques similar to the proof of Lemma \ref{gammap34asymptoteforlargemu}. Without affecting the limit, we can replace $S$ with $\cup \mathfrak{D}_i\times S^1$ and replace $(\gamma_{p3}+ \gamma_{p4})$ with 
\[
\frac{ \sqrt{-1} a_{p\bar{q}} (b- y_q') \sqrt{ \mathbb{A} }     } {  2\pi A^{3/2} |y-y'|_a^2     }.
\]
This leads to an asymptotic expression for $b\gg 1$,
\[
(\beta_{p3}+\beta_{p4})(\sqrt{-1}b, \sqrt{-1}b)\sim \int_{ \cup \mathfrak{D}_i\times S^1   }\frac{ \sqrt{-1} a_{p\bar{q}} (-b+ y_q') \sqrt{ \mathbb{A} }     } {   A |y-y'|_a^2     } d\mathcal{A},
\]
where the RHS is understood as an improper integral. To evaluate this integral we fix $b$ and calculate the $\Lambda\to +\infty$ asymptotic expression of the integral over the large
bounded domain 
$
( \cup \mathfrak{D}_i\times S^1  )\cap \{ |y'|_a <\Lambda    \}.
$     
The contribution from the end $\mathfrak{D}_1\times S^1$ is 
\[
\begin{split}
&\frac{\sqrt{-1}a_{p\bar{2} } \sqrt{\mathbb{A} }  }{ 2A  } \log ( \frac{\Lambda^2}{ (a_{1\bar{1}}+a_{1\bar{2}}+ a_{2\bar{1}}+ a_{2\bar{2}} ) b^2         }      )
\\
+ &\frac{\sqrt{-1} (  a_{p\bar{2} } \text{Re}a_{1\bar{2}}- a_{p\bar{1}} a_{2\bar{2}}) }{ A  } (  \frac{\pi}{2} + \arctan (  \frac{a_{2\bar{2}} + \text{Re}(a_{1\bar{2}}) }{ \sqrt{\mathbb{A} }  }    )     ) +o(1) .
\end{split}
\]
The contribution from $\mathfrak{D}_2\times S^1$ is
\[
\begin{split}
&\frac{\sqrt{-1} a_{p\bar{1} }\sqrt{\mathbb{A} }  }{ 2A  } \log ( \frac{\Lambda^2}{ (a_{1\bar{1}}+a_{1\bar{2}}+ a_{2\bar{1}}+ a_{2\bar{2}} ) b^2         }      )
\\
+ &\frac{\sqrt{-1} (  a_{p\bar{1} }\text{Re}a_{1\bar{2}} - a_{p\bar{2}} a_{1\bar{1}}) }{ A  } (  \frac{\pi}{2} + \arctan (  \frac{a_{1\bar{1}} + \text{Re}(a_{1\bar{2}}) }{ \sqrt{\mathbb{A} }  }    )     )+o(1).
\end{split}
\]
The contribution from $\mathfrak{D}_3\times S^1$ is
\[
-\frac{\sqrt{-1} \sqrt{\mathbb{A} }
	( a_{p\bar{1} } +a_{p\bar{2} }) }{ 2A  } \log ( \frac{\Lambda^2}{ (a_{1\bar{1}}+a_{1\bar{2}}+ a_{2\bar{1}}+ a_{2\bar{2}} ) b^2         }      )+o(1).
\]
Summing up, the log terms cancel out, so 
the improper integral 
\[
\int_{ \cup \mathfrak{D}_i\times S^1   }\frac{ \sqrt{-1} a_{p\bar{q}} (-b+ y_q') \sqrt{ \mathbb{A} }     } {   A |y-y'|_a^2     } d\mathcal{A},
\]
is equal to the constant $K_p(a)$ defined in the statement of the Lemma.
This shows limiting value
\[
\lim_{b\to +\infty}(\beta_{p3}+\beta_{p4})(\sqrt{-1}b, \sqrt{-1}b)= K_p(a).
\]
Comparing this with
\[
\lim_{b\to +\infty} \frac{-2\pi i e^{2\pi i \eta_p} }{f_S}  (\sqrt{-1}b, \sqrt{-1}b)= 0
\]
determines the constant.
\end{proof}

\begin{rmk}
The trigonometric factors in $K_p(a)$ have elementary geometric interpretations. The Euclidean metric $g_a'$ induces an inner product on 
$\R^2_{y_1,y_2}$. Then the angles between the asymptotic directions of $S$ are
\[
\begin{cases}
\angle( \mathfrak{D}_1, \mathfrak{D}_3   )= \frac{\pi}{2} + \arctan (  \frac{a_{2\bar{2}} + \text{Re} (a_{1\bar{2}} ) }{ \sqrt{\mathbb{A} }  }    )     ,
 \\
\angle( \mathfrak{D}_2, \mathfrak{D}_3   )= \frac{\pi}{2} + \arctan (  \frac{a_{1\bar{1}} + \text{Re} (a_{1\bar{2}} ) }{ \sqrt{\mathbb{A} }  }    )    .
\end{cases}
\]
\end{rmk}

\begin{rmk}
We have chosen a special ray $(\eta_1, \eta_2)=(\sqrt{-1}b, \sqrt{-1}b)$ to calculate the asymptotic value of $\beta_{p3}+\beta_{p4}$. More generally $\mathfrak{D}_1, \mathfrak{D}_2, \mathfrak{D}_3$ divide the plane $\R^2_{y_1, y_2}$ into three sectors, and the asymptotic value of function
\[
\beta_{p3}+\beta_{p4}= \frac{-2\pi i e^{2\pi i \eta_p}  }{ 1- e^{2\pi i \eta_1}- e^{2\pi i \eta_2}   } + K_p(a)
\]
along the ray $\{ (y_1, y_2)= s\vec{e} \text{ for } s>0 \}$ specified by a directional vector $\vec{e}$ depends on which sector $\vec{e}$ belongs to, and can have a jumping discontinuity as we cross $\mathfrak{D}_i$. This is known as \textbf{Stokes phenomenon} in complex analysis.  
\end{rmk}

\begin{prop}
The functions $\beta_{p3}$ and $\beta_{p4}$ \textbf{solve the overdetermined system} (\ref{overdeterminedsystennegativevertex}). Equivalently, the $(1, 0)$-forms $\zeta_3, \zeta_4$ defined by (\ref{holomorphicdifferentialsnegativevertex}) are \textbf{holomorphic differentials}.
\end{prop}

\begin{proof}
Starting from  the definition of the function $v$ in terms of $\gamma_i$ (\cf (\ref{gammai})), we can differentiate with respect to $\eta_p$ to get
\[
\frac{\partial v}{\partial \eta_p}= -2\pi A^{1/2}\int_S \frac{\partial \gamma}{\partial \eta_p}( \eta_1- \eta_1', \eta_2- \eta_2', \mu   ) d\mathcal{A}(\eta_1', \eta_2').
\]
Using the differential relations in Lemma \ref{gammap34differentialrelations}, 
\[
\begin{split}
2 \frac{\partial v}{\partial \eta_p} &= -2\pi A^{1/2} \int_S \frac{\partial \gamma_{p3}}{\partial \mu}( \eta_1- \eta_1', \eta_2- \eta_2', \mu   ) d\mathcal{A}(\eta_1', \eta_2') \\
& = -2\pi A^{1/2} \frac{\partial }{\partial \mu}  \int_S \gamma_{p3}( \eta_1- \eta_1', \eta_2- \eta_2', \mu   ) d\mathcal{A}(\eta_1', \eta_2') 
\\
& = \frac{\partial \beta_{p3} }{\partial \mu},
\end{split}
\]
and similarly
$
2 \frac{\partial v}{\partial \eta_p}= - \frac{\partial \beta_{p4} }{\partial \mu}.
$

Next we study $\frac{\partial \beta_{p3}}{ \partial \bar{\eta}_q  }$ in the complement of $\{ f_S=0, \mu\leq 0  \}$. We have
\[
\frac{\partial^2 \beta_{p3}}{ \partial \bar{\eta}_q  \partial \mu }= 2\frac{\partial^2 v}{ \partial \eta_p  \partial \bar{\eta}_q }= - \frac{1}{2}  \frac{\partial^2 w^{p\bar{q}}}{ \partial \mu  \partial \mu },
\]
where the second equality uses the distributional equation (\ref{Negativevertexdistributionalequationlinearised}). But by Lemma \ref{gammap34asymptoteforlargemu}, for fixed $\eta_1,\eta_2$,
\[
\lim_{\mu\to +\infty} \frac{\partial \beta_{p3}}{ \partial \bar{\eta}_q  }( \eta_1, \eta_2, \mu )=0,
\]
and the asymptotes we obtained in Section \ref{Negativevertexasymptote1}, \ref{Negativevertexasymptote2} easily imply 
\[
\lim_{\mu\to +\infty} \frac{\partial w^{p\bar{q}}}{ \partial \mu  }( \eta_1, \eta_2, \mu )=0.
\]
Thus we can integrate from $\mu=+\infty$ to obtain
\[
\frac{\partial \beta_{p3}}{ \partial \bar{\eta}_q   }= - \frac{1}{2}  \frac{\partial w^{p\bar{q}}}{  \partial \mu }.
\]
A completely parallel argument shows
\[
\frac{\partial \beta_{p4}}{ \partial \bar{\eta}_q   }=  \frac{1}{2}  \frac{\partial w^{p\bar{q}}}{  \partial \mu }.
\]

Finally by integrating the second part of Lemma \ref{gammap34differentialrelations} we see
\[
\frac{\partial \beta_{p3}}{ \partial \eta_q   }=   \frac{\partial  \beta_{q3}}{  \partial \eta_p }
, \quad \frac{\partial \beta_{p4}}{ \partial \eta_q   }=   \frac{\partial  \beta_{q4}}{  \partial \eta_p } , \quad p, q=1,2.
\]
\end{proof}

To compute the periods of the integrals $\int \zeta_3$ and $\int \zeta_4$, we recall from the topological description (\cf review Section \ref{Negativevertices}, \ref{Joycecritique}) that there are 3 generating $S^1$-cycles in $H_1(T^3)$, one of which is the $S^1$-fibre, and the other two come from lifting  $T^2\subset (\C^*)^2\times \R_\mu$ to the total space, which involve monodromy issues.

\begin{lem}\label{functionalequationNegativevertex}
For appropriate choices of constants $0\leq \theta_{31}^\infty, \theta_{32}^\infty\leq 2\pi$,
the $T^3$-periods of the holomorphic differentials
\begin{equation}
\begin{cases}
d\log z_3=  \zeta_3-  \sqrt{-1} (\theta_{13}^\infty d\eta_1+ \theta_{23}^\infty d\eta_2)
\\
d\log z_4=  \zeta_4+  \sqrt{-1} (\theta_{13}^\infty d\eta_1+ \theta_{23}^\infty d\eta_2)- (  K_1(a)d\eta_1 - K_2(a)d\eta_2   )
\end{cases}
\end{equation}
take values in $2\pi \sqrt{-1}\Z$; here $K_p(a)$ are the constants defined in Lemma \ref{betap3+betap4formula}.
In particular, the \textbf{holomorphic functions} $z_3$ and $z_4$
are defined without multivalue issues. For a suitable choice of multiplicative normalisation on $z_3, z_4$ we have the \textbf{functional equation}
\begin{equation}
z_3z_4 = f_S= 1- z_1- z_2 =1- e^{2\pi i \eta_1}- e^{2\pi i \eta_2} .
\end{equation}
\end{lem}

\begin{proof}
This Lemma is parallel to Lemma \ref{functionalequationpositivevertex}, so we will only highlight the key issues.
The constants $\theta_{13}^\infty$ and $\theta_{23}^\infty$ are the asymptotic holonomy as $\mu \to +\infty$ of the $S^1$-connection $\vartheta$, along the $S^1$-cycles in the base corresponding to the $x_1$ and $x_2$ variables respectively. These are introduced in order to cancel the twist of $\vartheta$ by a flat connection.

The functional equation follows from
\[
\begin{split}
\log z_3+ \log z_4=&  ( \beta_{13}+ \beta_{14} -K_1(a)  )d\eta_1 + ( \beta_{23}+ \beta_{24} -K_2(a)  )d\eta_2 \\
=& - 2\pi i \frac{ e^{2\pi i\eta_1} d\eta_1+ e^{2\pi i \eta_2} d\eta_2 }{f_S } \\
= & d\log f_S.
\end{split}
\]
which crucially uses Lemma \ref{betap3+betap4formula}.
\end{proof}

We have thus defined a holomorphic map away from the singular locus of the $S^1$-fibration on the negative vertex  $M^-$:
\[
M^-\setminus S\to \{ z_3 z_4= 1- z_1- z_2       \} \subset \C^*_{z_1}\times \C^*_{z_2} \times  \C^2_{z_3, z_4}.
\]
Here the functional equation allows us to extend the map holomorphically across $\{\mu\neq 0, f_S=0\}$. However the complex structure on $M^-$ is not a priori defined along $S\cap M^-$.

\begin{lem}\label{Negativevertexcontinuousextensionofz3z4}
The holomorphic functions $z_3, z_4$ on $M^-$ extend continuously over the singular locus $S$ where they attain the value zero. Morever $z_3, z_4$ are $C^{2,\alpha}$-regular with respect to $g^{(1)}$-metric.
\end{lem}

\begin{proof}
By construction $\log |z_3|$ is a function of $\eta_1, \eta_2, \mu$ with differential
\[
d\log |z_3|= V_{(1)} d\mu +  \text{Re}( \beta_{13} d\eta_1+ \beta_{23} d\eta_2  )+ \text{Im}( \theta^\infty_{13} d\eta_1+ \theta_{23}^\infty d\eta_2   ).
\]
In particular the positivity of $V_{(1)}$ in $M^-$ means $\log |z_3|$ is increasing in $\mu$. Around a given point $P\in S\cap M^- $, we first show continuity of $z_3$ at $P$. Observe
\[
\log |z_3|(\eta_1, \eta_2, \mu  )= \log |z_3|(\eta_1, \eta_2, A^{-1/4}  ) + \int_{A^{-1/4}}^{\mu} V_{(1)} d\mu .
\] 
Here $\log |z_3|( \eta_1, \eta_2, \mu=A^{-1/4}  )$ is locally $L^\infty$ by smoothness of $\log z_3$ in $\{ \mu>0 \}$. Applying Proposition \ref{NegativevertexleadingorderasymptotenearDelta1} and neglecting all locally bounded terms, as $(\eta_1, \eta_2, \mu)\to (\eta_1(P), \eta_2(P), 0    )$,
\[
\log |z_3|\sim  \int_{A^{-1/4}}^{\mu} \frac{A^{1/2}}{2R} d\mu \sim \frac{1}{2}\log (\frac{ A^{1/2}\mu+ R}{ A^{1/4} }  )  \to -\infty, 
\]
or equivalently $|z_3|\to 0$ as required. 
The case of $z_4$ is completely analogous.

Since $(g^{(1)}, \Omega^{(1)})$ is $C^{1,\alpha}$-regular by Proposition \ref{TransverseTaubNUTmetric},  holomorphicity implies that $z_3,z_4$ are $C^{2,\alpha}$-regular in the local chart of Section \ref{StructurenearDeltanegativevertex}.
\end{proof}

\begin{prop}\label{negativevertexcomplexstructure}
	The map $M^-\to \{  z_3z_4=1-z_1-z_2      \}$ is a \textbf{holomorphic open embedding}. The \textbf{$S^1$-action} is identified as
	\[
	e^{i\theta}\cdot (z_1, z_2, z_3, z_4)= ( z_1, z_2,   e^{i\theta}z_3, e^{-i\theta_1}z_4 ), 
	\] 
	and the \textbf{holomorphic volume form} is 
	$\Omega^{(1)}= - \frac{\sqrt{-1}}{4\pi^2 z_1 z_2 } dz_2 \wedge dz_3\wedge dz_4    $. 
\end{prop}

\begin{proof}
The $S^1$-action can be identified as in Proposition \ref{C3complexstructure}. The holomorphic volume form is characterised by $\iota_{ \frac{\partial}{\partial \theta}  } \Omega^{(1)} = d\eta_1\wedge d\eta_2$, which is compared to
\[
\begin{split}
\iota_{ \frac{\partial}{\partial \theta}  }  (\sqrt{-1} dz_2 \wedge dz_3\wedge dz_4)= & - d(z_3z_4)\wedge dz_2= -d(1-z_1-z_2) \wedge dz_2 = dz_1\wedge dz_2 \\
=& ( 2\pi \sqrt{-1} z_1 d\eta_1  )\wedge ( 2\pi \sqrt{-1} z_2 d\eta_2  )= -4\pi^2 z_1 z_2 d\eta_1 \wedge d\eta_2,
\end{split}
\]
to yield $\Omega^{(1)}= - \frac{\sqrt{-1}}{4\pi^2 z_1 z_2 } dz_2 \wedge dz_3\wedge dz_4  $.

This holomorphic volume form formula in particular shows the map $M^-\to \{ z_3z_4=1-z_1-z_2 \}$ is a local biholomorphism wherever the complex structure is defined. We finally need to show this map is injective. Since both $M^-$ and $\{ z_3z_4=1-z_1-z_2 \}$ fibre over $\C^*_{z_1}\times \C^*_{z_2}$ in a compatible way, it suffices to compare the $\C^*$-fibres. The map between the fibres is equivariant with respect to the $S^1$-action, so to conclude injectivity we only need to recall from the proof of Lemma \ref{Negativevertexcontinuousextensionofz3z4} that $\log |z_3|$ is a monotone function of $\mu$. 
\end{proof}

\section{Weighted H\"older norms and initial error estimate}\label{WeightedHoldernormsandinitialerrorsNegativevertex}

The following few Sections are aimed at perturbing the K\"ahler ansatz into a Calabi-Yau metric. This Section sets up the weighted H\"older norms and measure the \textbf{volume form error}
\begin{equation}\label{volumeformerrornegativevertex}
E^{(1)}= \frac{ \det( W^{p\bar{q}}_{(1)} ) }{   V_{(1)}} -1 = \frac{ A+ Aa^{p\bar{q}} w^{p\bar{q}}+ \det( w^{p\bar{q}} )  }{  A+ v   }-1 =  \frac{  \det( w^{p\bar{q}} )  }{  A+ v   }.
\end{equation}
 There are three  \textbf{weight parameters}: 
\[
\varrho= |(y_1, y_2, \mu)|_a', \quad R=\text{dist}_{g_a}(\cdot, S), \quad \tilde{\ell}= \kappa_a \text{dist}_{g_a'}(\cdot, \text{Im}(S)).
\]
The parameter $\tilde{\ell}$ is useful for measuring exponential decay rates (\cf Proposition \ref{Exponentialdecayfor higherFouriermodesinthefirstorderansatzNegativevertex}). 
The following definitions are  parallel to Section \ref{WeightedHoldernormsandinitialerror}.

Let $\delta\leq 0$. We shall define the \textbf{weighted H\"older norms} $\norm{T}_{C^{k,\alpha}_{\delta,0} }$  for $S^1$-invariant tensor fields $T$
on $M^-$, by prescribing the norm on a number of overlapping regions up to uniform equivalence. 

\begin{itemize}
	\item The region  $\{ R\lesssim A^{1/4} \}$ is covered by local charts $\{ r\lesssim A^{1/4} \}$ introduced in Section \ref{StructurenearDeltanegativevertex} and \ref{StructurenearDeltanegativevertexII}, where the ansatz metric is approximated by $g_{\text{NUT} }$. Let $\norm{T}_{C^{k,\alpha}_{\delta,0}}$ be uniformly equivalent to the norm $\norm{\Psi^*T}_{C^{k,\alpha}_{ \delta} (g_{\text{NUT} } ) }$ in Section \ref{StructurenearDeltanegativevertexII}.
	Inside $\{ R\lesssim A^{-1/2}  \}$ the metric ansatz is $C^{1,\alpha}$-regular, so correspondingly we should work with functions of at most $C^{2,\alpha}$-regularity and tensors of at most $C^{1,\alpha}$-regularity. Inside $\{ R\gtrsim A^{-1/2}  \}$ there is no restriction on regularity.

	\item The region $\{ \tilde{\ell}\gtrsim 1\}$ can be covered by subregions of diameter $\sim R$, where the $S^1$-bundle is topologically trivial. Over each subregion the metric is approximated by the periodic version of the constant solution $g_{\text{flat}}$  (\cf Section \ref{MetricbehaviourawayfromDelta}). The $x_1,x_2$ variables define two periodic direction.
	We decompose $T$ into the part $\bar{T}$ independent of $x_1,x_2$ (the `zeroth Fourier mode') and the oscillatory part $T-\bar{T}$ (the `higher Fourier mode'), and define the weighted H\"older norm separately on the two parts:
	\item
	On the zeroth Fourier mode, the norm $\norm{\bar{T}}_{C^{k,\alpha}_{\delta,0}}$ is equivalent to
	\[
	A^{-3\delta/4}(\sum_{j=0}^k \norm{  R^j \nabla^j \bar{T} }_{ L^\infty   }
	+
	[  R^k \nabla^k \bar{T}    ]_\alpha),
	\]
	where $[]_\alpha$ denotes the appropriately normalised H\"older seminorm.
	\item
	On the higher Fourier modes we build in the exponential decay. Fix a parameter $0<\kappa< 1$. The norm $\norm{T-\bar{T}}_{C^{k,\alpha}_{\delta,0}}$ in this region is equivalent to
	\[
	A^{-3\delta/4}
	\sup_{ 
		\tilde{\ell}\gtrsim 1  } e^{\kappa \tilde{\ell}   } (   \sum_{j=0}^k \norm{  A^{j/4} \nabla^j (T-\bar{T})}_{ L^\infty   }
	+ A^{k/4}
	[   \nabla^k ( T-\bar{T})    ]_\alpha    ).
	\]
	An estimate in this norm is the higher order version of
	$
	|T-\bar{T}| \leq CA^{3\delta/4} e^{-\kappa \tilde{\ell}}.
	$
\end{itemize}

\begin{notation}
	The norm $\norm{\cdot }_{C^{k,\alpha}_{\delta,0} }$ can refer to any type of tensors depending on the context, such as functions, 1-forms, symmetric 2-tensors, and in some cases can refer to the norm computed in a subregion. Strictly speaking this norm depends on $\kappa$, but we suppress this to avoid cluttering the notation.
\end{notation}

We will also need a \textbf{variant} weighted H\"older norm $\norm{T}_{C^{k,\alpha}_\delta}$. The only difference from $\norm{T}_{C^{k,\alpha}_{\delta,0} }$ is that in the region $\{\tilde{\ell}\gtrsim 1  \}$ on the zeroth Fourier mode, 
$\norm{\bar{T}}_{C^{k,\alpha}_{\delta}}$ is equivalent to
\[
A^{-\delta/2}(\sum_{j=0}^k \norm{  R^{j-\delta} \nabla^j \bar{T}}_{ L^\infty   }
+
[  R^{k-\delta} \nabla^k \bar{T}    ]_\alpha),
\]
so an estimate in this norm is the higher order version of $|\bar{T}|= O( A^{3\delta/4}(A^{-1/4}R) ^{\delta}  )$.
We have inserted an extra decay factor $(A^{-1/4}R)^{\delta}$.

\begin{notation}
	For a parameter $\nu$ with $1\ll \nu< \epsilon_0 A^{3/4}$, define the subregion of $M^-$
	\[
	M^-_\nu= \{   A^{-1/4}\varrho < e^\nu          \} \subset M^-.
	\]
	Its base is $\mathcal{B}^-_\nu= \{   A^{-1/4}\varrho < e^\nu          \} \subset (\C^*)^2_{\eta_1,\eta_2}\times \R_\mu$.
\end{notation}

\begin{lem}\label{volumeformerrornegativevertexlemma}
The \textbf{volume form error} $E^{(1)}$  satisfies the estimate on $M^-_\nu$:
\[
\norm{ E^{(1)}  }_{ C^{1,\alpha}_{-1,0}    }  \leq CA^{-3/4}\nu^2.
\]
In the subregion $\{ R\gtrsim A^{1/4}   \}\subset M^-_\nu$, and any fixed large $k$,
\[
\norm{ E^{(1)}  }_{ C^{k,\alpha}_{-1,0}    }  \leq CA^{-3/4}\nu^2.
\]
\end{lem}

\begin{proof}
(Sketch) In the region $\{ R\lesssim A^{1/4} \}$ the volume form estimate follows from Proposition \ref{NegativevertexleadingorderasymptotenearDelta1} and \ref{NegativevertexleadingorderasymptotenearDelta2}. In the region $\{ R\gtrsim A^{1/4} \}$ the absolute estimate follows by combining Section \ref{Negativevertexasymptote1}, \ref{Negativevertexasymptote2}, notably the exponential decay estimate in Proposition \ref{Exponentialdecayfor higherFouriermodesinthefirstorderansatzNegativevertex}, and the higher order estimate uses the $\Lap_a$-harmonicity on $w^{p\bar{q}}$.
\end{proof}

\section{Harmonic analysis I: periodic Euclidean region}\label{HarmonicanalysisInegativevertex}

The refined mapping properties of the \textbf{Euclidean Green operator} $\Lap_a^{-1}$ on $(\C^*)^2_{\eta_1, \eta_2}\times \R_\mu$ follow Section \ref{HarmonicanalysisI} almost verbatim:

\begin{prop}\label{harmonicanalysisImainNegativevertex}
(\textbf{Periodic Euclidean region}) 
Let $-3<\delta<0$. 	
Let $f$ be a function compactly supported in $\mathcal{B}^-_\nu\cap
\{  R> A^{-1/2}   \}
$
with $\norm{f}_{ C^{k,\alpha}_{\delta,0} }\leq 1$ (respectively $\norm{f}_{ C^{k,\alpha}_{\delta} }\leq 1$).
Then $\Lap_a^{-1}f$ satisfies the $g_a$-Hessian bound on $\mathcal{B}^-_\nu\cap
\{  R\gtrsim A^{-1/2}   \}
$,
\[
\norm{ \nabla^2_{g_a} \Lap_a^{-1} f}_{ C^{k,\alpha}_{\delta,0}  } \leq C\nu, \quad \text{resp. } \norm{ \nabla^2_{g_a} \Lap_a^{-1} f}_{ C^{k,\alpha}_{\delta}  } \leq C.
\]
\end{prop}

\begin{rmk}
The regularity of $\Lap_a^{-1}f$ in $\{ R\lesssim A^{-1/2} \}$ is well controlled by $\Lap_a$-harmonicity.
\end{rmk}

This allows us to correct the volume form error sufficiently away from $S$ as in proposition \ref{PerturbationintheEulideanregionprop}. From now on $1\ll \nu\ll A^{3/8}$.

\begin{prop}\label{firstcorrectionperiodicregionnegativevertex}
	Let $1\ll \nu \ll A^{3/8}$. Then there is a real valued function $\varphi_1$ on $\mathcal{B}^-_\nu$, solving the generalised Gibbons-Hawking equation on $ \mathcal{B}^-_{\nu}\cap \{ \tilde{\ell} >2   \} $
	\[
	V_{(2)}= V_{(1)}+  \frac{\partial^2 \varphi_1}{\partial \mu \partial \mu }, \quad  W^{p\bar{q}}_{(2)} =  W^{p\bar{q}}_{(1)}-4 \frac{\partial^2 \varphi_1}{\partial \eta_p \partial \bar{\eta}_q } ,\quad 
	\det(   W^{p\bar{q}}_{(2)}   )= V_{(2)}.
	\]
	Morever $\varphi_1$ is $\Lap_a$-harmonic on $\mathcal{B}^-_\nu\cap \{ \tilde{\ell} <1  \}$, and
	\[
	\norm{ \nabla^2_{g_a} \varphi_1 }_{ C^{k,\alpha}_{-1,0}(\mathcal{B}^-_\nu\cap \{  \tilde{\ell}\gtrsim 1 \} )  } \leq  C \nu^3 A^{-3/4},
	\]
	and $ |\nabla^2_{g_a} \varphi_1|_{g_a} \leq C\nu^3A^{-3/2}   $ on $\mathcal{B}^-_\nu$. In particular the matrix $(  W^{p\bar{q}}_{(2)})$ is positive definite and $V_{(2)}$ is positive on $\mathcal{B}^-_\nu$. 
	
\end{prop}

We obtain by the generalised Gibbons-Hawking construction  $(g^{(2)}, \omega^{(2)}, J, \Omega)$ associated to the data $V_{(2)}$ and $W^{p\bar{q}}_{(2)}$,
and identify its ambient space as $M^-_\nu$.
The new $S^1$-connection $\vartheta^{(2)}$ is related to $\vartheta$ by
\begin{equation*}
\vartheta^{(2)}- \vartheta= \sqrt{-1} ( \frac{\partial^2 \varphi^-}{\partial \eta_p \partial \mu  } d\eta_p -    \frac{\partial^2 \varphi^-}{\partial \bar{\eta}_p \partial \mu  } d\bar{\eta}_p     ).
\end{equation*}
The new volume form error $E^{(2)}$ is  supported in $\{ \ell\lesssim 1   \}$ with bound 
\begin{equation}\label{volumeformerrorE2negativevertex}
E^{(2)}= \frac{ \det( W^{p\bar{q}}_{(2)}  )}{ V_{(2)}  }-1,\quad  \norm{ E^{(2)} }_{ C^{1,\alpha}_{-1}(M^-_\nu) } \leq CA^{-3/4}\nu^2,
\end{equation}
and in particular $\norm{ E^{(2)} }_{ C^{\alpha}_{-1-\epsilon }(M^-_\nu) } \leq CA^{-3/4 }\nu^2 $.

Henceforth the holomorphic structures will be fixed, and can be identified building on results in Section \ref{ComplexgeometricperspectiveNegativevertex}. The new holomorphic differentials $d\log Z_3, d\log Z_4$ are related to $d\log z_3, d\log z_4$ by 
\[
\begin{cases}
d\log Z_3= d\log z_3+ d ( \frac{\partial  \varphi_1}{\partial \mu}  ),
\\
d\log Z_4= d\log z_4- d ( \frac{\partial  \varphi_1}{\partial \mu}  ),
\end{cases}
\]
whence we find holomorphic coordinates by integration
\begin{equation}\label{Z3Z4z3z4negativevertex}
Z_3= z_3 \exp(  \frac{\partial  \varphi_1}{\partial \mu}  ), \quad Z_4= z_4 \exp(    -\frac{\partial  \varphi_1}{\partial \mu}),
\end{equation}
which satisfy the functional equation
\[
Z_3Z_4= z_3z_4= 1-z_1- z_2= 1- e^{2\pi i \eta_1}- e^{2\pi i \eta_2}.
\]
Following Lemma \ref{Negativevertexcontinuousextensionofz3z4} and Proposition \ref{negativevertexcomplexstructure},

\begin{prop}
	(\textbf{Holomorphic structure})
	The map 
	\[
	M^-_\nu\to \{ Z_3Z_4= 1-z_1-z_2   \}\subset \C^2_{Z_3, Z_4}\times (\C^*)^2_{z_1, z_2}\]
	extends continuously over the singular locus $S\cap M^-_\nu$ and defines a holomorphic open embedding under the complex structure $J$. The $S^1$-action is identified as
	\[
	e^{i\theta}\cdot(z_1, z_2, Z_3, Z_4 )=(z_1,z_2, e^{i\theta}Z_3, e^{-i\theta}Z_4),
	\]
	and the holomorphic volume form is $\Omega=- \frac{\sqrt{-1}}{4\pi^2z_1z_2} dz_2\wedge dZ_3\wedge dZ_4.$ The K\"ahler structure is $C^{1,\alpha}$-regular near $S$.
\end{prop}

\begin{prop}\label{symplecticarea2}
	(\textbf{Symplectic structure}) The integral 
	$
	\int_{T^2} \omega^{(2)}= \text{Im} (a_{2\bar{1}}).
	$
\end{prop}

\begin{proof}
	The new K\"ahler form $\omega^{(2)}$ is cohomologous to $\omega^{(1)}$ by the formula
	\[
	\omega_- = \omega^{(1)} + d(  \sqrt{-1} \frac{\partial \varphi^-}{\partial \eta_p} d\eta_p - \sqrt{-1} \frac{\partial \varphi^-}{\partial \bar{\eta}_p} d\bar{\eta}_p        ).
	\]
	The claim then follows from Lemma \ref{symplecticarea1}.
\end{proof}

\section{Harmonic analysis II: Neighbourhood of $S$}\label{HarmonicanalysisIInegativevertex}

This Section approximately inverts $\mathcal{L}$ for source functions supported in the vicinity of $S$.

Recall the model metric $g_{\text{NUT}}$ defined in (\ref{tildegTaubdefinition}) as the product of the Taub-NUT metric with flat $\C$. Denote $r= \sqrt{  A(\mu_1^2+ |\xi_1|^2+ |\xi_2|^2 )}$ in the local chart around $P\in S$, as in Section \ref{StructurenearDeltanegativevertex}.

\begin{lem}(Model Laplacian)\label{LocalparametrixI}
Given $0<\epsilon \ll 1$, let $f$ be an $S^1$-invariant function on the model space supported in $\{  r\lesssim A^{-1/2}  \}$, with bound $\norm{f}_{ C^{k,\alpha}_0(g_{\text{NUT}} )   }\leq 1$. Then there is a function $u$ compactly supported in $\{ r\lesssim A^{1/4}   \}$, such that on an annulus region at any given dyadic scale $r\sim r_0$ (or the ball region $ r\lesssim A^{-1/2} $) we have a decay estimate
\[
\norm{ u}_{  C^{k+2,\alpha}_{0}( r\sim r_0, g_\text{NUT}) } \leq  CA^{-1} (A^{1/2}r_0+1)^{-3+\epsilon} ,
\]
and $ \Lap_{ g_{\text{NUT}} } u- f$ is only supported on one dyadic scale $\{  r\sim A^{1/4}   \}$ with the bound
\[
\norm{ \Lap_{ g_{\text{NUT}} } u- f}_{ C^{k,\alpha}_{-2}(g_{\text{NUT}})      } \leq CA^{ 3(-3+\epsilon )/4 } .
\]
\end{lem}

\begin{proof}
Let ${G}_{\text{NUT}}$ be the Green operator on the model space, so $u'= {G}_{\text{NUT}} f$  is an $S^1$-invariant function.  Applying Hein's package on Poisson equations as in Corollary \ref{harmonicanalysisonTaubNUTmodel} with a simple scaling argument, the function $u'$ decays like
\[
|u'| \leq CA^{ -1  } ( A^{1/2}r+1 )^{-3+\epsilon}.
\]
The decay exponent $-3+\epsilon$ here comes from the quintic volume growth rate of $g_{\text{NUT}}$.
Since the model space is smooth, we can bootstrap this to a weighted $C^{2,\alpha}$ estimate on $u'$. The function $u$ is obtained by cutting off $u'$ at a dyadic scale $ r\sim A^{1/4}   $. The cutoff error is controlled by the Hessian estimate.
\end{proof}

The next Lemma patches together a large number of local parametrices to produce an approximate right inverse of $\Lap_{g^{(2)}}$ in a neighbourhood of $S$, with sufficiently fast decay estimates. Its proof is similar to Lemma \ref{harmonicanalysislemma3}.

\begin{lem}\label{LocalparametrixII}
(Neighbourhood of $S$) 
Given $0<\epsilon\ll 1$ and $-3+\epsilon<\delta<0$, let $f$ be an $S^1$-invariant function compactly supported in $M^-_\nu \cap \{R\lesssim A^{-1/2}  \} $ with bound $\norm{f}_{ C^\alpha_\delta  }\leq 1$. Then there is a function $u$ on $\mathcal{B}$ supported in $\{ R\lesssim A^{1/4}   \}$, with decay estimates
$
\norm{ u}_{  C^{2,\alpha}_{-1+\epsilon}   } \leq CA^{-1},
$
such that
$
\norm{ \Lap_{g^{(2)} } u -f }_{ C^\alpha_\delta }  \ll 1.
$
\end{lem}

\begin{proof}
Take a large collection of points $\{ P_i\}_{i=1}^N$ on $S\cap \mathcal{B}^-_\nu$, such that for any point $P$ on $S\cap \mathcal{B}^-_\nu$, the number of points $P_i$ in the collection  within $g_a$-distance $O(A^{-1/2})$ to $P$ is at least one but no more than $C$. Then take cutoff functions $\chi_i$ on $\mathcal{B}^-_\nu$ supported in $\{ \text{dist}_{g_a}(\cdot, P_i)\lesssim A^{-1/2}  \}$ such that $\sum \chi_i=1$ on $\mathcal{B}^-_\nu\cap \{ R\lesssim A^{-1/2}  \}$. These allow us to decompose $f$ into a large number of localised contributions:
\[
f= \sum_{i=1}^N  \chi_i f, \quad \norm{ \chi_i f}_{ C^\alpha_0  } \leq C,
\]
using the fact that the $C^\alpha_\delta$-norm is not sensitive to $\delta$ inside $\{ R\lesssim A^{-1/2}  \}$.

For each term $\chi_i f$ we apply Lemma \ref{LocalparametrixI} to produce an approximate local solution $u_i$ on $\{ r_{P_i}\lesssim A^{1/4}  \}$, with bounds prescribed in Lemma \ref{LocalparametrixI}. Here $r_{P_i}(Q)$ is uniformly equivalent to $|P_i-Q|_{g_a}$. The candidate solution is
\[
u= u_1+ u_2+ \ldots u_N.
\]
By construction $u$ is supported in $\{  R\lesssim A^{1/4}  \}$.

We now bound $u$, focusing on the absolute estimate. Summing up the contributions
\[
|u_i|\leq CA^{ -1  } ( A^{1/2}  r_{P_i} +1 )^{-3+\epsilon}\norm{ \chi_i f}_{  C^\alpha_0  } \leq CA^{ -1  } ( A^{1/2}r_{P_i} +1 )^{-3+\epsilon}, 
\]
we estimate at a point $Q$:
\[
\begin{split}
|u(Q)|\leq &\sum_{ |P_i-Q|_{g_a}\lesssim A^{1/4}   } |u_i|(Q) 
\\
\leq & CA^{-1} \sum_{|P_i-Q|_{g_a}\lesssim A^{1/4}  } (A^{1/2}| P_i-Q  |_{g_a}+1)^{-3+\epsilon}
\\
\leq & C \int_{S\cap \{ |P-Q|_{g_a}\lesssim A^{1/4}  \} }(A^{1/2}| P-Q  |_{g_a}+1)^{-3+\epsilon}d\mathcal{A}(P)
\\
\leq & CA^{-1}  (A^{1/2}R(Q)+1)^{-1+\epsilon}.
\end{split}
\]
The higher order version is $
\norm{ u}_{  C^{2,\alpha}_{-1-\epsilon}    } \leq CA^{-1},
$
so $\norm{ u}_{  C^{2,\alpha}_{\delta}    } \leq CA^{-1}
$ for $\delta>-3+\epsilon$.

Next we estimate the error $\Lap_{g^{(2)} } u-f$. The error $\Lap_{g^{(2)} } u_i- \chi_i f$ has two sources: the cutoff error supported on $\{ r_{P_i}\sim A^{1/4}\}$ from Lemma \ref{LocalparametrixI}
\[
\norm{ \Lap_{ g_{\text{NUT}} } u_i- \chi_i f}_{ C^{\alpha}_{-2}(g_{\text{NUT}})      } \leq CA^{ 3(-3+\epsilon )/4 } ,
\]
and the metric deviation error $\Lap_{ g_{\text{NUT}} } u_i- \Lap_{g^{(2)} } f$, which is controlled because by Proposition \ref{TransverseTaubNUTmetric} the local diffeomorphism $\Psi$ is a $C^{1,\alpha}$-approximate isometry between $ g_{\text{NUT}}$ and $g^{(2)}$, and $f$ has weighted $C^{2,\alpha}$ control.  We focus on the absolute estimate:
\[
|\Lap_{g^{(2)}} u_i- \Lap_{g_ {\text{NUT} } }u_i | \leq  C A^{ -3/4 } (A^{1/2}r_{P_i}+1 )^{-3+ \epsilon}(A^{1/2}R+1)^{-2}
(  A|\xi_2|+ \nu   ) ,
\]
so that 
\[
|\Lap_{g^{(2)}} u_i- \chi_i f | \leq  C A^{ -3/4 } (A^{1/2}r_{P_i}+1 )^{-3+ \epsilon}(A^{1/2}R+1)^{-2}
(  A^{1/2}r_{P_i}+ \nu   ) .
\]
Using
\[
\begin{cases}
\sum_{  |P_i-Q|_{g_a}\lesssim A^{1/4}   } (A^{1/2}r_{P_i}(Q)+1)^{-3+\epsilon}\leq (A^{1/2}R (Q)+1)^{-1+\epsilon}, 
\\
\sum_{  |P_i-Q|_{g_a}\lesssim A^{1/4}   } (Ar_{P_i}(Q)+1)^{-2+\epsilon}\leq CA^{3\epsilon/4}, 
\end{cases}
\]
we sum up all contributions to deduce for $\delta>-3+\epsilon$,
\[
\begin{split}
& |\Lap_{g^{(2)}} u-  f |
\leq \sum_{ |P_i-Q|_{g_a}\lesssim A^{1/4}   } |\Lap_{g^{(2)}} u_i- \chi_i f | 
\\
\leq & CA^{-3/4}  \{  A^{3\epsilon/4}( A^{1/2}R+1 )^{-2}+ \nu( A^{1/2}R+1     )^{-3+ \epsilon}         \}
\ll ( A^{1/2}R+1 )^{\delta}.
\end{split}
\]
The H\"older version is 
$
\norm{ \Lap_{g^{(2)} } u-f}_{  C^\alpha_{\delta}   } \ll 1
$ as required.
\end{proof}

\section{Harmonic analysis III: perturbation to Calabi-Yau metric}\label{HarmonicanalysisIII}

We now \emph{shift to the complex geometric perspective} and solve the complex Monge-Amp\`ere equation by perturbative methods. Below is main result of the linear theory, which is parallel to Proposition \ref{harmonicanalysismain} and \ref{harmonicanalysisIImain}. The idea is to patch together as in Proposition \ref{harmonicanalysisIImain} the local parametrices provided by Proposition \ref{harmonicanalysisImainNegativevertex} and \ref{LocalparametrixII}.

\begin{prop}\label{Parametrixconstruction}
Given $-3<\delta< -1$ and  $1\ll \nu \ll A^{3/8}$, let $f$ be an $S^1$-invariant function compactly supported in  $M^-_\nu$ with norm $\norm{f}_{C^\alpha_{\delta}}= 1$. Then there is an $S^1$-invariant function $u$ approximately solving the Poisson equation:
\[
\norm{ \Lap_{g^{(2)}} u-f }_{ C^\alpha_{\delta} (M^-_\nu) } \ll 1, 
\]
with the Hessian bound
\[
\norm{  \nabla^2_{g^{(2)}}  u }_{  C^\alpha_{\delta}(M^-_\nu)  } \leq C, 
\quad 
\norm{  d  u }_{  C^{1,\alpha}_{\delta+1}(M^-_\nu)  } \leq CA^{-1/2}.
\]
The constants depend only on $\delta, \alpha, \kappa$ and the scale invariant ellipticity bound on $a_{p\bar{q}}$.
\end{prop}

Combined with the initial error estimate (\ref{volumeformerrorE2negativevertex})
this  allows us to set up a Banach iteration scheme to perturb $\omega^{(2)}$ to a Calabi-Yau metric, parallel to Theorem \ref{OoguriVafatypemetriconpositivevertextheorem}. This involves shrinking domain from $M^-_{\nu}$ to $M^-_{\nu-1}$ and changing  $\nu$ to $\nu+1$.

\begin{thm}\label{NegativevertexSolveGibbonsHawking}
(\textbf{Ooguri-Vafa type metric on the negative vertex})
Fix $k,\alpha, \kappa$ and $0<\epsilon\ll 1$, and
let $1\ll \nu \ll A^{3/8}$. Then there is an $S^1$-invariant Calabi-Yau metric on $M^-_\nu$ with $S^1$-invariant K\"ahler potential $\phi^-$,
\[
\omega_-= \omega^{(2)}+ \sqrt{-1} \partial \bar{\partial} \phi^-, \quad  \omega_-^3= \frac{3}{4} \sqrt{-1}\Omega\wedge \overline{\Omega},
\]
with metric deviation estimate	
$
\norm{ \nabla^2_{g^{(2)}} \phi^- }_{ C^\alpha_{-1-\epsilon}(\mathcal{B}^-_\nu)  } \leq  C \nu^2 A^{-3/4}. 
$	
The constants depend only on $\alpha, \epsilon, \kappa$ and the scale invariant ellipticity bound on $a_{p\bar{q}}$.
\end{thm}

\section{Ooguri-Vafa type metrics on the negative vertex}\label{OoguriVafatypemetriconNegativevertecSection}

We discuss the geometric aspects of the Ooguri-Vafa type metric $(g_-, \omega_-,\Omega)$.

Combining deviation estimates in Proposition \ref{Exponentialdecayfor higherFouriermodesinthefirstorderansatzNegativevertex},  \ref{firstcorrectionperiodicregionnegativevertex}  and  Theorem \ref{NegativevertexSolveGibbonsHawking},

\begin{cor}
(\textbf{Exponential decay to semiflat metric away from $S$}) In the situation of Theorem \ref{NegativevertexSolveGibbonsHawking}, for $ \tilde{\ell}\gtrsim 1  $ the deviation of $\omega_-$ from its zeroth Fourier mode $\overline{\omega}_-$ decays exponentially:
\[
|\omega_- - \overline{\omega}_- | \leq C A^{-3/4}\nu e^{-  \kappa\tilde{\ell}}  .
\]
\end{cor}

The transverse structure along $S$ follows immediately from  Proposition \ref{TransverseTaubNUTmetric} and the metric deviation estimates.

\begin{cor}
(\textbf{Transverse Taub-NUT metrics}) Under suitable gauge choices for the $S^1$-connection, the metric $g_-$ restricted over the disc $\{  \xi_2=0, R\lesssim A^{1/4}  \} $ is approximated by the Taub-NUT metric:
\[
| (g_- - g_{\text{NUT}})|_{\xi_2=0}  |_{ g_{\text{NUT}} } \leq  C    A^{-3/4}\nu .  
\]
\end{cor}

The smooth topology emerges a posteriori after solving the Monge-Amp\`ere equation as a consequence of regularity theory.

\begin{prop}
 The metric structure $(g_-, \omega_-, J, \Omega)$ \textbf{extends smoothly} to a Calabi-Yau metric on $M^-_\nu$.
\end{prop}

\begin{proof}
The holomorphic coordinates extend  across $S$, so induces a smooth structure on $M^-_\nu$. The metric $\omega_-$ has $C^\alpha$-regularity and is Calabi-Yau, so standard regularity theory of complex Monge-Amp\`ere equation implies that it is smooth.
\end{proof}

\section{Special Lagrangian geometry}\label{SpecialLagrangiangeometry}

This Section is an informal discussion concerning special Lagrangian 3-tori $L\simeq T^3$ on the negative vertex  $M^-_\nu$.

The generic special Lagrangian 3-tori $L$ are expected to be isotopic to the $T^3$ lying over the 2-tori in the 5-dimensional base defined by
\begin{equation}\label{affinesolutionspecialLagrangian}
(y_1, y_2, \mu)= \text{const}.
\end{equation}
The Lagrangian requirement then imposes a \textbf{homological constraint} in the light of Proposition \ref{symplecticarea2}:
\begin{equation}
\int_{T^2} \omega_-=  \text{Im}( a_{2\bar{1}}  )=0,
\end{equation}
or equivalently $(a_{p\bar{q}})$ is real symmetric.
The interpretation is that the Ooguri-Vafa type metrics we constructed on the negative vertex can be the metric model for the SYZ fibration \emph{only if} the homological constraint is satisfied; when this fails, they may still be the local model for other types of degenerating 3-fold Calabi-Yau metrics which \emph{do not} admit a global special Lagrangian 3-torus fibration.

From now on in this Section we assume the homological constraint, and proceed to speculate on the geometric features of the special Lagrangian 3-tori $L$, without attempting to prove existence results.

First, notice that outside a tubular neighbourhood of the singular locus $S$, the metric $g_-$ is a perturbation of the flat model $g_{\text{flat}}$ (\cf Example \ref{Constantsolution}), namely the generalised Gibbons-Hawking construction applied to the constant  solution
\[
V= A, \quad W^{p\bar{q}}= a_{p\bar{q}}.
\] 
On the flat model it is elementary to check that the map to $\R^3$ defined by $(y_1, y_2, \mu)$ have special Lagrangian fibres, which are flat 3-tori invariant under the \text{$S^1$-action}. In other words, to crudest approximation the map
\begin{equation}\label{specialLagrangianfibrationnegativevertex}
M^-_\nu \xrightarrow{(y_1,y_2,\mu)} \R^3
\end{equation}
is an \textbf{approximate special Lagrangian fibration}. Most of these $T^3$-fibres stay far away from the curvature radius along $S$, so it is likely that in the \textbf{generic region} these $T^3$ can be perturbed into a genuine special Lagrangian fibration with respect to the Calabi-Yau structure $(g_-, \omega_-, J, \Omega)$, while maintaining the $S^1$-invariance.

Near $S$ the features of the special Lagrangians $L$ have strong resonance with Joyce's work \cite{Joyce} (\cf Section \ref{Joycecritique}). 
It is natural to expect $L$ to be \textbf{$S^1$-invariant}. Around $P\in S$, the Calabi-Yau structure is transversely modelled on $g_{\text{NUT}}$ (\cf Section \ref{StructurenearDeltanegativevertexII}), so the $S^1$-reduction of the special Lagrangian condition 
\[
\omega_-|_L=0,  \quad \text{Im}(\Omega_-)|_{L}=0
\]
approximately reads:
\begin{equation}\label{specialLagrangiancondition2}
\begin{cases}
\mu= \text{constant}, \\
\{(A+ \frac{1}{ 2\sqrt{\mu^2+|\xi_1|^2 }  }) d\xi_1\wedge d\bar{\xi}_1 +Ad\xi_2\wedge d\bar{\xi}_2 \} |_{L/S^1}=0, \\
\text{Im}( d\xi_1\wedge d\xi_2)|_{L/S^1}=0.
\end{cases}
\end{equation}
To render the analogy with Joyce \cite{Joyce} more transparent, we introduce real variables $\tilde{x}_1, \tilde{x}_2, \tilde{u}_1, \tilde{u}_2$ such that
$
\xi_1= \tilde{x}_1- \sqrt{-1} \tilde{u}_1,  \xi_2= \tilde{x}_2+ \sqrt{-1} \tilde{u}_2.
$
Representing $L$ locally by
\[
\tilde{u}_1= \tilde{u}_1(  \tilde{x}_1, \tilde{x}_2), \quad \tilde{u}_2= \tilde{u}_2(  \tilde{x}_1, \tilde{x}_2), \quad \mu= \text{constant},
\]
then (\ref{specialLagrangiancondition2}) takes the form of the \textbf{nonlinear Cauchy-Riemann equation}
\begin{equation}\label{nonlinearCauchyRiemann}
\frac{\partial \tilde{u}_1}{\partial  \tilde{x}_1 }=  \frac{\partial \tilde{u}_2}{\partial  \tilde{x}_2 }, \quad \frac{\partial \tilde{u}_1}{\partial  \tilde{x}_2 }=  -( 1+ \frac{1}{ 2A \sqrt{\mu^2+ \tilde{x}_1^2+ \tilde{u}_1^2  }  }   ) \frac{\partial \tilde{u}_1}{\partial  \tilde{x}_2 },
\end{equation}
which is very similar to the key equations in \cite{Joyce}.
Morever $L$ should asymptotically match up with fibres of (\ref{specialLagrangianfibrationnegativevertex}) at far distance from $S$, described by the affine condition (\ref{affinesolutionspecialLagrangian}).

We can use this information to speculate on the nature of \textbf{singularities} in line with Joyce \cite{Joyce}. For $\mu\neq 0$ the equations are nonsingular, so the special Lagrangians $L$ will be smooth. When $\mu=0$ the $T^3$-fibres of (\ref{specialLagrangianfibrationnegativevertex}) intersect $S$ if and only if $(y_1, y_2)$ lies in the amoeba $\text{Im}(S) $, and we expect a perturbation of such fibres to produce special Lagrangians $L$ with singularities. In the subcase where $(y_1,y_2)$ lies in the interior of the amoeba, there are two transverse intersection points $L\cap S$, at which we expect to create a pair of special Lagrangian $T^2$-cone singularities. At the boundary of the amoeba these two intersection points merge together, and the singularities disappear outside the amoeba.

The fine details of a special Lagrangian $L$ near a transverse intersection point with $S$ is conjecturally modelled by an entire solution to the nonlinear Cauchy-Riemann equation (\ref{nonlinearCauchyRiemann}) over $\R^2_{\tilde{x}_1, \tilde{x}_2}$, which has a local special Lagrangian $T^2$-cone singularity at the origin and is asymptotic to  (\ref{affinesolutionspecialLagrangian})
 at infinity. $L$ is then obtained by gluing this local picture to the corresponding fibres of (\ref{specialLagrangianfibrationnegativevertex}) away from $S$.


A salient feature of Joyce \cite{Joyce} is that the special Lagrangian fibration can fail to be defined by smooth maps (\cf Section \ref{Joycecritique}). This is compatible with this Chapter. The key point is that the absence of an a priori smooth topology forces us to work with tensors of low regularity (\cf Section \ref{StructurenearDeltanegativevertexII}), and the smooth structure along $S$ only emerges a posteriori after solving the Monge-Amp\`ere equation.
Thus one neither expects to produce a model special Lagrangian fibration defined by smooth maps, nor expects smoothness properties to persist after perturbation inside function spaces of low regularity.

\section{Incompleteness and running coupling}\label{runningcouplingnegativevertex}

The incompleteness of the Ooguri-Vafa type metric on the negative vertex has a strong analogy with the positive vertex as discussed in detail in Section \ref{runningcoupling}. The key point is that asymptotes of the first order corrections $v$ and $w^{p\bar{q}}$ lead naturally to a renormalisation flow equation, which in turn predicts the drifting of coupling constants $a_{p\bar{q}}$ over many log scales.

If we follow the discussion of Section \ref{runningcoupling}, but replace the asymptote (\ref{alphabar}) by (\ref{gammaibarbarformula}), then we find that in the following variables
\[
p_1= \sqrt{a_{2\bar{2}}}, \quad p_2= \sqrt{ a_{1\bar{1}}}, \quad p_3= \sqrt{ a_{1\bar{1}}+ a_{1\bar{2}}+ a_{2\bar{1}}+ a_{2\bar{2}} }, 
\]
the \textbf{renormalisation flow equation for the negative vertex} is given
as
\begin{equation}
\begin{cases}
\frac{d}{d\lambda}  p_1^2 = - \frac{1}{2p_2}- \frac{1}{2p_3}, 
\\
\frac{d}{d\lambda}  p_2^2 = - \frac{1}{2p_1}- \frac{1}{2p_3}, 
\\
\frac{d}{d\lambda}  p_3^2 = - \frac{1}{2p_1}- \frac{1}{2p_2},
\\
\frac{d}{d\lambda}  \text{Im}(a_{2\bar{1}} ) = 0,
\end{cases}
\end{equation}
where $\lambda$ is the log scale parameter. The rest of this Section is concerned with geometric interpretations.

The renormalisation flow equation implies that \[
\text{Im}(a_{2\bar{1}})=\text{constant}.
\]
This is compatible with the fact that $\text{Im}(a_{2\bar{1}})$ is the \emph{cohomological invariant} determined by integrating the K\"ahler form on the $T^2$-cycle.

More interestingly, the evolution of $p_1, p_2, p_3$ is formally \emph{identical} to the renormalisation flow equation (\ref{renormalisationflowequation}) for the positive vertex. This can be explained in terms of \textbf{semiflat mirror symmetry} (\cf Section \ref{Genericregion}) as follows, assuming the homological constraint $\text{Im}(a_{2\bar{1}})=0$, namely $a_{p\bar{q}}$ is a real symmetric matrix.

In general, given a semiflat SYZ fibration, the mirror SYZ fibration is obtained by replacing the torus fibres by their dual tori, interchanging the symplectic moment coordinates on the SYZ base with the complex affine coordinates on the SYZ base, and keeping the same Riemannian metric on the base. We apply this to the constant solution relevant to the positive vertex case (\cf Example \ref{Constantsolution})
\[
V^{ij}= a_{ij}, \quad W= A= \det(a_{ij}),
\]
whose SYZ base is equipped with the Euclidean metric
\[
a_{ij} d\mu_i \otimes d\mu_j + A dy^2
\]
written in the two symplectic moment coordinates $\mu_1, \mu_2$ and a complex affine coordinate $y= \text{Im}(\eta)$.
The SYZ mirror is the constant solution relevant to the negative vertex
\[
W^{i\bar{j}}= a_{ij}, \quad V= A,
\]
whose SYZ base is equipped with the Euclidean metric
\[
a_{ij} dy_i \otimes dy_j + A d\mu ^2
\]
written in the two complex affine coordinate $y_1= \text{Im}(\eta_1), y_2=\text{Im}(\eta_2)$ and a symplectic moment coordinate $\mu$. The crucial point is that mirror symmetry means the matrices $a_{ij}$ appearing in both cases are the same.

Now the Ooguri-Vafa type metrics are perturbations of some constant solution at any given log scale, and the coupling constants $a_{ij}$ drift slowly according to the renormalisation flow as the log scale changes. The formal coincidence of the renormalisation flow equations for both the positive vertex and the negative vertex agrees with semiflat mirror symmetry.

\begin{rmk}
The exlusion of the natural possibility that $\text{Im}(a_{2\bar{1}})\neq 0$ suggests that there may be generalisations of semiflat mirror symmetry to situations where special Lagrangian fibrations cannot exist (\cf Section \ref{SpecialLagrangiangeometry}).
\end{rmk}

\begin{rmk}
The positive and the negative vertices have drastically different features at refined scales: for example the positive vertex contains a fully nonlinear region modelled on the Taub-NUT type metric on $\C^3$, while the negative vertex metric is obtained by a perturbative analysis. Nonetheless they share the same renormalisation flow equation, which controls large scale behaviours. The insight is that \emph{mirror symmetry should govern metric behaviours at large scales, but not necessarily at refined scales}. In this perspective mirror symmetry owes its predicative power  to the fact that questions in algebraic or symplectic geometry are mostly insensitive to small scale metric fluctuations.
\end{rmk}

\textbf{Acknowledgement}.
The author thanks his PhD supervisor Simon Donaldson and co-supervisor Mark Haskins for their inspirations, and Song Sun for discussions.


\backmatter

\end{document}